\documentclass[a4paper,reqno]{amsart}
\renewcommand\rightmark{Canonical differential calculi via functorial geometrization}

\usepackage{quiver}

\usepackage[T1]{fontenc}
\usepackage[utf8]{inputenc}
\usepackage{lmodern}
\usepackage[english]{babel}
\usepackage[margin=1.8cm]{geometry}%
\usepackage{xcolor}
\usepackage{graphicx}
\usepackage{adjustbox}

\usepackage{amsmath,amsthm,amssymb,mathrsfs,yhmath,stmaryrd}
\usepackage{thmtools}

\usepackage{bbold}
\usepackage[bbgreekl]{mathbbol}

\usepackage{latexsym}
\usepackage{mathtools}
\usepackage{amsfonts}
\usepackage{enumerate}
\usepackage[shortlabels]{enumitem}
\usepackage[hypertexnames=false]{hyperref}
\usepackage{cleveref}
\usepackage{colonequals,bbm}
\usepackage{mdwlist}

\usepackage{tikz}
\usetikzlibrary{cd}
\usetikzlibrary{arrows}

\numberwithin{equation}{section}

\allowdisplaybreaks

\makeatletter
\renewcommand{\tocsection}[3]{%
	\indentlabel{\@ifnotempty{#2}{\bfseries\ignorespaces#1 #2\quad}}\bfseries#3}
\renewcommand{\tocsubsection}[3]{%
	\indentlabel{\@ifnotempty{#2}{\ignorespaces#1 #2\quad}}#3}

\newcommand\@dotsep{4.5}
\def\@tocline#1#2#3#4#5#6#7{\relax
	\ifnum #1>\c@tocdepth 
	\else
	\par \addpenalty\@secpenalty\addvspace{#2}%
	\begingroup \hyphenpenalty\@M
	\@ifempty{#4}{%
		\@tempdima\csname r@tocindent\number#1\endcsname\relax
	}{%
		\@tempdima#4\relax
	}%
	\parindent\z@ \leftskip#3\relax \advance\leftskip\@tempdima\relax
	\rightskip\@pnumwidth plus1em \parfillskip-\@pnumwidth
	#5\leavevmode\hskip-\@tempdima{#6}\nobreak
	\leaders\hbox{$\m@th\mkern \@dotsep mu\hbox{.}\mkern \@dotsep mu$}\hfill
	\nobreak
	\hbox to\@pnumwidth{\@tocpagenum{\ifnum#1=1\bfseries\fi#7}}\par
	\nobreak
	\endgroup
	\fi}
\AtBeginDocument{%
\expandafter\renewcommand\csname r@tocindent0\endcsname{0pt}
}
\def\l@subsection{\@tocline{2}{0pt}{2.5pc}{5pc}{}}
\makeatother

\makeatletter
\@namedef{subjclassname@2020}{\textup{2020} Mathematics Subject Classification}
\makeatother

\newcommand{\R}{\mathbb{R}}
\newcommand{\C}{\mathcal{C}}

\newcommand{\bk}{\mathbbm{k}}

\newcommand{\A}{\mathcal{A}}
\newcommand{\D}{\mathcal{D}}
\newcommand{\E}{\mathcal{E}}
\newcommand{\B}{\mathcal{B}}
\newcommand{\V}{\mathcal{V}}
\newcommand{\K}{\mathcal{K}}
\newcommand{\G}{\mathcal{G}}
\newcommand{\BB}{{\bf B}}

\newcommand{\mono}{\rightarrowtail}
\newcommand{\epi}{\twoheadrightarrow}

\DeclareMathOperator*{\colim}{co{\lim}}

\newcommand{\kalg}{\bk\operatorname{Alg}}
\newcommand{\kcalg}{\bk\operatorname{CAlg}}
\newcommand{\kvect}{\bk\textnormal{-}\!\operatorname{Vect}}
\newcommand{\Rvect}{\mathbb{R}\textnormal{-}\!\operatorname{Vect}}
\newcommand{\GrMod}{\operatorname{GrMod}}
\newcommand{\Mod}{\operatorname{Mod}}

\newcommand{\comod}{\operatorname{CMd}}
\newcommand{\Ab}{\operatorname{Ab}}

\newcommand{\Ban}{\operatorname{Ban}_{1}}
\newcommand{\Banb}{\operatorname{Ban}_{b}}
\newcommand{\Mon}{\operatorname{Mon}}
\newcommand{\Alg}{\operatorname{Alg}}
\newcommand{\Bimon}{\operatorname{BiMon}}
\newcommand{\CMon}{\operatorname{CMon}}
\newcommand{\VMon}{\Mon(\V)}
\newcommand{\CoMon}{\operatorname{CoMon}}

\newcommand{\VBimon}{\Bimon(\V)}

\newcommand{\AMod}{{}_A\!\Mod}
\newcommand{\IMod}{{}_I\!\Mod}
\newcommand{\BMod}{{}_B\!\Mod}
\newcommand{\RMod}{{}_R\!\Mod}
\newcommand{\SMod}{{}_S\!\Mod}
\newcommand{\ModA}{\Mod_A}
\newcommand{\ModI}{\Mod_I}
\newcommand{\ModB}{\Mod_B}

\newcommand{\AModA}{{}_A\!\Mod_A}

\newcommand{\Bimod}[2]{{}_{#1}\!\Mod_{#2}}
\newcommand{\Bicomod}[2]{{}^{#1}\!\Mod^{#2}}
\newcommand{\HBimod}[4]{{}^{#1}_{#3}\!\Mod^{#2}_{#4}}
\newcommand{\HBimodA}{{}^{A}_{A}\!\Mod^{A}_{A}}
\newcommand{\HBimodB}{{}^{B}_{B}\!\Mod^{B}_{B}}
\newcommand{\Leftmod}[1]{{}_{#1}\!\Mod}

\newcommand{\BModB}{{}_B\!\Mod_B}
\newcommand{\AModB}{{}_A\!\Mod_B}

\newcommand{\CalcA}{\operatorname{Calc}^1_A}
\newcommand{\CalcB}{\operatorname{Calc}^1_B}

\newcommand{\Ch}{\operatorname{CoCh}}

\newcommand{\dr}{\operatorname{dR}}

\newcommand{\EMod}{\Mod(\E)}
\newcommand{\EModE}{{}_{E_0}\!\Mod_{E_0}}
\newcommand{\EEModE}{{}_E\!\Mod_E}

\newcommand{\EAlg}{\operatorname{Calc}^\bullet}

\newcommand{\Calc}{\operatorname{Calc}^1}
\newcommand{\HCalc}{\operatorname{HCalc}^1}

\newcommand{\Sub}{\operatorname{Sub}}

\newcommand{\ass}{\operatorname{Ass}}
\newcommand{\Op}{\operatorname{Op}}

\newcommand{\Com}{\operatorname{Com}}
\newcommand{\Brd}{\operatorname{Brd}}

\newcommand{\suitable}{weakly $\V$-geometrical}
\newcommand{\verysuitable}{$\V$-geometrical}

\newcommand{\wgeomfun}{\V\operatorname{-Geom}_w}
\newcommand{\geomfun}{\V\operatorname{-Geom}}

\newcommand{\id}{\mathrm{id}}

\newcommand{\op}{\mathrm{op}}
\newcommand{\rev}{\tx{rev}}

\newcommand{\Ret}{\mathbf{Ret}}

\newcommand{\tx}{\textnormal}

\newcommand{\norm}[1]{\left\lVert{#1}\right\rVert}

\theoremstyle{definition}

\newtheorem{defi}{Definition}[subsection]
\newtheorem{eg}[defi]{Example}

\theoremstyle{plain}

\newtheorem{theo}[defi]{Theorem}
\newtheorem{prop}[defi]{Proposition}
\newtheorem{cor}[defi]{Corollary}
\newtheorem{lemma}[defi]{Lemma}
\newtheorem{notation}[defi]{Notation}

\newtheorem*{theo*}{Theorem}
\newtheorem*{prop*}{Proposition}
\newtheorem*{cor*}{Corollary}
\newtheorem*{lemma*}{Lemma}
\newtheorem*{notation*}{Notation}

\theoremstyle{remark}
\newtheorem{rmk}[defi]{Remark}

\begin{document}
\thanks{The authors thank John Bourke, Mauro Mantegazza, Reamonn Ó Buachalla, Shahn Majid, Ji\v{r}í Rosicky, Thomas Webber, and Henrik Winther for useful discussions. K. J. F. was partially supported by the DFG
priority program Geometry at Infinity SPP 2026: ME 4899/1-2. G. L. was supported by the PNRR - Young Researcher Seal of Excellence Grant.
G. T. acknowledges with gratitude the support of the EPSRC, under the grant EP/X027139/1}
\title{Canonical differential calculi via functorial geometrization}
\author{Keegan J.~Flood, Gabriele Lobbia, Giacomo Tendas}
\address{Faculty of Mathematics and Computer Science\\
  	UniDistance Suisse\\
 	 Schinerstrasse 18\\
	  3900 Brig\\
 	 Switzerland}
  \email{keegan.flood@unidistance.ch}
\address{Department of Mathematics\\
	University of Milan\\
	Via Cesare Saldini 50\\
	20133 Milan\\ 
	Italy}
  \email{lobbia.math@gmail.com}
\address{Department of Mathematics\\ 
	University of Manchester\\ 
	Alan Turing Building, Manchester\\ 
	M13 9PL, UK}
	\email{giacomo.tendas@manchester.ac.uk}

\subjclass[2020]{Primary 18M05, 16D90, 18C40, 58B34, 58B32, 16E45; Secondary 18M60, 18C35}

\begin{abstract}
Given a category $\E$, we establish sufficient conditions on a faithful isofibration $\E\rightarrow\Mon(\V)$ valued in the category of monoids internal to a monoidal additive category $\V$ such that $\E$ admits a canonical functor to the category of first order differential calculi in $\V$.
In particular, these conditions hold when $\E$ and $\V$ are locally presentable, the monoidal product is accessible and preserves finite colimits in each variable, and the
faithful isofibration is accessible and continuous.
Generalizing the procedure of extending a first order differential calculus to its maximal prolongation to this setting, we obtain a canonical functor from $\E$ to the category of differential calculi in $\V$.
This yields a simultaneous generalization of the de Rham complex on $C^{\infty}$-rings, the K\"{a}hler differentials on commutative algebras, and the universal differential calculus on associative algebras.
As a consequence, such categories $\E$ admit natural analogues of the notions of smooth map and diffeomorphism, as well as a functorial de Rham theory. 
Moreover, whenever two such faithful isofibrations to $\Mon(\V)$ factor suitably, their corresponding de Rham functors are related via a comparison map.
Developing this theory requires first extending the noncommutative geometry formalism of differential calculi from associative algebras to the setting of monoids internal to monoidal additive categories.
\end{abstract}

\maketitle
\tableofcontents

\newpage
\section{Introduction}

The respective canonical notions of differential form play a fundamental role in both differential and algebraic geometry.
For $C^\infty$-rings one works with the de Rham forms (cf.~\cite[Chapter 5.2]{joyce2019algebraic} \cite{stel2013cosimplicial}). 
Likewise, for commutative algebras one considers the Kähler differentials (cf.~\cite[Section 8]{hartshorne1977graduate}). 
In noncommutative geometry, one typically equips an associative algebra $A$ with further structure such that it can then be regarded as encoding a ``noncommutative space''. 
A general axiomatic approach is via the notion of a \emph{differential calculus} (also called \emph{exterior algebra} as in \cite{beggs2020quantum}), introduced by Woronowicz \cite{woronowicz1989differential}, which is a dg-algebra $\Omega^\bullet$ over $\Omega^0=A$ generated in degree zero via the differential and the product in the dg-algebra.
One challenge the differential calculus framework presents is that a given associative algebra typically admits many distinct first order differential calculi.
So, \emph{a priori}, in the differential calculus framework, there is no intrinsic notion of smooth map, diffeomorphism, or de Rham theory attached to a class of algebras (e.g. a subcategory of the category of $\bk$-algebras closed under isomorphism\footnote{In particular, this means that the inclusion functor is a faithful isofibration.}).

Given an associative algebra $A$, there is a canonical choice of first order differential calculus though, namely the universal first order differential calculus, obtained as the kernel of the multiplication on $A$.  
It is widely studied (see for instance \cite{cuntz1995algebra}) although it is often regarded to be of considerably less geometric interest\footnote{For instance, the de Rham cohomology for the universal differential calculus is concentrated in degree $0$ (cf.~\cite[Theorem 1.33]{beggs2020quantum}). Further, the corresponding differential operators are degenerate in the sense that every $\bk$-linear map $E\rightarrow F$ between $A$-modules $E$ and $F$ is differential operator of order at most $1$ (cf.~\cite[Proposition 4.3]{flood2025jet})} than the preceding classical commutative examples. 
Therefore, a standard practice in the noncommutative setting is to consider additional structures (bimodule relations, covariance conditions, centrality constraints, or analyticity assumptions) in order to obtain more geometrically interesting notions of differential calculus. 
Unfortunately, such an approach is, in general, insufficient to yield existence or uniqueness of a compatible notion of form, and such procedures are typically not functorial.
The foundational work on differential calculi \cite[p.~126]{woronowicz1989differential} specifically remarks upon the ``unpleasant contrast'' with the classical case which this lack of functoriality\footnote{In this quote, Woronowicz is specifically referring to the case of quantum groups.} results in.
	There are of course also functorial constructions of differential calculi, (cf.~\cite{karoubi1987homologie,ginzburg2005double}), but these are \emph{ad hoc}  in the sense that they apply to a particular type of category (e.g., algebras of the associative operad), and do not provide a general theory which depends on the categorical input data.

In this article we treat the geometry of a category $\E$ as a \emph{relative} notion: it emerges when $\E$ is viewed in relation to a category of monoids $\Mon(\V)$ via a faithful isofibration $(-)_0\colon \E\to \Mon(\V)$, where $\V$ is a monoidal additive category. 
Under suitable hypotheses on $(-)_0$ we obtain a canonical, functorial assignment of first order differential calculi to objects of $\E$, and hence, via maximal prolongation\footnote{See, e.g.,\cite[Lemma 1.32]{beggs2020quantum} and \cite[Section 4]{schauenburg1996differential} for the notion of maximal prolongation in the associative algebra setting.}, a canonical differential calculus.
This shifts the emphasis from attaching first order differential calculi to individual algebras, or even categories of algebras, to extracting differential geometric structure canonically from categorical input; our main existence theorem provides sufficient conditions. Suppose $\V$ and $\E$ are locally presentable, the monoidal product of $\V$ is accessible and preserves finite colimits in each variable, and the faithful isofibration $(-)_0\colon\E\to\Mon(\V)$ is accessible and continuous.
Then the construction yields a canonical functor $\Calc_\E\colon\E\to\Calc(\V)$
\textup{(cf.~\Cref{thm:canonical-calc} and \Cref{remark:functor-Calc_E})}.
Thus the existence of canonical differential calculi follows from structural properties of the categories and the underlying-monoid functor.
The construction recovers the classical canonical differential calculi as special cases.
For $\E=\Mon(\V)$ and $(-)_0=1_{\Mon(\V)}$ it yields the universal differential calculus (cf.~\Cref{universal-calc-left-adj});
for $\V$ braided and $\E=\CMon(\V)$, the category of braided commutative monoids, it yields (the braided generalization of) the Kähler differentials (cf.~\Cref{ComMon-canonical-calc});\footnote{For $\V=\bk$-$\Mod$, the first order part is the module of K\"ahler differentials; when $2$ is invertible, its maximal prolongation is the classical algebra of K\"ahler forms. If $2$ is not invertible then the resulting dg-algebra is graded-commutative but not necessarily strictly graded-commutative. This can be seen as motivation for the development of a higher order generalization of our construction that also yields non-maximal differential calculi.}
and for $\V=\mathbb{R}\text{-Vect}$ and $\E=C^\infty$-$\operatorname{Ring}$, the category of $C^\infty$-rings, it yields the classical de Rham complex (cf.~\Cref{eg:final-Cinfty-rings})\footnote{The de Rham complex is well known to be the maximal prolongation of the first order differential calculus of de Rham $1$-forms.}.
The construction also recovers known calculi on classes of noncommutative algebras.
The central-bimodule calculus on complex algebras considered in \cite[Proposition~2]{dubois1994derivations} is recovered as a special case of our example of algebras equipped with a chosen central subalgebra.
Its maximal prolongation agrees with the higher central forms of \cite[Section~7, Proposition~7 and Corollary~2]{dubois2001lectures}
(cf.~\Cref{dvmkc:central-calculus} and \Cref{dvmkc:central-prolongation}).
For commutator-nilpotent algebras in the sense of
\cite[Definitions~1.1.1 and~1.1.5(a)]{kapranov1998noncommutative}, over a field of characteristic zero, the canonical first order differential calculus agrees with the degree-one component of the differential calculus defined in
\cite[equation~(2)]{cortinas2003rham}.
The latter differential calculus is a natural quotient of the maximal prolongation
(cf.~\Cref{dvmkc:kapranov} and \Cref{dvmkc:prolongation}).
This motivates future work on a higher order generalization of our construction that also yields non-maximal differential calculi.
As further examples, in the braided monoidal setting, we describe the canonical calculi on a natural class of categories of monoids, including as a special case the category of dyslectic monoids. 
The construction also applies to algebras equipped with additional operations.
For $\V=\bk\text{-}\Mod$, this includes models of any Lawvere theory under the theory of unital associative $\bk$-algebras (cf.~\Cref{cor:lawvere-theory-examples}).
Examples include $C^\infty$-rings, associative Poisson algebras, and Rota--Baxter algebras (cf.~\Cref{eg:Cinfty-lawvere}, \Cref{eg:Poisson-calculus}, and \Cref{eg:rota-baxter}).
The latter two examples show that changing the additional operations on a fixed underlying algebra can significantly change its canonical first order differential calculus.

The main idea is that the geometry of $\E$ is seen to arise from a well-behaved faithful isofibration $(-)_0\colon \E\rightarrow\Mon(\V)$. 
Thus, the classification and study of such functors $(-)_0$ emerges as a new question of interest. 
Note that, in practice, one often restricts to sufficiently regular (in a suitable sense) subcategories of a given $\E$, such as smooth manifolds, smooth affine varieties, etc., and this provides a heuristic for obtaining a canonical differential calculus in the noncommutative setting as well.
Namely, one suitably modifies\footnote{This is analogous to synthetic differential geometry, where, for instance, one replaces the category of smooth manifolds with a better behaved category into which the category of smooth manifolds suitably embeds (cf.~\cite{moerdijk2013models}).} the category under consideration until it satisfies the conditions of the main results, \Cref{thm:canonical-calc} and \Cref{remark:functor-Calc_E}, sufficient to admit a canonical differential calculus.

\subsection{Some implications}
In classical differential geometry, diffeomorphism invariance is a unifying principle: from fundamental constructions such as de Rham cohomology and the Lie bracket of vector fields, to moduli theory, to physical applications such as general relativity and the Hamiltonian formalism. 
By contrast, in the noncommutative setting of an associative algebra equipped with a differential calculus, there is no evident analogue of the notion of a smooth map or diffeomorphism\footnote{The notion of diffeomorphism has been considered in some restricted noncommutative settings, e.g. \cite{majid1998quantum}.}. 
There is, of course, the notion of a differential algebra map (cf.~\cite[p.~3-4]{beggs2020quantum}), namely an algebra homomorphism $\phi\colon A\to B$ together with an $A$-bimodule map $\phi_*\colon \Omega_A\to\Omega_B$ such that the standard compatibility diagram commutes. 
But such maps need not arise from a functorial assignment of calculi to algebras, and therefore do not provide an intrinsic analogue of smooth maps. 
This obstructs the development of noncommutative counterparts of theories that classically rely on diffeomorphism invariance,
and can restrict noncommutative geometry to the study of algebras in isolation, detached from any specific geometric ‘‘flavor'' (differential, algebraic, analytic, etc.), prohibiting a Grothendieck-style perspective in which morphisms take precedence over objects.

However, given a faithful isofibration $(-)_0\colon \E\to\Mon(\V)$ admitting the canonical differential calculus construction described above, several classical geometric phenomena reappear.
First, morphisms in $\E$ then play the role of smooth maps (cf. \Cref{rmk:diffeos}) in this more general setting, providing a framework for the extension of classical results that require a notion of morphism. 
Second, automorphism groups in $\E$ serve as analogues of diffeomorphism groups and can be used to build invariants, moduli spaces, etc.
Third, there are implications for situations in which one constructs new objects from old; for instance, given $A\in\E$ and a module over $A$ (thought of as sections of a bundle over $A$), then a reasonable notion of  ‘‘algebra of functions on the total space of the module'' would yield an object in the slice category $\E/A$.

In classical differential geometry, the de Rham functor is fundamental. 
It underlies the Stokes’ theorem, certain topological invariants  and characteristic classes, as well as aspects of Chern–Weil theory, Chern–Simons theory, and index theory.
In the noncommutative setting, there already exists an object-wise notion of de Rham cohomology for a given differential calculus. 
However, the utility of the de Rham construction in the classical case lies, in part, in its functoriality: morphisms of spaces induce morphisms in cohomology, ensuring invariance under diffeomorphisms.

Our construction extends this feature to the noncommutative setting. 
Building upon our canonical first order differential calculus construction, the maximal prolongation procedure extends this first order differential calculus to a differential calculus, the composition of which yields the de Rham functor (cf.~\Cref{def:de-Rham-funct}) whose corresponding cohomology (cf.~\Cref{rmk:deRhamCohomology}) provides noncommutative homological invariants which are analogous to the topological invariants in the classical setting.
Thus, this work provides a novel route to approach a noncommutative analogue of the aspects of Chern–Weil theory, Chern–Simons theory, index theory, and related constructions which classically arise from the de Rham cohomology.
One could then explore in what sense this geometric approach connects back to the extant $K$-theoretic formulations (see for example~\cite{connes1985non}). 

Our construction also yields a way to compare related $\E$'s:
given a suitable factorization of $\E_1\rightarrow\Mon(\V)$ through $\E_2\rightarrow\Mon(\V)$, 
there is a comparison map (cf.~\Cref{rmk:comparison-of-de-Rham-functors} and \Cref{rmk:comparison-of-de-Rham-cohomology-functors}) from de Rham cohomology functor of $\E_2$, restricted to $\E_1$, to that of $\E_1$; e.g., the forgetful functor from $C^\infty$-rings to commutative  $\mathbb R$-algebras induces a comparison map from algebraic to smooth de Rham cohomology (cf.~\Cref{ex:smooth_algebraic_de_rham}).
Thus, our theory provides a setting in which to relate homological invariants in different geometric contexts. 
In principle, one could attempt to prove analogues of Grothendieck's comparison theorem (cf.~\cite[Theorem 1]{grothendieck1966rham}).

Together, these results show that several of the obstacles in extending classical differential geometric constructions to the noncommutative setting are obviated by working with a suitably well-behaved category $\E$ equipped with a faithful isofibration to $\Mon(\V)$ satisfying sufficient assumptions. 
In the setting where $\V$ is further a monoidal model category, one could consider homotopy invariance of the de Rham functors considered above as well as a generalization of the cotangent complex (cf. \cite{quillen1970co,illusie2006complexe,grothendieck1968categories}) to this setting, but that goes beyond the scope of this article and will therefore be addressed in subsequent work.

\subsection{Overview of results}
In \Cref{sec:Preliminaries} we recall the relevant background material on monoidal categories, as well as monoids, and their corresponding module categories, therein.
Our construction and main results require that we first generalize the notion of first order differential calculi to the setting of monoids internal to a monoidal additive category $\V$ and show that the standard results concerning first order differential calculi extend to this broader setting (cf.~\Cref{sec:diff-calc}).
This is of independent interest: it extends the differential calculus framework from $\bk$-algebras to monoids in monoidal additive categories.
Consequently the theory applies to $G$-equivariant algebras for $G$ a compact group, algebras graded by a monoid, dg-algebras, algebras under a fixed algebra, small algebroids, small dg-categories, etc.
	Hence our theory provides a unified framework for treating the noncommutative differential geometric generalizations of many approaches to classical, graded-commutative, higher, and relative geometry.

In \Cref{sec:ext-restr-calc} we show that a morphism of monoids induces an adjunction between the corresponding categories of first order differential calculi and, further, show how it arises from the extension-restriction of scalars adjunction.
In \Cref{subsec:calc_and_left_adjoints} we show that the left adjoint to the square-zero extension functor $\Mod(\V)\rightarrow\Mon(\V)$ factors through the forgetful functor $\Calc(\V)\rightarrow\Mod(\V)$ via the left adjoint to the projection functor $\Calc(\V)\rightarrow\Mon(\V)$.
\Cref{sect:bimon} shows that the notion of bicovariant first order differential calculi, and a few basic results about them, extend to the setting of monoidal additive categories.
Note that to apply our main results to $\Bimon(\V)$, one would need to realize them as a subcategory of $\Mon(\V)$ to obtain a canonical first order differential calculus, and only then ask whether the resulting calculus is bicovariant\footnote{This is reasonable since bimonoids, and in particular Hopf monoids, are interpreted as noncommutative analogues of Lie groups (or, more generally, symmetry encoding objects) and, by analogy with the classical case, one would expect functorial differential calculus construction on the entire category of spaces, not just on a subcategory of symmetry encoding objects therein.}.

 The category $\Mod(\E)$ of square-zero extensions in $\E$, obtained by pulling back the square-zero extension functor (cf.~\Cref{square_zero}) along $(-)_0$ at the level of retractions, in order to remember the zero section and the projection (cf.~\Cref{pullback_construction}), can then naturally be regarded as the category of bimodules over objects in $\E$  .  
In \Cref{def:suitable} we introduce the notion of \emph{\suitable} functor, which is enough to imply (cf. \Cref{F'-gives-calculus1})  that a left adjoint $L'$ (if it exists) to the  forgetful functor $U'\colon\Mod(\E)\to\E$ necessarily factors through the category of generalized first order differential calculus objects internal to $\Mod(\E)$; under these hypotheses the unit of the adjunction encodes the corresponding differential.
In \Cref{F'-gives-calculus2} we then establish further conditions such that $L'$ factors through the category of first order differential calculus objects internal to $\Mod({\E})$ (cf.~\Cref{def:first-ord-diff-calc-E}). 
\Cref{inducing-F'1} then gives sufficient conditions (primarily local presentability and accessibility) on $\V$, $\E$, and $(-)_0$ to yield the existence of the left adjoints necessary for the preceding results.
\Cref{thm:canonical-calc} and \Cref{remark:functor-Calc_E} then show when these (generalized) first order differential calculus objects internal to $\E$ induce first order differential calculi in the usual sense.
More precisely, 
we obtain a functor $\E\rightarrow \Calc(\V)$ suitably factoring via a left adjoint through $\Mod(\E)$.
In \Cref{sec:alg-closure} we apply the construction to polynomial identities, Horn theories, and chosen central subalgebras, and identify the resulting calculi in the Dubois-Violette--Michor
and Kapranov--Corti\~nas examples.
In \Cref{subsec:operads} we show that for suitable $\V$, the theory simplifies considerably when $\E\rightarrow\Mon(\V)$ is induced by an operad under the associative operad (\Cref{cor:operad-examples}) or, for $\V=\bk\text{-}\Mod$, by a Lawvere theory under the theory of unital associative $\bk$-algebras (\Cref{cor:lawvere-theory-examples}).

In \Cref{sec:extalg} we extend our treatment of first order differential calculi to differential calculi, and show that the maximal prolongation procedure is left adjoint to the truncation $\EAlg(\V)\rightarrow \Calc(\V)$, as in the standard noncommutative setting \cite[Lemma 1.32]{beggs2020quantum}. 
We conclude in \Cref{subsec:deRham} with a treatment of the de Rham functors, their cohomology, and induced comparison maps between them arising from our theory.

\subsection{Alternative approaches}\label{alternative}		
Another possible setting in which (generalized) differential calculi have been studied, is that of K\"ahler and (co)differential categories, see~\cite{blute2011kahler,blute2015derivations,ONeill:master-thesis,ONeill:phd-thesis}. Indeed, starting from a symmetric monoidal category $\V$ enriched over commutative monoids and endowed with an {\em algebra modality} (a monad on $\V$ with additional structure), they can define notions of derivations and modules of K\"ahler differentials. 

 	We take a different approach mainly for two reasons. On the one hand, we want to go beyond the commutative case, defining our notions over a monoidal additive category $\V$ (that need not be symmetric), and by taking modules over not-necessarily-commutative monoids; for this the setting of~\cite{blute2015derivations} does not seem suitable, as symmetry and commutativity are assumed from the beginning. 
	Some steps in this direction have been made in \cite[Section~7]{ONeill:master-thesis}  (mainly the definition of \emph{$A$-derivation} for a monoid $A$; that is,  \Cref{def:enriched-1st-diff-calc} without the surjectivity condition) and \cite[Section~3.2]{ONeill:phd-thesis}. 
 	On the other hand, in order to make this paper accessible to non-experts in category theory, we attempt to use the minimal categorical structure possible. 
 	
 	Our construction is closely related to the notion of Beck modules (cf.~\cite{beck1967triples,quillen1970co}).
When $\E$ has pullbacks preserved by $(-)_0$, an $\E$-bimodule determines a Beck module precisely when the canonical fibrewise addition and inversion of its underlying square-zero extension lift to morphisms in $\E$.
This holds in the classical algebraic and smooth examples, but need not hold in general: for associative Poisson algebras the canonical calculus need not be a Beck module (cf.~\Cref{rmk:Poisson-not-Beck}).	 

\subsection{Summary}
This work applies category theory and categorical algebra to address a long-standing gap in the noncommutative differential geometry literature. 
It provides sufficient conditions under which a category may be regarded as one of noncommutative differential-geometric objects, and develops a functorial procedure for realizing them as such.
An important consequence is to reveal the limitations of the traditional approach of working with a single algebra equipped with a (first order) differential calculus.
While it is of interest to explore how far one can go using only the data of a (first order) differential calculus, the geometric flavor of classical settings fundamentally depends on the functorial assignment of differential calculi to algebras. 
The framework developed here provides the tools to reproduce this dependence in the noncommutative setting.
Thus, we see that the search for a canonical differential calculus can be treated via the choice of a suitable faithful isofibration to $\Mon(\V)$, rather than algebra-wise, or even category-wise.

\section{Preliminaries: Monoids and modules in a monoidal category}\label{sec:Preliminaries}

In this section we will introduce the relevant background material, including monoidal categories, and monoids and modules in them. 
We conclude the section (cf. \Cref{sec:ext-restr-calc}) with a treatment of the extension-restriction of scalars adjunction for bimodules in a monoidal category. 

\subsection{Monoidal categories}

In this subsection we recall the definitions of monoidal category (and related concepts) and fix our notation.

\begin{defi}\label{defi:moncat}
A \emph{monoidal category} $\V=(\V,\otimes,I,\alpha,\lambda,\rho)$ consists of the following data:
\begin{enumerate} 
	\item A category $\V$.
	\item A functor $\otimes\colon \V\times\V\longrightarrow\V$ termed the monoidal (or tensor) product.
	\item An object $I\in\V$ termed the unit. 
	\item A natural isomorphism termed the associator with components $\alpha_{X,Y,Z}\colon (X\otimes Y)\otimes Z \longrightarrow X\otimes(Y\otimes Z)$. 
	\item A natural isomorphism termed the left unitor with components $\lambda_X\colon I\otimes X\longrightarrow X$. 
	\item A natural isomorphism termed the right unitor with components $\rho_X\colon X\otimes I\longrightarrow X$. 
\end{enumerate}
Further, the following coherence diagrams must commute.
\begin{center}
\scalebox{0.875}{
\begin{tikzcd}[ampersand replacement=\&]
	\& {(W\otimes (X\otimes Y))\otimes Z} \\
	{((W\otimes X)\otimes Y)\otimes Z} \&\& {W\otimes ((X\otimes Y)\otimes Z)} \\
	{(W\otimes X)\otimes(Y\otimes Z)} \&\& {(W\otimes (X\otimes(Y\otimes Z)))}
	\arrow["{\alpha_{W,X\otimes Y,Z}}"{pos=0.7}, from=1-2, to=2-3]
	\arrow["{\alpha_{W,X,Y}\otimes 1_Z}"{pos=0.3}, from=2-1, to=1-2]
	\arrow["{\alpha_{W\otimes X,Y,Z}}"', from=2-1, to=3-1]
	\arrow["{1_W\otimes\alpha_{X,Y,Z}}", from=2-3, to=3-3]
	\arrow["{\alpha_{W,X,Y\otimes Z}}"', from=3-1, to=3-3]
\end{tikzcd}}
\scalebox{0.875}{
\begin{tikzcd}[ampersand replacement=\&]
	{(X\otimes I)\otimes Y} \&\& {X\otimes (I\otimes Y)} \\
	\& {X\otimes Y}
	\arrow["{\alpha_{X,I,Y}}", from=1-1, to=1-3]
	\arrow["{\rho_X\otimes 1_Y}"'{pos=0.4}, from=1-1, to=2-2]
	\arrow["{1_X\otimes\lambda_Y}"{pos=0.4}, from=1-3, to=2-2]
\end{tikzcd}}
\end{center}
\end{defi}

\begin{defi}
	A monoidal category $\V=(\V,\otimes,I,\alpha,\lambda,\rho)$ is called:
	\begin{itemize}
		\item \emph{left closed} if, for any $A\in\V$, the functor $A\otimes-\colon\V\to\V$ has a right adjoint $[A,-]_l$ that is, there are natural isomorphisms $\V(A\otimes X,Y)\cong\V(X,[A,Y]_l)$;
		
		\item \emph{right closed} if, for any $A\in\V$, the functor $-\otimes A\colon\V\to\V$ has a right adjoint $[A,-]_r$; that is, there are natural isomorphisms $\V(X\otimes A,Y)\cong\V(X,[A,Y]_r)$;

	\item \emph{biclosed} if it is both left and right closed.
	\end{itemize}
\end{defi}

\begin{defi}
	(\cite[Definition~2.1]{JoyalStreet:brad-mon-cat}) A \emph{braided monoidal category}  is a monoidal category $\V$ equipped with a \emph{braiding}; that is, for any $X,Y\in\V$, natural isomorphisms $\beta_{X,Y}\colon X\otimes Y\to Y\otimes X$ satisfying the following axioms. 
	\begin{center}
	\begin{tikzcd}[ampersand replacement=\&]
		{(X\otimes Y)\otimes Z} \& {X\otimes (Y\otimes Z)} \\
		{(Y\otimes X)\otimes Z} \& {(Y\otimes Z)\otimes X} \\
		{Y\otimes (X\otimes Z)} \& {Y\otimes (Z\otimes X)}
		\arrow["{\alpha_{X,Y,Z}}", from=1-1, to=1-2]
		\arrow["{\beta_{X,Y}\otimes 1_Z}"', from=1-1, to=2-1]
		\arrow["{\beta_{X,Y\otimes Z}}", from=1-2, to=2-2]
		\arrow["{\alpha_{Y,X,Z}}"', from=2-1, to=3-1]
		\arrow["{\alpha_{Y,Z,X}}", from=2-2, to=3-2]
		\arrow["{1_Y\otimes\beta_{X,Z}}"', from=3-1, to=3-2]
	\end{tikzcd}
		\hspace{0.5cm}
		\begin{tikzcd}[ampersand replacement=\&]
			{X\otimes (Y\otimes Z)} \& {(X\otimes Y)\otimes Z} \\
			{X\otimes(Z\otimes Y)} \& {Z\otimes (X\otimes Y)} \\
			{(X\otimes Z)\otimes Y} \& {(Z\otimes X)\otimes Y}
			\arrow["{\alpha_{X,Y,Z}^{-1}}", from=1-1, to=1-2]
			\arrow["{1_X\otimes\beta_{Y,Z}}"', from=1-1, to=2-1]
			\arrow["{\beta_{X\otimes Y,Z}}", from=1-2, to=2-2]
			\arrow["{\alpha_{X,Z,Y}^{-1}}"', from=2-1, to=3-1]
			\arrow["{\alpha_{Z,X,Y}^{-1}}", from=2-2, to=3-2]
			\arrow["{\beta_{X,Z}\otimes 1_Y}"', from=3-1, to=3-2]
		\end{tikzcd}
	\end{center}
	A braiding is called a \emph{symmetry} if, for any $X,Y\in\V$, $\beta_{X,Y}^{-1}=\beta_{Y,X}$. 
	In this case the two axioms above are equivalent. 
We call a monoidal category equipped with a symmetry a \emph{symmetric monoidal category}. 
\end{defi}

\subsection{Monoids in monoidal categories}
There are various equivalent ways of encoding the classical notion of unital associative $\bk$-algebra over a commutative ring $\bk$. 
In particular, one can realize a unital associative $\bk$-algebra $A$ as monoid internal to the (symmetric) monoidal category $(\Mod,\otimes_{\bk})$ of $\bk$-modules with monoidal product given by the tensor product of $\bk$-modules.
One benefit of this perspective is that we can readily generalize from algebras internal to $\Mod$ to algebras internal to a monoidal category $\V$.  
\begin{defi}\label{defi:monoid}
Given a monoidal category $\V=(\V,\otimes,I,\alpha,\lambda,\rho)$, a \emph{monoid} $A=(A,m,i)$ in $\V$ consists of the following data:
\begin{enumerate}
	\item An object $A\in\V$. 
	\item A morphism $m\colon A\otimes A\to A$ termed the \emph{multiplication}. 
	\item A morphism $i\colon I\to A$ termed the unit. 
\end{enumerate}
Further, the following diagrams must commute.
\begin{center}
\[\begin{tikzcd}
	{(A\otimes A)\otimes A} && {A\otimes A} && {A\otimes I} & {A\otimes A} & {I\otimes A} \\
	{A\otimes (A\otimes A)} & {A\otimes A} & A &&& A
	\arrow["{m\otimes 1_A}", from=1-1, to=1-3]
	\arrow["\alpha"', from=1-1, to=2-1]
	\arrow["m", from=1-3, to=2-3]
	\arrow["{1_A\otimes i}", from=1-5, to=1-6]
	\arrow["{\rho_A}"', from=1-5, to=2-6]
	\arrow["m"{description}, from=1-6, to=2-6]
	\arrow["{i\otimes 1_A}"', from=1-7, to=1-6]
	\arrow["{\lambda_A}", from=1-7, to=2-6]
	\arrow["{1_A\otimes m}"', from=2-1, to=2-2]
	\arrow["m"', from=2-2, to=2-3]
\end{tikzcd}\]
\end{center} 
\end{defi}

We will see that several aspects of the theory of monoids and modules over a monoidal category can be captured using enriched category theory. We briefly recall the definition of enriched category below.

\begin{rmk}
	Given a monoidal category $\V$ as above, we can consider categories enriched over $\V$, or $\V$-categories~(cf. \cite{Kel82:libro}), which consists of
	\begin{itemize}
		\item a collection of objects $\C_0$;
		\item for any $X,Y\in\C_0$, an hom-object $\C(X,Y)\in\V$;
		\item for any $X,Y,Z\in\C_0$, a composition morphism $-\circ-\colon\C(Y,Z)\otimes\C(X,Y)\to\C(X,Z)$ in $\V$;
		\item for any $X\in\C_0$, a morphism $1_X\colon I\to\C(X,X)\in\V$;
	\end{itemize}
	satisfying axioms making $-\circ-$ associative and $1_X$ identities (for the precise definition and examples see \cite[Section~1.2]{Kel82:libro}). 
\end{rmk}

\begin{prop}
	\label{monoid-V-cat}
	Let $\V$ be a monoidal category.
	Then the following are equivalent.
	\begin{enumerate} 
		\item A monoid $(A,m,i)$ in $\V$.
		\item A $1$-object $\V$-category $\BB A$.
	\end{enumerate}
\end{prop} 
	\begin{proof}
		If $\C$ is a $\V$-category with one object, denoted by $*$, then the hom-object $A:=\C(*,*)$ is a monoid in $\V$: the multiplication $m$ is defined to be the composition map $\C(*,*)\otimes \C(*,*)\to \C(*,*)$, while the unit $i$ is the identity $1_*\colon I\to \C(*,*)$.
		The associativity and unit rule hold for $(A,m,i)$ as they correspond to the fact that composition is associative and preserves the identity. 
		Conversely, if $(A,m,i)$ is a monoid in $\V$ it easily follows that $\textnormal{Ob}(\BB A)=\{*\}$ and $\BB A(*,*):=A$, with composition and identity maps defined by $m$ and $i$ respectively, define a $\V$-category $\BB A$.
	\end{proof}

A morphism between monoids in $\V$ is a morphism that preserves the multiplication and the unit. 
In detail, a morphism of monoids $f\colon(A,m_A,i_A)\to(B,m_B,i_B)$ is morphism $f\colon A\to B$ in $\V$ that makes the following diagrams
\begin{center}
	\begin{tikzpicture}[baseline=(current  bounding  box.south), scale=2]
		
		\node (a0) at (0,0.7) {$A\otimes A$};
		\node (b0) at (0,0) {$B\otimes B$};
		\node (c0) at (1,0.7) {$A$};
		\node (d0) at (1,0) {$B$};
		
		\node (a'0) at (2,0.7) {$I$};
		\node (b'0) at (2.8,0) {$B$};
		\node (c'0) at (2.8,0.7) {$A$};
		
		\path[font=\scriptsize]
		
		(a0) edge [->] node [left] {$f\otimes f$} (b0)
		(a0) edge [->] node [above] {$m_A$} (c0)
		(b0) edge [->] node [below] {$m_B$} (d0)
		(c0) edge [->] node [right] {$f$} (d0)
		
		(a'0) edge [->] node [below] {$i_B$} (b'0)
		(a'0) edge [->] node [above] {$i_A$} (c'0)
		(c'0) edge [->] node [right] {$f$} (b'0);
	\end{tikzpicture}	
\end{center}
commute. 
It is straightforward to see that we then obtain a category $\VMon$ whose objects are monoids in $\V$ and whose arrows are morphisms of monoids. This category is studied for instance in~\cite[Section~15]{street_2007}.

\begin{rmk}
	Following \Cref{monoid-V-cat}, to give a morphism of monoids $f\colon(A,m_A,i_A)\to(B,m_B,i_B)$ is the same as giving a $\V$-functor $F\colon \BB A\to \BB B$ between the $\V$-categories induced by the monoids $A$ and $B$. 
	Indeed, such a $\V$-functor is completely determined by assigning $F(*)=*$ and by giving a map 
	$$ f\colon A=\BB A(*,*)\longrightarrow B=\BB B(*,*) $$
	in $\V$ preserving composition and identity, which translate into the preservation of multiplication and unit of the monoids. 
	As a consequence $\VMon$ can equivalently be described as the full subcategory of $\V\tx{-}\mathbf{Cat}$ spanned by the $\V$-categories with a single object. 
\end{rmk}

\begin{eg}
	\label{eg:unit-monoid}
	The unit $I$ of the monoidal structure defines a monoid $(I,\lambda_I,1_I)$; the monoid axioms are implied by those defining the monoidal structure on $\V$. 
	Moreover, for any monoid $(A,m,i)$ in $\V$, the unit map $i\colon I\to A$ defines a morphism of monoids $i\colon (I,\lambda_I,1_I)\to (A,m,i)$.
\end{eg}

\begin{eg}\label{MonoidInChain}
	Given a monoidal category $\V$ which is an additive category, we can consider the additive category $\Ch(\V)$ of non-negatively graded cochain complexes in $\V$: an object $A\in\Ch(\V)$ is a family $A=(A^n)_{n\geq 0}$ together with maps $d_A^n\colon A^n\to A^{n+1}$ such that $d_A^{n+1}d_A^n=0$. One often writes $A=\bigoplus_{n\geq 0}A^n$ fo indicate the cochain complex; that should be understood only as notation, not as an infinite direct sum in $\V$ (which may not exist).
	
	Note that this can be seen the category of $\mathbf{Ab}$-enriched functors $[\mathbb N,\V]$, where $\mathbb N$ is the $\mathbf{Ab}$-category with objects the natural numbers and homs: $\mathbb N(n,m)=\mathbb Z$ for $m=n,n+1$ and $0$ otherwise (with trivial composition maps).
	
	From this it follows that $\Ch(\V)=[\mathbb N,\V]$ inherits a tensor product from $\V$ defined as:
	$$ (A\otimes B)^n= \bigoplus_{i+j=n} A^i\otimes B^j$$
	in $\V$, with differential 
	$$d_{A\otimes B}^n= \bigoplus_{i+j=n} d_A^i\otimes 1_{B^j}+ (-1)^i 1_{A^i}\otimes d_B^j.$$
	The unit of this tensor product is the cochain complex $\bar I$ with $I$ in degree $0$, and $0$ in all other degrees.
	
	Then, unpacking the definitions, a monoid in $\Ch(\V)$ amounts to a cochain complex $(A=\bigoplus_{n\geq 0}A^n,d)$ together with\begin{enumerate}
		\item a unit $i\colon I\to A^0$ such that $d(i)=0$;
		\item an operator $\wedge^i_j\colon A^i\otimes A^j\to A^{i+j}$ which is associative, respects the unit $i$, and satisfies 
		$$ d^{i+j}(\wedge^i_j)= \wedge^{i+1}_j(d^i\otimes  1_{A^j})+ (-1)^i \wedge^i_{j+1}(1_{A^i}\otimes d^j). $$
	\end{enumerate}
	Indeed, condition (i) follows from the fact that $i\colon \bar I\to A$ needs to be a map of cochain complexes. 
	Similarly, the various $\wedge^i_j$ need to be induced by a morphism of cochain complexes
	$$ \wedge \colon A\otimes A\to A,$$
	which by definition needs to satisfy $d\circ \wedge = \wedge \circ d_{A\otimes A}$. Following the definition of $d_{A\otimes A}$ this translates into (ii).
	
	Such a structure is commonly known as a \emph{dg-algebra}.
\end{eg}

In the definition below, we can see how to recover the notion of commutative associative algebras (and similar commutative notions) when the considered monoidal category $\V$ is braided. 
 
 \begin{defi}
	\label{def:comm-mon}
 Let $\V$ be a braided monoidal category. 
 Given $n\in\mathbb{Z}$, we define an \emph{$n$-commutative monoid} in $\V$ to be a monoid $(A,m,i)\in\Mon(\V)$ satisfying the following commutativity axiom 
 \[\begin{tikzcd}[ampersand replacement=\&]
 	{A\otimes A} \&\& {A\otimes A} \\
 	\& A
 	\arrow["{\beta^n_{A,A}}", from=1-1, to=1-3]
 	\arrow["m"', from=1-1, to=2-2]
 	\arrow["m", from=1-3, to=2-2]
 \end{tikzcd}\]
 where $\beta^n_{A,A}$ denotes the composite of $\beta_{A,A}$ with itself $n$-times if $n$ is positive, or the composite of $\beta_{A,A}^{-1}$ with itself $(-n)$-times if $n$ is negative. We define $\CMon_n(\V)$ as the full subcategory of $\VMon$ whose objects are $n$-commutative monoids. 
\end{defi}
\begin{rmk}
If $n=0$, then $\beta^0_{A,A}=\id_{A\otimes A}$, hence $\CMon_0(\V)= \Mon(\V)$.
If $n=1$, these are simply known as \emph{commutative monoids} and we write $\CMon(\V)\colonequals \CMon_1(\V)$.
 If $n=2$, these are known as \emph{dyslectic monoids} (cf.~\cite{pareigis1995braiding}).
If $\V$ is symmetric, then for any $n\in\mathbb{N}$, $\CMon(\V)=\CMon_{2n+1}(\V)$ and $\Mon(\V)=\CMon_{2n}(\V)$.
\end{rmk}
\begin{rmk}\label{rmk:comm-mon-full-subcat}
It is straightforward to show that $\CMon_n(\V)$ is equivalent to $\CMon_{-n}(\V)$, for any $n\in\mathbb{Z}$.
Then, given $q\in\mathbb{Z}$, one can show that $\CMon_n(\V)$ is a full subcategory of $\CMon_{qn}(\V)$.
It follows that $\CMon_n(\V)$ is a full subcategory of $\CMon_0(\V)=\Mon(\V)$, for any $n\in\mathbb{Z}$.
\end{rmk} 

\begin{eg}
\begin{enumerate}[(i)]
	\item[]
	\item For any braided monoidal category, the monoid structure on the monoidal unit $(I,\lambda_I,1_I)$ described in \Cref{eg:unit-monoid} is commutative. 
	
	\item For the symmetric monoidal category $\kvect$ of $\bk$-vector spaces, we have $\CMon(\kvect)=\kcalg$ the category of commutative algebras.

	\item Following \Cref{MonoidInChain}, for $\V$ braided, a commutative monoid in $\Ch(\V)$ is a commutative dg-algebra. 
	
\end{enumerate}
\end{eg}
 
\subsection{Module categories associated to a monoid}

Just as one defines modules for ring in classical algebra, one can also extend this notion to modules over monoids in a monoidal category. 
We recall the preliminary definitions, indicate the relevant references and provide the results which we will require in the subsequent sections.

\begin{defi}
	Let $(A,m,i)$ be a monoid in $\V$. 
	A {\em left module over $A$}, or {\em left $A$-module}, consists of an object $M\in \V$ together with a morphism $\mu\colon A\otimes M\to M$ in $\V$ making the diagrams below commute.
	\begin{center}
		\begin{tikzpicture}[baseline=(current  bounding  box.south), scale=2]
			
			\node (a0) at (-1.5,0.7) {$(A\otimes A)\otimes M$};
			\node (a1) at (0,0.7) {$A\otimes (A\otimes M)$};
			\node (b0) at (-1.5,0) {$A\otimes M$};
			\node (c0) at (1.5,0.7) {$A\otimes M$};
			\node (d0) at (1.5,0) {$M$};
			
			\node (a'0) at (2.8,0.7) {$I\otimes M$};
			\node (b'0) at (4,0) {$M$};
			\node (c'0) at (4,0.7) {$A\otimes M$};
			
			\path[font=\scriptsize]
			
			(a0) edge [->] node [left] {$m\otimes 1_M$} (b0)
			(a0) edge [->] node [above] {$\alpha$} (a1)
			(a1) edge [->] node [above] {$1_A\otimes \mu$} (c0)
			(b0) edge [->] node [below] {$\mu$} (d0)
			(c0) edge [->] node [right] {$\mu$} (d0)
			
			(a'0) edge [->] node [below] {$\lambda$} (b'0)
			(a'0) edge [->] node [above] {$i\otimes 1_M$} (c'0)
			(c'0) edge [->] node [right] {$\mu$} (b'0);
		\end{tikzpicture}	
	\end{center}
	A morphism of left $A$-modules $f\colon (M,\mu)\to (N,\nu)$ consists of a map $f\colon M\to N$ in $\V$ making the following diagram 
	\begin{center}
		\begin{tikzpicture}[baseline=(current  bounding  box.south), scale=2]
			
			\node (a0) at (0,0.7) {$A\otimes M$};
			\node (b0) at (0,0) {$A\otimes N$};
			\node (c0) at (1,0.7) {$M$};
			\node (d0) at (1,0) {$N$};

			\path[font=\scriptsize]
			
			(a0) edge [->] node [left] {$1_A\otimes f$} (b0)
			(a0) edge [->] node [above] {$\mu$} (c0)
			(b0) edge [->] node [below] {$\nu$} (d0)
			(c0) edge [->] node [right] {$f$} (d0);
		\end{tikzpicture}	
	\end{center}
	commute. Left $A$-modules and their morphisms form a category $\AMod$.
\end{defi}

\begin{rmk}\label{leftmodules-V-funct}
	If the monoidal category $\V$ is  {\em right closed}, meaning that every functor $-\otimes A$ has a right adjoint  $[A,-]:=[A,-]_r$, then left $A$-modules can be equivalently described as certain $\V$-functors.
	
	First notice that $\V$ itself is a $\V$-category (that we denote again by $\V$) with hom-objects $\V(A,B):=[A,B]$, see~\cite[Section~1.6]{Kel82:libro}. Then, the data of a left $A$-module is equivalent to that of a $\V$-functor ${\bf M} \colon\BB A\to \V$.
	Indeed, $\bf M$ is completely determined by the object $M:=\bf M(*)$ in $\V$ and the map between hom-objects
	$$ A=\BB A(*,*)\longrightarrow [{\bf M}(*),{\bf M}(*)]=[M,M], $$
	which, by the adjunction, corresponds to a morphism $\mu\colon A\otimes M\to M$ in $\V$. 
	Then the axioms expressing that $\bf M$ is a $\V$-functor translate into those saying that $(M,\mu)$ is a left $A$-module.
	
	Similarly, to give a morphism of left $A$-modules $f\colon (M,\mu)\to (N,\nu)$ is the same as giving a $\V$-natural transformation $\bf M\Rightarrow \bf N$ between the induced $\V$-functors. 
	It follows that
	$$ \AMod \cong [\BB A,\V] $$
	where $[\BB A,\V]$ is the category of $\V$-functors $\BB A\to \V$ and $\V$-natural transformations between them.
\end{rmk}

The notion of {\em right module over $A$}, or {\em right $A$-module}, is completely analogous and can be obtained by writing all the instances of $A$ in the definition above on the right-hand-side of $M$ rather thank on the left. 
Thus, a right $A$-module consists of an object $M\in \V$ together with a map $\mu\colon M\otimes A\to M$ making the corresponding squares commute.

\begin{rmk}\label{rightmodules-V-funct}
	Given a monoidal category $\V=(\V,\otimes,\alpha,\rho,\lambda,I)$ we can define another monoidal structure on $\V$ by taking
	$$ \V^{\rev}:=(\V,\widehat\otimes,\alpha^{-1},\lambda,\rho,I) $$
	where we define $A\ \widehat{\otimes}\ B := B\otimes A$; that is, $\widehat\otimes\cong \otimes\circ \sigma\colon \V\times \V\to\V$, with $\sigma\colon\V\times\V\to \V\times\V$ switching the components. 
	
	Then, monoids in $\V$ are the same as monoids in $\V^{\rev}$ (since they involve only tensoring with the same objects), so that $\VMon\cong\tx{Mon}( \V^{\rev})$.
	While, to give a right module over $A$, seen as a monoid in $\V$, is the same as to give left module over $A$, seen as a monoid in $\V^{\rev}$.
	
	As a consequence of Remark~\ref{leftmodules-V-funct}, if the monoidal category $\V$ is  left closed (meaning that $\V^{\rev}$ is  right closed), then 
	$$ \ModA \cong [\BB A,\V^{\rev}] $$
	where $[\BB A,\V^{\rev}]$ is the category of $\V^{\rev}$-functors $\BB A\to \V^{\rev}$ and $\V^{\rev}$-natural transformations between them.
\end{rmk}

\begin{eg}
	Given a monoid $(A,m,i)$, the pair $(A,m)$ is both a left and a right $A$-module.
	Indeed, the module axioms translate into the associativity and unity axioms defining the monoid structure.
\end{eg}

\begin{rmk}\label{I-mod}
	Consider the monoid $(I,\lambda,1_I)$ induced by the unit of the monoidal structure on $\V$. 
	Then, for any left $I$-module, the unit axiom of the definition forces the morphism $\mu\colon I\otimes M\to M$ to be the isomorphism $\lambda$. It follows that a left $I$-module is simply an object of $\V$, so that $\IMod\cong \V$.
	The same argument shows that also $\ModI\cong \V$.
\end{rmk}

\begin{defi}\label{bimod}
	Let $(A,m_A,i_A)$ and $(B,m_B,i_B)$ be monoids in $\V$. 
	An {\em $A$-$B$ bimodule} consists of an object $M\in \V$ together with morphisms $\mu\colon A\otimes M\to M$ and $\nu\colon M\otimes B\to M$ for which:\begin{enumerate}
		\item $(M,\mu)$ is a left $A$-module;
		\item $(M,\nu)$ is a right $B$-module;
		\item the following diagram, expressing associativity between $\mu$ and $\nu$,
			\begin{center}
				\begin{tikzpicture}[baseline=(current  bounding  box.south), scale=2]
					
					\node (a0) at (-1.5,0.7) {$(A\otimes M)\otimes B$};
					\node (a1) at (0,0.7) {$A\otimes (M\otimes B)$};
					\node (b0) at (-1.5,0) {$M\otimes B$};
					\node (c0) at (1.5,0.7) {$A\otimes M$};
					\node (d0) at (1.5,0) {$M$};
					
					\path[font=\scriptsize]
					
					(a0) edge [->] node [left] {$\mu\otimes 1_B$} (b0)
					(a0) edge [->] node [above] {$\alpha$} (a1)
					(a1) edge [->] node [above] {$1_A\otimes \nu$} (c0)
					(b0) edge [->] node [below] {$\nu$} (d0)
					(c0) edge [->] node [right] {$\mu$} (d0);
				\end{tikzpicture}	
			\end{center}
		commutes.
	\end{enumerate}
	
	A morphism of $A$-$B$ bimodules $f\colon (M,\mu^M,\nu^M)\to (N,\mu^N,\nu^N)$ consists of a map $f\colon M\to N$ in $\V$ which defines a morphism between the underlying left $A$-modules and right $B$-modules. 
	The category of $A$-$B$ bimodules and morphisms between them is denoted by $\AModB$.
\end{defi}

\begin{eg}\label{f-bimod}
	Given a morphism of monoids $f\colon (A,m_A,i_A)\to (B,m_B,i_B)$, then the triple
	$$ (B,m_B\colon B\otimes B\to B, \nu_B\colon B\otimes A\xrightarrow{1_B\otimes f}B\otimes B\xrightarrow{m_B} B) $$
	defines a $B$-$A$ bimodule. 
	Indeed, we know that $(B,m_B)$ is a left $B$-module, while the properties (ii) and (iii) of the definition above hold since $f$ is a morphism of monoids.\\
	Symmetrically, the triple 
	$$ (B, \mu_B\colon A\otimes B\xrightarrow{f\otimes 1_B}B\otimes B\xrightarrow{m_B} B,m_B\colon B\otimes B\to B) $$
	defines an $A$-$B$ bimodule. 
\end{eg}

\begin{rmk}
	Following Remark~\ref{I-mod}, we obtain the isomorphisms of categories
	$${}_A\!\Mod_I\cong \AMod,\hspace{40pt}    {}_I\!\Mod_A\cong \ModA, \hspace{40pt} {}_I\!\Mod_I\cong \V.$$
	This is because the commutativity of (iii) in Definition~\ref{bimod} above becomes redundant if $\mu=\lambda\colon I\otimes M\to M$ or $\nu=\rho\colon M\otimes I\to M$.
\end{rmk}

\begin{rmk}
	Assume that $\V$ is a biclosed monoidal category with equalizers, so that in particular we can apply Remarks~\ref{leftmodules-V-funct} and~\ref{rightmodules-V-funct}.
	Then the category $\ModB\cong[\BB B,\V^{\rev}]$ is actually a $\V$-category, with hom-objects defined by the equalizers
	{\begin{center}
		\begin{tikzpicture}[baseline=(current  bounding  box.south), scale=2]
			
			\node (0) at (-3.7,0) {$\ModB(M,N)$};
			\node (z) at (-2.2,0) {$[M,N]_r$};	
			\node (b) at (0,0) {$[M\otimes B,N]_r$;};

			\path[font=\scriptsize]

			(0) edge [>->] node [below] {} (z)
			
			([yshift=1.5pt]z.east) edge [->] node [above] {$[\nu^M, N]$} ([yshift=1.5pt]b.west)
			([yshift=-1.5pt]z.east) edge [->] node [below] {$ [M,\nu^N]\circ (-\otimes B)$} ([yshift=-1.5pt]b.west);
		\end{tikzpicture}
	\end{center}}
	see for instance \cite[Section~3]{gordon1999gabriel}. 
	Then, unwrapping the definitions, one obtains 
	$$ \AModB \cong [\BB A, \ModB]$$
	expressing the category of $A$-$B$ bimodules as that of $\V$-functors $\BB A\to \ModB$ and $\V$-natural transformations between them.
	
	If $\V$ has all finite limits and the tensor product is symmetric, then $\AModB$ inherits a structure of $\V$-category.
	This can be seen thanks to the isomorphism above and following \cite[Section~2.2]{Kel82:libro}. 
	In this case the hom-objects of $\AModB$ are given by the pullback
	{\begin{center}
		\begin{tikzpicture}[baseline=(current  bounding  box.south), scale=2]
			
			\node (a0) at (0,0.8) {$\AModB(M,N)$};
			\node (a0') at (0.3,0.6) {$\lrcorner$};
			\node (b0) at (1.6,0.8) {$\AMod(M,N)$};
			\node (c0) at (0,0) {$\ModB(M,N)$};
			\node (d0) at (1.6,0) {$[M,N]_l\cong[M,N]_r,$};
			
			\path[font=\scriptsize]
			
			(a0) edge [->] node [above] {} (b0)
			(a0) edge [->] node [left] {} (c0)
			(b0) edge [>->] node [right] {} (d0)
			(c0) edge [>->] node [below] {} (d0);
		\end{tikzpicture}	
	\end{center}} 
	where  $\AMod(M,N)\rightarrowtail[M,N]_l$ is defined by an equalizer similar to that used to define  $\ModB(M,N)\rightarrowtail[M,N]_r$, and $[M,N]_l\cong[M,N]_r$ since $\V$ is symmetric.
\end{rmk}

\begin{rmk}
	\label{rmk:bimodules-as-modules-if-V-symm}
	In classical algebra, when we have a consider algebras over commutative rings, we can see bimodules as modules. 
	Analogously, if we consider $\V$ a symmetric monoidal category\footnote{We are not aware whether this holds for $\V$ only braided instead of symmetric.}, then bimodules in $\V$ can be seen as particular left (or right) modules.
	In order to see this we will start with a duality construction.
	
	If $\V$ is braided\footnote{Notice that for this duality construction we need $\V$ to be only \textit{braided}, while to see bimodules as modules we will need it to be symmetric.} (with braiding isomorphism $\beta_{A,B}\colon A\otimes B\to B\otimes A$), given any monoid $A=(A,m_A,i_A)\in\Mon(\V)$ we can define a new monoid structure $A^\op$ on $A$ using the symmetry, i.e. $A^\op:=(A,m_A\circ \beta_{A,A},i_A)$. The unit axioms for $A^\op$ follow from naturality of $\beta$ and the axioms for $A$ (switched). 
	For the associativity we use again the naturality of $\beta$, the associativity of $A$ and coherence of the braiding. 
	A clarifying example is when we consider a ring $R$ (i.e. a monoid in the symmetric monoidal category of abelian groups):  %
$R^\op$ is the ring  with the multiplication reversed, i.e. $a\cdot_\op b:=b\cdot a$. 
	
	Then, if $\V$ is moreover symmetric, we can see that $\AModB\cong \Leftmod{B^\op\otimes A}$, where the monoid $B^\op\otimes A$ is defined through the symmetry and the two multiplications. 
	This can be proven directly with lengthy calculations, but assuming $\V$ is moreover closed and has equalizers, we can give the following abstract proof.
	\begin{enumerate}[1.]
		\item Notice that, by definition $\BB (B^\op)=(\BB B)^\op$, where the latter is the dual $\V$-category (where composition is reversed using the symmetry). 
		\item Then, we can see that $\ModA\cong \Leftmod{A^\op}$, since using \Cref{leftmodules-V-funct} we get the following chain of equivalences. 
		{$$\ModA\cong [\BB A,\V^\rev]\cong[(\BB A)^\op,\V]=[\BB (A^\op),\V]= \Leftmod{A^\op}$$}
		\item The source categories appearing below have one object, so the hom-objects of the enriched functor categories are defined by equalizers in $\V$. The enriched currying isomorphism therefore applies in this case (cf.\ \cite[Section~2.3]{Kel82:libro}), and gives
		$$\AModB\cong [\BB A,\ModB]\cong[\,\BB A,[\BB(B^\op),\V]\,]\cong[\,\BB(B^\op)\otimes\BB A,\V\,]\cong[\,\BB(B^\op\otimes A),\V\,]\cong\Leftmod{B^\op\otimes A}.$$
		Here $\otimes$ denotes the tensor product of $\V$-categories; the symmetry interchanges its two factors in the currying isomorphism. The second to last equivalence follows from $\BB(B^\op)\otimes\BB A\cong\BB(B^\op\otimes A)$.
	\end{enumerate}
\end{rmk}

The following proposition follows from \cite[Section~11]{shulman2008framed}, see also \cite[Definition~4.1.9]{brandenburg2014tensor}.

\begin{prop}\label{tensor-bimodules}
	Let $\V$ be a monoidal category with coequalizers that are preserved by the tensor product. 
	For any triple of monoids $A$, $B$, and $C$ in $\V$ there is a functor
	$$ -\otimes _B-\colon \AModB \times {}_B\!\Mod_C \longrightarrow {}_A\!\Mod_C$$
	defined by sending an $A$-$B$ bimodule $(M,\mu^M,\nu^M)$ and a $B$-$C$ bimodule $(N,\mu^N,\nu^N)$ to the $A$-$C$ bimodule $M\otimes_B N$ defined as the coequalizer
	\begin{center}
		\begin{tikzpicture}[baseline=(current  bounding  box.south), scale=2]
			
			\node (z) at (-2,0) {$(M\otimes B)\otimes N$};	
			\node (b) at (0,0) {$M\otimes N$};
			\node (b1) at (1.3,0) {$M\otimes_B N$};

			\path[font=\scriptsize]

			(b) edge [->>] node [below] {} (b1)
			
			([yshift=1.5pt]z.east) edge [->] node [above] {$\nu^M\otimes 1_N$} ([yshift=1.5pt]b.west)
			([yshift=-1.5pt]z.east) edge [->] node [below] {$(1_M\otimes \mu^N )\circ \alpha$} ([yshift=-1.5pt]b.west);
		\end{tikzpicture}
	\end{center}
in $\V$. The left and right actions are obtained by using the universal property of coequalizers, the fact that these are preserved by the tensor products, and the left and right actions of $M$ and $N$ respectively.
\end{prop}

\begin{rmk}\label{biclosed-preserves}
	If $\V$ is biclosed, then the hypotheses on the tensor product are automatically satisfied: the functors $A\otimes-$ and $-\otimes B$ preserve any colimits that exist in $\V$ (being left adjoints to the internal homs).
\end{rmk}

As a consequence we can define a monoidal structure on $\AModA$; this is a special case of \cite[Theorem~11.5]{shulman2008framed}.

\begin{prop}
	\label{prop:otimes_A-preserves-coequ}
	Let $\V$ be a monoidal category with coequalizers that are preserved by the tensor product, and $A=(A,m,i)$ be a monoid in $\V$. 
	Then $\AModA$ has a monoidal structure which is defined by
	$$ -\otimes _A-\colon \AMod_A \times {}_A\!\Mod_A \longrightarrow {}_A\!\Mod_A$$
	and whose unit is the $A$-bimodule $(A,m,m)$. 
	Moreover, $\AModA$ has coequalizers and $\otimes_A$ preserves them in each variable.\\
	If in addition $\V$ is left (resp. right) closed and has equalizers, then also $\AModA$ is left (resp. right) closed and has equalizers.
\end{prop}

Next we study what properties $\AModB$ inherits from $\V$.

\begin{prop}\label{forgetful}
	Let $A$ and $B$ be monoids in $\V$, and let $U\colon \AModB\to\V$ be the forgetful functor taking a bimodule to its underlying object. 
	Then:\begin{enumerate}
		\item $U$ is a monadic right adjoint;
		\item the left adjoint $F\colon \V\to \AModB$ sends an object $M$ to the $A$-$B$ bimodule 
		$$(A\otimes M\otimes B, m_A\otimes 1_M\otimes 1_B, 1_A\otimes 1_M\otimes m_B);$$
		\item $\AModB$ has, and $U$ preserves, any shape of colimits that $\V$ has and that are preserved by the tensor product.
	\end{enumerate}
\end{prop}
\begin{proof}
	Points (i) and (ii) follow from the standard fact asserting that the functor 
	$$ T:=(A\otimes -)\otimes B\colon \V\longrightarrow \V,$$ 
	together with the units and multiplications of $A$ and $B$, induces a monad on $\V$ whose algebras are precisely the $A$-$B$ bimodules. On one hand, if $(M,t\colon TM\to M)$ is a $T$-algebra, then we can define maps
	$$ \mu\colon A\otimes M\cong (A\otimes M)\otimes I\xrightarrow{(1_A\otimes 1_M)\otimes i_B} TM\xrightarrow{\ t\ } M $$
	and 
	$$ \nu\colon M\otimes B\cong( I\otimes M)\otimes B\xrightarrow{(i_A\otimes 1_M)\otimes 1_B} TM\xrightarrow{\ t\ } M $$
	that make $(M,\mu,\nu)$ into a bimodule.
	Conversely, if $(M,\mu,\nu)$ is a bimodule then the map $t\colon TM\to M$, defined by taking the diagonal of the commutative square in Definition~\ref{bimod}(iii), makes $M$ into a $T$-algebra.
	
	Point (iii) is now a consequence of \cite[Proposition~4.3.2]{Bor94:libro}. 
\end{proof}

The hypotheses on the tensor product of $\V$ below are satisfied whenever the monoidal structure is biclosed, see Remark~\ref{biclosed-preserves}.

\begin{cor}\label{abelian-modules}
	Let $\V$ be an abelian monoidal category for which the tensor product functors preserve cokernels and finite direct sums. 
	Then for any monoids $A$ and $B$ in $\V$, the category $\AModB$ is abelian and the forgetful functor $U\colon \AModB\to \V$ is an exact functor.
\end{cor}
\begin{proof}
	Recall that a category is called abelian if and only if (\cite[Definition~1.4.1]{Bor94:libro}):\begin{enumerate}
		\item it has a zero object as well as finite products and coproducts;
		\item it has kernels and cokernels;
		\item every monomorphism is a kernel, and every epimorphism is a cokernel.
	\end{enumerate}
	By Proposition~\ref{forgetful} above, it follows that $\AModB$ has, and $U$ preserves, all the limits and colimits listed in points (i) and (ii). 
	
	To show the validity of (iii) we use that in addition $U$ reflects isomorphisms (being monadic). Indeed, let $m\colon M\to N$ be a monomorphism in $\AModB$ and let $k\colon K\to N$ be the kernel of the cokernel of $m$; we need to prove that the map $f\colon M\to K$, induced by the universal property of the limit, is an isomorphism.
	Since $U$ preserves kernels and cokernels, and $\V$ is abelian, it follows that $Uf$ is an isomorphism in $\V$. Thus $f$ itself is an isomorphism in $\AModB$ since $U$ reflects isomorphisms. 
	The same argument applies to the case of epimorphisms.
\end{proof}

\begin{cor}
	\label{cor:otimes_A-preserves-epimorp}
	Let $\V$ be a monoidal category with coequalizers that are preserved by the tensor product, and $A=(A,m,i)$ be a monoid in $\V$. 
	Then the tensor product functors 	
	$$-\otimes_A M, M\otimes_A-\colon \AMod_A\longrightarrow {}_A\!\Mod_A$$ 
	preserve any shape of colimits that exist in $\V$ and are preserved by $\otimes$ in each variable.\\
	In addition, if $\V$ is additive and so is $\otimes$, then $\otimes_A$ preserves all existing finite colimits in each variable. 
\end{cor}
\begin{proof}
	Let $\C$ be a small category such that $\C$-colimits exist in $\V$ and are preserved in each variable by $\otimes$. Then, by Proposition~\ref{forgetful}(iii), $\C$-colimits exist in $\AMod_A$ and are preserved by the forgetful functor 
	$$ U\colon \AMod_A\longrightarrow \V.$$
	Since $U$ is also conservative, to show that $M\otimes_A-$ preserves $\C$-colimits it is enough to prove that $U(M\otimes_A -)\colon\AMod_A\to \V$ preserves them. Given a diagram $H\colon\C\to\AMod_A$, note that $U(M\otimes_A \colim H)$ is defined as the coequalizer
	\begin{center}
		\begin{tikzpicture}[baseline=(current  bounding  box.south), scale=2]
			
			\node (z) at (-2.4,0) {$(M\otimes A)\otimes U(\colim H)$};	
			\node (b) at (0.2,0) {$M\otimes U(\colim H)$};
			\node (b1) at (2.1,0) {$U(M\otimes_A \colim H)$};

			\path[font=\scriptsize]

			(b) edge [->>] node [below] {} (b1)
			
			([yshift=1.5pt]z.east) edge [->] node [above] {$\nu^M\otimes 1_{U(\colim H)}$} ([yshift=1.5pt]b.west)
			([yshift=-1.5pt]z.east) edge [->] node [below] {$(1_M\otimes \mu )\circ \alpha$} ([yshift=-1.5pt]b.west);
		\end{tikzpicture}
	\end{center}
	in $\V$, by Proposition~\ref{tensor-bimodules}. But $U$, $(M\otimes A)\otimes-$, and $M\otimes-$ preserve $\C$-colimits, and coequalizers commute with any colimits. Thus $U(M\otimes_A \colim H)$ coincides with the colimit of the coequalizers
	\begin{center}
		\begin{tikzpicture}[baseline=(current  bounding  box.south), scale=2]
			
			\node (z) at (-2.1,0) {$(M\otimes A)\otimes UHc$};	
			\node (b) at (0,0) {$M\otimes UHc$};
			\node (b1) at (1.4,0) {$M\otimes_A UHc$.};

			\path[font=\scriptsize]

			(b) edge [->>] node [below] {} (b1)
			
			([yshift=1.5pt]z.east) edge [->] node [above] {$\nu^M\otimes 1_{UHc}$} ([yshift=1.5pt]b.west)
			([yshift=-1.5pt]z.east) edge [->] node [below] {$(1_M\otimes \mu )\circ \alpha$} ([yshift=-1.5pt]b.west);
		\end{tikzpicture}
	\end{center}
	taken for $c\in\C$. In other words
	$$ U(M\otimes_A \colim H)\cong \colim M\otimes_A UH- $$
	concluding the first part of the statement. The same arguments apply to $-\otimes_A M$.\\
	If $\V$ is additive and in addition $\otimes$ preserves finite direct sums, then $\V$ has all finite colimits and $\otimes$-preserves them in each variable (since finite coproducts and coequalizers generate all finite colimits). Thus the result follows from the previous part.
\end{proof}

\begin{rmk}\label{Otimes_A cancels A}
	Consider $(A\otimes A, m\otimes 1_A, 1_A\otimes m)$ as an $A$-bimodule. Then for any $M\in\AModA$ we have an isomorphism
	$$ M\otimes_A(A\otimes A)\cong M\otimes A $$
	in $\V$. Indeed, the coequalizer defining $M\otimes_A(A\otimes A)$ is given by 
	\begin{center}
		\begin{tikzpicture}[baseline=(current  bounding  box.south), scale=2]
			
			\node (z) at (-2.5,0) {$M\otimes A\otimes A\otimes A$};	
			\node (b) at (0,0) {$M\otimes A\otimes A$};
			\node (b1) at (1.9,0) {$M\otimes A$.};

			\path[font=\scriptsize]

			(b) edge [->>] node [above] {$\nu^M\otimes 1_A$} (b1)
			
			([yshift=1.5pt]z.east) edge [->] node [above] {$\nu^M\otimes 1_{A\otimes A}$} ([yshift=1.5pt]b.west)
			([yshift=-1.5pt]z.east) edge [->] node [below] {$1_M\otimes m\otimes 1_A$} ([yshift=-1.5pt]b.west);
		\end{tikzpicture}
	\end{center}
	This is actually a split coequalizer, with splitting $1_M\otimes i\otimes 1_A\colon M\otimes A\cong M\otimes (I\otimes A)\to M\otimes (A\otimes A)$.
\end{rmk}

\subsection{The extension and restriction of scalars adjunction}
\label{sec:ext-restr-scalars}

In this section we will describe the generalization of the extension and restriction to scalars adjunction for bimodules over monoids in a monoidal category. 
We recall that given two rings $R$ and $S$ and a ring homomorphism $f\colon R\to S$, we get the following adjunction
\[\begin{tikzcd}[ampersand replacement=\&]
	\SMod \& \RMod
	\arrow[""{name=0, anchor=center, inner sep=0}, "{f_!}"', curve={height=6pt}, from=1-2, to=1-1]
	\arrow[""{name=1, anchor=center, inner sep=0}, "{f^\ast}"', curve={height=6pt}, from=1-1, to=1-2]
	\arrow["\bot"{description}, draw=none, from=1, to=0]
\end{tikzcd}\]
where $f_!$ is the extension of scalars and $f^\ast$ the restriction of scalars. 
Note that the explicit description of the extension of scalars is $f_!=S\otimes_R-$. 

In the monoidal setting we can do something similar, starting with a monoidal category $\V$, two monoids in it $A$ and $B$ together with a monoid morphism $f\colon A\to B$. Moreover, we will see how, if we are given a second monoid morphism $f'\colon A'\to B'$, this can be generalized to an adjunction between $\Bimod{A}{A'}$ and $\Bimod{B}{B'}$. 
Clearly, if $\V$ is symmetric and closed then the adjunction between bimodules follows directly by the left (or right) modules one using \Cref{rmk:bimodules-as-modules-if-V-symm}. 
That said, as shown in this section, we can obtain the bimodules adjunction without assuming $\V$ symmetric. 

First we treat the restriction of scalars; the version for left modules (i.e. the case $A'=B'=I$ and $f'=1_I$) can also be found in \cite[Chapter~15, Page~103]{street_2007}. 

\begin{prop}\label{prop:def-restr-of-scalars}
	For any pair of monoid morphisms $(A,m_A,i_A)\xrightarrow{f}(B,m_B,i_B)$ and $(A',m_{A'},i_{A'})\xrightarrow{f'}(B',m_{B'},i_{B'})$ in $\Mon(\V)$, there is an induced functor
	$$(f,f')^\ast\colon\Bimod{B}{B'}\to\Bimod{A}{A'}$$
	which we call \emph{restriction of scalars}. 
\end{prop}

\begin{proof} 
	Let $(N,B\otimes N\xrightarrow{\nu} N,N\otimes B'\xrightarrow{\mu}N)$ be a $B$-$B'$ bimodule. 
	We define $(f,f')^\ast(N,\nu,\mu)$ as $N$ itself equipped with the following $A$-$A'$ actions. 
	\begin{center}
	\begin{tikzcd}[ampersand replacement=\&]
		{A\otimes N} \& {B\otimes N} \& N
		\arrow["{f\otimes 1_N}", from=1-1, to=1-2]
		\arrow["\nu", from=1-2, to=1-3]
	\end{tikzcd}
	\hspace{0.5cm}
\begin{tikzcd}[ampersand replacement=\&]
	{N\otimes A'} \& {N\otimes B'} \& N
	\arrow["{1_N\otimes f'}", from=1-1, to=1-2]
	\arrow["\mu", from=1-2, to=1-3]
\end{tikzcd}
	\end{center}
	
	The associativity and unital axioms for each actions follow from the ones for $N$ plus the fact that $f$ and $f'$ are monoid morphisms. 
	Moreover, the diagram below show that the left and right actions are compatible with each other. 
	\begin{center}
	\scalebox{0.9}{
	\begin{tikzcd}[ampersand replacement=\&]
		{(A\otimes N)\otimes A'} \&\& {A\otimes (N\otimes A')} \&\& {A\otimes (N\otimes B')} \& {A\otimes N} \\
		\&\& {B\otimes(N\otimes A')} \\
		{(B\otimes N)\otimes A'} \&\& {(B\otimes N)\otimes B'} \&\& {B\otimes (N\otimes B')} \& {B\otimes N} \\
		{N\otimes A'} \&\& {N\otimes B'} \&\&\& N
		\arrow["\alpha", from=1-1, to=1-3]
		\arrow["{(f\otimes 1_N)\otimes 1_{A'}}"', from=1-1, to=3-1]
		\arrow["{1_A\otimes(1_N\otimes f')}", from=1-3, to=1-5]
		\arrow["{f\otimes1_{N\otimes A'}}"{description}, from=1-3, to=2-3]
		\arrow["{1_A\otimes\mu}", from=1-5, to=1-6]
		\arrow["{f\otimes1_{N\otimes B'}}"{description}, from=1-5, to=3-5]
		\arrow["{f\otimes 1_N}", from=1-6, to=3-6]
		\arrow["{1_B\otimes (1_N\otimes f')}"{description}, from=2-3, to=3-5]
		\arrow["\alpha"{description}, from=3-1, to=2-3]
		\arrow["{1_{B\otimes N}\otimes f'}"', from=3-1, to=3-3]
		\arrow["{\nu\otimes 1_{A'}}"', from=3-1, to=4-1]
		\arrow["\alpha"', from=3-3, to=3-5]
		\arrow["{\nu\otimes 1_{B'}}"{description}, from=3-3, to=4-3]
		\arrow["{1_B\otimes\mu}"', from=3-5, to=3-6]
		\arrow["\nu", from=3-6, to=4-6]
		\arrow["{1_N\otimes f'}"', from=4-1, to=4-3]
		\arrow["\mu"', from=4-3, to=4-6]
	\end{tikzcd}
	}
	\end{center}
	
	The definition of $(f,f')^\ast$ on morphisms is the identity. The only thing to check is that if $t\colon (N,\nu,\mu)\to(N',\nu',\mu')$ is a morphism in $\Bimod{B}{B'}$, then $t\colon(N,\nu\circ f\otimes N,\mu\circ N\otimes f')\to(N',\nu'\circ f\otimes N',\mu'\circ N\otimes f')$ is a morphism in $\Bimod{A}{A'}$. This is true because the four squares below commute 
	\begin{center}
\begin{tikzcd}[ampersand replacement=\&]
	{A\otimes N} \& {B\otimes N} \& N \\
	{A\otimes N'} \& {B\otimes N'} \& {N'}
	\arrow["{f\otimes 1_N}", from=1-1, to=1-2]
	\arrow["{1_A\otimes t}"', from=1-1, to=2-1]
	\arrow["\nu", from=1-2, to=1-3]
	\arrow["{1_B\otimes t}"{description}, from=1-2, to=2-2]
	\arrow["t", from=1-3, to=2-3]
	\arrow["{f\otimes 1_{N'}}"', from=2-1, to=2-2]
	\arrow["{\nu'}"', from=2-2, to=2-3]
\end{tikzcd}
	\hspace{0.5cm}
	\begin{tikzcd}[ampersand replacement=\&]
		{N\otimes A'} \& {N\otimes B'} \& N \\
		{N'\otimes A'} \& {N'\otimes B'} \& {N'}
		\arrow["{1_N\otimes f'}", from=1-1, to=1-2]
		\arrow["{t\otimes 1_{A'}}"', from=1-1, to=2-1]
		\arrow["\mu", from=1-2, to=1-3]
		\arrow["{t\otimes 1_{B'}}"{description}, from=1-2, to=2-2]
		\arrow["t", from=1-3, to=2-3]
		\arrow["{1_{N'}\otimes f'}"', from=2-1, to=2-2]
		\arrow["{\mu'}"', from=2-2, to=2-3]
	\end{tikzcd}
	\end{center}
	by functoriality of $\otimes$ and the requirement that $t\in\Bimod{B}{B'}$. 
\end{proof}

Other standard facts about restriction of scalars are listed below. By a {\em continuous} functor we mean one that preserves all existing limits.

\begin{prop}\label{restriction-properties}
	The following hold for any monoid morphisms $f\colon A\to B$ and $f'\colon A'\to B'$:
	\begin{enumerate}
		\item $(f,f')^\ast\colon\Bimod{B}{B'}\to\Bimod{A}{A'}$ is continuous, faithful, and conservative (that is, reflects isomorphisms);
		\item if $\V$ has coequalizers that are preserved by the tensor product, then $(f,f')^\ast(-)\cong (B\otimes_B-)\otimes_{B'}B'$, where, following \Cref{f-bimod}, we are seeing $B$ as an $A$-$B$ bimodule (via $f$) and $B'$ as a $B'$-$A'$ bimodule (via $f'$);
		\item if $\V$ has finite colimits and $\otimes$ preserves them in each variable, then  $(f,f')^\ast$ also preserves them;
		\item if $\V$ is abelian and the tensor product functors preserve finite colimits, then $(f,f')^\ast$ is an exact functor between abelian categories. 
	\end{enumerate}
\end{prop}

\begin{proof}
\begin{enumerate}[(i)]
	\item[]
	\item Consider the forgetful functors $U_{A,A'}\colon\Bimod{A}{A'}\to\V$ and  $U_{B,B'}\colon\Bimod{B}{B'}\to\V$, then by construction the triangle below commutes.
	\begin{center}
		\begin{tikzpicture}[baseline=(current  bounding  box.south), scale=2]
			
			\node (a0) at (0,0.7) {$\Bimod{B}{B'}$};
			\node (b0) at (1.2,0.7) {$\Bimod{A}{A'}$};
			\node (c0) at (0.6,0) {$\V$};
			
			\path[font=\scriptsize]
			
			(a0) edge [->] node [above] {$(f,f')^\ast$} (b0)
			(a0) edge [->] node [left] {$U_{B,B'}$} (c0)
			(b0) edge [->] node [right,xshift=0.1cm] {$U_{A,A'}$} (c0);
		\end{tikzpicture}	
	\end{center} 
	By Proposition~\ref{forgetful}, the functors $U_{A,A'}$ and $U_{B,B'}$ are both monadic, and hence in particular continuous, faithful, and reflect isomorphisms. This is enough to imply that also $(f,f')^\ast$ has these properties. 
	
	\item This follows easily by the definition of the tensor product of bimodules (Proposition~\ref{tensor-bimodules}).
	Indeed, if we see $B$ as an $A$-$B$ bimodule (with left $A$-action induced by restricting along $f$) it follows that $B\otimes_B M\cong M$ for any $M\in\BMod$ (see Remark~\ref{Otimes_A cancels A}). Moreover the left $A$-action on $B\otimes_B M$ corresponds to that defined on $(f,f')^\ast M$. 
	Then, similarly we can see $(B\otimes_B M)\otimes_{B'}B'\cong M\otimes_{B'}B'\cong M$ with $A$-$A'$ actions the same as $(f,f')^\ast M$. 
	
	\item  We consider the same commutative triangle as in (i) knowing moreover that $U_{A,A'}$ and $U_{B,B'}$ both preserve finite colimits by Proposition~\ref{forgetful}. It follows that $(f,f')^\ast$ also preserves them. 
	
	\item This is now a direct consequence of point (iii). 
\end{enumerate}

\end{proof}

In the classical case the extension of scalars can be described as \emph{tensoring} with the codomain $B$. We get a similar description also in our setting; see also~\cite[Remark~4.1.22]{brandenburg2014tensor}.

\begin{prop}\label{prop:def-ext-of-scalars}
	Let $\V$ be a monoidal category with coequalizers that are preserved by the tensor product.
	For any pair of monoid morphisms $(A,m_A,i_A)\xrightarrow{f}(B,m_B,i_B)$ and $(A',m_{A'},i_{A'})\xrightarrow{f'}(B',m_{B'},i_{B'})$ in $\Mon(\V)$, we can define the functor
	$$(f,f')_!:=(B\otimes_A-)\otimes_{A'}B'\colon\Bimod{A}{A'}\to\Bimod{B}{B'},$$
	called \emph{extension of scalars}, which is left adjoint to $(f,f')^\ast$. 
\end{prop}
\begin{proof}(Sketch)
	Following Example~\ref{f-bimod}, the morphism $f$ induces a structure of $B$-$A$ bimodule on $B$, while the morphism $f'$ induces a structure of $A'$- $B'$ bimodule on $B'$. 
	Thus, using Proposition~\ref{tensor-bimodules} twice, the functor
	$$ (f,f')_!:=(B\otimes_A-)\otimes_{A'}B'\colon\Bimod{A}{A'}\longrightarrow\Bimod{B}{B'}$$
	is well-defined.
	Now, we want to prove that this is a left adjoint to $(f,f')^\ast$.
	In order to do so, we first notice that, by definition,  %
	$(f,f')_!=(1_B,f')_!\circ(f,1_{A'})_!$ and $(f,f')^\ast=(f,1_{A'})^\ast\circ(1_B,f')^\ast$. 
	Then, we can show that we get two adjunctions as below, which give the desired one $(f,f')_!\dashv(f,f')^\ast$ as composite. 
	\[\begin{tikzcd}[ampersand replacement=\&, column sep=large]
		{\Bimod{A}{A'}} \&\& {\Bimod{B}{A'}} \&\& {\Bimod{B}{B'}}
		\arrow[""{name=0, anchor=center, inner sep=0}, "{(f,1_{A'})_!}", curve={height=-6pt}, from=1-1, to=1-3]
		\arrow[""{name=1, anchor=center, inner sep=0}, "{(f,1_{A'})^\ast}", curve={height=-6pt}, from=1-3, to=1-1]
		\arrow[""{name=2, anchor=center, inner sep=0}, "{(1_B,f')_!}", curve={height=-6pt}, from=1-3, to=1-5]
		\arrow[""{name=3, anchor=center, inner sep=0}, "{(1_B,f')^\ast}", curve={height=-6pt}, from=1-5, to=1-3]
		\arrow["{(f,f')_!}", curve={height=-40pt}, from=1-1, to=1-5]
		\arrow["{(f,f')^\ast}", curve={height=-40pt}, from=1-5, to=1-1]
		\arrow["\bot"{description}, draw=none, from=0, to=1]
		\arrow["\bot"{description}, draw=none, from=2, to=3]
	\end{tikzcd}\]
	 
	 Both adjunctions are obtained in a similar way, hence we will show only the leftmost one. For convenience, we will write $f_!=(f,1_{A'})_!$ and $f^\ast=(f,1_{A'})^\ast$. 
	 The unit of the adjunction $\eta_M\colon M\to f^\ast f_!M$, for $M\in\Bimod{A}{A'}$, is given by the composite
	 $$ \eta_M\colon M\xrightarrow{\lambda_M^{-1}} I\otimes M\xrightarrow{i_B\otimes 1_M} B\otimes M \xrightarrow{\ q\ }B\otimes_A M=f^\ast f_!M,$$
	 where $q$ is the map defining $B\otimes_A M$ as a quotient of $B\otimes M$ (see Proposition~\ref{tensor-bimodules}).  
	  Conversely, the counit of the adjunction $\epsilon_N\colon f_! f^\ast N\to N$, for $N\in\Bimod{B}{A'}$, is the map
	 $$ \epsilon_N\colon f_! f^\ast N=B\otimes_A f^\ast N=B\otimes_A N \longrightarrow N $$
	 induced by the universal property of the coequalizer defining $B\otimes_A N$ applied to the left $B$-action $\nu_N\colon B\otimes N\to N$ (this coequalizes the pair of maps since $\nu_N$ distributes over $m_B$). \footnote{The proof that the co/unit described give an adjunction is completely analogous to the left modules extension-restriction of scalars adjunction described in, for example, \cite[Chapter~15, Page~103]{street_2007}.}
\end{proof}

\begin{rmk}
	We can give explicit descriptions of the unit and counit of the extension-restriction of scalars adjunction for generic pairs of morphisms $(f,f')$.
	\begin{itemize}
		\item The unit $\eta_M\colon M\to (B\otimes_AM)\otimes_{A'}B'=:(f,f')^\ast (f,f')_!M$, for $M\in\Bimod{A}{A'}$, is given by the composite
		$$ \eta_M\colon M\cong (I\otimes M)\otimes I\xrightarrow{(i_B\otimes 1_M)\otimes i_{B'}} (B\otimes M)\otimes B' \xrightarrow{\ q\otimes 1_{B'} }(B\otimes_A M)\otimes B'\xrightarrow{q'}(B\otimes_A M)\otimes_{A'} B',$$
		where $q$ is the map defining $B\otimes_A M$ as a quotient of $B\otimes M$ and similarly $q'$ for $-\otimes_{A'}-$  (see Proposition~\ref{tensor-bimodules}). 
		\item For the counit we first define the map $\widetilde{\epsilon_N}\colon B\otimes_A N \longrightarrow N$ as the map induced by the universal property of the coequalizer defining $B\otimes_A N$ applied to the left action $\nu_N\colon B\otimes N\to N$ (this coequalizes the pair of maps since $\nu_N$ distributes over $m_B$). 
		Notice that this is the counit of the adjunction $(f,1)_!\dashv(f,1)^\ast$.  
		Then, similarly, using $\widetilde{\epsilon_N}$ and the right $B'$-action of $B\otimes_AN$ we get the counit 
		$$\epsilon_N\colon (B\otimes_AN)\otimes_{A'}B'\to N$$ 
		as the map induced by the universal property of $(B\otimes_AN)\otimes_{A'}B'$ as a coequalizer. 
	\end{itemize}
\end{rmk}

\begin{rmk}
	The adjunction $(f,f')_!\dashv (f,f')^\ast$ induces a monad $T:= (f,f')^\ast (f,f')_!\colon \Bimod{A}{A'}\to\Bimod{A}{A'}$ with unit $\eta$ as defined in the proof of \Cref{prop:def-ext-of-scalars} and multiplication:
	{$$ T^2M= [\,B\otimes_A[\,(B\otimes_A M)\otimes_{A'}B'\,]\,]\otimes_{A'}B'\xrightarrow{(f,f')^\ast(\epsilon_{(f,f')_!M})}(B\otimes_A M)\otimes_{A'}B'=TM. $$} 
	Similarly, we have a comonad $S:= (f,f')_! (f,f')^\ast\colon \Bimod{B}{B'}\to\Bimod{B}{B'}$ whose counit is $\epsilon$ (again defined in the proof of \Cref{prop:def-ext-of-scalars}) and whose comultiplication is defined by
	$$ SN= (B\otimes_A (f,f')^\ast N)\otimes_{A'}B'\xrightarrow{(f,f')_!(\eta_{(f,f')^\ast N})} [\,B\otimes_A[\,(B\otimes_A (f,f')^\ast N)\otimes_{A'}B'\,]\,]\otimes_{A'}B'=S^2N. $$
\end{rmk}

\begin{rmk}
	When $A'=B'=I$ and $f'=1_I$, we recover the classical extension-restriction of scalars adjunction for left modules, see for instance \cite[Chapter~15, Page~103]{street_2007}. Similarly we get the one for right modules setting $A=B=I$ and $f=1_I$. 
	\begin{center}
	\begin{tikzpicture}[baseline=(current  bounding  box.south), scale=2]

		\node (f) at (0,0.4) {$\BMod$};
		\node (g) at (1.5,0.4) {$\AMod$};
		\node (h) at (0.75,0.45) {$\bot$};
		
		\path[font=\scriptsize]

		([yshift=-1.3pt]f.east) edge [->] node [below] {$f^\ast=(f,1_I)^\ast$} ([yshift=-1.3pt]g.west)
		([yshift=1.5pt]f.east) edge [bend left,<-] node [above] {$f_!=(f,1_I)_!$} ([yshift=1.5pt]g.west);
	\end{tikzpicture}
	\hspace{1cm}
	\begin{tikzpicture}[baseline=(current  bounding  box.south), scale=2]

		\node (f) at (0,0.4) {$\ModB$};
		\node (g) at (1.5,0.4) {$\ModA$};
		\node (h) at (0.75,0.45) {$\bot$};
		
		\path[font=\scriptsize]

		([yshift=-1.3pt]f.east) edge [->] node [below] {$f'^\ast=(1_I,f')^\ast$} ([yshift=-1.3pt]g.west)
		([yshift=1.5pt]f.east) edge [bend left,<-] node [above] {$f'_!=(1_I,f')_!$} ([yshift=1.5pt]g.west);
	\end{tikzpicture}
	\end{center}

	Moreover, setting $A=A'=B'=I$, $f=i_B\colon I\to B$ and $f'=1_I$, we recover the free-forgetful adjunction between $\V={}_I\!\Mod$ and $\BMod$. Similarly, we also get the free-forgetful adjunction with $\ModB$. 
	\begin{center}
		
		\begin{tikzpicture}[baseline=(current  bounding  box.south), scale=2]

			\node (f) at (0,0.4) {$\BMod$};
			\node (g) at (1.5,0.4) {$\V$};
			\node (h) at (0.8,0.45) {$\bot$};
			
			\path[font=\scriptsize]

			([yshift=-1.3pt]f.east) edge [->] node [below] {$(i_B)^\ast=(i_B,1_I)^\ast$} ([yshift=-1.3pt]g.west)
			([yshift=1.5pt]f.east) edge [bend left,<-] node [above] {$(i_B)_!=(i_B,1_I)_!$} ([yshift=1.5pt]g.west);
		\end{tikzpicture}
		\hspace{1cm}
		\begin{tikzpicture}[baseline=(current  bounding  box.south), scale=2]

			\node (f) at (0,0.4) {$\ModB$};
			\node (g) at (1.5,0.4) {$\V$};
			\node (h) at (0.8,0.45) {$\bot$};
			
			\path[font=\scriptsize]

			([yshift=-1.3pt]f.east) edge [->] node [below] {$(i_{B'})^\ast=(1_I,i_{B'})^\ast$} ([yshift=-1.3pt]g.west)
			([yshift=1.5pt]f.east) edge [bend left,<-] node [above] {$(i_{B'})_!=(1_I,i_{B'})_!$} ([yshift=1.5pt]g.west);
		\end{tikzpicture}
	\end{center}	

\end{rmk}

\section{First order differential calculi}
\label{sec:diff-calc}

In this section we will define the monoidal version of first order differential calculus over a monoid in an additive monoidal category $\V$ with kernel and cokernels (i.e. $\Ab$-enriched and with finite limits and colimits). 
Additive monoidal here means that $\otimes$ is an additive functor in each variable, following for instance \cite{Ho:embedding-thm-ab-mon-cat}.
We moreover assume throughout this section that the tensor product preserves cokernels in each variable.

\begin{eg}\label{eg:additive-monoidal-cats}
	The following are examples of monoidal additive categories. In most cases, the fact that the tensor product is additive in both variables follows from the fact that the monoidal structure is biclosed (and hence tensoring is cocontinuous in both variables).
\begin{enumerate}[(i)]

	\item The category of finite dimensional representations of a group $G$ (or finite dimensional algebraic representations of an affine algebraic group $G$) is closed monoidal abelian.
	\item The category $\Mod$ of modules over a commutative ring $\bk$ is monoidal abelian. 
	This is also symmetric and closed.  
	\item The category $\GrMod$ of graded modules over a graded commutative ring equipped with the graded tensor product is monoidal abelian.
	This is also symmetric and closed.  
	\item The category (co)chain complexes over a commutative ring $\bk$ is monoidal abelian. 
	This is also symmetric and closed.  
	\item The category of condensed abelian groups is monoidal abelian.
	This is also symmetric and closed, see~\cite[p.~13]{scholze2019condensed}.
	\item The category of solid abelian groups is symmetric monoidal abelian,  see~\cite[p.~34 and p.~42]{scholze2019condensed}. The fact that the tensor product is additive in both variables follows from the fact that it is obtained by first taking the tensor product as condensed abelian groups (which is additive) and then applying an (additive) left adjoint. 
	\item Let $\V$ be symmetric monoidal closed and additive (resp.~abelian), and $\C$ be a small category enriched on $\V$. Then the category $\V\tx{-Prof}(\C)$ of endo-profunctors (i.e. distributors) on $\C$ is monoidal additive (resp.~abelian) and closed. By definition  $\V\tx{-Prof}(\C):=[\C\otimes\C^\op,\V]$, is the category of $\V$-enriched functors out of $\C\otimes\C^\op$ (see~\cite[1.4]{Kel82:libro}) into $\V$, where the monoidal structure is given by profunctor composition, see~\cite{shulman2008framed} in particular Example~5.9(2).
	\item The category $\Ban$ of Banach spaces with short linear operators as morphisms is symmetric monoidal but not $\mathbf{Ab}$-enriched. $\Banb$ of Banach spaces with bounded linear operators is a quasi-abelian category that is symmetric monoidal closed if endowed with the projective tensor product (see for instance~\cite[Section~4.4]{ST2012traces}). 
	\item The category of Bornological abelian group is quasi-abelian monoidal \cite{Bambozzi2014OnAG} This is also symmetric and closed. 
	\item The category of complete Bornological spaces is closed symmetric quasi-abelian monoidal \cite{savage2023koszul}. 
	\item The category of inductive limits of Banach spaces is closed symmetric quasi-abelian monoidal \cite{savage2023koszul}.  
	\item Let $\V$ be monoidal pre-abelian such that the monoidal product preserves coequalizers, and $A$ and object in $\VMon$, then $\AModA$ is monoidal pre-abelian. 
	If $\V$ is left (resp. right) closed, then $\AModA$ is left (resp. right) closed as well.
\end{enumerate}
\end{eg}

\subsection{Preliminary definitions}

We start by defining what it means to be a first order differential calculus over a monoid $(A,m,i)$ in an additive monoidal category $\V$, generalizing directly the ordinary definition for associative $\bk$-algebras (for $\bk$ a commutative ring). First we fix some notation.

\begin{notation}
	For a bimodule $(M,\mu,\nu)\in\AModA$, and morphisms $f\colon X\to M$ and $g\colon Y\to A$, with abuse of notation, we write: 
	\begin{center}
		$ g\cdot f := \mu(g\otimes f)$,\\
		$ f\cdot g := \nu(f\otimes g)$.
	\end{center}
	Even though we use the same notation for two different actions, whether we are using $\mu$ or $\nu$ is understood from the codomains of the morphisms we are considering. 
\end{notation}

\begin{defi}\label{def:enriched-1st-diff-calc}
A \emph{first order differential calculus} $\Omega^1_d$ over a monoid $A$ in $\V$ is an object $(\Omega^1_d,\mu_d,\nu_d)$ in $\AModA$ equipped with a \emph{differential}, that is a morphism $d\colon A\rightarrow \Omega^1_d$ in $\V$, satisfying the following axioms.
\begin{enumerate}
\item (Leibniz rule) The equality $dm= d\cdot 1_A + 1_A\cdot d$, depicted below, holds in $\V$.
\begin{equation}\label{eq:leibniz-enriched}
\begin{tikzcd}[ampersand replacement=\&]
	{A\otimes A} \& A \& {\Omega_d^1}
	\arrow["m", from=1-1, to=1-2]
	\arrow["{d\cdot 1_A+1_A\cdot d}"', curve={height=18pt}, from=1-1, to=1-3]
	\arrow["d", from=1-2, to=1-3]
\end{tikzcd}
\end{equation}
\item (Surjectivity condition) $\Omega^1_d$ is generated as a left module by $dA$, the image of $d$; that is, $1_A\cdot d$ is an epimorphism in $\V$. 
\begin{equation}\label{eq:surj-cond-enriched}
\begin{tikzcd}[ampersand replacement=\&]
	{A\otimes A} \& {A\otimes\Omega_d^1} \& {\Omega_d^1}
	\arrow["{1_A\otimes d}", from=1-1, to=1-2]
	\arrow["{1_A\cdot d}"', curve={height=18pt}, two heads, from=1-1, to=1-3]
	\arrow["{\mu_d}", from=1-2, to=1-3]
\end{tikzcd}
\end{equation}
\end{enumerate}
We call \emph{generalized first order differential calculus} a pair $(\Omega^1_d,d\colon A\to\Omega^1_d)$ as above which satisfies only the Leibniz rule \eqref{eq:leibniz-enriched}
\end{defi} 

A generalized first order differential calculus is also termed an \emph{$A$-derivation} in \cite[Definition~7.0.1]{ONeill:master-thesis}.

\begin{rmk}
	Following \cite[Definition~7.0.1]{ONeill:master-thesis}, one could notice that for the definition of first order differential calculus (\Cref{def:enriched-1st-diff-calc}) we did not need an $\Ab$-enriched category but a $\CMon$-enriched one would have been enough. 
	This said, the $\Ab$-enrichment requirement will be necessary  later on for the existence of the universal calculus (see \Cref{sec:univ first ord calc}).
\end{rmk}

\begin{rmk}
	Given monoids $A,B$ in $\V$, the forgetful functor $U\colon\AModB\to\V$ preserves and reflects epimorphisms: it preserves them by Proposition~\ref{forgetful}~(iii) and reflects them since every faithful functor reflects epimorphisms.
	It follows that, in the definition above, to say that $1_A\cdot d$ is an epimorphism in $\V$ is the same as saying that it is an epimorphism in $\AMod$. \\
	Similarly, in \Cref{prop:leibz-then-left-right} below, the map $(d\cdot 1_A)$ is an epimorphism in $\V$ if and only if it is such in $\ModA$, and $(1_A\cdot d\cdot 1_A)$ is an epimorphism in $\V$ if and only if it is such in $\AModA$. 
\end{rmk}

In the definition above, it might look unnatural to require only the surjectivity condition for the left module structure on $\Omega^1_d$ and nothing for the right one. 
This choice is justified by the next proposition. 

\begin{prop}\label{prop:leibz-then-left-right}
Let $(M,l_M,r_M)\in\AModA$ be a $A$-bimodule equipped with a map $d\colon A\to M$ satisfying the Leibniz rule \eqref{eq:leibniz-enriched}. The following are equivalent:\begin{enumerate}
	\item $(1_A\cdot d)\colon A\otimes A\to M$ is an epimorphism in $\V$;	
	\item $(d\cdot 1_A)\colon A\otimes A\to M$ is an epimorphism in $\V$;
	\item $(1_A\cdot d\cdot 1_A)\colon A\otimes A\otimes A\to M$ is an epimorphism in $\V$.
\end{enumerate}
\end{prop}

\begin{proof}
We shall prove the equivalence $(i)\Leftrightarrow(iii)$; then $(ii)\Leftrightarrow(iii)$ will be totally analogous.

$(i)\Rightarrow(iii)$. We can write
$$ 1_A\cdot d = (1_A\cdot d\cdot 1_A)\circ (1_A\otimes 1_A\otimes \eta) $$
where $\eta$ is the unit of $A$. Thus, if $1_A\cdot d$ is an epimorphism, so is $1_A\cdot d\cdot 1_A$.

$(iii)\Rightarrow(i)$. Using the Leibniz rule we can write
\begin{align}
	1_A\cdot d\cdot 1_A&=1_A\cdot(dm- 1_A\cdot d)\tag{Leibniz}\\
	&= 1_A\cdot (d\circ m)- m\cdot d\tag{linear and $1_A\cdot1_A=m$}\\
	&= (1_A\cdot d)\circ (1_A\otimes m)-(1_A\cdot d)\circ (m\otimes 1_A)\tag{functoriality}\\
	&= (1_A\cdot d)\circ (1_A\otimes m - m\otimes 1_A);\notag
\end{align}
therefore, if $1_A\cdot d\cdot 1_A$ is an epimorphism, so is $1_A\cdot d$.
\end{proof}

\begin{rmk}\label{surjectivity-bimodule-map}
	The advantage in using $(1_A\cdot d\cdot 1_A)\colon A\otimes A\otimes A\to M$, rather than the other two, is that this is a map of $A$-bimodules. Indeed, we can see $A\otimes A\otimes A$ as an $A$-bimodule with left action $m\otimes 1_A\otimes 1_A$ and right action $1_A\otimes 1_A\otimes m$; in other words, $A\otimes A\otimes A$ is the free $A$-bimodule on the object $A\in\V$.
\end{rmk}

A property about first order differential calculi:

\begin{lemma}
	Let $(\Omega^1_d,d)$ be a first order differential calculus on $A$. 
	Then $d\circ i=0$, where $i\colon I\to A$ is the unit of the monoid $A$.
\end{lemma}
\begin{proof}
	By restricting the Leibniz rule along $i\otimes i\colon I\cong I\otimes I\to A\otimes A$ and applying the functoriality of the tensor product we obtain
	\begin{align}
			d\circ i&=d\circ m\circ (i\otimes i)\tag{monoid unit}\\
			&= (d\cdot 1_A)\circ (i\otimes i)+(1_A\cdot d)\circ (i\otimes i)\tag{Leibniz}\\
			&= (d\circ i)\cdot (1_A\circ i)+( 1_A\circ i)\cdot (d\circ i)\tag{functoriality}\\
			&= (d\circ i)\cdot i +i\cdot (d\circ i)\notag\\
			&= d\circ i + d\circ i \tag{$i$ unit for $\mu_d$ and $\nu_d$}
	\end{align}
	By additivity this implies $d\circ i=0$.
\end{proof}

\subsection{The universal first order differential calculus}\label{sec:univ first ord calc}

In this subsection we extend the notion of universal first order differential calculus from the setting of associative algebras to monoids in monoidal additive categories. 

\begin{notation}\label{i-tensor}
	For simplicity of notation, given a morphism $f\colon B\to C$, we will write $i\otimes f\colon B\to A\otimes C$ to actually mean the composite $(i\otimes f)\circ \lambda_B^{-1}$. Similarly, by $f\otimes i\colon B\to C\otimes A$ we really mean the composite $(f\otimes i)\circ \rho_B^{-1}$.
\end{notation}

Let $(A,m,i)$ be a monoid in a monoidal additive category $\V$. 
Then, let us consider the kernel $\iota^1_u\colon\Omega_u^1\mono A\otimes A$ of the multiplication map $m\colon A\otimes A\to A$. 
By left/right unital axioms for $m$ we have that $m\circ( i\otimes 1_A)=1_A=m\circ (1_A\otimes i)$; hence by the universal property of the kernel there exists a unique map $d_u\colon A\to \Omega_u^1$ making the triangle below commutative.
\[\begin{tikzcd}[ampersand replacement=\&]
	{\Omega_u^1} \& {A\otimes A} \& A \\
	A
	\arrow["{\iota^1_u}", tail, from=1-1, to=1-2]
	\arrow["m", from=1-2, to=1-3]
	\arrow["{\exists!\;d_u}", from=2-1, to=1-1]
	\arrow["{i\otimes 1_A-1_A\otimes i}"', from=2-1, to=1-2]
\end{tikzcd}\]
In the next proposition, we show that $(\Omega_u^1,d_u)$ is a first order differential calculus over $A$, and  
we refer to this as the \emph{universal first order differential calculus} over the monoid $A$.

\begin{prop}\label{prop:univ-1-d-cal-well-def}
With the definitions given above, $(\Omega_u^1,d_u)$ is a first order differential calculus over $A$. Moreover $(d_u\cdot 1_A)$ is a split epimorphism in $\V$ with right inverse $-\iota_u^1$.
\end{prop}

\begin{proof}
	
First, we need to show that $\Omega_u^1$ inherits a structure of $A$-bimodule. 
Note that $A\otimes A$ has a canonical $A$-bimodule structure whose left and right actions are respectively $m\otimes 1_A$ and $1_A\otimes m$. 
Moreover, by associativity, $m\colon A\otimes A\to A$ is a map in $\AModA$. 
Since kernels in $\AModA$ are computed as in $\V$ (see \Cref{forgetful}), then $\iota_u\colon\Omega_u^1\mono A\otimes A\in\AModA$. 
Therefore, in particular, $\Omega_u^1$ is an $A$-bimodule. 
Explicitly, its left action $l_u\colon A\otimes \Omega_u^1\to \Omega_u^1$ is defined as the dashed map below
\begin{center}
	\begin{tikzpicture}[baseline=(current  bounding  box.south), scale=2]
		
		\node (a1) at (0,0.8) {$A\otimes \Omega_u^1$};
		\node (b1) at (1.5,0.8) {$A\otimes A\otimes A$};
		\node (c1) at (3,0.8) {$A\otimes A$};
		\node (a0) at (0,0) {$\Omega_u^1$};
		\node (b0) at (1.5,0) {$A\otimes A$};
		\node (c0) at (3,0) {$A$};
		
		\path[font=\scriptsize]
		
		(a1) edge [->] node [above] {$1_A\otimes \iota_u$} (b1)
		(b1) edge [->] node [above] {$1_A\otimes m$} (c1)
		(a0) edge [>->] node [below] {$\iota_u^1$} (b0)
		(b0) edge [->] node [below] {$m$} (c0)
		(a1) edge [dashed, ->] node [left] {$l_u$} (a0)
		(b1) edge [->] node [right] {$m\otimes 1_A$} (b0)
		(c1) edge [->] node [right] {$m$} (c0);
		
	\end{tikzpicture}	
\end{center} 
induced by the universal property of the kernel using that the top horizontal composite is $0$. 
	
Next we shall prove that the Leibniz rule holds for $d_u\colon A\to \Omega_u^1$.
Since $\iota_u$ is a monomorphism, this amounts to show that
$$ \iota_u^1\circ d_u\circ m=\iota_u^1\circ (d_u\cdot 1_A)+\iota_u^1\circ (1_A\cdot d_u). $$
On one hand we have
\begin{align}
	\iota_u^1 \circ d_u \circ m &= (i\otimes 1_A)\circ  m - (1_A\otimes i)\circ  m \tag{def of $d_u$}\\
	& = (i\otimes m) - (m\otimes i),\tag{$m$ acts on $A$}	
\end{align}
where we are identifying $A\otimes I$ and $I\otimes A$ with $A$ as in Notation~\ref{i-tensor}.
For the other, note that
\begin{align}
	\iota_u^1\circ (d_u\cdot 1_A) &= (1_A\otimes m)\circ (\iota_u^1\otimes 1_A)\circ (d_u\otimes 1_A) \tag{def of $r_u$} \\
	&= (1_A \otimes m)\circ ((\iota_u^1\circ  d_u) \otimes 1_A) \tag{functoriality} \\
	&= (1_A \otimes m)\circ (i\otimes 1_A\otimes 1_A - 1_A\otimes i\otimes 1_A) \tag{def of $d_u$} \\
	&= i\otimes m - (1_A\otimes m)\circ (1_A\otimes i\otimes 1_A) \tag{functoriality}\\
	&= i\otimes m - 1_A\otimes 1_A.\tag{$i$ unit for $m$}
\end{align}
Similarly,
\begin{align}
	\iota_u^1\circ (1_A\cdot d_u) 
	&=  1_A\otimes 1_A - m\otimes i.\notag
\end{align}
Putting everything together we obtain
\begin{align}
	\iota_u^1\circ (d_u\cdot 1_A)+\iota_u^1\circ (1_A\cdot d_u)& = i\otimes m - 1_A\otimes 1_A + 1_A\otimes 1_A - m\otimes i\notag\\
	&= i\otimes m - m\otimes i\notag\\
	&= \iota_u^1\circ  d_u \circ m.\notag
\end{align}

It remains to prove the surjectivity condition; that is to show that $(d_u\cdot 1_A)$ is an epimorphism. 
We achieve that by proving the equality $-(d_u\cdot 1_A)\iota_u^1=1_{\Omega_u^1}$, so that $(d_u\cdot 1_A)$ is actually a split epimorphism.
Consider the following chain of equalities
\begin{align}
	\iota_u^1\circ (d_u\cdot 1_A)\circ \iota_u^1 &= (i\otimes m - 1_A\otimes 1_A)\circ \iota_u^1 \tag{seen above}\\
	&= (i\otimes m)\circ \iota_u^1 - \iota_u^1\tag{functoriality}\\
	&= (i\otimes( m\circ \iota_u^1)) - \iota_u^1\tag{$\iota_u^1$ acts on $A$}\\
	&= -\iota_u^1.\tag{$m\iota_u^1=0$}
\end{align}
Since $\iota_u^1$ is a monomorphism, this implies that $-(d_u\cdot 1_A)\circ \iota_u^1=1_{\Omega_u^1}$.
\end{proof}

\begin{prop}\label{universal-prop}
The first order differential calculus $(\Omega_u^1,d_u)$ is universal; that is, for any other first order differential calculus $(\Omega_d^1,d)$, there exists a unique morphism $f\colon \Omega_u^1\to\Omega_d^1$ in $\AModA$ making the triangle
\[\begin{tikzcd}[ampersand replacement=\&]
	\& A \\
	{\Omega_u^1} \&\& {\Omega_d^1}
	\arrow["{d}", from=1-2, to=2-3]
	\arrow["{d_u}"', from=1-2, to=2-1]
	\arrow["{\exists! f}"', from=2-1, to=2-3]
\end{tikzcd}\]
commute in $\V$.
Such a morphism is always an epimorphism.
\end{prop}
\begin{proof}
	Define $f$ as the composite
	$$ f:=(1_A\cdot d)\circ \iota_u^1\colon \Omega_u^1\xrightarrow{\ \iota_u^1\ } A\otimes A\xrightarrow{1_A\otimes d}A\otimes \Omega_d^1\xrightarrow{\ l_d\ }\Omega_d^1. $$
	First let us observe that $f$ makes the triangle commute in $\V$:
	\begin{align}
		f\circ d_u&=(1_A\cdot d)\circ \iota_u^1\circ d_u\tag{def of $f$}\\
		&= (1_A\cdot d)\circ(i\otimes 1_A - 1_A\otimes i)\tag{def of $\iota_u^1$}\\
		&= (1_A\cdot d)\circ(i\otimes 1_A) - (1_A\cdot d)\circ (1_A\otimes i)\tag{functoriality}\\
		&= i\cdot d - 1_A\cdot (d\circ i)\notag \tag{functoriality}\\
		&= d -0 \tag{$i$ unit for $l_d$ and $d\circ i=0$}\\
		&= d\notag.
	\end{align}
	
	Next we prove that $f$ is a map of bimodules; this amounts to showing that $l_d\circ (1_A\otimes  f)= f\circ l_u$ and $r_d\circ (f\otimes 1_A)= f\circ r_u$.
	Then
	\begin{align}
		f\circ l_u &= (1_A\cdot d)\circ\iota_u^1\circ l_u\tag{def of $f$}\\
		&= (1_A\cdot d)\circ (m\otimes 1_A)\circ (1_A\otimes\iota_u^1)\tag{def of $\iota_u^1$}\\
		&= (m\cdot d)\circ (1_A\otimes\iota_u^1)\tag{functoriality}\\
		&= 1_A\cdot(1_A\cdot d)\circ (1_A\otimes\iota_u^1)\tag{action of $l_d$}\\
		&= l_d\circ (1_A\otimes (1_A\cdot d))\circ (1_A\otimes\iota_u^1)\tag{def of $-\cdot -$}\\
		&= l_d\circ (1_A\otimes  f)\tag{functoriality},
	\end{align} 
	and
	\begin{align}
		f\circ r_u &= (1_A\cdot d)\circ\iota_u^1\circ r_u\tag{def of $f$}\\
		&= (1_A\cdot d)\circ (1_A\otimes m)\circ (\iota_u^1\otimes 1_A)\tag{def of $r_u$}\\
		&= (1_A\cdot d\circ m)\circ (\iota_u^1\otimes 1_A)\tag{functoriality}\\
		&= (1_A\cdot (d\cdot 1_A) + 1_A\cdot (1_A\cdot d) ) \circ (\iota_u^1\otimes 1_A)\tag{Leibniz}\\
		&= ((1_A\cdot d)\cdot 1_A + m\cdot d ) \circ (\iota_u^1\otimes 1_A)\tag{associativity} \\
		&= r_d\circ ((1_A\cdot d)\otimes 1_A)\circ (\iota_u^1\otimes 1_A) + (m\circ \iota_u^1)\cdot d\tag{def of $-\cdot -$}\\
		&= r_d\circ (f\otimes 1_A)\tag{$m\circ \iota_u^1=0$}.
	\end{align} 
	
	Now, given any map of bimodules $g\colon \Omega_u^1\to\Omega_d^1$ such that $g\circ d_u=d$, then
	$$ g\circ (1_A\cdot d_u) = 1_A\cdot(g\circ d_u)= 1_A\cdot d ,$$
	where the first equality holds since $g$ is a map of bimodules, and as such respects multiplication on the left. 
	
	This implies the uniqueness of $f$: given any other $g$ as above we have
	$$ g\circ (1_A\cdot d_u) = 1_A\cdot d = f\circ (1_A\cdot d_u), $$
	but $(1_A\cdot d_u)$ is an epimorphism (by the surjectivity condition on $\Omega_u^1$); thus $f=g$.
	
	The fact that $f$ is an epimorphism is also a consequence of the equality $f\circ (1_A\cdot d_u)= 1_A\cdot d$ since $1_A\cdot d$ is an epimorphism (by the surjectivity condition on $\Omega_d^1$).
\end{proof}

\begin{lemma}\label{lem:epi-from-univ-always-calc}
	Let $(\Omega_u^1,d_u)$ be the universal calculus on a monoid $A$ and $p\colon\Omega_u^1\epi D$ an epimorphism in $\AModA$.
	Then, $(D,d)$, with $d:=p d_u$, is a first order differential calculus on $A$. 
\end{lemma}
\begin{proof}
	Write $\mu_D$ and $\nu_D$ for the left and right actions of $D$. Let us prove that $(D,d)$ satisfies the Leibniz rule:
	\begin{align}
		d\circ m&=p\circ d_u \circ m\tag{def of $d$}\\
		&= p \circ (d_u\cdot 1_A + 1_A\cdot d_u)\tag{Leibniz for $d_u$}\\
		&= p\circ  (d_u\cdot 1_A) + p\circ (1_A\cdot d_u)\tag{additivity}\\
		&= p\circ  \nu_u\circ (d_u\otimes 1_A) + p\circ  \mu_u\circ (1_A\otimes d_u)   \tag{def of $\cdot$}\\
		&= \nu_D\circ (p\otimes 1_A)\circ (d_u\otimes 1_A) + \mu_D\circ (1_A\otimes p)\circ (1_A\otimes d_u) \tag{$p$ bimodule map}\\
		&= \nu_D\circ ( (p\circ d_u)\otimes 1_A) + \mu_D\circ (1_A\otimes (p\circ d_u)) \tag{functoriality}\\
		&= d\cdot 1_A + 1_A\cdot d. \tag{def of $d$ and $\cdot$}
	\end{align}
	It remains to show that $1_A\cdot d$ is an epimorphism. But it was shown in the equalities above that
	$$ 1_A\cdot d = p\circ (1_A\cdot d_u) $$
	where $p$ is an epimorphism by hypothesis and $(1\cdot d_u)$ is since $\Omega^1_u$ is a first order differential calculus. Thus $1_A\cdot d$ is also an epimorphism and $(D,d)$ is a first order differential calculus on $A$.
\end{proof}

\begin{rmk}\label{kernel-counit}
We now show that the kernel of the counit of the restriction-extension to scalars $\AMod\to\V$ is the functor $\Omega^1_u\otimes_A-$. 
It is enough to prove this pointwise, and it corresponds to showing that for each left $A$-module $M$, the object $\Omega^1_u\otimes_A M$ is the kernel of the scalar multiplication $\mu_M\colon A\otimes M\to M$. 
Consider the solid part of the diagram below.
\begin{center}
	\begin{tikzpicture}[baseline=(current  bounding  box.south), scale=2]
		
		\node (a2) at (-0.4,1.6) {$\Omega_u^1\otimes A\otimes  M$};
		\node (b2) at (1.4,1.6) {$A^{\otimes 3}\otimes M$};
		\node (a1) at (-0.4,0.8) {$\Omega_u^1\otimes  M$};
		\node (b1) at (1.4,0.8) {$A^{\otimes 2}\otimes M$};
		\node (c1) at (3.2,0.8) {$A\otimes M$};
		\node (a0) at (-0.4,0) {$\Omega_u^1\otimes_A M$};
		\node (b0) at (1.4,0) {$A\otimes M$};
		\node (c0) at (3.2,0) {$M$};
		
		\path[font=\scriptsize]
		
		(a1) edge [->] node [above] {$\iota_u\otimes 1_M$} (b1)
		(b1) edge [->] node [above] {$m\otimes 1_M$} (c1)
		(a0) edge [dashed, ->] node [below] {$g$} (b0)
		(b0) edge [->] node [below] {$\mu_M$} (c0)
		(a1) edge [->>] node [left] {$q$} (a0)
		(b1) edge [->] node [right] {$1_A\otimes \mu_M$} (b0)
		(c1) edge [->] node [right] {$\mu_M$} (c0)
		
		([xshift=1.5pt]a2.south) edge [->] node [right] {$r_u\otimes 1_M$} ([xshift=1.5pt]a1.north)
		([xshift=-1.5pt]a2.south) edge [->] node [left] {$1_{\Omega^1_u}\otimes \mu_M$} ([xshift=-1.5pt]a1.north)
		
		([xshift=1.5pt]b2.south) edge [->] node [right] {$1_A\otimes m\otimes 1_M$} ([xshift=1.5pt]b1.north)
		([xshift=-1.5pt]b2.south) edge [->] node [left] {$1_{A^{\otimes 2}}\otimes \mu_M$} ([xshift=-1.5pt]b1.north)
		
		(a2) edge [->] node [above] {$\iota_u\otimes 1_{A\otimes M}$} (b2);
	\end{tikzpicture}	
\end{center} 
Here, the left vertical fork is the coequalizer defining $\Omega_u^1\otimes_A M$. 
The top squares (one with the left vertical maps, one with the right vertical maps) commute by definition of $\mu_M$ and $r_u$; same for the bottom right square. 
Thus, since $(1_A\otimes \mu_M)\circ(\iota_u\otimes 1_M)$ coequalizes the two arrows on top of  $\Omega^1_u\otimes M$, by the universal property of $q$ there exists a map  $g$ as dashed above making the square commute.
The composite  $\mu_M\circ g$ is $0$ since also the middle horizontal composite is.

Now, let $k\colon K\to A\otimes M$ be the kernel of $\mu_M$; then there is an induced map $s\colon \Omega^1_u\otimes_A M\to K$ such that $k\circ s=g$.
Conversely, we  define the composite
$$ t\colon K\xrightarrow{ k } A\otimes M\xrightarrow{d_u\otimes 1_M}\Omega^1_u\otimes M\xrightarrow{-q}  \Omega^1_u\otimes_A M.$$
By definition of $g$ and of $d_u$, we have that
$$ g\circ q\circ(d_u\otimes 1_M)=(1_A\otimes\mu_M)\circ\big((i\otimes 1_A-1_A\otimes i)\otimes 1_M\big)=(i\otimes 1_M)\circ\mu_M-1_{A\otimes M}, $$
and $\mu_M\circ k=0$; hence $k\circ s\circ t=-g\circ q\circ(d_u\otimes 1_M)\circ k=k$, and $s\circ t=1$ since $k$ is a monomorphism. 
Then, using that $q$ coequalizes $r_u\otimes 1_M$ and $1_{\Omega^1_u}\otimes\mu_M$, and that $(d_u\cdot 1_A)\circ\iota_u=-1_{\Omega^1_u}$ by \Cref{prop:univ-1-d-cal-well-def}, we compute
\begin{align*}
	q\circ(d_u\otimes 1_M)\circ g\circ q&=q\circ(d_u\otimes 1_M)\circ(1_A\otimes \mu_M)\circ(\iota_u\otimes 1_M)=q\circ(1_{\Omega^1_u}\otimes\mu_M)\circ(d_u\otimes 1_{A\otimes M})\circ(\iota_u\otimes 1_M)\\
	&=q\circ(r_u\otimes 1_M)\circ(d_u\otimes 1_{A\otimes M})\circ(\iota_u\otimes 1_M)=q\circ\big(((d_u\cdot 1_A)\circ\iota_u)\otimes 1_M\big)=-q.
\end{align*}
This yields $t\circ s\circ q=-q\circ(d_u\otimes 1_M)\circ g\circ q=q$, and $t\circ s=1$ since $q$ is an epimorphism.
Thus $s\circ t=1$ and $t\circ s=1$, proving our claim. 
\end{rmk}

We denote by $\CalcA$ the category whose objects are first order differential calculi over $A$, and whose morphisms are maps of $A$-bimodules $f\colon\Omega_d^1\to\Omega_{d'}^1$ which make the following triangle commute. 
\begin{center}
	\begin{tikzpicture}[baseline=(current  bounding  box.south), scale=2]
		
		\node (c0) at (0.6,0) {$A$};
		\node (a0) at (0,-0.7) {$\Omega_d^1$};
		\node (b0) at (1.2,-0.7) {$\Omega_{d'}^1$};
		
		\path[font=\scriptsize]
		
		(a0) edge [->] node [below] {$f$} (b0)
		(c0) edge [->] node [above] {$d\ \ \ $} (a0)
		(c0) edge [->] node [above] {$\ \ \ d'$} (b0);
	\end{tikzpicture}	
\end{center}

\begin{prop}
	The following properties hold for $\CalcA$:\begin{enumerate}
		\item $\Omega_u^1$ is an initial object; 
		\item every morphism is an epimorphism when regarded in $\AModA$;
		\item there is at most one morphism between any two objects.
	\end{enumerate}
	Therefore $\CalcA$ is a preorder with initial element $\Omega_u^1$ and terminal element the trivial first order differential calculus $(0\in\AModA,0\colon A\to 0)$.
\end{prop}
\begin{proof}
	Point (i) is a consequence of \Cref{universal-prop}. 
	For (ii), note that for any first order differential calculus $\Omega_d^1$ the unique map $\Omega_u^1\to\Omega_d^1$ is an epimorphism by Proposition~\ref{universal-prop}. 
	In general, if we are given a morphism $f\colon \Omega_d^1\to \Omega_{d'}^1$, then the commutativity of the triangle
	\begin{center}
		\begin{tikzpicture}[baseline=(current  bounding  box.south), scale=2]
			
			\node (c0) at (0.6,0) {$\Omega_u^1$};
			\node (a0) at (0,-0.7) {$\Omega_d^1$};
			\node (b0) at (1.2,-0.7) {$\Omega_{d'}^1$};
			
			\path[font=\scriptsize]
			
			(a0) edge [->] node [below] {$f$} (b0)
			(c0) edge [->>] node [above] {$!_d\ \ \ $} (a0)
			(c0) edge [->>] node [above] {$\ \ \ !_{d'}$} (b0);
		\end{tikzpicture}	
	\end{center} 
	in $\AModA$, where both the left and right legs are epimorphisms, implies that $f$ is an epimorphism too.
	
	Finally, (iii) is again a consequence of the commutativity of the triangle above: given any pair of maps $f,g\colon \Omega_d^1\to \Omega_{d'}^1$, since $!_df=\ !_{d'}=\ !_{d}g$ and $!_d$ is an epimorphism, it follows that $f=g$.
\end{proof}

\begin{rmk}
	Fix a monoid $A$ in $\V$. Given a first order differential calculus $\Omega_{d}^1$ over $A$, we get a subobject of $\Omega_{u}^1$ by taking the kernel of the epimorphism $p_d\colon\Omega_{u}\twoheadrightarrow\Omega_{d}^1$. 
	On the other hand, if we start with a subobject $i\colon N\hookrightarrow\Omega_{u}^1$ of the universal first order differential calculus, taking the cokernel of $i$ gives a first order differential calculus on $A$ (by \Cref{lem:epi-from-univ-always-calc}). 
	In other words, we get functors as below, which we can prove form an adjunction. 
\[\begin{tikzcd}[ampersand replacement=\&]
	\CalcA \& {\ \Sub(\Omega_{u}^1)}
	\arrow[""{name=0, anchor=center, inner sep=0}, "{\tx{ker}}"', curve={height=12pt}, from=1-1, to=1-2]
	\arrow[""{name=1, anchor=center, inner sep=0}, "{\tx{coker}}"', curve={height=12pt}, from=1-2, to=1-1]
	\arrow["\bot"{description}, draw=none, from=0, to=1]
\end{tikzcd}\]
In fact, since both categories are posetal, to get an adjunction is enough to prove that, for any subobject $i\colon N\hookrightarrow\Omega_{u}^1$ and first order differential calculus $\Omega_{d}^1$,
there exists $N\to\tx{ker}(p_{d})\in\Sub(\Omega_{u}^1)$ if and only if there exists a morphism $\tx{coker}(i)\to \Omega_{d}^1\in\CalcA$. 
This can be seen looking at the diagram below and using the universal properties of kernels and cokernels. 
\[\begin{tikzcd}[ampersand replacement=\&]
	{N} \& {\Omega_{u}^1} \& {\text{coker}(i)} \\
	{\text{ker}(p_{d})} \& {\Omega_{u}^1} \& {\Omega_{d}^1}
	\arrow["{i}", tail, from=1-1, to=1-2]
	\arrow[dashed, from=1-1, to=2-1]
	\arrow[two heads, from=1-2, to=1-3]
	\arrow[equals, from=1-2, to=2-2]
	\arrow[dashed, from=1-3, to=2-3]
	\arrow[tail, from=2-1, to=2-2]
	\arrow["{p_{d}}"', from=2-2, to=2-3]
\end{tikzcd}\]
	Moreover, if $\V$ is abelian, this adjunction can be shown to be an equivalence (using that every epimorphism is the cokernel of its kernel). 
\end{rmk}

\subsection{The extension-restriction of scalars on first order differential calculi}
\label{sec:ext-restr-calc}

In this subsection we will show that, considering $A,B\in\VMon$ and $f\colon A\to B$ a monoid morphism, then the extension-restriction of scalars adjunction $f_!\dashv f^\ast$ \footnote{See \Cref{sec:ext-restr-scalars}. With abuse of notation, we write $f_!:=(f,f)_!$ and $f^\ast:=(f,f)^\ast$.} (below left) induces an adjunction between their corresponding categories of first order differential calculi (below right). 
\[\begin{tikzcd}[ampersand replacement=\&]
	\BModB \& \AModA \& \CalcB \& \CalcA
	\arrow[""{name=0, anchor=center, inner sep=0}, "{f^\ast}"', curve={height=6pt}, from=1-1, to=1-2]
	\arrow[""{name=1, anchor=center, inner sep=0}, "{f_!}"', curve={height=6pt}, from=1-2, to=1-1]
	\arrow[""{name=2, anchor=center, inner sep=0}, "{F^\ast}"', curve={height=6pt}, from=1-3, to=1-4]
	\arrow[""{name=3, anchor=center, inner sep=0}, "{F_!}"', curve={height=6pt}, from=1-4, to=1-3]
	\arrow["\bot"{description}, draw=none, from=0, to=1]
	\arrow["\bot"{description}, draw=none, from=2, to=3]
\end{tikzcd}\]

In order to show this adjunction, we will use the following strategy: define the assignments on objects of $F_!$ and $F^*$, prove functoriality of one of them and the adjunction natural isomorphism, which then will guarantee the other assignment to be a functor as well. 

\begin{rmk}\label{rmk:induced-morph-betw-univ-calc}
	Let $A,B\in\VMon$ and $f\colon A\to B$ a monoid morphism. Then, there exists a unique map between the universal calculi $f_u\colon \Omega_{A,u}^1\to \Omega_{B,u}^1$ in $\V$ making the square below commutative. 
	\[\begin{tikzcd}[ampersand replacement=\&]
		{\Omega^1_{A,u}} \& {A\otimes A} \& A \\
		{\Omega_{B,u}^1} \& {B\otimes B} \& B
		\arrow[tail, from=1-1, to=1-2]
		\arrow["{\exists!\,f_u}"', dashed, from=1-1, to=2-1]
		\arrow["{m_A}", from=1-2, to=1-3]
		\arrow[from=1-2, to=2-2]
		\arrow[from=1-3, to=2-3]
		\arrow[tail, from=2-1, to=2-2]
		\arrow["{m_B}"', from=2-2, to=2-3]
	\end{tikzcd}\]
	Indeed, this is induced by the universal property of kernels.
\end{rmk}

\begin{prop}\label{prop:fu-bimod-map}
	Given $A,B\in\VMon$ and $f\colon A\to B$ a monoid morphism, the  morphism $f_u$ defined in \Cref{rmk:induced-morph-betw-univ-calc} is a morphism of $A$-bimodules $f_u\colon\Omega_{A,u}^1\to f^\ast(\Omega_{B,u}^1)\in\AModA$. 
\end{prop}

\begin{proof}
	We will prove that $f_u$ preserve the right action, the proof for the left one is completely analogous. 
	Recalling the definition of the $A$-right action on $f^\ast(\Omega_{B,u}^1)$ (see \Cref{prop:def-restr-of-scalars}), to show that $f_u$ preserve the right action, since $\Omega_{B,u}^1\mono B\otimes B$ is a monomorphism, is equivalent to prove that the outer diagram below is commutative. 
\[\begin{tikzcd}[ampersand replacement=\&]
	{\Omega^1_{A,u}\otimes A} \& {A\otimes A\otimes A} \\
	{\Omega_{B,u}^1\otimes A} \& {\Omega^1_{A,u}} \& {A\otimes A} \\
	{\Omega_{B,u}^1\otimes B} \& {\Omega_{B,u}^1} \& {B\otimes B} \\
	\& {B\otimes B\otimes B}
	\arrow[from=1-1, to=1-2]
	\arrow["{f_u\otimes1_A}"', from=1-1, to=2-1]
	\arrow[from=1-1, to=2-2]
	\arrow["{m_A\otimes 1_A}", from=1-2, to=2-3]
	\arrow["{1\otimes f}"', from=2-1, to=3-1]
	\arrow[tail, from=2-2, to=2-3]
	\arrow[from=2-2, to=3-2]
	\arrow["{f\otimes f}", from=2-3, to=3-3]
	\arrow[from=3-1, to=3-2]
	\arrow[from=3-1, to=4-2]
	\arrow[tail, from=3-2, to=3-3]
	\arrow["{m_B\otimes 1_B}"', from=4-2, to=3-3]
\end{tikzcd}\]
To show this, it is enough to notice that the following diagram is commutative.
\[\begin{tikzcd}[ampersand replacement=\&]
	{\Omega^1_{A,u}\otimes A} \&\&\& {A\otimes A\otimes A} \&\& {A\otimes A} \\
	{\Omega_{B,u}^1\otimes A} \&\& {\Omega_{A,u}^1\otimes B} \& {A\otimes A\otimes B} \& {A\otimes B} \\
	{\Omega_{B,u}^1\otimes B} \&\&\& {B\otimes B\otimes B} \&\& {B\otimes B}
	\arrow[from=1-1, to=1-4]
	\arrow["{f_u\otimes1_A}"', from=1-1, to=2-1]
	\arrow["{1\otimes f}"{description}, from=1-1, to=2-3]
	\arrow["{m_A\otimes 1_A}", from=1-4, to=1-6]
	\arrow["{1\otimes f}"', from=1-4, to=2-4]
	\arrow["{1\otimes f}"', from=1-6, to=2-5]
	\arrow["{f\otimes f}", from=1-6, to=3-6]
	\arrow["{1\otimes f}"', from=2-1, to=3-1]
	\arrow[from=2-3, to=2-4]
	\arrow["{f_u\otimes 1}"{description}, from=2-3, to=3-1]
	\arrow["{m_A\otimes 1}"', from=2-4, to=2-5]
	\arrow["{f\otimes f\otimes1}"', from=2-4, to=3-4]
	\arrow["{f\otimes 1}"', from=2-5, to=3-6]
	\arrow[from=3-1, to=3-4]
	\arrow["{m_B\otimes 1_B}"', from=3-4, to=3-6]
\end{tikzcd}\]
\end{proof} 

Now, we are ready to define the assignments on objects of $F_!$ and $F^*$ in the following two remarks. 

\begin{rmk}[$F_!\colon\CalcA\to\CalcB$]\label{rmk:def-F!}
	Let $A,B\in\VMon$ and $f\colon A\to B$ be a monoid morphism. Since the  morphism $f_u$ defined in \Cref{rmk:induced-morph-betw-univ-calc} is a $A$-bimodule morphism $f_u\colon\Omega_{A,u}^1\to f^\ast(\Omega_{B,u}^1)\in\AModA$ (by \Cref{prop:fu-bimod-map}), then it corresponds to a unique $\widehat{f_u}\colon f_!\Omega_{A,u}^1\to\Omega_{B,u}^1\in\BModB$ through the adjunction $f_!\dashv f^*$. 
	Now, for any $\Omega_{A,d}^1\in\CalcA$, we define $F_!\Omega_{A,d}^1$ through the pushout below. 
	\[\begin{tikzcd}[ampersand replacement=\&]
		{f_!\Omega^1_{A,u}} \& {f_!\Omega^1_{A,d}} \\
		{\Omega^1_{B,u}} \& {F_!\Omega_{A,d}^1}
		\arrow[two heads, from=1-1, to=1-2]
		\arrow["{\widehat{f_u}}"', from=1-1, to=2-1]
		\arrow["{\widehat{f_d}}", from=1-2, to=2-2]
		\arrow[two heads, from=2-1, to=2-2]
		\arrow["\lrcorner"{anchor=center, pos=0.125, rotate=180}, draw=none, from=2-2, to=1-1]
	\end{tikzcd}\]
	Since epimorphisms are preserved by $f_!=-\otimes_A-$ (see \Cref{prop:otimes_A-preserves-coequ}) and are stable under pushouts, in the diagram above we have an epimorphism $\Omega_{B,u}^1\twoheadrightarrow F_!\Omega_{A,d}^1$. Hence, this provides a calculus on $B$ (by \Cref{lem:epi-from-univ-always-calc}). 
	We might also use the following notation  $F_!\Omega_{A,d}^1:=\Omega_{B,f_!d}$. 
\end{rmk}

\begin{rmk}[$F^*\colon\CalcB\to\CalcA$]\label{rmk:def-Fast}
	Let $A,B\in\VMon$ and $f\colon A\to B$ be a monoid morphism. 
	In order to define $F^*$, we take $\Omega_{B,\delta}\in\CalcB$ and we consider the pullback below in $\AModA$, where $N_{B,\delta}$ is the kernel of $q_\delta$. Therefore we obtain a monomorphism $f^\ast j_\delta$, since $f^\ast$ preserves them (\Cref{restriction-properties}). 
	Since, monomorphisms are preserved by pullbacks, the induced map $P\mono _A\Omega_u^1$ is a monomorphism. 
	Then, if we take the cokernel of this map we obtain an epimorphism $\Omega_{A,u}^1\epi F^*\Omega_{B,\delta}^1$; hence giving a calculus structure on $F^*\Omega_{B,\delta}^1$ (by \Cref{lem:epi-from-univ-always-calc}). 
\[\begin{tikzcd}[ampersand replacement=\&]
	P \& {\Omega^1_{A,u}} \& {F^*\Omega_{B,\delta}^1} \\
	{f^*N_{B,\delta}} \& {f^*\Omega_{B,u}} \& {f^*\Omega^1_{B,\delta}}
	\arrow["{i_{f^\ast\delta}}", tail, from=1-1, to=1-2]
	\arrow["{k^\delta}"', from=1-1, to=2-1]
	\arrow["\lrcorner"{anchor=center, pos=0.125}, draw=none, from=1-1, to=2-2]
	\arrow["{p_{f^\ast \delta}}", two heads, from=1-2, to=1-3]
	\arrow["{f_u}", from=1-2, to=2-2]
	\arrow["{\exists!f^\delta}"{pos=0.4}, dashed, from=1-3, to=2-3]
	\arrow["{f^\ast j_\delta}"', tail, from=2-1, to=2-2]
	\arrow["{f^\ast q_\delta}"', two heads, from=2-2, to=2-3]
\end{tikzcd}\]
	We might also use the following notation $F^*\Omega_{B,\delta}^1:=\Omega^1_{A,f^\ast \delta}$.

	Moreover, we can prove that $P$ is actually the kernel of $\Omega_{A,u}^1\epi\Omega^1_{A,f^\ast \delta}$. 
	In order to check this, we start with a morphism $l\colon L\to\Omega_{A,u}^1$ such that $p_{f^\ast\delta}\circ l=0$. 
	Hence, $0=f^\delta\circ p_{f^\ast\delta}\circ l=f^\ast q_\delta\circ f_u\circ l$, and so by the universal property of the kernel $f^\ast j_\delta$, there exists a unique map $k_l\colon L\to f^\ast N_{B,\delta}$ such that $f^\ast j_\delta \circ k_l=f_u\circ l$.
	Therefore, by the universal property of the pullback we get a unique map $L\to P$ with the required property. 
	We then denote $P=N_{A,f^\ast\delta}$, as the kernel of $p_{f^\ast\delta}$.  
\end{rmk}

\begin{prop}
Assume that every epimorphism in $\AModA$ is the cokernel of its kernel (for instance, if $\V$ is abelian and the tensor product preserves cokernels and finite direct sums, see \Cref{abelian-modules}). Then the assignments in \Cref{rmk:def-F!} and \ref{rmk:def-Fast} can be completed to an adjunction 
\[\begin{tikzcd}[ampersand replacement=\&]
	\CalcB \& \CalcA
	\arrow[""{name=0, anchor=center, inner sep=0}, "{F^\ast}"', curve={height=6pt}, from=1-1, to=1-2]
	\arrow[""{name=1, anchor=center, inner sep=0}, "{F_!}"', curve={height=6pt}, from=1-2, to=1-1]
	\arrow["\bot"{description}, draw=none, from=0, to=1]
\end{tikzcd}\]
between the categories of first order differential calculi over $A$ and $B$. 
\end{prop}

\begin{proof}
We start showing that the assignment of objects $F_!$ is actually a functor $\CalcA\to\CalcB$. 
Since both categories are posetal, it is enough to show that given a morphism $\Omega_{A,d}^1\epi\Omega_{A,d'}^1\in\CalcA$, then there is a morphism $\Omega_{B,f_!d}^1\epi\Omega_{B,f_!d'}^1\in\CalcB$. 
This follows by the universal property of the pushout $\Omega_{B,f_!d}^1$ as shown below, where the outer diagram is the pushout diagram defining $\Omega_{B,f_!d'}^1$ and the top triangle commutes since $\Omega_{A,d}^1\twoheadrightarrow\Omega_{A,d'}^1$ is a morphism in $\CalcA$. 
\[\begin{tikzcd}[ampersand replacement=\&]
	{f_!\Omega^1_{A,u}} \& {f_!\Omega^1_{A,d}} \& {f_!\Omega^1_{A,d'}} \\
	{\Omega^1_{B,u}} \& {\Omega_{B,f_!d}^1} \\
	\&\& {\Omega_{B,f_!d'}^1}
	\arrow[two heads, from=1-1, to=1-2]
	\arrow[curve={height=-18pt}, two heads, from=1-1, to=1-3]
	\arrow["{\widehat{f_u}}"', from=1-1, to=2-1]
	\arrow[two heads, from=1-2, to=1-3]
	\arrow["{\widehat{f_d}}", from=1-2, to=2-2]
	\arrow[from=1-3, to=3-3]
	\arrow[two heads, from=2-1, to=2-2]
	\arrow[curve={height=6pt}, two heads, from=2-1, to=3-3]
	\arrow["\lrcorner"{anchor=center, pos=0.125, rotate=180}, draw=none, from=2-2, to=1-1]
	\arrow["{\exists!}"{description}, two heads, from=2-2, to=3-3]
\end{tikzcd}\]

Now we have left to prove that there is a natural isomorphism $\CalcB(F_!\Omega_{A,d}^1,\Omega_{B,\delta}^1)\cong\CalcA(\Omega_{A,d}^1,F^*\Omega_{B,\delta}^1)$, since this, together with the first part, will ensure that also $F^*$ is a functor.
Again, since both $\CalcA$ and $\CalcB$ are posetal, it suffices to prove that 
\begin{center}
	there exists $F_!\Omega_{A,d}^1\to\Omega_{B,\delta}^1\in\CalcB$ $\quad\Longleftrightarrow\quad$ there exists $\Omega_{A,d}^1\to F^*\Omega_{B,\delta}^1\in\CalcA$. 
\end{center}
\begin{itemize}
	\item[($\Rightarrow$)] Let us consider a morphism $h\colon F_!\Omega_{A,d}^1\to\Omega_{B,\delta}^1\in\CalcB$. Then, the map $h\circ\widehat{f_d}\colon f_!\Omega^1_{A,d}\to\Omega_{B,\delta}$ corresponds to a map $h_d\colon \Omega^1_ {A,d}\to f^\ast\Omega^1_{B,\delta}$ through the adjunction $f_!\dashv f^\ast$. 
	In order to find the desired morphism, we will proceed as follows, looking at the diagram below. 

	\begin{enumerate}[(1)]
		\item By definition of $f^\delta$, the adjunction  $f_!\dashv f^\ast$, and the definition of $h_d$, we can see that $h_d\circ p_d=f^\ast q_\delta\circ f_u$. Hence, by the universal property of the kernel $f^\ast N_{B,\delta}$, there exists a unique $k_d$ such that $f^\ast i_\delta\circ k_d=f_u\circ i_d$. 
		\item Therefore, by the universal property of the pullback $ N_{A,f^\ast \delta}$, there exists a unique $k\colon N_{A,d}\to N_{A,f^\ast \delta}$ such that $i_{f^\ast\delta}\circ k=i_d$.
		\item Finally,  $p_{f^\ast\delta}\circ i_d=p_{f^\ast\delta}\circ i_{f^\ast\delta}\circ k=0$. Since $p_d$ is the cokernel of its kernel $i_d$, there exists a unique map $\Omega^1_{A,d}\to F^\ast\Omega^1_{B,\delta}$ whose composite with $p_d$ is $p_{f^\ast\delta}$, as wanted. 
	\end{enumerate}
\[\begin{tikzcd}[ampersand replacement=\&]
	{N_{A,d}} \&\&\&\& {\Omega^1_{A,d}} \\
	\& {N_{A,f^\ast\delta}} \& {\Omega^1_{A,u}} \& {F^*\Omega_{B,\delta}^1} \\
	\& {f^*N_{B,\delta}} \& {f^*\Omega_{B,u}} \& {f^*\Omega^1_{B,\delta}}
	\arrow["{\textcolor{teal}{(2)}\exists!k}"{description}, dashed, from=1-1, to=2-2]
	\arrow["{i_d}", tail, from=1-1, to=2-3]
	\arrow["{\textcolor{teal}{(1)}\,\exists!k_d}"', curve={height=12pt}, dashed, from=1-1, to=3-2]
	\arrow["{\textcolor{teal}{(3)}\,\exists!}"{description}, dashed, from=1-5, to=2-4]
	\arrow["{h_d}", curve={height=-12pt}, from=1-5, to=3-4]
	\arrow["{i_{f^\ast\delta}}"', tail, from=2-2, to=2-3]
	\arrow["{k^\delta}"', from=2-2, to=3-2]
	\arrow["\lrcorner"{anchor=center, pos=0.125}, draw=none, from=2-2, to=3-3]
	\arrow["{p_d}", two heads, from=2-3, to=1-5]
	\arrow[two heads, from=2-3, to=2-4]
	\arrow["{f_u}", from=2-3, to=3-3]
	\arrow["{f^\delta}"{pos=0.4}, from=2-4, to=3-4]
	\arrow["{f^\ast i_\delta}"', tail, from=3-2, to=3-3]
	\arrow["{f^\ast q_\delta}"', two heads, from=3-3, to=3-4]
\end{tikzcd}\]
	\item[($\Leftarrow$)] Let us start with a morphism $t\colon\Omega_{A,d}^1\to F^\ast\Omega_{B,\delta}=\Omega_{A,f^\ast\delta}$. 
	Then, postcomposing with $f^\delta$ we get a morphism $f^\delta\circ t\colon \Omega_{A,d}^1\to\Omega^1_{A,f^\ast\delta}\to f^\ast\Omega_{B,\delta}^1$. 
	This corresponds uniquely, through the adjunction $f_!\dashv f^\ast$, to a morphism $\widetilde{f^\delta t}\colon f_!\Omega_{A,d}^1\to \Omega_{B,\delta}^1$. 
	We notice that, since the diagram below left commutes (by definition of $f^\delta$), then the outer diagram below right commutes as well (using the adjunction $f_!\dashv f^\ast$).
	Therefore, thanks to the universal property of the  pushout below right, we find the desired morphism $F_!\Omega_{A,d}^1\to\Omega_{B,\delta}^1$.
\begin{center}
	
	\begin{tikzcd}[ampersand replacement=\&]
		\& {\Omega^1_{A,d}} \\
		{\Omega^1_{A,u}} \& {F^*\Omega_{B,\delta}^1} \\
		{f^*\Omega_{B,u}} \& {f^*\Omega^1_{B,\delta}}
		\arrow["t", from=1-2, to=2-2]
		\arrow["{p_d}", two heads, from=2-1, to=1-2]
		\arrow[two heads, from=2-1, to=2-2]
		\arrow["{f_u}"', from=2-1, to=3-1]
		\arrow["{f^\delta}"{pos=0.4}, from=2-2, to=3-2]
		\arrow["{f^\ast q_\delta}"', two heads, from=3-1, to=3-2]
	\end{tikzcd}\hspace{1cm}
	\begin{tikzcd}[ampersand replacement=\&]
		{f_!\Omega^1_{A,u}} \& {f_!\Omega^1_{A,d}} \\
		{\Omega^1_{B,u}} \& {F_!\Omega_{A,d}^1} \\
		\&\& {\Omega_{B,\delta}^1}
		\arrow["{f_!p_d}", two heads, from=1-1, to=1-2]
		\arrow["{\widehat{f_u}}"', from=1-1, to=2-1]
		\arrow["{\widehat{f_d}}", from=1-2, to=2-2]
		\arrow["{\widetilde{f^\delta t}}", curve={height=-12pt}, from=1-2, to=3-3]
		\arrow[two heads, from=2-1, to=2-2]
		\arrow["{q_\delta}"', curve={height=12pt}, two heads, from=2-1, to=3-3]
		\arrow["\lrcorner"{anchor=center, pos=0.125, rotate=180}, draw=none, from=2-2, to=1-1]
		\arrow["{\exists!}"{description}, dashed, from=2-2, to=3-3]
	\end{tikzcd}
	
\end{center}

\end{itemize}
\end{proof}

\begin{rmk}
	It is interesting to notice that the two squares below do not commute, but there are natural transformations as shown. This follows immediately by looking at the definitions of $F_!$ and $F^\ast$.
	\begin{center}
		\begin{tikzcd}[ampersand replacement=\&]
			\CalcA \& \CalcB \\
			\AModA \& \BModB
			\arrow["{F_!}", from=1-1, to=1-2]
			\arrow[""{name=0, anchor=center, inner sep=0}, from=1-1, to=2-1]
			\arrow[""{name=1, anchor=center, inner sep=0}, from=1-2, to=2-2]
			\arrow["{f_!}"', from=2-1, to=2-2]
			\arrow[shorten <=25pt, shorten >=25pt, Rightarrow, from=0, to=1]
		\end{tikzcd}
		\hspace{0.5cm}
		\begin{tikzcd}[ampersand replacement=\&]
			\CalcB \& \CalcA \\
			\BModB \& \AModA
			\arrow["{F^\ast}", from=1-1, to=1-2]
			\arrow[""{name=0, anchor=center, inner sep=0}, from=1-1, to=2-1]
			\arrow[""{name=1, anchor=center, inner sep=0}, from=1-2, to=2-2]
			\arrow["{f^\ast}"', from=2-1, to=2-2]
			\arrow[shorten <=25pt, shorten >=25pt, Rightarrow, from=1, to=0]
		\end{tikzcd}
	\end{center}
\end{rmk}

\subsection{First order differential calculi and left adjoints}\label{subsec:calc_and_left_adjoints}

Let $\V$ be an additive monoidal category with finite limits and colimits, such that the tensor product preserves finite colimits in both variables. 
Denote by $\Calc(\V)$ the category whose objects are pairs 
$$(A,\Omega_d)$$
where $A\in\Mon(\V)$ and $\Omega_d$ is a first order differential calculus on $A$. A morphism 
$$ (f,g)\colon (A,\Omega_d)\to (A',\Omega_{d'}) $$ 
is given by $f\colon A\to A'$ in $\tx{Mon}(\V)$ and $g\colon \Omega_d\to f^*\Omega_{d'}$ in $\AModA$, such that $gd=d'f$ in $\V$. 
Note that a map $g\colon \Omega_d\to \Omega_{d'}$ in $\V$ defines a morphism  $g\colon \Omega_d\to f^*\Omega_{d'}$ in $\AModA$ as above if and only if the two squares below commute.
\begin{center}
	\begin{tikzpicture}[baseline=(current  bounding  box.south), scale=2]
		
		\node (01) at (-1,0.8) {$A\otimes \Omega_d$};
		\node (a1) at (0.1,0.8) {$\Omega_d$};
		\node (b1) at (1.2,0.8) {$\Omega_{d}\otimes A$};
		\node (c1) at (2.4,0.8) {$\Omega_{d}$};
		
		\node (00) at (-1,0) {$A'\otimes\Omega_{d'}$};
		\node (a0) at (0.1,0) {$\Omega_{d'}$};
		\node (b0) at (1.2,0) {$\Omega_{d'}\otimes A'$};
		\node (c0) at (2.4,0) {$\Omega_{d'}$};

		\path[font=\scriptsize]

		(01) edge [->] node [above] {$\cdot$} (a1)
		(01) edge [->] node [left] {$f\otimes g$} (00)
		(b1) edge [->] node [above] {$\cdot$} (c1)
		
		(a1) edge [->] node [right] {$g$} (a0)
		(b1) edge [->] node [left] {$g\otimes f$} (b0)
		(c1) edge [->] node [right] {$g$} (c0)
		
		(00) edge [->] node [below] {$\cdot$} (a0)
		(b0) edge [->] node [below] {$\cdot$} (c0);
	\end{tikzpicture}	
\end{center}

\begin{prop}\label{prop:univ-calc-adjunction}
	The first component projection $\pi\colon \Calc(\V)\to\Mon(\V)$ has:\begin{enumerate}
		\item a right adjoint, which is given by sending $A$ to $(A,0)$; 
		\item a left adjoint, which is given by sending $A$ to the universal calculus $(A,\Omega^1_{A,u})$.
	\end{enumerate} 
\end{prop}
\begin{proof}
	(i). We need to show that
	$$ \Mon(\V)(A,B)\cong \Calc(\V)((A,\Omega_d), (B,0)) $$
	naturally in $(A,\Omega_d)$ in $\Calc(\V)$ and $B\in\Mon(\V)$.  But, since the first order differential calculus on $B$ is trivial, every map $(A,\Omega_d)\to (B,0) $ is of the form $(f,0)$, and the commutativity condition is always satisfied. Thus the isomorphism follows.
	
	(ii). We need to prove that 
	$$\Calc(\V)((A,\Omega^1_{A,u}), (B,\Omega_d)) \cong \Mon(\V)(A,B) $$
	naturally in the two variables. Clearly, given a map
	$$ (f,g)\colon (A,\Omega^1_{A,u})\to (B,\Omega_{d}) $$
	in $\Calc(\V)$ we obtain a map $f\colon A\to B$ in $\Mon(\V)$, by forgetting $g$. To conclude we need to prove that for any such $f$ there is a unique $g\colon \Omega^1_{A,u}\to f^*\Omega_d$ such that $gd_u=df$ in $\V$. \\
	Consider the following diagram 
	\begin{center}
		\begin{tikzpicture}[baseline=(current  bounding  box.south), scale=2]
			
			\node (01) at (-1,0.8) {$A$};
			\node (a1) at (0.1,0.8) {$\Omega^1_{A,u}$};
			\node (b1) at (1.2,0.8) {$A\otimes A$};
			
			\node (00) at (-1,0) {$B$};
			\node (a0) at (0.1,0) {$\Omega^1_{B,u}$};
			\node (b0) at (1.2,0) {$B\otimes B$};
			
			\node (a-1) at (0.1,-0.5) {$\Omega_d$};
			
			\path[font=\scriptsize]

			(01) edge [->] node [above] {$d_u^A$} (a1)
			(01) edge [->] node [left] {$f$} (00)
			(a1) edge [>->] node [above] {$\iota^A_u$} (b1)
			
			(a1) edge [dashed, ->] node [right] {$ f_u$} (a0)
			(b1) edge [->] node [right] {$f\otimes f$} (b0)
			
			(00) edge [->] node [above] {$d_u^B$} (a0)
			(a0) edge [>->] node [below] {$\iota^B_u$} (b0)
			
			(00) edge [->] node [below] {$d$} (a-1)
			(a0) edge [->] node [right] {$e$} (a-1);
		\end{tikzpicture}	
	\end{center} 
where the solid part of the diagram commutes since $f$ is a morphism of monoids, and the map $e$ is uniquely determined as a morphism in $\Calc_B$. The morphism $f_u$ is given by \Cref{prop:fu-bimod-map} and defines a map of $A$-bimodules $\Omega^1_{A,u}\to f^*\Omega^1_{B,u}$. Thus we can set
$$ g\colon\Omega^1_{A,u}\xrightarrow{\ f_u\ } f^*\Omega^1_{B,u}\xrightarrow{\ f^*(e)\ }f^*\Omega_d.$$
It remains to show that such a $g$ is uniquely determined from $f$. Consider the diagram
\begin{center}
	\begin{tikzpicture}[baseline=(current  bounding  box.south), scale=2]
		
		\node (a1) at (-0.3,0.8) {$A\otimes A$};
		\node (b1) at (1.2,0.8) {$A\otimes \Omega_u$};
		\node (c1) at (2.4,0.8) {$\Omega_u$};
		
		\node (a0) at (-0.3,0) {$B\otimes B$};
		\node (b0) at (1.2,0) {$B\otimes \Omega_d$};
		\node (c0) at (2.4,0) {$\Omega_d$};
		
		\path[font=\scriptsize]

		(a1) edge [->] node [above] {$1_A\otimes d_u$} (b1)
		(b1) edge [->] node [above] {$\cdot$} (c1)
		
		(a1) edge [->] node [left] {$f\otimes f$} (a0)
		(b1) edge [->] node [right] {$f\otimes g$} (b0)
		(c1) edge [->] node [right] {$g$} (c0)
		
		(a0) edge [->] node [below] {$1_B\otimes d$} (b0)
		(b0) edge [->] node [below] {$\cdot$} (c0);
	\end{tikzpicture}	
\end{center} 
where the left-hand-side commutes since $gd_u=df$, and the right-hand-side commutes since $g$ is a morphism of left $A$-modules. But the top composite is an epimorphism (surjectivity condition on $\Omega_u$). It follows that $g$ is unique with the property that $g(1_A\cdot d_u)= (1_B\cdot d)(f\otimes f)$;
and thus is uniquely determined by $f$ and by being a morphism in $\Calc(\V)$. 
\end{proof}

Similarly to how we defined $\Calc(\V)$, we now introduce the category $\Mod(\V)$ whose objects are pairs 
$$(A,M)$$
where $A\in\Mon(\V)$ and $M\in\AModA$. A morphism 
$$ (f,g)\colon (A,M)\to (B,N) $$ 
in $\Mod(\V)$ is given by $f\colon A\to B$ in $\Mon(\V)$ and $g\colon M\to f^*N$ in $\AModA$. Identities and compositions are inherited from the categories $\Mon(\V)$ and $\AModA$, for a monoid $A$.

\begin{defi}\label{square_zero}
The \emph{square-zero extension functor} is the functor
$$ U\colon \Mod(\V)\longrightarrow\Mon(\V) $$
that sends $(A,M)$ to the monoid whose underlying object is $A\oplus M\in \V$, and whose unit and multiplication are

\begin{center}
	\begin{tikzpicture}[baseline=(current  bounding  box.south), scale=2]
		
		\node (00) at (0,0.5) {$(A\oplus M)\otimes (A\oplus M)$};
		\node (a0) at (0,0) {$(A\otimes A)\oplus ((A\otimes M)\oplus (M\otimes A)\oplus (M\otimes M))$};
		\node (b0) at (0,-0.7) {$A\oplus M.$};

		\node (c'0) at (-2.8,0.4) {$I$};
		\node (b'0) at (-2.8,-0.4) {$A\oplus M$};
		
		\path[font=\scriptsize]
		
		(00) edge [->] node [right] {$\cong $} (a0)
		(a0) edge [->] node [right] {$(m_A, \mu^l_M+\mu^r_M+0)$} (b0)
		
		(c'0) edge [->] node [left] {$(i_A,0_M)$} (b'0);
	\end{tikzpicture}	
\end{center}
A morphism $(f,g)\colon (A,M)\to (B,N)$ in $\Mod(\V)$ is sent to $U(f,g):=f\oplus g$ (note that $N=f^*N$ as objects of $\V$).
\end{defi}

\begin{prop}\label{universal-calc-left-adj}
	The functor $U\colon \Mod(\V)\to\Mon(\V)$ has a left adjoint $L$ which is given pointwise by sending a monoid $A$ to $LA=(A,\Omega^1_{A,u})$, where $\Omega^1_{A,u}$ is the universal first order differential calculus on $A$.
\end{prop}
\begin{proof}
	We need to prove that there is a bijection
	$$ \Mod(\V)((A,\Omega^1_{A,u}),(B,N))\cong \Mon(\V)(A,U(B,N)) $$
	natural in $(B,N)\in\Mod(\V)$.
	
	Given $(f,g)\colon (A,\Omega^1_{A,u})\to (B,N)$ we define
	$$h\colon A\xrightarrow{(1_A,d_u)}U(A,\Omega^1_{A,u})\xrightarrow{U(f,g)}U(B,N), $$ 
	where $(1_A,d_u)\colon A \to U(A,\Omega^1_{A,u})$ is a morphism in $\Mon(\V)$ since that amounts to $d_u$ satisfying the Leibniz rule; in fact, this is going to be the unit of the adjunction at the component $A$.
	
	Conversely, given a morphism $h=(h_1,h_2)\colon A\to B\oplus N=U(B,N)$ in $\Mon(\V)$ we define $(f,g)\colon (A,\Omega^1_{A,u})\to (B,N)$ by $f:=h_1$ and 
	$$ g\colon\Omega^1_{A,u}\xrightarrow{\iota_u^1} A\otimes A\xrightarrow{h_1\cdot h_2} N. $$
 {%
}	Then $(f,g)$ is a map in $\Mod(\V)$  , that is, $g$ is a morphism of $A$-bimodules $\Omega^1_{A,u}\to h_1^*N$. Note that $h_1\cdot h_2\colon A\otimes A\to h_1^*N$ is a morphism of left $A$-modules, but in general not of right $A$-modules, so we argue as follows. Since $h$ is a  morphism of monoids, so is $h_1$, and $h_2\circ m=h_1\cdot h_2+h_2\cdot h_1$. Left linearity: $(h_1\cdot h_2)\circ(m\otimes 1_A)=h_1\cdot(h_1\cdot h_2)$, hence
	$$g\circ l_u=(h_1\cdot h_2)\circ(m\otimes 1_A)\circ(1_A\otimes\iota^1_u)=h_1\cdot g.$$
	Right linearity:
	$$(h_1\cdot h_2)\circ(1_A\otimes m)=h_1\cdot(h_1\cdot h_2)+h_1\cdot(h_2\cdot h_1)=(h_1\cdot h_2)\circ(m\otimes 1_A)+(h_1\cdot h_2)\cdot h_1;$$
	precomposing with $\iota^1_u\otimes 1_A$ and using $m\circ\iota^1_u=0$ we obtain $g\circ r_u=(h_1\cdot h_2)\circ(1_A\otimes m)\circ(\iota^1_u\otimes 1_A)=g\cdot h_1$.
	
	We need to prove that these operations are one the inverse of the other; the fact that they are natural follows by construction since we are applying functorial operations. Given $h\colon A\to B\oplus N$, this amounts to show that the composite 
	$$ A\xrightarrow{d_u}\Omega^1_{A,u}\xrightarrow{\iota_u^1} A\otimes A\xrightarrow{h_1\otimes h_2}B\otimes N\xrightarrow{\mu_N} N. $$
	coincides with $h_2$. Using the definition of $d_u$ we can rewrite this as in the top left:
	\begin{align}
	 (h_1\cdot h_2)\circ (\eta_A\otimes 1_A-1_A\otimes \eta_A)&= (h_1\cdot h_2)\circ (\eta_A\otimes 1_A) - (h_1\cdot h_2)\circ (1_A\otimes \eta_A)\notag\\
	&= (\eta_B\cdot  h_2 )- (h_1\cdot h_2(\eta_A))\tag{$\pi_1h$ monoid map}\\
	&= h_2 - (h_1\cdot h_2(\eta_A))\tag{$\eta_B$ unit}\\
	&= h_2\tag{$h$ map of monoids}.
	\end{align} 
	where the last equality is obtained using that $h$ is a map of monoid and how the monoid structure of $B\oplus N$ is defined.
	
	Conversely, given $(f,g)\colon (A,\Omega^1_{A,u})\to (B,N)$, it is enough to prove that 
	$$ \Omega^1_{A,u}\xrightarrow{\iota_u^1} A\otimes A\xrightarrow{f\otimes (g\circ d_u)}B\otimes N\xrightarrow{\mu_N} N. $$
	is equal to $g$. Consider then the following chain of equalities
	\begin{align}
		(f\cdot_N(g\circ d_u))\circ \iota_u&= (1_A\cdot_{f^*N}(g\circ d_u))\circ \iota_u\tag{def of $\lambda_{f^*N}$}\\
		&= (1_A\cdot_{f^*N}g)\circ(1_A\otimes d_u)\circ  \iota_u\tag{functoriality}\\
		&= g\circ (1_A\cdot d_u)\circ \iota_u\tag{$g$ module map}\\
		&= g\circ (-d_u\cdot 1_A)\circ \iota_u\tag{def of $\iota_u$}\\
		&= g\tag{$(-d_u\cdot 1_A)\circ \iota_u=1_{\Omega_u^1}$}.
	\end{align} 
\end{proof}

\begin{rmk}\label{rmk:L-factors-through-calc}
	In \Cref{universal-calc-left-adj} above, we defined a functor $L\colon\Mon(\V)\to\Mod(\V)$, which on objects is the same as the left adjoint of $\pi\colon\Calc(\V)\to\Mon(\V)$ defined in \Cref{prop:univ-calc-adjunction}, which we call here $L_u\colon\Mon(\V)\to\Calc(\V)$. 
	Moreover, we will show that $L$ on morphisms is also the same as $L_u$, showing therefore that the following triangle is commutative.
	 \[\begin{tikzcd}[ampersand replacement=\&]
	 	\& {\Mod(\V)} \\
	 	{\Mon(\V)} \& {\Calc(\V)}
	 	\arrow["L", from=2-1, to=1-2]
	 	\arrow["{L_u}"', from=2-1, to=2-2]
	 	\arrow[tail, from=2-2, to=1-2]
	 \end{tikzcd}\]
	In order to see this, we recall that \Cref{universal-calc-left-adj} tells us that, for any $f\colon A\to B\in\Mon(\V)$, then $Lf$ is defined as the top horizontal composite in the diagram below. 
	\[\begin{tikzcd}[ampersand replacement=\&]
		{\Omega^1_{A,u}} \&\& {A\otimes A} \&\& {B\otimes\Omega^1_{B,u}} \& {\Omega^1_{B,u}} \\
		{\Omega^1_{B,u}} \&\& {B\otimes B} \&\& {B\otimes B\otimes B} \\
		\&\&\&\&\& {B\otimes B}
		\arrow["{\iota^1_{A,u}}", tail, from=1-1, to=1-3]
		\arrow["{f_u}"', from=1-1, to=2-1]
		\arrow["{f\otimes (d_u^B\circ f)}", from=1-3, to=1-5]
		\arrow["{f\otimes f}"{description}, from=1-3, to=2-3]
		\arrow["{l_u^B}", from=1-5, to=1-6]
		\arrow["{1_B\otimes \iota^1_{B,u}}"{description}, from=1-5, to=2-5]
		\arrow["{\iota^1_{B,u}}", tail, from=1-6, to=3-6]
		\arrow["{\iota^1_{B,u}}"'{pos=0.6}, tail, from=2-1, to=2-3]
		\arrow["{\iota^1_{B,u}}"', curve={height=12pt}, tail, from=2-1, to=3-6]
		\arrow["{1_B\otimes d_u^B}"{description}, from=2-3, to=1-5]
		\arrow["{1_B\otimes(i_B\otimes 1_B-1_B\otimes i_B)}"', from=2-3, to=2-5]
		\arrow["{m^B\otimes1_B}"{description}, from=2-5, to=3-6]
	\end{tikzcd}\]
	
	Since $\iota_{B,u}^1$ is a mono, we get that $Lf=f_u$ if and only if $\iota_{B,u}^1\circ Lf=\iota_{B,u}^1\circ f_u$. 
	The latter equality follows because all the squares and triangles in the diagram above commute for the following reasons (from left to right, top to bottom): definition of $f_u$; $\otimes$ functoriality; definition of $d_u^B$; definition of $l_u^B$; first addendum goes to the identity by unit laws of $B$ and second to zero by definition of $\iota_{B,u}^1$.
	Therefore, $Lf=f_u$ and so $Lf$ coincides with the action on morphisms of the left adjoint $\Mon(\V)\to\Calc(\V)$ described in \Cref{prop:univ-calc-adjunction}.
\end{rmk}

Consider now another forgetful functor
$$ V\colon\Mod(\V)\longrightarrow\Mon(\V) $$
which acts by taking the first component projection, so that $V(A,M)=A$. For any  $A\in\Mon(\V)$, we define the fiber $V^{-1}(A)$ as the subcategory of $\Mod(\V)$ consisting of those objects whose image is $A$ and those morphisms whose image is $1_A$. It follows by definition that $V^{-1}(A)\cong\AModA$.

In fact, $V$ is a Grothendieck fibration. Meaning that it is given by applying the Grothendieck construction to the pseudofunctor
$$ \Phi\colon \Mon(\V)^\op\longrightarrow\mathbf{Cat} $$
defined by $\Phi(A):=V^{-1}(A)=\AModA$ and $\Phi(f)=f^*$. 

\begin{rmk}\label{lp-cat}
	Below we consider the notion of locally presentable category. Given a regular cardinal $\lambda$, we say that a category $\K$ is {\em locally $\lambda$-presentable} if it is cocomplete and has a strong generator $\G$ made of $\lambda$-presentable objects. This means that a map $f$ in $\K$ is an isomorphisms if and only if $\K(G,f)$ is for any $G\in \G$, and for each functor $\K(G,-)\colon \K\to\bf{Set}$, for $G\in\G$, preserves $\lambda$-filtered colimits. A category $\K$ is locally presentable if it is locally $\lambda$-presentable for some $\lambda$. A functor $F\colon\K\to{\mathcal{L}}$ is called accessible if it preserves $\lambda$-filtered colimits for some regular cardinal $\lambda$. For a standard reference on local presentability we refer the reader to~\cite{AR94:libro, MP89:libro}. 
\end{rmk}

\begin{prop}\label{l.p.mod.mon}
	Assume that $\V$ is locally presentable and the tensor product is accessible in both arguments. 
	Then $\Mon(\V)$ and $\Mod(\V)$ are also locally presentable. Moreover, the functors $$U,V\colon \Mod(\V)\to \Mon(\V)$$ introduced above are both continuous and accessible.
\end{prop}
\begin{proof}
	{The category $\Mon(\V)$ is locally presentable by the result recalled in \cite[Section~3]{porst2019colimits}.}
	
	Regarding $\Mod(\V)$, since it is obtained by applying the Grothendieck construction to $ \Phi\colon \Mon(\V)^\op\longrightarrow\mathbf{Cat} $, to prove that it is locally presentable it is enough to show that $\Phi(A)$ is locally presentable for any object $A$, $\Phi(f)$ is continuous and accessible for any arrow $f$, and that $\Phi(-)$ sends {sufficiently} filtered colimits in $\Mon(\V)$ to bilimits in $\mathbf{Cat}$ (this follows from~\cite[Theorem~5.3.4]{MP89:libro} plus the fact that completeness is preserved for the total category of a fibration if all $\Phi(A)$ are complete and all $\Phi(f)$ continuous).
	
	By \Cref{forgetful}, for each $A\in\Mon(\V)$ the category $\AModA$ is the category of algebras for the accessible monad $A\otimes-\otimes A$ on $\V$. Hence it is locally presentable by \cite[2.78]{AR94:libro}. Similarly, given $f\colon A\to B$ in $\Mon(\V)$ the functor $\Phi(f)=f^*\colon\BMod_B\to\AModA$ is continuous (by \Cref{restriction-properties}) and accessible (since {sufficiently filtered} colimits are computed as in $\V$ both in $\AModA$ and $\BMod_B$). 
	
	Finally, given a {sufficiently} filtered colimit $A\cong\colim_{i\in D}A_i $ in $\Mon(\V)$, we need to check that
	$$ \AModA \simeq \lim_{i\in D}\  (_{A_i}\!\Mod_{A_i}). $$ 
	An object of the limit on the right, is by definition a family $(M_i\in_{A_i}\!\Mod_{A_i})_{i\in D}$ such that $M_j=f_{ij}^*M_i$ for any $f_{ij}\colon A_j\to A_i$ in the filtered diagram $D$. It follows that all the $M_i$ have same underlying object $M\in\V$, and that the left and right multiplications are compatible with the connecting maps. One can then define $\mu_M:=\colim_i (\mu_{M_i})$ and $\nu_M:=\colim_i (\nu_{M_i})$ to obtain an $A$-bimodule $(M,\mu,\nu)$. Let $(f_i\colon A_i\to A)_{i\in D}$ be the colimit cocone in $\Mon(\V)$, then it is easy to see that $M_i=f_i^*M$ for any $i\in D$. Therefore we get a bijection between the objects of $ \AModA$ and those of $\lim_{i\in D}\ ( _{A_i}\!\Mod_{A_i})$. It is easy to see that such bijection extends also at the level of arrows providing an equivalence of categories.
	
	The fact that $V$ is accessible and preserves limits follows again from \cite[Theorem~5.3.4]{MP89:libro}. Moreover, $U$ is continuous and accessible by \cite[Theorem~1.66]{AR94:libro} since it has a left adjoint (Proposition~\ref{universal-calc-left-adj}).
\end{proof}

\subsection{Bimonoids and Hopf bimodules}\label{sect:bimon}

In this subsection we consider the case of an object equipped with both a multiplication and comultiplication, and modules over these structures. 
We then use these notions to define the notion of Hopf bimodule in a braided monoidal category.

First, we need to recall the dual notions of monoids and bimodules, called comonoids and cobimodules. 
These are actually instances of monoids and bimodules; indeed, we define the category $\CoMon(\V)$ of comonoids in a monoidal category $\V$ as $$\CoMon(\V):=\Mon(\V^\op)^\op;$$ see \cite[Section~2.2]{porst2008categories}.
Similarly, given comonoids $C$ and $D$ in $\V$, the category $\Bicomod{C}{D}$ of $C$-$D$ bicomodules is defined as $$\Bicomod{C}{D}:=\Bimod{C}{D}(\V^\op)^\op\footnote{We explicitly write here where the bimodules are taken to make the duality clearer. As in the case of monoids then one has to take the dual again to get the right notion of morphism.}$$  where we see $C$ and $D$ as monoids in $\V^\op$. 
Explicitly, a comonoid is an object of $\V$ with a \emph{comultiplication} $\delta\colon C\to C\otimes C  \in\V$ and a {\em counit} $\epsilon\colon C\to I$ satisfying dual axioms to associativity and right and left unital axioms, which we might refer to as \emph{coassociativity} and \emph{counital} axioms.
Moreover, comonoid maps from $(C,\delta,\epsilon)$ to $(D,\phi,\kappa)$ are still maps $h\colon C\to D$ compatible with  comultiplication and counits. 
Similarly, a $C$-$D$ bicomodule $M$ has coactions $\lambda\colon M\to C\otimes M$ and $\rho\colon M\to M\otimes D$ satisfying dual axioms. 

Furthermore, in this section we will need to assume that $\V$ is braided, so that $\VMon$ inherits a monoidal structure from $\V$ (see discussion before Remark 5.2 in \cite{JoyalStreet:brad-mon-cat})\footnote{The case for $\V$ symmetric and regarding comonoids was also stated in \cite{Fox1976}.}. 
We recall that, given two monoids $(A,m,i),(B,n,j)\in\Mon(\V)$, their tensor product is defined as below, where we omit the needed associator isomorphisms.  
$$(A\otimes B,\ A\otimes B\otimes A \otimes B\xrightarrow{1_A\otimes \beta_{B,A}\otimes 1_B}A\otimes A\otimes B \otimes B\xrightarrow{m\otimes n}A\otimes B,\ I\xrightarrow{i\otimes j}A\otimes B ).$$

Now, we are ready to recall the definition of bimonoid, which can both be seen as a monoid in $\CoMon(\V)$ or a comonoid in $\Mon(\V)$; see \cite[Section~4]{porst2008categories} for details.  

\begin{defi}
	A bimonoid $(A,m,i,\delta,\epsilon)$ in $\V$ is a comonoid object in $\VMon$. That is,  a quintuple $(A,m,i,\delta,\epsilon)$ such that:\begin{enumerate}
		\item $(A,m,i)$ is a monoid in $\V$,
		\item $(A,\delta,\epsilon)$ is a comonoid in $\V$,
		\item $\delta$ and $\epsilon$ are monoid morphisms (or equivalently, $m$ and $i$ are comonoid morphisms).
	\end{enumerate}  
	A morphism between any two bimonoids is a map between the underlying objects which is both a monoid and a comonoid morphism. We denote by $\VBimon$ the category of bimonoids and morphisms between them.
\end{defi}

\begin{defi}\label{def:hopf-bimodule}
	Let $\phi\colon A\to B$ be a morphism of bimonoids in $\V$. A \emph{Hopf bimodule along $\phi$} is a quintuple $(M,l,r,\lambda,\rho)$ where:\begin{enumerate}
		\item $(M,l,r)$ is an $A$-bimodule,
		\item $(M,\lambda,\rho)$ is a $B$-bicomodule,
		\item $\lambda$ and $\rho$ are morphisms of $A$-bimodules for the diagonal actions on $B\otimes M$ and $M\otimes B$, induced by $\delta_A$, the actions on $M$, and multiplication on $B$ along $\phi$. Explicitly, these four compatibility conditions are given by:
		\begingroup
		\begin{align}
		\lambda\circ l&=\big((m_B\circ(\phi\otimes1_B))\otimes l\big)\circ(1_A\otimes\beta_{A,B}\otimes1_M)\circ(\delta_A\otimes\lambda),\label{eq:hopf-bimod-ll}\\
		\lambda\circ r&=\big((m_B\circ(1_B\otimes \phi))\otimes r\big)\circ(1_B\otimes\beta_{M,A}\otimes1_A)\circ(\lambda\otimes\delta_A),\label{eq:hopf-bimod-lr}\\
		\rho\circ l&=\big(l\otimes(m_B\circ(\phi\otimes1_B))\big)\circ(1_A\otimes\beta_{A,M}\otimes1_B)\circ(\delta_A\otimes\rho),\label{eq:hopf-bimod-rl}\\
		\rho\circ r&=\big(r\otimes(m_B\circ(1_B\otimes \phi))\big)\circ(1_M\otimes\beta_{B,A}\otimes1_A)\circ(\rho\otimes\delta_A).\label{eq:hopf-bimod-rr}
		\end{align}
		\endgroup
	\end{enumerate}
	\begingroup\emergencystretch=1em
	A morphism between such is a map between the underlying objects which is both a morphism of $A$-bimodules and $B$-bicomodules. We denote by $\HBimod{B}{B}{A}{A}$ the category of Hopf bimodules along $\phi$ and morphisms between them. When $B=A$ and $\phi=1_A$, these are the \emph{$A$-Hopf bimodules}, and we denote the corresponding category by $\HBimodA$.
\par\endgroup
\end{defi}

\begin{rmk}\label{rmk:bi-co-mod-as-hopf-mod}
	For a bimonoid $A$, taking $\phi=\epsilon_A\colon A\to I$ gives $\HBimod{I}{I}{A}{A}\cong\Bimod{A}{A}$, while taking $\phi=i_A\colon I\to A$ gives $\HBimod{A}{A}{I}{I}\cong\Bicomod{A}{A}$
\end{rmk}

\begin{rmk}\label{rmk:bicomonoids-as-hopf-bimodules}
	If we consider a bimonoid $A$ in $\V$, then we can consider the category of $A$-Hopf bimodules $\HBimodA$. 
	Clearly, using its multiplication and comultiplication, $A$ can be seen both as a $A$-bimodule and $A$-bicomodule. 
	With these structures, $(A,m,m,\delta,\delta)$ is an $A$-Hopf bimodule since all four compatibility conditions reduce to compatibility of the multiplication and comultiplication.
\end{rmk}

\begin{prop}\label{prop:restr-for-hopf-mod}
	For morphisms of bimonoids $A'\xrightarrow{g}A\xrightarrow{\phi}B\xrightarrow{h}B'$, there is a functor
	\[
	\qquad
	\begin{bmatrix}h&h\\g&g\end{bmatrix}^{\ast}\colon\HBimod{B}{B}{A}{A}\to\HBimod{B'}{B'}{A'}{A'}
	\]
	which sends a Hopf bimodule along $\phi$ $(M,l,r,\lambda,\rho)$ to the Hopf bimodule along $h\circ \phi\circ g$
	\[
	\big(M,l\circ(g\otimes1_M),r\circ(1_M\otimes{g}),(h\otimes1_M)\circ\lambda,(1_M\otimes{h})\circ\rho\big),
	\]
	and is the identity on morphisms.
\end{prop}

\begin{proof}
The underlying $A'$-bimodule structure is the one defined by the restriction of scalars for bimodules (\Cref{prop:def-restr-of-scalars}).
	Similarly, the $B'$-bicomodule structure is given by the restriction of scalars on bicomodules, which is defined using the description of bicomodules as bimodules in $\V^\op$.
	Hence, the only thing that we need to check is that the new coactions are morphisms of $A'$-bimodules with respect to the diagonal actions along $k:=h\circ \phi\circ g$.
	We show that the left coaction $\lambda':=(h\otimes 1_M)\circ\lambda$ is a morphism of left $A'$-modules, where $l':=l\circ(g\otimes 1_M)$; the remaining three verifications are completely analogous.
	Consider the following chain of equalities:
	\begin{align}
		\lambda'\circ l'&=(h\otimes 1_M)\circ\lambda\circ l\circ(g\otimes 1_M)\tag{def}\\
		&=\big((h\circ m_B\circ(\phi\otimes 1_B))\otimes l\big)\circ(1_A\otimes\beta_{A,B}\otimes 1_M)\circ((\delta_A\circ g)\otimes\lambda)\tag{by \eqref{eq:hopf-bimod-ll}}\\
		&=\big((h\circ m_B\circ(\phi\otimes 1_B))\otimes l\big)\circ(1_A\otimes\beta_{A,B}\otimes 1_M)\circ(g\otimes g\otimes 1_{B\otimes M})\circ(\delta_{A'}\otimes\lambda)\tag{$g$ comon. morph.}\\
		&=\big((h\circ m_B\circ((\phi\circ g)\otimes 1_B))\otimes l'\big)\circ(1_{A'}\otimes\beta_{A',B}\otimes 1_M)\circ(\delta_{A'}\otimes\lambda)\tag{naturality of $\beta$}\\
		&=\big((m_{B'}\circ(k\otimes h))\otimes l'\big)\circ(1_{A'}\otimes\beta_{A',B}\otimes 1_M)\circ(\delta_{A'}\otimes\lambda)\tag{$h$ mon. morph.}\\
		&=\big((m_{B'}\circ(k\otimes 1_{B'}))\otimes l'\big)\circ(1_{A'}\otimes\beta_{A',B'}\otimes 1_M)\circ(\delta_{A'}\otimes\lambda').\tag{naturality of $\beta$}
	\end{align}
	The last composite is the diagonal left action of $A'$ on $B'\otimes M$ along $k$ precomposed with $1_{A'}\otimes\lambda'$, as required.
	To conclude, we notice that any morphism in $\HBimod{B}{B}{A}{A}$ is, by construction, also a morphism of $A'$-bimodules and of $B'$-bicomodules for the restricted structures.
\end{proof}

\begin{rmk}
	By construction, the following squares commute, where the vertical arrows are the functors forgetting either the bicomodule or bimodule structure of the Hopf bimodules, and the bottom functors are the restriction of scalars applied to bimodules and bicomodules. 
	\begin{center}
\begin{tikzcd}[ampersand replacement=\&]
	{\HBimod{B}{B}{A}{A}} \&\&\& {\HBimod{B'}{{B'}}{A'}{{A'}}} \\
	{\Bimod{A}{A}} \&\&\& {\Bimod{A'}{{A'}}}
	\arrow["{\begin{bmatrix} h & h \\ g & g \end{bmatrix}^\ast }"{description}, from=1-1, to=1-4]
	\arrow[from=1-1, to=2-1]
	\arrow[from=1-4, to=2-4]
	\arrow["{(g,g)^\ast}"', from=2-1, to=2-4]
\end{tikzcd}
\hspace{0.5cm}
\begin{tikzcd}[ampersand replacement=\&]
	{\HBimod{B}{B}{A}{A}} \&\&\& {\HBimod{B'}{{B'}}{A'}{{A'}}} \\
	{\Bicomod{B}{B}} \&\&\& {\Bicomod{B'}{{B'}}}
	\arrow["{\begin{bmatrix} 		h & h \\ 		g & g 	\end{bmatrix}^\ast }"{description}, from=1-1, to=1-4]
	\arrow[from=1-1, to=2-1]
	\arrow[from=1-4, to=2-4]
	\arrow["{(h,h)^\ast}"', from=2-1, to=2-4]
\end{tikzcd}
	\end{center}

Taking $B=B'=I$, $h=1_I$, and $\phi=\epsilon_A$ recovers restriction of scalars for bimodules along a bimonoid morphism $g$. Taking $A'=A=I$, $g=1_I$, and $\phi=i_B$ similarly recovers corestriction of scalars for bicomodules along a bimonoid morphism $h$.
\end{rmk}

\begin{rmk} 
	For any bimonoid morphism $\phi\colon(A,m_A,i_A,\delta_A,\epsilon_A)\rightarrow(B,m_B,i_B,\delta_B,\epsilon_B)\in\VBimon$, since $\phi$ is both a monoid and comonoid morphism, applying \Cref{prop:restr-for-hopf-mod} to $A\xrightarrow{1_A}A\xrightarrow{1_A}A\xrightarrow{\phi}B$ and $A\xrightarrow{\phi}B\xrightarrow{1_B}B\xrightarrow{1_B}B$ respectively, we get two functors as below, valued in the category of Hopf bimodules along $\phi$.
	\begin{center}
	$\phi^\ast_{\comod}:=\begin{bmatrix}
		\phi & \phi \\
		1_A & 1_A
	\end{bmatrix}^\ast\colon\HBimodA\to\HBimod{B}{B}{A}{A}$ \\
$\phi^\ast_{\Mod}:=\begin{bmatrix}
	1_B & 1_B \\
	\phi & \phi
\end{bmatrix}^\ast\colon\HBimodB\to\HBimod{B}{B}{A}{A}$
	\end{center}

\end{rmk}

\begin{prop}\label{Hopf-universal-def}
	Let $A$ be a bimonoid and $(\Omega^1_{u},d_u)$ be the universal first order differential calculus over the underlying monoid on $A$. Then:\begin{enumerate}
		\item $\Omega^1_{u}$ is an $A$-Hopf bimodule and $\iota_u\colon \Omega^1_u\to A\otimes A$ is a morphism of $A$-Hopf bimodules;
		\item $d_u\colon A\to\Omega^1_u$ is a morphism of $A$-bicomodules.
	\end{enumerate}
\end{prop}
\begin{proof}
	To prove the first point, recall that $A\otimes A$ is an $A$-bimodule with actions $m\otimes 1_A$ and $1_A\otimes m$. It is moreover an $A$-bicomodule, with the coactions $\lambda_{A^{\otimes 2}}$ and $\rho_{A^{\otimes 2}}$ recalled below, and a direct verification, using the bimonoid axioms of $A$ and the naturality of the braiding, shows that these structures make $A\otimes A$ into an $A$-Hopf bimodule.
	It is then enough to define left and right coactions on $\Omega^1_u$ that make $\iota_u$ into a morphism of $A$-comodules. This will suffice since then the bicomodule axioms and condition (iii) of \Cref{def:hopf-bimodule} for $\Omega^1_u$ are implied by the fact that they hold for $A\otimes A$, using that $\iota_u$ is a morphism of $A$-bimodules and a split monomorphism in $\V$ (so that its tensor products with identity morphisms are again monomorphisms).
	
	We define the right coaction as the composite
	$$ \rho_u\colon\Omega^1_u\xrightarrow{\iota_u}A^{\otimes 2}\xrightarrow{\delta\otimes\delta}A^{\otimes 4}\xrightarrow{1_A\otimes\beta\otimes 1_A}A^{\otimes 4}\xrightarrow{1_{A^{\otimes 2}}\otimes m}A^{\otimes 3}\xrightarrow{(1_A\cdot d)\otimes 1_A} \Omega^1_u\otimes A,$$
	where the middle composite $\rho_{A^{\otimes 2}}:=(1_{A^{\otimes 2}}\otimes m) \circ(1_A\otimes\beta\otimes 1_A)\circ (\delta\otimes\delta)$ is the right coaction of $A\otimes A$.
	To show that $\rho_u$ is compatible with $\rho_{A^{\otimes 2}}$ through $\iota_u$ we need to prove that 
	$$ \rho_{A^{\otimes 2}} \circ \iota_u = (\iota_u\otimes 1_A) \circ \rho_u.$$
	To that end, consider the following equalities
	\begin{align}
		(\iota_u\otimes 1_A) \circ \rho_u&= (\iota_u\otimes 1_A)\circ ((1_A\cdot d)\otimes 1_A )\circ \rho_{A^{\otimes 2}}\circ \iota_u\tag{def of $\rho_u$}\\
		&= (\iota_u\otimes 1_A)\circ (l_u\otimes 1_A)\circ(1_A\otimes d\otimes 1_A )\circ \rho_{A^{\otimes 2}}\circ \iota_u\tag{def of $\cdot$}\\
		&= (m\otimes 1_{A^{\otimes 2}})\circ (1_A\otimes \iota_u\otimes 1_A)\circ(1_A\otimes d\otimes 1_A )\circ \rho_{A^{\otimes 2}}\circ \iota_u\tag{ $\iota_u\in\AModA$}\\
		&= (m\otimes 1_{A^{\otimes 2}})\circ (1_A\otimes(i\otimes 1_A-1_A\otimes i)\otimes 1_A)\circ \rho_{A^{\otimes 2}}\circ \iota_u\tag{def of $d$}\\
		&= (1_{A^{\otimes 3}}- m\otimes i\otimes 1_A)\circ \rho_{A^{\otimes 2}}\circ \iota_u\tag{$i$ unit}.
	\end{align} 
	To conclude it is enough to prove that $(m\otimes i\otimes 1_A)\circ \rho_{A^{\otimes 2}}\circ \iota_u=0$. Thus:
	{\begin{align}
		(m\otimes i\otimes 1_A)\circ \rho_{A^{\otimes 2}}\circ \iota_u&=(m\otimes i\otimes 1_A)\circ (1_{A^{\otimes 2}}\otimes m) \circ(1_A\otimes\beta\otimes 1_A)\circ (\delta\otimes\delta)\circ \iota_u\tag{def of $\rho_{A^{\otimes 2}}$}\\
		&= (1_A\otimes i\otimes 1_A)\circ (m\otimes m) \circ(1_A\otimes\beta\otimes 1_A)\circ (\delta\otimes\delta)\circ \iota_u\tag{functoriality}\\
		&= (1_A\otimes i\otimes 1_A)\circ\delta\circ m \circ \iota_u\tag{$m$ map of comonoids}\\
		&= 0\tag{$m\circ\iota_u=0$}.
	\end{align}} 
	Symmetrically, we define the left coaction as the composite
	$$ \lambda_u\colon\Omega^1_u\xrightarrow{\iota_u}A^{\otimes 2}\xrightarrow{\delta\otimes\delta}A^{\otimes 4}\xrightarrow{1_A\otimes\beta\otimes 1_A}A^{\otimes 4}\xrightarrow{m\otimes 1_{A^{\otimes 2}}}A^{\otimes 3}\xrightarrow{1_A\otimes (1_A\cdot d)} A\otimes \Omega^1_u,$$
	where, as for $\rho_u$, we use the retraction $1_A\cdot d$ of $\iota_u$ (recall that $(1_A\cdot d)\circ\iota_u=-(d\cdot 1_A)\circ\iota_u=1_{\Omega^1_u}$ by the Leibniz rule and \Cref{prop:univ-1-d-cal-well-def}, so that using $d\cdot 1_A$ here would produce the wrong sign), and where the middle composite $\lambda_{A^{\otimes 2}}:=(m\otimes 1_{A^{\otimes 2}})\circ(1_A\otimes\beta\otimes 1_A)\circ(\delta\otimes\delta)$ is the left coaction of $A\otimes A$.
	An analogous chain of equalities shows that $\lambda_u$ is compatible with $\lambda_{A^{\otimes 2}}$ through $\iota_u$.
	
	For point (ii), we shall prove that $d$ is a morphism of right $A$-comodules, the same arguments apply to show that it is also a morphism of left $A$-comodules. Thus, we need to show that $ (\rho_u)\circ d = (d\otimes 1_A)\circ \delta$;
	since $\iota_u$ is a split monomorphism, this is equivalent to showing that 
	$$  (\iota_u\otimes 1_A)\circ(\rho_u)\circ d = (\iota_u\otimes 1_A)\circ (d\otimes 1_A)\circ \delta. $$
	Now
	\begin{align}
		(\iota_u\otimes 1_A)\circ(\rho_u)\circ d &= (\rho_{A^{\otimes 2}})\circ \iota_u\circ d \tag{$(1_A\cdot d)\circ \iota_u=1_{\Omega_u^1}$}\\
		&= (1_{A^{\otimes 2}}\otimes m) \circ(1_A\otimes\beta\otimes 1_A)\circ (\delta\otimes\delta)\circ \iota_u\circ d\tag{def of $\rho_{A^{\otimes 2}}$}\\
		&= (1_{A^{\otimes 2}}\otimes m) \circ(1_A\otimes\beta\otimes 1_A)\circ (\delta\otimes\delta) (i\otimes 1_A-1_A\otimes i)\tag{def of $d$}\\
		&= (1_{A^{\otimes 2}}\otimes m) \circ(1_A\otimes\beta\otimes 1_A)\circ  (i\otimes i \otimes \delta-\delta\otimes i\otimes i).\tag{$i$ comon. morph.}
	\end{align} 
	Expanding the subtraction and considering them separately we obtain:
	\begin{align}
		(1_{A^{\otimes 2}}\otimes m) \circ(1_A\otimes\beta\otimes 1_A)\circ  (i\otimes i \otimes \delta) &=(1_{A^{\otimes 2}}\otimes m) \circ(1_A\otimes\beta\otimes 1_A)\circ (i\otimes i\otimes 1_A\otimes 1_A)\circ \delta\tag{functoriality}\\
		&= (1_{A^{\otimes 2}}\otimes m) \circ (i\otimes 1_A\otimes i\otimes 1_A)\circ\delta\tag{$\beta\circ (i\otimes 1_A)=1_A\otimes i$}\\
		&= (i\otimes 1_A\otimes 1_A)\circ \delta.\tag{$i$ unit for $m$}\\
		&= i\otimes \delta.\tag{functoriality}
	\end{align} 
	Similarly,
	{%
	$$ (1_{A^{\otimes 2}}\otimes m) \circ(1_A\otimes\beta\otimes 1_A)\circ  (\delta\otimes i \otimes i)= (1_A\otimes i\otimes 1_A)\circ\delta. $$}
	Thus, continuing the original chain of equalities we obtain:
	\begin{align}
		(\rho_{A^{\otimes 2}})\circ \iota_u\circ d &= i\otimes \delta-(1_A\otimes i\otimes 1_A)\circ\delta\notag\\
		&= ((i\otimes 1_A-1_A\otimes i)\otimes 1_A)\circ \delta\tag{functoriality}\\
		&= (\iota_u\otimes 1_A)\circ (d\otimes 1_A)\circ \delta.\tag{def of $d$}
	\end{align} 
\end{proof}

Next we define the notion of first order differential calculus that is relevant for the case of Hopf bimodules.

\begin{defi}
	Let $A$ be a bimonoid. An {\em Hopf first order differential calculus over $A$} is the data of an object $\Omega\in\HBimod{A}{A}{A}{A}$ and $d\colon A\to\Omega$ in $\V$ such that:\begin{enumerate}
		\item $(\Omega,d)$ is a first order differential calculus over the underlying monoid on $A$;
		\item $d\colon A \to \Omega$ is a morphism of $A$-bicomodules.
	\end{enumerate}
\end{defi}

It follows from Proposition~\ref{Hopf-universal-def} that $(\Omega^1_{u},d_u)$ is a Hopf first order differential calculus over $A$.

\begin{lemma}\label{lemma:Hopf.diff.induced.map}
	Let $A$ be a bimonoid and $A_0$ be its underlying monoid. Let $(\Omega_d,d)$ be a first order differential calculus over $A_0$, and assume that $\Omega_d\in\HBimod{A}{A}{A}{A}$ is equipped with an $A$-Hopf bimodule structure extending its $A$-bimodule structure. Then the uniquely induced morphism $f\colon\Omega^1_u\to \Omega_d$ in $\AModA$, with $f\circ d_u=d$, is a morphism of $A$-Hopf bimodules if and only if $(\Omega_d,d)$ is an Hopf first order differential calculus over $A$.
\end{lemma}
\begin{proof}
	Recall that the induced $f$ is unique by \Cref{universal-prop}. Assume first that $f$ is a morphism of $A$-Hopf bimodules; we only need to show that then $d$ is an $A$-comodule morphism. To prove that it is a morphism of left comodules we need to prove that the diagram below left commutes.
	\begin{center}
		\begin{tikzpicture}[baseline=(current  bounding  box.south), scale=2]
			
			\node (a1) at (0.2,0.8) {$A$};
			\node (b1) at (1.5,0.8) {$\Omega_d$};
			\node (a0) at (0.2,0) {$A\otimes A$};
			\node (b0) at (1.5,0) {$A\otimes \Omega_d$};
			
			\node (a11) at (3.2,0.8) {$A$};
			\node (b11) at (4.5,0.8) {$\Omega^1_u$};
			\node (c11) at (5.8,0.8) {$\Omega_d$};
			\node (a01) at (3.2,0) {$A\otimes A$};
			\node (b01) at (4.5,0) {$A\otimes \Omega^1_u$};
			\node (c01) at (5.8,0) {$A\otimes \Omega_d$};
			
			\path[font=\scriptsize]
			
			(a1) edge [->] node [above] {$d$} (b1)
			(a0) edge [->] node [below] {$1_A\otimes d$} (b0)
			(a1) edge [->] node [left] {$\delta$} (a0)
			(b1) edge [->] node [right] {$\lambda_{\Omega_d}$} (b0)
			
			(a11) edge [->] node [above] {$d_u$} (b11)
			(a01) edge [->] node [below] {$1_A\otimes d_u$} (b01)
			(a11) edge [->] node [left] {$\delta$} (a01)
			(b11) edge [->] node [right] {$\lambda_{\Omega_u}$} (b01)
			(b11) edge [->] node [above] {$f$} (c11)
			(b01) edge [->] node [below] {$1_A\otimes f$} (c01)
			(c11) edge [->] node [right] {$\lambda_{\Omega_d}$} (c01);
		\end{tikzpicture}	
	\end{center} 
	But this is equal to the pasting of the two squares on the right, which commute since $d_u$ an $f$ are morphisms of $A$-comodules. The same arguments apply when considering the right coaction instead.
	
	Conversely, assume that $(\Omega_d,d)$ is an Hopf first order differential calculus over $A$. Recall the left coaction $\lambda_{A^{\otimes 2}}$ of $A\otimes A$ from the proof of \Cref{Hopf-universal-def}, and consider similarly the diagonal left coaction of $\Omega_d\otimes A$:
	$$\lambda_{\Omega_d\otimes A}:=(m\otimes 1_{\Omega_d\otimes A})\circ(1_A\otimes\beta_{\Omega_d,A}\otimes 1_A)\circ(\lambda_{\Omega_d}\otimes\delta)\colon \Omega_d\otimes A\longrightarrow A\otimes\Omega_d\otimes A.$$
	In the diagram below,
	\begin{center}
		\begin{tikzpicture}[baseline=(current  bounding  box.south), scale=2]
			
			\node (a11) at (3,0.8) {$A\otimes A$};
			\node (b11) at (4.5,0.8) {$\Omega_d\otimes A$};
			\node (c11) at (6,0.8) {$\Omega_d$};
			\node (a01) at (3,0) {$A^{\otimes 3}$};
			\node (b01) at (4.5,0) {$A\otimes \Omega_d\otimes A$};
			\node (c01) at (6,0) {$A\otimes \Omega_d$};
			
			\path[font=\scriptsize]
			
			(a11) edge [->] node [above] {$d\otimes 1_A$} (b11)
			(a01) edge [->] node [above] {$1_A\!\otimes\! d\!\otimes\! 1_A$} (b01)
			(a11) edge [->] node [left] {$\lambda_{A^{\otimes 2}}$} (a01)
			(b11) edge [->] node [right] {$\lambda_{\Omega_d\otimes A}$} (b01)
			(b11) edge [->] node [above] {$r_{\Omega_d}$} (c11)
			(b01) edge [->] node [above] {$1_A\otimes r_{\Omega_d}$} (c01)
			(c11) edge [->] node [right] {$\lambda_{\Omega_d}$} (c01)
			
			(a11) edge [bend left, ->>] node [above] {$d\cdot 1_A$} (c11)
			(a01) edge [bend right, ->>] node [below] {$1_A\otimes (d\cdot 1_A)$} (c01);
		\end{tikzpicture}	
	\end{center} 
	both squares commute: that on the left since $d$ is a left $A$-comodule morphism and by naturality of the braiding, and that on the right by \eqref{eq:hopf-bimod-lr}, that is, since $\lambda_{\Omega_d}$ is a morphism of right $A$-modules with respect to the diagonal action on $A\otimes\Omega_d$. In other words, $d\cdot 1_A\colon A\otimes A\to\Omega_d$ is a morphism of left $A$-comodules.
	Now consider the diagram below, where the outer square coincides with the diagram above since $f\circ d_u=d$ and $f$ is a morphism of right $A$-modules.
	\begin{center}
		\begin{tikzpicture}[baseline=(current  bounding  box.south), scale=2]
			
			\node (a11) at (3,0.8) {$A\otimes A$};
			\node (b11) at (4.5,0.8) {$\Omega^1_u$};
			\node (c11) at (6,0.8) {$\Omega_d$};
			\node (a01) at (3,0) {$A^{\otimes 3}$};
			\node (b01) at (4.5,0) {$A\otimes \Omega^1_u$};
			\node (c01) at (6,0) {$A\otimes \Omega_d$};
			
			\path[font=\scriptsize]
			
			(a11) edge [->>] node [above] {$d_u\cdot 1_A$} (b11)
			(a01) edge [->] node [below] {$1_A\otimes (d_u\cdot 1_A)$} (b01)
			(a11) edge [->] node [left] {$\lambda_{A^{\otimes 2}}$} (a01)
			(b11) edge [->] node [right] {$\lambda_{\Omega_u}$} (b01)
			(b11) edge [->] node [above] {$f$} (c11)
			(b01) edge [->] node [below] {$1_A\otimes f$} (c01)
			(c11) edge [->] node [right] {$\lambda_{\Omega_d}$} (c01);
		\end{tikzpicture}	
	\end{center} 
	Now, the square on the left commutes by the argument above applied to $(\Omega^1_u,d_u)$, which is a Hopf first order differential calculus by \Cref{Hopf-universal-def}, and $d_u\cdot 1_A$ is an epimorphism by the surjectivity condition on $\Omega^1_u$. Thus the fact that the outer square commutes implies that the square on the right commutes as well; that is, $f$ is a morphism of left $A$-comodules. A symmetric argument, using $1_A\cdot d$ and the right coaction $\rho_{A^{\otimes 2}}$ of $A\otimes A$, shows that $f$ is also a morphism of right $A$-comodules.
\end{proof}

As a consequence we can prove a universal property of the universal first order differential calculus over a bimonoid.

\begin{prop}\label{prop:univ-hopf-mod}
	Given a bimonoid $A$, then $(\Omega^1_u,d_u)$ is universal among the Hopf first order differential calculi over $A$. That is, for any Hopf first order differential calculus $(\Omega_d,d)$ over $A$ there exists a unique morphism $f\colon\Omega^1_u\to\Omega_d$ in $\HBimod{A}{A}{A}{A}$ such that $f\circ d_u=d$.
\end{prop}
\begin{proof}
	This is a direct consequence of \Cref{lemma:Hopf.diff.induced.map} above.
\end{proof}

Denote by $\HCalc(\V)$ the category whose objects are pairs 
$$(A,\Omega_d)$$
where $A\in\VBimon$ and $\Omega_d$ is a Hopf first order differential calculus over $A$. A morphism 
$$ (\phi,g)\colon (A,\Omega_d)\to (B,\Omega_{d'}) $$ 
is given by $\phi\colon A\to B$ in $\VBimon$ and $g\colon \phi^\ast_{\comod}(\Omega_d)\to \phi^\ast_{\Mod}(\Omega_{d'})$ in $\HBimod{B}{B}{A}{A}$, such that $gd=d'\phi$ in $\V$.
In the next Proposition we prove an analogous result to \Cref{prop:univ-calc-adjunction} for bimonoids and Hopf bimodules. 

\begin{rmk}\label{rmk:maps-of-hopf-calc-descr}
	Note that a map $g\colon \Omega_d\to \Omega_{d'}$ in $\V$ defines a morphism  $g\colon \phi^\ast_{\comod}(\Omega_d)\to \phi^\ast_{\Mod}(\Omega_{d'})$ in $\HBimod{B}{B}{A}{A}$ as above if and only if the four squares below commute. 
	The first two say that $g$ is a morphism of $A$-bimodules, while the last two that $g$ is a morphism of $B$-bicomodules. 
	\begin{center}
		\begin{tikzpicture}[baseline=(current  bounding  box.south), scale=2]
			
			\node (01) at (-1,0.8) {$A\otimes \Omega_d$};
			\node (a1) at (0.1,0.8) {$\Omega_d$};
			\node (b1) at (1.2,0.8) {$\Omega_{d}\otimes A$};
			\node (c1) at (2.4,0.8) {$\Omega_{d}$};
			
			\node (00) at (-1,0) {$B\otimes\Omega_{d'}$};
			\node (a0) at (0.1,0) {$\Omega_{d'}$};
			\node (b0) at (1.2,0) {$\Omega_{d'}\otimes B$};
			\node (c0) at (2.4,0) {$\Omega_{d'}$};

			\path[font=\scriptsize]

			(01) edge [->] node [above] {$l_{\Omega_d}$} (a1)
			(01) edge [->] node [left] {$\phi\otimes g$} (00)
			(b1) edge [->] node [above] {$r_{\Omega_{d}}$} (c1)
			
			(a1) edge [->] node [right] {$g$} (a0)
			(b1) edge [->] node [left] {$g\otimes \phi$} (b0)
			(c1) edge [->] node [right] {$g$} (c0)
			
			(00) edge [->] node [below] {$l_{\Omega_{d'}}$} (a0)
			(b0) edge [->] node [below] {$r_{\Omega_{d'}}$} (c0);
		\end{tikzpicture}	
\hspace{0.5cm}
		\begin{tikzpicture}[baseline=(current  bounding  box.south), scale=2]
		
		\node (01) at (-1,0.8) {$\Omega_d$};
		\node (a1) at (0.1,0.8) {$A\otimes \Omega_d$};
		\node (b1) at (1.2,0.8) {$\Omega_{d}$};
		\node (c1) at (2.4,0.8) {$\Omega_{d}\otimes A$};
		
		\node (00) at (-1,0) {$\Omega_{d'}$};
		\node (a0) at (0.1,0) {$B\otimes\Omega_{d'}$};
		\node (b0) at (1.2,0) {$\Omega_{d'}$};
		\node (c0) at (2.4,0) {$\Omega_{d'}\otimes B$};

		\path[font=\scriptsize]

		(01) edge [->] node [above] {$\lambda_{\Omega_d}$} (a1)
		(01) edge [->] node [left] {$g$} (00)
		(b1) edge [->] node [above] {$\rho_{\Omega_{d}}$} (c1)
		
		(a1) edge [->] node [right] {$\phi\otimes g$} (a0)
		(b1) edge [->] node [left] {$g$} (b0)
		(c1) edge [->] node [right] {$g\otimes \phi$} (c0)
		
		(00) edge [->] node [below] {$\lambda_{\Omega_{d'}}$} (a0)
		(b0) edge [->] node [below] {$\rho_{\Omega_{d'}}$} (c0);
	\end{tikzpicture}	
	\end{center} 
\end{rmk}

\begin{prop}
	The first component projection $\pi\colon \HCalc(\V)\to\VBimon$ has:\begin{enumerate}
		\item a right adjoint, which is given by sending $A$ to $(A,0)$; 
		\item a left adjoint, which is given by sending $A$ to the universal Hopf first order differential calculus $(A,\Omega^1_{A,u})$.
	\end{enumerate} 
\end{prop}

\begin{proof}
(i). We need to show that $\VBimon(A,B)\cong \HCalc(\V)((A,\Omega_d), (B,0))$
naturally in $(A,\Omega_d)$ in $\HCalc(\V)$ and $B\in\VBimon$. But, since the first order differential calculus on $B$ is trivial, every map $(A,\Omega_d)\to (B,0) $ is of the form $(\phi,0)$, and the commutativity condition is always satisfied. Therefore the isomorphism follows. 

(ii). We need to prove that $$\HCalc(\V)((A,\Omega^1_{A,u}), (B,\Omega_d)) \cong \Bimon(\V)(A,B).$$ 
Notice that a morphism of Hopf first order differential calculi $(\phi,g)$ is exactly a morphism of differential calculi such that $\phi\in\CoMon(\V)$ and $g\in\Bicomod{B}{B}$, and similarly a map of bimonoids $\phi$ is a map of monoids which is also a map of comonoids. 
\[\begin{tikzcd}[ampersand replacement=\&]
	{\HCalc(\V)((A,\Omega^1_{A,u}), (B,\Omega_d))} \& {\Calc(\V)((A,\Omega^1_{A,u}), (B,\Omega_d))} \\
	{\Bimon(\V)(A,B)} \& {\Mon(\V)(A,B)}
	\arrow[hook, from=1-1, to=1-2]
	\arrow[dashed, from=1-1, to=2-1]
	\arrow["\cong", from=1-2, to=2-2]
	\arrow[hook, from=2-1, to=2-2]
\end{tikzcd}\]
Therefore, it suffices to prove that the adjunction isomorphism of \Cref{prop:univ-calc-adjunction} (vertical right arrow above) restricts to the required one (vertical left above); that is, we need to prove that given $(\phi,g)\in\Calc(\V)((A,\Omega^1_{A,u}), (B,\Omega_d))$ with $(A,\Omega^1_{A,u}), (B,\Omega_d)\in\HCalc(\V)$),
\begin{center}
	$\phi\in\CoMon(\V)$ and $g\in\Bicomod{B}{B}$ $\Leftrightarrow$ $\phi\in\CoMon(\V)$.
\end{center}
The only non trivial part is proving that if $\phi\in\CoMon(\V)$, then $g\in\Bicomod{B}{B}$. 
Let us recall that in \Cref{prop:univ-calc-adjunction} is constructed as $g=e\phi_u$ where $e\colon \Omega^1_{B,u}\to \Omega_d$ is the map induced by the universal property of $\Omega_u^B$ and $\phi_u$ is defined using the universal property of $\Omega^1_{B,u}$ as a kernel, see below. 
\begin{center}
	\begin{tikzpicture}[baseline=(current  bounding  box.south), scale=2]
		
		\node (a1) at (0.1,0.8) {$\Omega^1_{A,u}$};
		\node (b1) at (1.2,0.8) {$A\otimes A$};
		\node (c1) at (2.4,0.8) {$A$};
		
		\node (a0) at (0.1,0) {$\Omega^1_{B,u}$};
		\node (b0) at (1.2,0) {$B\otimes B$};
		\node (c0) at (2.4,0) {$B$};

		\path[font=\scriptsize]

		(a1) edge [>->] node [above] {$\iota^A_u$} (b1)
		(b1) edge [->] node [above] {$m_A$} (c1)
		
		(a1) edge [dashed, ->] node [left] {$\phi_u$} (a0)
		(b1) edge [->] node [right] {$\phi\otimes \phi$} (b0)
		(c1) edge [->] node [right] {$\phi$} (c0)
		
		(a0) edge [>->] node [below] {$\iota^B_u$} (b0)
		(b0) edge [->] node [below] {$m_B$} (c0);
	\end{tikzpicture}	
\end{center} 
By \Cref{lemma:Hopf.diff.induced.map}, we know that $e$ is a morphism in $\Bicomod{B}{B}$; therefore we are only left to prove that $\phi_u$ is in $\Bicomod{B}{B}$ as well. 
We will only show that $\phi_u$ respects the left coactions since the proof for the right ones is completely analogous. 
Using \Cref{rmk:maps-of-hopf-calc-descr} and postcomposing with the mono $1_B\otimes \iota_u^B$\footnote{Notice that $\iota_u^B$ is a split mono, hence it is preserved by any functor.}, $\phi_u$ is a morphism of left $B$-comodules if and only if the front part of the cube below commutes. 
\[
\scalebox{0.8}{
\begin{tikzcd}[ampersand replacement=\&]
	{\Omega^1_{A,u}} \&\& {A\otimes A} \\
	\& {A\otimes\Omega^1_{A,u}} \&\& {A^{\otimes3}} \\
	{\Omega^1_{B,u}} \&\& {B\otimes B} \\
	\& {B\otimes\Omega^1_{B,u}} \&\& {B^{\otimes3}}
	\arrow["{\iota_u^A}", tail, from=1-1, to=1-3]
	\arrow[from=1-1, to=2-2]
	\arrow["{\phi_u}"', from=1-1, to=3-1]
	\arrow["{\lambda_{A^{\otimes2}}}", from=1-3, to=2-4]
	\arrow["{\phi\otimes \phi}"{description, pos=0.2}, dotted, from=1-3, to=3-3]
	\arrow["{1_{A}\otimes\iota_u^A}"{description}, tail, from=2-2, to=2-4]
	\arrow["{\phi\otimes \phi_u}"{pos=0.3}, from=2-2, to=4-2]
	\arrow["{\phi\otimes \phi\otimes \phi}", from=2-4, to=4-4]
	\arrow["{\iota_u^B}"{description, pos=0.3}, dotted, tail, from=3-1, to=3-3]
	\arrow[from=3-1, to=4-2]
	\arrow["{\lambda_{B^{\otimes2}}}"{description}, dotted, from=3-3, to=4-4]
	\arrow["{1_B\otimes\iota_u^B}"', tail, from=4-2, to=4-4]
\end{tikzcd}
}\]
Then, we conclude noticing that: the back commutes since $\phi$ is a bimonoid morphism and naturality of $\beta$, plus definition of $\phi_u$; the top commutes since $\iota_u^A$ is a morphism of comodules; the bottom commutes since $\iota_u^B$ is a morphism of comodules. 
Therefore the front commutes as well. 
\end{proof}

\section{Constructing canonical first order differential calculi}\label{sec:canonical_calculi}

In this section we derive sufficient conditions on a category $\E$ equipped with a functor $(-)_0\colon\E\to\Mon(\V)$ in order to obtain a notion of first order differential calculus internal to a natural module category associated to $\E$. 
Moreover, we also show sufficient conditions such that this notion induces a first order differential calculus over a monoid in $\V$ (as in \Cref{sec:diff-calc}).
We end the section describing a family of examples coming from the theory of operads. 

\subsection{The pullback construction}\label{pullback_construction}

In this subsection we consider a category $\E$ equipped with a faithful isofibration $(-)_0\colon\E\to\Mon(\V)$ and construct a natural notion of \emph{bimodule} category for objects of $\E$,
which generalizes the case of bimodules over a monoid in $\V$.

For a category $\C$, let $[\Ret,\C]$ be its category of retractions, where $\Ret$ is the free category on a retraction. An object of this category consists of the data  $X\xrightarrow{s}\widehat X\xrightarrow{p}X$ in $\C$, with $ps=1_X$; we call $X$ the base and $\widehat X$ the total component. A morphism is a pair $(f\colon X\to X',\widehat f\colon\widehat X\to \widehat X')$ commuting with the sections and the projections (that is, a natural transformation).

The zero section $s_{A,M}=(1_A,0)$ and projection $p_{A,M}$ of the square-zero extension define a functor
\[
\widehat U\colon\Mod(\V)\longrightarrow[\Ret,\Mon(\V)],\qquad
(A,M)\longmapsto(A\xrightarrow{s_{A,M}}A\oplus M\xrightarrow{p_{A,M}}A),
\]
whose base and total components are induced by $V$ and $U$, respectively, and which sends $(f,g)$ to $(f,f\oplus g)$. 

We define $\EMod$ as the pullback
{\begin{equation}\label{pullback-can-calc}
\begin{tikzcd}[ampersand replacement=\&, column sep=huge]
\EMod \& \Mod(\V) \\
{[\Ret,\E]} \& {[\Ret,\Mon(\V)]}
\arrow["K", from=1-1, to=1-2]
\arrow["\widehat U'"', from=1-1, to=2-1]
\arrow["\lrcorner"{anchor=center, pos=0.125}, draw=none, from=1-1, to=2-2]
\arrow["\widehat U", from=1-2, to=2-2]
\arrow["{[\Ret,(-)_0]}"', from=2-1, to=2-2]
\end{tikzcd}
\end{equation}}
Thus $X\in\EMod$ consists of $KX=(A,M)$ and a retraction
\[
s_X\colon V'X\rightleftarrows U'X\colon p_X
\]
in $\E$ lying over $(s_{A,M},p_{A,M})$. Composing $\widehat U'\colon\EMod\to[\Ret,\E]$ with the evaluation functors sending a retraction $E\to\widehat E\to E$ to $E$ and $\widehat E$, respectively, gives $V',U'\colon\EMod\to\E$. These satisfy $(-)_0V'=VK$ and $(-)_0U'=UK$.
As the notation suggests, we think of $X$ as a module over its base $V'X$.

\begin{rmk}\label{rmk:split-vs-naive}
A morphism $X\to Y$ in $\EMod$ is determined by $\hat u\colon U'X\to U'Y$ and $(f,g)\colon KX\to KY$ with $\hat u_0=f\oplus g$: its base component is $p_Y\hat us_X$, and compatibility with the retractions follows by faithfulness. 

If $(-)_0$ is the inclusion of a full replete subcategory closed under retracts (as it will be in many of our cases), $\EMod$ is therefore the full subcategory of $\Mod(\V)$ on those $(A,M)$ with $A\oplus M\in\E$, equivalently the pullback of $U$ along $(-)_0$.
\end{rmk}

\begin{eg}
	\label{eg:central-bimod-as-pullback}
	Let us consider a braided monoidal category $\V$ and let us look into the full subcategory $\CMon_n(\V)\hookrightarrow\VMon$ of $n$-commutative monoids in monoids. 
	This inclusion, since it is fully faithful and replete, is in particular faithful and a conservative isofibration. 
	We now look at the definition of $\Mod(\CMon_n(\V))$ in this case. 
	\[\begin{tikzcd}[ampersand replacement=\&]
		{\Mod(\CMon_n(\V))} \& {\Mod(\V)} \\
		\CMon_n(\V)  \& \VMon
		\arrow["{f.f.}", hook, from=1-1, to=1-2]
		\arrow["{U'}"', from=1-1, to=2-1]
		\arrow["\lrcorner"{anchor=center, pos=0.125}, draw=none, from=1-1, to=2-2]
		\arrow["U", from=1-2, to=2-2]
		\arrow["{f.f.}"', hook, from=2-1, to=2-2]
	\end{tikzcd}\]
	Since the inclusion of $n$-commutative monoids is fully faithful and closed under submonoids (\Cref{ComMon-suitable}), \Cref{rmk:split-vs-naive} shows that the square above is a pullback, whose objects are $(A,M)\in\Mod(\V)$ such that $U(A,M)=A\oplus M\in\CMon_n(\V)$. 
	Moreover, we can describe more explicitly what the commutativity condition on $A\oplus M$ means. 
	By definition, this means that the diagram below is commutative. 
\[\begin{tikzcd}[ampersand replacement=\&]
	{(A\oplus M)\otimes (A\oplus M)} \& {(A\otimes A)\oplus(A\otimes M)\oplus(M\otimes A)\oplus(M\otimes M)} \\
	{(A\oplus M)\otimes (A\oplus M)} \& {(A\otimes A)\oplus(A\otimes M)\oplus(M\otimes A)\oplus(M\otimes M)} \&\& {A\oplus M}
	\arrow[from=1-1, to=1-2]
	\arrow["\beta^n"', from=1-1, to=2-1]
	\arrow["{(m,l+r+0)}", from=1-2, to=2-4]
	\arrow[from=2-1, to=2-2]
	\arrow["{(m,l+r+0)}"', from=2-2, to=2-4]
\end{tikzcd}\]
	Looking at the two different components, one sees that this equivalent to require:
	\begin{itemize}
		\item $A\in\CMon_n(\V)$;
		\item If $n$ is odd, $M$ satisfies the following axioms:
		\begin{center}
			\begin{tikzcd}[ampersand replacement=\&]
				{M\otimes A} \& {A\otimes M} \\
				\& M
				\arrow["\beta^n", from=1-1, to=1-2]
				\arrow["r"', from=1-1, to=2-2]
				\arrow["l", from=1-2, to=2-2]
			\end{tikzcd}
			\hspace{0.5cm}
			\begin{tikzcd}[ampersand replacement=\&]
				{A\otimes M} \& {M\otimes A} \\
				M
				\arrow["\beta^n", from=1-1, to=1-2]
				\arrow["l"', from=1-1, to=2-1]
				\arrow["r", from=1-2, to=2-1]
			\end{tikzcd}
		\end{center}
		\item If $n$ is even, $M$ satisfies the following axioms:
		\begin{center}
			\begin{tikzcd}[ampersand replacement=\&]
				{M\otimes A} \& {M\otimes A} \\
				\& M
				\arrow["\beta^n", from=1-1, to=1-2]
				\arrow["r"', from=1-1, to=2-2]
				\arrow["r", from=1-2, to=2-2]
			\end{tikzcd}
			\hspace{0.5cm}
			\begin{tikzcd}[ampersand replacement=\&]
				{A\otimes M} \& {A\otimes M} \\
				M
				\arrow["\beta^n", from=1-1, to=1-2]
				\arrow["l"', from=1-1, to=2-1]
				\arrow["l", from=1-2, to=2-1]
			\end{tikzcd}
		\end{center}
	\end{itemize}
	We will refer to $\Mod(\CMon_n(\V))$ as the category of \emph{$n$-commutative bimodules} in $\V$.
	If $n=2$, then the category of $n$-commutative bimodules is equivalent to the category of dyslectic bimodules over dyslectic monoids.
	If $\V$ is symmetric (and not just braided), and $n$ is odd, then just one of the relevant axioms would be sufficient, and further the category of $n$-commutative bimodules is equivalent to the category of central bimodules over commutative monoids.
	Moreover, for $\V$ symmetric and $n$ even, $\Mod(\CMon_n(\V))$ is equivalent to $\Mod(\V)$.
\end{eg}

\begin{eg}\label{eg:Cinfty-rings-pullback}
	We now consider $\V=\bf{Ab}$ and $\E=C^\infty\tx{-Ring}$ to be the category of smooth rings; this is defined as follows. Let $\mathbb E$ be the full subcategory of the category of smooth manifolds spanned by the finite products of $\mathbb R$; then we define 
	$$ C^\infty\tx{-Ring}:=\tx{FP}(\mathbb E,\bf{Set})$$
	as the category of finite-product-preserving functors from $\mathbb E$ into $\bf{Set}$. More explicitly, an object $A\in C^\infty\tx{-Ring}$ is a set $A$ together with for any $\rho\in C^\infty(\mathbb R^n,\mathbb R)$ a function
	$$ \rho_A\colon A^n\longrightarrow A $$
	for which the assignment $\rho\mapsto \rho_A$ is functorial. 
	
	Now note that every $C^\infty$-ring $A$ has an underlying commutative ring $A_0$ whose operations are induced by interpreting the ring operations of $\mathbb R$ (which are smooth functions); for details see~\cite{stel2013cosimplicial}. In particular we have an underlying functor
	$$ (-)_0\colon C^\infty\tx{-Ring}\longrightarrow \tx{Ring}:=\Mon(\bf{Ab}) $$
	which factors through the inclusion $\tx{CRing}\rightarrowtail\tx{Ring}$; this is faithful, conservative, and an isofibration.
	
 It follows that we can define $\Mod(C^\infty\tx{-Ring})$ as in \eqref{pullback-can-calc}. 
By \Cref{eg:central-bimod-as-pullback}, an object consists of a $C^\infty$-ring $A$, a central $A_0$-bimodule $M$, and a $C^\infty$-ring structure on $A_0\oplus M$ for which the zero section and projection are smooth-ring morphisms. 
Such a structure exists and is unique, namely
\[
\rho_{A\oplus M}(\bar a\oplus\bar m)
=\rho_A(\bar a)\oplus\sum_{i=1}^n
\left(\frac{\partial\rho}{\partial x_i}\right)_A(\bar a)\cdot m_i,
\]
for every $\rho\in C^\infty(\mathbb R^n,\mathbb R)$; see \cite[Theorem~20]{stel2013cosimplicial}. 
Uniqueness follows by applying Taylor's formula to the square-zero ideal $M$, using that the zero section is a smooth-ring morphism. 
We denote the resulting object by $(A,M)$.
A morphism $(f,g)\colon(A,M)\to(B,N)$ consists of a smooth-ring morphism $f\colon A\to B$ and an $A_0$-bimodule map $g\colon M\to f_0^*N$; the displayed formula ensures that $f\oplus g$ is a smooth-ring morphism.

Finally, $A_0$ is naturally an $\mathbb R$-algebra, with scalar multiplication induced by the smooth maps $\cdot r\colon\mathbb R\to\mathbb R$. The formula $r\cdot m=(r\cdot_A1)\cdot_Mm$ makes $M$ an $\mathbb R$-vector space. Hence the same description of $\Mod(C^\infty\tx{-Ring})$ is obtained by applying \eqref{pullback-can-calc} over $\Mon(\mathbb R\tx{-Vect})$ instead of $\Mon(\bf{Ab})$.
\end{eg}

We now introduce the property which will guarantee that the main theorem of this section will apply to our chosen functor $(-)_0\colon \E\to \VMon$. 
While this might seem a bit technical, in practice we always use the (stronger) condition outlined in Remark~\ref{rmk:other-suitable} below. { It will be useful to keep in mind the following commutative square.
\begin{center}
\begin{tikzcd}[ampersand replacement=\&, column sep=huge]
	\EMod \& \Mod(\V) \\
	{\E} \& {\Mon(\V)}
	\arrow["K", from=1-1, to=1-2]
	\arrow["U'"', from=1-1, to=2-1]
	\arrow["U", from=1-2, to=2-2]
	\arrow["{(-)_0}"', from=2-1, to=2-2]
\end{tikzcd}
\end{center}
}

\begin{defi}
	\label{def:suitable}
	We say that $(-)_0\colon \E\to\Mon(\V)$ is {\em \suitable} if for any  $E$ in $\E$, any $X$ in $\EMod$ with $KX=(A,M)$, and $f\colon E\to U'X$ in $\E$ with underlying map $f_0=(f_1,f_2)\colon E_0\to A\oplus M$ in $\Mon(\V)$ where $f_1$ is a monomorphism of monoids, there exists $Y$ in $\EMod$ with $KY=(E_0,N)$, together with morphisms as dashed below
	\begin{center}
		\begin{tikzpicture}[baseline=(current  bounding  box.south), scale=2]
			
			\node (c0) at (0.7,0) {$U'Y$};
			\node (a0) at (0,-0.5) {$E$};
			\node (b0) at (1.4,-0.5) {$U'X$};
			
			\path[font=\scriptsize]
			
			(a0) edge [->] node [below] {$f$} (b0)
			(a0) edge [dashed, ->] node [above] {$k\ \ \ $} (c0)
			(c0) edge [dashed, ->] node [above] {$\ \ \ U'l$} (b0);
		\end{tikzpicture}	
	\end{center} 
	in $\E$ with $k_0=(1_{E_0},k_2)\colon E_0\to E_0\oplus N$. (Then one necessarily has $(U'l)_0=f_1\oplus l_2\colon E_0\oplus N\to A\oplus M$ for some $l_2$.)
\end{defi}

\begin{rmk}\label{rmk:other-suitable}
It suffices that, for $X$, $E$ and $f$ as above, the canonical factorization
\[
E_0\xrightarrow{(1_{E_0},f_2)}E_0\oplus f_1^*M
\xrightarrow{f_1\oplus1_M}A\oplus M
\]
lifts to $E\xrightarrow{k}Y'\xrightarrow{l'}U'X$ in $\E$, and that the zero section and projection of $E_0\oplus f_1^*M$ lift to a retraction $E_Y\rightleftarrows Y'$ in $\E$. These lifts define an object $Y\in\EMod$ with $KY=(E_0,f_1^*M)$ and, by \Cref{rmk:split-vs-naive}, a morphism $l\colon Y\to X$ with total component $l'$. Hence $(-)_0$ is \suitable.
\end{rmk}

The following provides a class of examples of \suitable\ functors

\begin{prop}\label{prop:implies-suitable}
	A faithful isofibration $(-)_0\colon\E\to\Mon(\V)$ is \suitable\ provided that one of the following conditions holds:\begin{enumerate}
		\item $(-)_0$ is a Grothendieck fibration;
		\item $(-)_0$ is a full subcategory closed under subobjects.
	\end{enumerate}
\end{prop}
\begin{proof}
	We prove that the stronger property of \Cref{rmk:other-suitable} is satisfied. Let $X$, $E$, and $f$ be as there, and let $s_X\colon V'X\rightleftarrows U'X\colon p_X$ be the retraction of $X$.

	For (i), let $l'\colon Y'\to U'X$ be a cartesian lift of $f_1\oplus 1_M\colon E_0\oplus f_1^*M\to A\oplus M$, and let $c\colon E_Y\to V'X$ be a cartesian lift of $f_1\colon E_0\to A$. Since $f$ lies over $(f_1\oplus1_M)\circ(1_{E_0},f_2)$, the universal property of $l'$ gives $k\colon E\to Y'$ with $l'k=f$ and $k_0=(1_{E_0},f_2)$. Similarly, $s_X\circ c$ lies over $(f_1\oplus 1_M)\circ s_{E_0,f_1^*M}$, so it factors through $l'$ via a map $E_Y\to Y'$ lying over the zero section; and $p_X\circ l'$ lies over $f_1\circ p_{E_0,f_1^*M}$, so it factors through $c$ via a map $Y'\to E_Y$ lying over the projection.

	For (ii), it is enough to notice that $f_1\oplus 1_M$ is a monomorphism, so that $E_0\oplus f_1^*M\in\E$; since $\E$ is full, everything lifts to $\E$, with $E_Y=E$.
\end{proof}

\begin{eg}\label{ComMon-suitable}
	As in \Cref{eg:central-bimod-as-pullback}, let us consider the fully faithful (conservative) isofibration $\CMon_n(\V)\hookrightarrow\VMon$. This is closed under subobjects: if $h\colon A\to B$ is a monomorphism in $\Mon(\V)$ and $B\in \CMon_n(\V)$, then we can consider the following diagram
	\begin{center}
		\begin{tikzpicture}[baseline=(current  bounding  box.south), scale=2, on top/.style={preaction={draw=white,-,line width=#1}}, on top/.default=6pt]
			
			\node (a1) at (0,0.8) {$A\otimes A$};
			\node (b1) at (1.4,0.8) {$A\otimes A$};
			\node (a0) at (0.7,-0.1) {$A$};
			
			\node (a1') at (1.3,1.5) {$B\otimes B$};
			\node (b1') at (2.7,1.5) {$B\otimes B$};
			\node (a0') at (2,0.6) {$B$};
			
			\path[font=\scriptsize]
			
			(a1') edge [->] node [above] {$\beta^n_{B,B}$} (b1')
			(a1') edge [->] node [left ] {$m_B$} (a0')
			(b1') edge [->] node [right] {$m_B$} (a0')
			
			(a1) edge [->] node [above] {$\beta^n_{A,A}$} (b1)
			(a1) edge [->] node [left] {$m_{A}$} (a0)
			(b1) edge [->] node [right] {$m_{A}$} (a0)
			
			(a1) edge [->] node [above] {$h\otimes h\ $} (a1')
			(b1) edge [->, on top] node [above] {$h\otimes h\ $} (b1')
			(a0) edge [>->] node [below] {$h$} (a0');
		\end{tikzpicture}	
	\end{center}
	whose front and back faces appear \Cref{def:comm-mon}, and whose other faces commute by naturality and since $h$ is a morphism of monoids. The back triangle commutes by hypothesis on $B$; thus, since $h$ is a monomorphism, also the front commutes, so $A\in\CMon_n(\V)$. It follows from \Cref{prop:implies-suitable}(ii) above that $\CMon_n(\V)\hookrightarrow\VMon$ is weakly geometrical.
\end{eg}

\begin{eg}\label{eg:Cinfty-rings-suitable}
	Consider the faithful, conservative, isofibration $(-)_0\colon C^\infty\tx{-Ring}\to \Mon(\mathbb R\tx{-Vect}) $ as in \Cref{eg:Cinfty-rings-pullback}. We shall prove that $(-)_0\colon C^\infty\tx{-Ring}\to \Mon(\mathbb R\tx{-Vect}) $ is \suitable\ by showing that the properties of \Cref{rmk:other-suitable} hold. 
	
	Consider $E\in C^\infty\tx{-Ring}$, a pair $(A,M)\in\Mod(C^\infty\tx{-Ring})$ and a morphism 
	$$f=(f_1,f_2)\colon E\to U'(A,M)=A\oplus M$$
	where $f_1$ is a monomorphism, and $A\oplus M$ is endowed with the smooth-ring structure defined in \Cref{eg:Cinfty-rings-pullback}. Note that $f_1\colon E\to A$ is a smooth-ring morphism since $f$ is and so is the projection $U'(A,M)\to A$.\\
	Then, using the explicit description of $\Mod(C^\infty\tx{-Ring})$, we know that $(E,f_1^*M)\in\Mod(C^\infty\tx{-Ring})$ since $f_1^*M$ is a central bimodule over $E_0$. Moreover, the pair
	$$(f_1,1_M)\colon (E,f_1^*M)\to (A,M)$$
	defines a morphism in $\Mod(C^\infty\tx{-Ring})$ since $f_1$ is a map of smooth rings and $1_M\colon f_1^*M\to f_1^*M$ is clearly a map of bimodules.\\ 
	Finally, we need to show that $(1_E,f_2)\colon E\to U'(E,f_1^*M)$ defines a morphism of smooth rings. For any $\rho\in C^\infty(\mathbb R^n,\mathbb R)$ consider the following diagram
	\begin{center}
		\begin{tikzpicture}[baseline=(current  bounding  box.south), scale=2]
			
			\node (02) at (0,1.6) {$E^n$};
			\node (a2) at (1.8,1.6) {$E$};
			\node (01) at (0,0.8) {$(E\oplus f_1^*M)^n$};
			\node (a1) at (1.8,0.8) {$E\oplus f_1^*M$};
			\node (00) at (0,0) {$(A\oplus M)^n$};
			\node (a0) at (1.8,0) {$A\oplus M$};			
			
			\path[font=\scriptsize]
			
			(02) edge [->] node [above] {$\rho_E$} (a2)
			(02) edge [->] node [left] {$(1_E,f_2)^n$} (01)
			(a2) edge [->] node [right] {$(1_E,f_2)$} (a1)
			(01) edge [->] node [above] {$\rho_{U'(E,f_1^*M)}$} (a1)
			(01) edge [>->] node [left] {$(f_1\oplus 1_M)^n$} (00)
			(a1) edge [>->] node [right] {$f_1\oplus 1_M$} (a0)
			(00) edge [->] node [below] {$\rho_{U'(A,M)}$} (a0);
		\end{tikzpicture}	
	\end{center} 
	where the outer and the bottom squares commute because $(f_1,f_2)$ and $(f_1 \oplus 1_M)$ are maps of smooth rings; since $(f_1 \oplus 1_M)$ is a monomorphism it follows that also the top square commutes. Thus $(1_E,f_2)$ is a smooth-ring morphism.
\end{eg}

\subsection{First order $\E$-differential calculi}
In this subsection, we define a generalized notion of first order differential calculus internal to $\Mod(\E)$. 
We also give sufficient conditions to get a notion of first order universal calculus again internal to $\Mod(\E)$, and to get an induced first order differential calculus over a monoid in $\V$. 

For $E\in\E$, let $\EEModE$ be the fibre of $V'\colon\EMod\to\E$ over $E$ (the pullback below left). Equivalently, it is the pullback on the right
\begin{equation}\label{def-EModE}
		\begin{tikzcd}[ampersand replacement=\&]
			\EEModE \& \EMod \\
			1 \& \E
			\arrow[" ", from=1-1, to=1-2]
			\arrow[" "', from=1-1, to=2-1]
			\arrow["\lrcorner"{anchor=center, pos=0.125}, draw=none, from=1-1, to=2-2]
			\arrow["V' ", from=1-2, to=2-2]
			\arrow["E"', from=2-1, to=2-2]
		\end{tikzcd}\qquad\qquad
		\begin{tikzcd}[ampersand replacement=\&]
			\EEModE \& \EModE \\
			E/\E/E \& E_0/\Mon(\V)/E_0
			\arrow["J_E", from=1-1, to=1-2]
			\arrow["\widehat U'_E"', from=1-1, to=2-1]
			\arrow["\lrcorner"{anchor=center, pos=0.125}, draw=none, from=1-1, to=2-2]
			\arrow["\widehat U_E", from=1-2, to=2-2]
			\arrow["(-)_0"', from=2-1, to=2-2]
		\end{tikzcd}
\end{equation}	
where $E/\E/E$ denotes retractions $E\to\widehat E\to E$ with composite $1_E$ and morphisms fixing $E$, and similarly for $E_0/\Mon(\V)/E_0$. The functor $\widehat U_E$ sends $M$ to its canonical square-zero retraction. Thus $M\in\EEModE$ consists of $M_0:=J_E(M)\in\EModE$ and a retraction $E\rightleftarrows\widehat E$ over the canonical one on $E_0\oplus M_0$; we write $U'_E(M):=\widehat E$. The fibre inclusion $\EEModE\to\EMod$ is faithful.

\begin{defi}
	Let $E\in\E$, a {\em generalized first order $\E$-differential calculus over $E$} is the data of $\Omega\in\EEModE$ together with a morphism $d\colon E_0\to \Omega_0$ in $\V$ such that: 
	\begin{itemize}
		\item (strong Leibniz). $E_0\xrightarrow{(1_{E_0},d)} E_0\oplus \Omega_0=U_E'(\Omega)_0$ is a morphism of monoids and, moreover, lifts to $\hat d\colon E\to U'_E(\Omega)$ in $\E$.
	\end{itemize}
\end{defi}

By faithfulness, the lift $\hat d$ is a section of the projection $p_\Omega\colon U'_E(\Omega)\to E$, since $(p_\Omega\hat d)_0=1_{E_0}$. Conversely, every section of $p_\Omega$ has underlying map $(1_{E_0},d)$ satisfying the strong Leibniz condition. The distinguished section $s_\Omega$ corresponds to $d=0$.

In the remark below we show how the condition above is a ``stronger'' Leibniz rule than the one in $\Mod(\V)$, which motivates its name. 

\begin{rmk}\label{Leibniz-square}
	In the context of the above definition the map $d$ always satisfies the Leibniz rule. In fact, $d\colon E_0\to M_0$ satisfies the Leibniz rule if and only if 
	$$ (1_{E_0},d)\colon E_0\to E_0\oplus M_0=U(E_0,M_0) $$
	is a morphism of monoids. To see this, note that $(1_{E_0},d)$ is a morphism of monoids if and only if the square
		\begin{center}
		\begin{tikzpicture}[baseline=(current  bounding  box.south), scale=2]
			
			\node (a0) at (0,0.7) {$E_0\otimes E_0$};
			\node (b0) at (0,0) {$(E_0\otimes E_0)\oplus((E_0\otimes M_0) \oplus(M_0\otimes E_0)\oplus(M_0\otimes M_0))$};
			\node (c0) at (3.5,0.7) {$E_0$};
			\node (d0) at (3.5,0) {$E_0\oplus M_0$};
			
			\path[font=\scriptsize]
			
			(a0) edge [->] node [left] {$(1_{E_0\otimes E_0}, 1_{E_0}\otimes d, d\otimes 1_{E_0}, d\otimes d) $} (b0)
			(a0) edge [->] node [above] {$m_{E_0}$} (c0)
			(b0) edge [->] node [below] {$(m_{E_0}, \mu^l+\mu^r+0)$} (d0)
			(c0) edge [->] node [right] {$(1_{E_0},d)$} (d0);
		\end{tikzpicture}	
	\end{center}
	commutes (the unit is always preserved), where the bottom left corner is an expansion of $(E_0\oplus M_0)\otimes (E_0\oplus M_0)$. Commutativity of the first component is trivial, while that of the second component says exactly that 
	$$ dm= 1_{E_0}\cdot d+d\cdot 1_{E_0}, $$
	so that $d$ satisfies the Leibniz rule.
\end{rmk}

\begin{prop}\label{F'-gives-calculus1}
	Assume that $(-)_0\colon \E\to\Mon(\V)$ is \suitable. If $U'\colon\EMod\to \E$ has a left adjoint $L'$, then for any $E\in\E$: \begin{itemize}
		\item $L'E$ can be identified with an object $\Omega=(\widehat U'L'E,\Omega_0)\in\EEModE$.
Under this identification, $V'L'E=E$ and $KL'E=(E_0,\Omega_0)$, and the unit $\hat d:=\eta_E\colon E\to U'L'E=U'_E\Omega$ has underlying components
		$$\hat d_0\colon E_0\xrightarrow{\ (1_{E_0},d)\ }(U'L'E)_0\cong UKL'E=E_0\oplus\Omega_0$$
		in $\Mon(\V)$.
	\end{itemize}
	It follows that the pair $(\Omega,d\colon E_0\to \Omega_0)$ is a generalized first order $\E$-differential calculus over $E$.
\end{prop}
\begin{proof}
	Write $KL'E=(A,\Omega_0)$ for some $A\in\Mon(\V)$ and $\Omega_0\in\AModA$. Following this notation, then $(\eta_E)_0$ is a morphism of monoids
	$$ (h,d)\colon E_0\to A\oplus \Omega_0=U(A,\Omega_0) $$
	with components $h\colon E_0\to A$ and $d\colon E_0\to \Omega_0$ lying in $\V$.  
	
	By one of the equivalent definitions of adjointness, we know that $(L'E,\eta_E)$ is an initial object in the comma category $E/U'$ whose objects are pairs $(X,f)$ with $X\in\EMod$ and $f\colon E\to U'X$ in $\E$.\\ 
	Let $X$ be the trivial split extension of $E$, with $KX=(E_0,0)$, obtained by lifting $E_0\cong E_0\oplus0$ along the isofibration. Identifying $U'X\cong E$, the pair $(X,1_E)$ defines an object of $E/U'$.
	Thus there exists a unique morphism $f\colon L'E\to X$ in $\EMod$ such that $U'f\circ \eta_E=1_E$. The morphism $f$ has components $U'f$ and $Kf$ with $U'f\colon U'L'E\to E$ in $\E$, and $Kf=(g,0)\colon (A,\Omega_0)\to (E_0,0)$ in $\Mod(\V)$, subject to $(U'f)_0=g\oplus 0$. It follows that in $\Mon(\V)$ the following triangle commutes
	\begin{center}
		\begin{tikzpicture}[baseline=(current  bounding  box.south), scale=2]
			
			\node (c0) at (0.8,0) {$A\oplus \Omega_0$};
			\node (a0) at (0,-0.5) {$E_0$};
			\node (b0) at (1.6,-0.5) {$E_0$};
			
			\path[font=\scriptsize]
			
			(a0) edge [->] node [below] {$1_{E_0}$} (b0)
			(a0) edge [->] node [above] {$(h,d)\ \ \ $} (c0)
			(c0) edge [->] node [above] {$\ \ \ g\oplus 0$} (b0);
		\end{tikzpicture}	
	\end{center} 
	and therefore also $g\circ h=1_{E_0}$, making $h$ a split monomorphism of monoids. 
	By $(-)_0\colon \E\to\Mon(\V)$ being weakly $\V$-geometrical, applied to $\eta_E\colon E\to U'L'E$, we find $Y\in\EMod$ with $KY=(E_0,N)$ and maps $k\colon E\to U'Y$ and $l\colon Y\to L'E$ for which $\eta_E=U'l\circ k$ and the induced underlying triangle is of the form
	\begin{center}
		\begin{tikzpicture}[baseline=(current  bounding  box.south), scale=2]
			
			\node (c0) at (0.8,0) {$E_0\oplus N$};
			\node (a0) at (0,-0.5) {$E_0$};
			\node (b0) at (1.6,-0.5) {$A\oplus \Omega_0$};
			
			\path[font=\scriptsize]
			
			(a0) edge [->] node [below] {$(h,d)$} (b0)
			(a0) edge [->] node [above] {$(1_{E_0},k_2)\ \ \ \ \ \ \ \ $} (c0)
			(c0) edge [->] node [above] {$\ \ \ \ \ \ \ \ \  h\oplus l_2$} (b0);
		\end{tikzpicture}	
	\end{center} 
	By initiality of $\eta_E$ there exists a unique map $m\colon L'E\to Y$ for which $U'm\circ \eta_E=k$. By uniqueness it follows that $l\circ m=1_{L'E}$ and thus that $h$ is also a split epimorphism.
	
	It follows that $h$ is an isomorphism of monoids, with inverse $g$. Let $s\colon E'\rightleftarrows U'L'E\colon p$ be the retraction of $L'E$. Then $\theta:=p\eta_E\colon E\to E'$ lies over $h$, and $V'f\colon E'\to E$ lies over $g$. By faithfulness they are inverse isomorphisms.

Lift $(h\oplus1)^{-1}\colon A\oplus\Omega_0\to E_0\oplus h^*\Omega_0$ to an isomorphism $\psi\colon U'L'E\to\widehat E$ in $\E$. The retraction
\[
E\xrightarrow{\psi s\theta}\widehat E
\xrightarrow{\theta^{-1}p\psi^{-1}}E
\]
lies over the canonical square-zero retraction and defines an object of $\EEModE$ isomorphic to $L'E$. The transported unit $\psi\eta_E$ lies over $(1_{E_0},d)$, as desired. The final claim follows from the definition of a generalized calculus.
\end{proof}

\begin{rmk}\label{rmk:underlying-generalized-calculus}
In particular, the composite $KL'\colon\E\to\Mod(\V)$,
equipped with the differentials induced by the unit, gives
a generalized first order differential calculus
$(\Omega_0,d\colon E_0\to\Omega_0)$ in $\V$ over each $E_0$,
by \Cref{Leibniz-square}. These calculi are functorial in $E$
by naturality of the unit.
\end{rmk}

\begin{defi}\label{def:geometrical}
	We say that $(-)_0\colon \E\to\Mon(\V)$ is {\em \verysuitable} if it is \suitable\ and for any $E\in\E$ the forgetful functor $V_E\colon \EEModE\to\V$, sending $M$ to the underlying object of $M_0$, has a left adjoint $L_E$.
\end{defi}

\begin{defi}
	\label{def:first-ord-diff-calc-E}
	Let $(-)_0\colon \E\to\Mon(\V)$ be \verysuitable. A {\em first order $\E$-differential calculus over $E$} is the data of $\Omega\in\EEModE$ together with a morphism $d\colon E_0\to \Omega_0$ in $\V$ such that:\begin{enumerate}
		\item $(\Omega,d)$ is a generalized first order $\E$-differential calculus over $E$.
		\item (Surjectivity). The transpose $d^t\colon L_E(E_0)\to \Omega$ of $d\colon E_0\to\Omega_0$ is an epimorphism in $\EEModE$.
	\end{enumerate}
\end{defi}

Note that (ii) is a generalized surjectivity condition saying that $d$ generates $\Omega$ as an $E$-bimodule. When $\E=\Mon(\V)$ and the tensor product preserves finite colimits in each variable, this agrees with the usual notion of first order differential calculus. Indeed, the Leibniz rule follows from Remark~\ref{Leibniz-square} above, and in this case $E=E_0$, $L_E(E)=E^{\otimes 3}$, and $d^t=1_E\cdot d\cdot 1_E$, so that $d^t$ is an epimorphism if and only if $d$ satisfies the surjectivity condition by Proposition~\ref{prop:leibz-then-left-right}.

We can compare the surjectivity condition in $\EEModE$ and $\EModE$ as follows:

\begin{lemma}\label{lemma:surjectivity-comparison}
	Assume that the forgetful functor $J_E\colon\EEModE\to\EModE$ has a left adjoint $F_E$, and let $(\Omega, d\colon E_0\to\Omega_0)$ be a generalized first order $\E$-differential calculus over $E$. Then:
	\begin{enumerate}
		\item If $(\Omega_0,d)$ is a first order differential calculus over $E_0$, then $(\Omega,d)$ is a first order $\E$-differential calculus over $E$;
		\item Assume that the tensor product preserves finite colimits in each variable, that the unit $\eta_{E_0^{\otimes 3}}\colon E_0^{\otimes 3}\to J_EF_E(E_0^{\otimes 3})$ is an epimorphism and $J_E$ preserves epimorphisms. Then $(\Omega,d)$ is a first order $\E$-differential calculus over $E$ if and only if $(\Omega_0,d)$ is a first order differential calculus over $E_0$.
	\end{enumerate}
	The hypotheses of (ii) are satisfied whenever $\V$ is abelian, tensor products preserve cokernels and finite direct sums, and $J_E\colon\EEModE\to\EModE$ is fully faithful and closed under subobjects and quotients.
\end{lemma}
\begin{proof}
	Recall that $E_0^{\otimes 3}$ is the value of the left adjoint to $\EModE\to \V$ at $E_0$. Then, since left adjoints compose, we have that $L_E(E_0)\cong F_E(E_0^{\otimes 3})$. As a consequence the triangle 
	\begin{center}
		\begin{tikzpicture}[baseline=(current  bounding  box.south), scale=2]
			
			\node (c0) at (1.1,0.1) {$J_E(L_E(E_0))$};
			\node (a0) at (0,-0.5) {$E_0^{\otimes 3}$};
			\node (b0) at (2.2,-0.5) {$\Omega_0=J_E(\Omega)$};
			
			\path[font=\scriptsize]
			
			(a0) edge [->] node [below] {$1_{E_0}\cdot d\cdot 1_{E_0}$} (b0)
			(a0) edge [->] node [above] {$\eta_{E_0^{\otimes 3}}\ \ \ \ \ $} (c0)
			(c0) edge [->] node [above] {$\ \ \ \ \ \ J_E(d^t)$} (b0);
		\end{tikzpicture}	
	\end{center} 
	commutes.\\
	In the hypotheses of (i), we know that the bottom map is an epimorphism; thus also $J_E(d^t)$ is such. Since $J_E$ is faithful is follows that $d^t$ is an epimorphism too, making $(\Omega,d)$ is a first order $\E$-differential calculus over $E$.
	
	Under the assumptions of (ii), it is clear that if $(\Omega,d)$ is a first order $\E$-differential calculus over $E$, then $(1_{E_0}\cdot d\cdot 1_{E_0})$ is an epimorphism (being the composite of two epimorphisms) and hence that $(\Omega_0,d)$ is a first order differential calculus over $E_0$.
	
	For the final statement, assume that $\V$ is abelian and $J_E\colon\EEModE\to\EModE$ is fully faithful and closed under subobjects and quotients. Then $\EModE$ is abelian too (\Cref{abelian-modules}) and in particular has (epi,mono) factorizations. Consider such factorization for the unit at $E_0^{\otimes 3}$
	\begin{center}
		\begin{tikzpicture}[baseline=(current  bounding  box.south), scale=2]
			
			\node (c0) at (0.9,0) {$M$};
			\node (a0) at (0,-0.5) {$E_0^{\otimes 3}$};
			\node (b0) at (1.8,-0.5) {$J_E(F_E(E_0^{\otimes 3}))$};
			
			\path[font=\scriptsize]
			
			(a0) edge [->] node [below] {$\eta_{E_0^{\otimes 3}}$} (b0)
			(a0) edge [->] node [above] {$e\ \ \ \ \ $} (c0)
			(c0) edge [->] node [above] {$\ \ \ \ \ \ m$} (b0);
		\end{tikzpicture}	
	\end{center} 
	then by hypothesis $M=J_E(N)$ and $m=J_E(n)$, where $n$ is a monomorphism. It follows from the universal property of the adjunction that there is a map $t\colon F_E(E_0^{\otimes 3})\to N$ in $\EEModE$ such that $J_E(t)\circ \eta_{E_0^{\otimes 3}}=e$. Since $J_E(n\circ t)\circ\eta_{E_0^{\otimes 3}}=m\circ e=\eta_{E_0^{\otimes 3}}$, the universal property of the unit gives $n\circ t=1$; therefore $n$ is also a split epimorphism and hence an isomorphism. Thus $m$ is also an isomorphism and $\eta_{E_0^{\otimes 3}}$ is an epimorphism.\\
	Finally, let $e\colon M\to N$ be an epimorphism in $\EEModE$. The cokernel $c\colon J_E(N)\to C$ of $J_E(e)$ lies in the image of $J_E$ by closure under quotients. By fullness, $c$ lifts to a morphism $c'$ in $\EEModE$; faithfulness gives $c'e=0$, hence $c'=0$ because $e$ is epic. Thus $c=0$ and $J_E(e)$ is an epimorphism.
\end{proof}

\begin{prop}\label{F'-gives-calculus2}
	Assume that $(-)_0\colon \E\to\Mon(\V)$ is \verysuitable. If $U'\colon\EMod\to \E$ has a left adjoint $L'$, then for any $E\in\E$ the object $\Omega:=L'E\in\EEModE$ and the unit $\eta_E\colon E\to U'L'E$ induce a first order $\E$-differential calculus 
	$$(\Omega,d\colon E_0\to \Omega_0)$$
	over $E$.
\end{prop}
\begin{proof}
	Given Proposition~\ref{F'-gives-calculus1} we only need to prove the surjectivity condition, meaning that the transpose $d^t\colon L_E(E_0)\to \Omega$ of $d$ is an epimorphism in $\EEModE$. 
	
	Consider two maps $f,g\colon\Omega\to M$ in $\EEModE$ for which $f\circ d^t=g\circ d^t$.  Therefore it follows that the equality
	{$$ (1_{E_0}\oplus f_0)\circ (1_{E_0},d)=(1_{E_0}\oplus g_0)\circ (1_{E_0},d) $$}
	holds in $\Mon(\V)$. By faithfulness, and the fact that $f$ and $g$ can be seen as maps in $\Mod(\E)$, we also deduce the equality
	$$ U'(f)\circ\eta_E=U'(g)\circ \eta_E. $$
	But $\eta_E=(1_E,d)$ is an initial object in $E/U'$; thus $f=g$ and hence $d^t$ is an epimorphism. 
\end{proof}

\subsection{The locally presentable case}

In this section we will consider the case when our base additive monoidal category is furthermore \emph{locally presentable}, we have already seen the definition in \Cref{lp-cat}. 

\begin{eg}\label{eg:locally-pres-add-mon-cats}
	Let us recall the examples in \Cref{eg:additive-monoidal-cats}; in most cases the local presentability of the category is well-known and follows from certain characterization theorems of~\cite{AR94:libro}. In Examples (i)-(vii) the tensor product is cocontinuous in each variable (since the monoidal structures are closed). In (viii)-(ix) if the tensor product is cocontinuous in both variables in $\V$, then the same holds in the constructed categories. 
	\begin{enumerate}
		\item The category of all $\bk$-linear representations of a group $G$ is locally presentable since it is the category of functors from the single-object category induced by $G$ to $\tx{Vect}$.
		\item The category $\Mod$ of modules over a commutative ring is locally presentable; this is a particular case of (ix).
		\item The category $\GrMod$ of graded modules over a monoid is locally presentable; indeed this can be defined as a category of additive functors into $\Ab$; this is a particular case of (ix).
		\item The category of (co)chain complexes is locally presentable; again this can be defined as a category of additive functors into $\Ab$; this is a particular case of (ix).
		\item The category of condensed abelian groups is locally presentable if restricted to $\lambda$-small groups, for a given cardinal $\lambda$; this is because it can be presented as the category of sheaves into $\Ab$ out of a small site, and this is locally presentable. 
		\item Let $\V$ be symmetric monoidal closed, additive and locally presentable, and $\C$ be a small category enriched on $\V$. Then the category $\V\tx{-Prof}(\C)$ is also locally presentable, since $\C$ is small  and we take $\V$-enriched functors valued in a locally presentable category.
		\item The category $\Ban$ is locally presentable; see~\cite[Example~1.48]{AR94:libro}.
		\item If $\V$ is locally presentable, then so are $\VMon$, $\AModA$ for any monoid $A$, and $\Mod(\V)$; see \Cref{l.p.mod.mon}. 
		\item If $\V$ is locally presentable, then so are the category $\tx{Gr}(\V)$ of objects of $\V$ graded by a monoid, the category $\Ch(\V)$ of cochain complexes in $\V$, and (more generally) the category $[\C,\V]$ of additive (or even $\V$-enriched) functors into $\V$ out of a small additive (or small $\V$-enriched) category $\C$. 
	\end{enumerate}
\end{eg}

\begin{lemma}\label{inducing-F'1}
	Let $\V$ be a locally presentable monoidal additive category such that the tensor product is accessible in each variable. If $\E$ is locally presentable and $(-)_0\colon\E\to\Mon(\V)$ is accessible and continuous, then also $\EMod$ is locally presentable and
	$$ U'\colon \EMod\to \E $$
	has a left adjoint. Moreover, for any $E\in\E$, the category $\EEModE$ is locally presentable and $V_E\colon \EEModE\to \V$ has a left adjoint.
\end{lemma}
\begin{proof}
	By Proposition~\ref{l.p.mod.mon}, $\Mon(\V)$ and $\Mod(\V)$ are locally presentable and $U,V\colon\Mod(\V)\to\Mon(\V)$ are accessible and continuous. Since $\Ret$ is small, $[\Ret,\E]$ and $[\Ret,\Mon(\V)]$ are locally presentable, and $[\Ret,(-)_0]$ and $\widehat U$ are accessible and continuous. Moreover, $[\Ret,(-)_0]$ is an isofibration: lift the two components of an isomorphism of retractions along $(-)_0$ and transport the section and projection. Hence the pullback \eqref{pullback-can-calc} is a bipullback. By \cite[Theorem~6.10]{Bir84:tesi}, $\EMod$ is locally presentable and its projections are accessible and continuous. Thus $U'$, obtained by evaluating the projection to $[\Ret,\E]$ at the total object, has a left adjoint by \cite[Theorem~1.66]{AR94:libro}.

	For $\EEModE$ we use the pullback on the right of \eqref{def-EModE}. The double slices $E/\E/E$ and $E_0/\Mon(\V)/E_0$ are locally presentable by \cite[Proposition~1.57]{AR94:libro}, and the induced functor between them is accessible, continuous, and an isofibration. The functor $\widehat U_E$ is also accessible and continuous: a limit of split extensions with fixed base $E_0$ has the form $E_0\oplus\lim_i M_i$, and an $\alpha$-filtered colimit has the form $E_0\oplus\colim_i M_i$, for $\alpha$ such that all the categories and functors involved are $\alpha$-accessible. Applying \cite[Theorem~6.10]{Bir84:tesi} again shows that $\EEModE$ is locally presentable and $J_E$ is accessible and continuous. Hence so is $V_E=W_{E_0}\circ J_E$, where $W_{E_0}\colon\EModE\to\V$ is the forgetful functor. Therefore $V_E$ has a left adjoint by \cite[Theorem~1.66]{AR94:libro}.
\end{proof}

\begin{lemma}\label{lem:pullbacks-imply-geometrical}
	Let $\V$ be a monoidal additive category, let $\E$ be a category with pullbacks, and let $(-)_0\colon\E\to\Mon(\V)$ be a faithful isofibration preserving pullbacks. Then $(-)_0$ is \suitable.
\end{lemma}
\begin{proof}
	We verify the condition of \Cref{rmk:other-suitable}. 
	Let $X\in\EMod$ with $KX=(A,M)$, and let $f\colon E\to U'X$ have underlying components $(f_1,f_2)$. Define $u:=p_X\circ f\colon E\to V'X$ and consider the pullback on the left below.
\[
\begin{tikzcd}[ampersand replacement=\&]
 Y' \& U'X \\
 E \& V'X
 \arrow["{l'}", from=1-1, to=1-2]
 \arrow["r"', from=1-1, to=2-1]
 \arrow["{p_X}", from=1-2, to=2-2]
 \arrow["u"', from=2-1, to=2-2]
 \arrow["\lrcorner"{anchor=center, pos=0.125}, draw=none, from=1-1, to=2-2]
\end{tikzcd}
\qquad\qquad\quad
\begin{tikzcd}[ampersand replacement=\&]
	E_0\oplus f_1^*M \& A\oplus M \\
	E_0 \& A
	\arrow["{(f_1,1)}", from=1-1, to=1-2]
	\arrow["p_{E_0\oplus f_1^*M}"', from=1-1, to=2-1]
	\arrow["p_{A,M}", from=1-2, to=2-2]
	\arrow["f_1"', from=2-1, to=2-2]
	\arrow["\lrcorner"{anchor=center, pos=0.125}, draw=none, from=1-1, to=2-2]
\end{tikzcd}
\]
	Since $(-)_0$ preserves this pullback, the image of this square (up to isomorphism) coincides with the pullback square on the right. Thus $(Y')_0$ is canonically isomorphic to $E_0\oplus f_1^*M$, with $r_0=p_{E_0,f_1^*M}$ and $l'_0=f_1\oplus1_M$. Using the isofibration property, we may choose $Y'$ so that these identifications are equalities.
	The pairs $(1_E,s_X\circ u)$ and $(1_E,f)$ induce morphisms $s,k\colon E\to Y'$, respectively. Thus $rs=1_E$, $s_0=s_{E_0,f_1^*M}$, and $k_0=(1_{E_0},f_2)$. The retraction $(s,r)$ defines an object $Y\in\EMod$ with base $E$ and $KY=(E_0,f_1^*M)$. Moreover, $(u,l')$ is a morphism of retractions, giving $l\colon Y\to X$ with $Kl=(f_1,1_M)$. Since $f=l'k=U'l\circ k$, the required condition follows.
	Thus pulling back the projection realizes restriction of scalars on the underlying bimodule.
\end{proof}

\begin{theo}
	\label{thm:canonical-calc}
	Let $\V$ be a locally presentable monoidal additive category such that the tensor product is accessible in each variable. Let $\E$ be locally presentable and let $(-)_0\colon\E\to\Mon(\V)$ be a faithful isofibration which is accessible and continuous. Then $(-)_0$ is \verysuitable\ and $U'\colon\EMod\to\E$ has a left adjoint $L'$. For any $E\in\E$, the object $L'E$ can be chosen in $\EEModE$, as in \Cref{F'-gives-calculus1}, and $\Omega:=L'E$, together with the differential $d\colon E_0\to\Omega_0$ induced by the unit, is a first order $\E$-differential calculus over $E$.
\end{theo}
\begin{proof}
	Since $\E$ is locally presentable, it has pullbacks, and $(-)_0$ preserves them by continuity. Thus \Cref{lem:pullbacks-imply-geometrical} shows that $(-)_0$ is \suitable. By \Cref{inducing-F'1}, the functors $U'$ and $V_E$ have left adjoints, so that $(-)_0$ is \verysuitable. The remaining claims follow from Propositions~\ref{F'-gives-calculus1} and~\ref{F'-gives-calculus2}.
\end{proof}

Consider the hypotheses of the Theorem above. Given $E\in\E$ the surjectivity condition holds for $(\Omega,d\colon E_0\to \Omega_0)$ over $E$, but may not hold over $E_0$. That is, $d$ generates $\Omega$ as an $E$-bimodule, but may not generate it as an $E_0$-bimodule.
This said, the next proposition shows how we can still produce a first order differential calculus over $E_0$.

\begin{cor}\label{remark:functor-Calc_E}
	Under the hypothesis of \Cref{thm:canonical-calc}, and assuming moreover that the tensor product preserves finite colimits in each variable, there exists a functor $\Calc_\E\colon\E\to\Calc(\V)$ together with two natural transformations $\chi$ and $\beta$ as below. 
	Moreover, if the hypotheses of \Cref{lemma:surjectivity-comparison}(ii) are satisfied, then $\beta$ is a natural isomorphism.
	\begin{center}
		\begin{tikzpicture}[baseline=(current  bounding  box.south), scale=2]
			
			\node (a0) at (-0.3,0) {$\E$};
			
			\node (b0) at (1,0.7) {$\Mon(\V)$};
			\node (c0) at (1,0) {$\Calc(\V)$};
			\node (d0) at (1,-0.7) {$\Mod(\E)$};
			
			\node (e0) at (2.2,-0) {$\Mod(\V)$};
			
			\node (00) at (0.6,0.4) {$\Downarrow \chi$};
			\node (10) at (1,-0.35) {$\Downarrow \beta$};
			\node (20) at (1.5,0.4) {$=$};
			
			\path[font=\scriptsize]
			
			(a0) edge [bend left=30, ->] node [above] {$(-)_0\ \ \ \ \ \ \ \  $} (b0)
			(b0) edge [bend left=30, ->] node [above] {$\ \ \ \ L$} (e0)
			(b0) edge [->] node [right] {$L_u$} (c0)
			(a0) edge [->] node [above] {$\Calc_\E$} (c0)
			(c0) edge [>->] node [above] {$ J$} (e0)
			(a0) edge [bend right=30, ->] node [below] {$L'\ \ \ $} (d0)
			(d0) edge [bend right=30, ->] node [below] {$\ \ \ K$} (e0);
		\end{tikzpicture}	
	\end{center}
\end{cor}

\begin{proof}	
	Fix $E\in\E$ and the first order $\E$-differential calculus $(\Omega\in\EEModE,d\colon E_0\to\Omega_0)$ given by \Cref{thm:canonical-calc}. 
	Let $(\Omega_{E_0,u},d_u)$ be the universal first order differential calculus over $E_0$, so that $\Omega_{E_0,u}\in\EModE$. Then, by Proposition~\ref{universal-calc-left-adj}, the morphism of monoids $(1_{E_0},d)\colon E_0\to E_0\oplus\Omega_0=U(E_0,\Omega_0)$ induces a unique morphism 
	$$ h\colon \Omega_u\longrightarrow\Omega_0$$
	of $E_0$-bimodules. Consider then the (epi, strong mono) factorization below, which exists since $\EModE$ is locally presentable
	\begin{center}
		\begin{tikzpicture}[baseline=(current  bounding  box.south), scale=2]
			
			\node (c0) at (0.8,0) {$\Omega_0'$};
			\node (a0) at (0,-0.5) {$\Omega_{E_0,u}$};
			\node (b0) at (1.6,-0.5) {$\Omega_0$};
			
			\path[font=\scriptsize]
			
			(a0) edge [->] node [below] {$h$} (b0)
			(a0) edge [->>] node [above] {$q\ \ \ $} (c0)
			(c0) edge [>->] node [above] {$\ \ \ m$} (b0);
		\end{tikzpicture}	
	\end{center} 
	in $\EModE$. Then, by Lemma~\ref{lem:epi-from-univ-always-calc}, $(\Omega'_0,d':=q\circ d_u)$ defines a first order differential calculus over $E_0$ which is a sub- $E_0$-module of the calculus $(\Omega,d)$ over $E$. Note that under the hypotheses of \Cref{lemma:surjectivity-comparison}(ii) we know that $(\Omega_0,d)$ is already a first order differential calculus over $E_0$, and so $h$ is an epimorphism and $\Omega_0'\cong\Omega_0$.
	
	To define the functor on a morphism $\phi\colon E\to F$, write $h_E=m_Eq_E$ and $h_F=m_Fq_F$ for these factorizations. The morphism $L'\phi$ induces $g_\phi\colon\Omega_{E,0}\to\phi_0^*\Omega_{F,0}$, and naturality of the unit gives
	\[ g_\phi h_E=\phi_0^*(h_F)(\phi_0)_u. \]
	Restriction of scalars $\phi_0^*$ is accessible and continuous between locally presentable categories, hence a right adjoint; it therefore preserves strong monomorphisms. Orthogonality of $q_E$ and $\phi_0^*(m_F)$ gives a unique morphism $\delta_\phi\colon\Omega'_{E,0}\to\phi_0^*\Omega'_{F,0}$ such that
	\[ \delta_\phi q_E=\phi_0^*(q_F)(\phi_0)_u,
	\qquad \phi_0^*(m_F)\delta_\phi=g_\phi m_E. \]
	Thus $\Calc_\E(\phi):=(\phi_0,\delta_\phi)$ is a morphism of calculi. Uniqueness gives functoriality and the required natural transformations.
	Precisely, the natural transformations are given componentwise by (recall that $L_u$ assigns to each $A\in\Mon(\V)$ the pair $(A,\Omega_{A,u})$)
$$\chi_E=(1_{E_0}\colon E_0\to E_0, q\colon \Omega_{E_0,u}\to\Omega_0')$$ 
and
$$\beta_E=(1_{E_0}\colon E_0\to E_0, m\colon \Omega_0'\to\Omega_0)$$ 
for any $E\in\E$. 
Note that, by the argument above, $\beta$ is a natural isomorphism when the hypotheses of \Cref{lemma:surjectivity-comparison}(ii) are satisfied, since in that case $h$ would be epimorphism. 
\end{proof}

In the next proposition we will see how, given two \suitable\ functors (again satisfying the hypotheses of \Cref{remark:functor-Calc_E}), we can obtain a natural transformation between their calculus functors defined in \Cref{remark:functor-Calc_E}.

\begin{prop}\label{rmk:comparison-two-canonical}
	Let $e_1\colon\E_1\to\Mon(\V)$ and $e_2\colon\E_2\to\Mon(\V)$ be two \suitable\ functors satisfying the hypotheses of \Cref{remark:functor-Calc_E}. A functor $k\colon\E_1\to\E_2$ and an invertible natural transformation $\xi\colon e_1\Rightarrow e_2k$, as below left, induce a natural transformation
	\[ \widetilde\kappa\colon\Calc_{\E_2}\circ k\Longrightarrow\Calc_{\E_1} \]
	whose component at $E$ lies over $\xi_E^{-1}$, as below right. Moreover, this assignment is functorial.
\[\begin{tikzcd}[ampersand replacement=\&,row sep=tiny]
	{\E_1} \&\&\& {\E_1} \\
	\& {\Mon(\mathcal{V})} \&\&\&\& {\Calc(\V)} \\
	{\E_2} \&\&\& {\E_2}
	\arrow[""{name=0, anchor=center, inner sep=0}, "{e_1}", from=1-1, to=2-2]
	\arrow["k"', from=1-1, to=3-1]
	\arrow[""{name=1, anchor=center, inner sep=0}, "{\Calc_{\E_1}}", from=1-4, to=2-6]
	\arrow["k"', from=1-4, to=3-4]
	\arrow[""{name=2, anchor=center, inner sep=0}, "{e_2}"', from=3-1, to=2-2]
	\arrow[""{name=3, anchor=center, inner sep=0}, "{\Calc_{\E_2}}"', from=3-4, to=2-6]
	\arrow["\xi"{description}, shift right=2, Rightarrow, from=0, to=2]
	\arrow["{\widetilde{\kappa}}"'{xshift=0.1cm}, between={0.3}{0.7}, Rightarrow, from=3, to=1]
\end{tikzcd}\]
\end{prop}

\begin{proof}
First, recall that, in the 2-category of categories, pullbacks along isofibrations are also bipullbacks\footnote{For the definition of bilimits see for instance \cite{Kelly_1989} or \cite[Chapter~5.1]{JohnsonYau_2d_cats}, where they are called \emph{pseudo bilimits}.}.
Since $[\Ret,e_i]$ is an isofibration, the pullback \eqref{pullback-can-calc} is therefore a bipullback.
Applying $[\Ret,k]$ and $[\Ret,\xi]$ to the defining retractions induces, by its universal property, a functor $\widehat k\colon\Mod(\E_1)\to\Mod(\E_2)$ making the diagram below left commute up to coherent isomorphism.
Writing $K_i\colon\Mod(\E_i)\to\Mod(\V)$ for the projections, we also have a compatible isomorphism $K_2\widehat k\cong K_1$.
\begin{center}
\begin{tikzcd}[ampersand replacement=\&]
 {\Mod(\E_1)} \& {\Mod(\E_2)} \\
 {\E_1} \& {\E_2}
 \arrow[""{name=0, anchor=center, inner sep=0}, "{\widehat k}", from=1-1, to=1-2]
 \arrow["{U'_1}"', from=1-1, to=2-1]
 \arrow["{U'_2}", shift right=2, from=1-2, to=2-2]
 \arrow[""{name=1, anchor=center, inner sep=0}, "k"', from=2-1, to=2-2]
 \arrow["\cong"{description}, draw=none, from=0, to=1]
\end{tikzcd}
\hspace{0.2cm}
\begin{tikzcd}[ampersand replacement=\&]
 {\Mod(\E_1)} \& {\Mod(\E_2)} \\
 {\E_1} \& {\E_2}
 \arrow["{\widehat k}", from=1-1, to=1-2]
 \arrow["{L'_1}", from=2-1, to=1-1]
 \arrow["k"', from=2-1, to=2-2]
 \arrow["\gamma"', between={0.5}{0.7}, Rightarrow, from=2-2, to=1-1]
 \arrow["{L'_2}"', from=2-2, to=1-2]
\end{tikzcd}
$:=$
\begin{tikzcd}[ampersand replacement=\&,column sep=small]
 {\E_1} \& {\Mod(\E_1)} \& {\Mod(\E_2)} \\
 \& {\E_1} \& {\E_2} \& {\Mod(\E_2)}
 \arrow["{L'_1}", from=1-1, to=1-2]
 \arrow[""{name=0, anchor=center, inner sep=0}, curve={height=12pt}, equals, from=1-1, to=2-2]
 \arrow[""{name=1, anchor=center, inner sep=0}, "{\widehat k}", from=1-2, to=1-3]
 \arrow["{U'_1}"', from=1-2, to=2-2]
 \arrow["{U'_2}", shift right=2, from=1-3, to=2-3]
 \arrow[""{name=2, anchor=center, inner sep=0}, curve={height=-12pt}, equals, from=1-3, to=2-4]
 \arrow[""{name=3, anchor=center, inner sep=0}, "k"', from=2-2, to=2-3]
 \arrow["{L'_2}"', from=2-3, to=2-4]
 \arrow["{\eta'_1}", between={0.3}{0.7}, Rightarrow, from=0, to=1-2]
 \arrow["\cong"{description}, draw=none, from=1, to=3]
 \arrow["{\varepsilon'_2}", between={0.3}{0.7}, Rightarrow, from=2-3, to=2]
\end{tikzcd}
\end{center}
Since, by \Cref{thm:canonical-calc}, both $U'_1$ and $U'_2$ admit left adjoints, we obtain the comparison shown above right.
Explicitly, $\gamma_{E_1}\colon L'_2kE_1\to\widehat kL'_1E_1$ is defined as the composite
\[
\begin{tikzcd}[ampersand replacement=\&]
 {L'_2kE_1} \&\& {L'_2kU'_1L'_1E_1\cong L'_2U'_2\widehat kL'_1E_1} \&\& {\widehat kL'_1E_1,}
 \arrow["{L'_2k(\eta'_1)_{E_1}}", from=1-1, to=1-3]
 \arrow["{(\varepsilon'_2)_{\widehat kL'_1E_1}}", from=1-3, to=1-5]
\end{tikzcd}
\]
where the first morphism is the image under $L'_2k$ of the unit $(\eta'_1)_{E_1}$, and the second is the counit of $L'_2\dashv U'_2$.

Following the notation of \Cref{thm:canonical-calc,remark:functor-Calc_E}, write $K_1L'_1E_1=(e_1E_1,\Omega_{e_1E_1})$ and $K_2L'_2kE_1=(e_2kE_1,\Omega_{e_2kE_1})$.
Using $K_2\widehat k\cong K_1$, we also write $\gamma_{E_1}$ for the induced morphism in $\Mod(\V)$ and its second component; its monoid component is $\xi_{E_1}^{-1}$.
The square below left commutes by definition of $\gamma$ and the triangle identities for $L'_2\dashv U'_2$.
\begin{center}
\begin{tikzcd}[ampersand replacement=\&]
 {e_1E_1} \& {U(e_1E_1,\Omega_{e_1E_1})} \\
 {e_2kE_1} \& {U(e_2kE_1,\Omega_{e_2kE_1})}
 \arrow["{e_1(\eta'_1)_{E_1}}", from=1-1, to=1-2]
 \arrow["{\xi_{E_1}^{-1}}"', from=2-1, to=1-1]
 \arrow["{U\gamma_{E_1}}", from=2-2, to=1-2]
 \arrow["{e_2(\eta'_2)_{kE_1}}"', from=2-1, to=2-2]
\end{tikzcd}
\hspace{0.5cm}
\begin{tikzcd}[ampersand replacement=\&]
 {\Omega_{e_1E_1,u}} \& {\Omega_{e_1E_1}} \\
 {\Omega_{e_2kE_1,u}} \& {\Omega_{e_2kE_1}}
 \arrow["h_1", from=1-1, to=1-2]
 \arrow["{(\xi_{E_1}^{-1})_u}"', from=2-1, to=1-1]
 \arrow["{\gamma_{E_1}}", from=2-2, to=1-2]
 \arrow["h_2"', from=2-1, to=2-2]
\end{tikzcd}
\end{center}
Using the adjunction $L\dashv U$, we obtain the corresponding commutative square above right.
Write $\Omega'_{e_1E_1}$ and $\Omega'_{e_2kE_1}$ for the bimodules underlying $\Calc_{\E_1}(E_1)$ and $\Calc_{\E_2}(kE_1)$, respectively.
After transporting the upper row along $\xi_{E_1}^{-1}$, both factorizations lie in the category of $e_2kE_1$-bimodules.
Since restriction of scalars along an isomorphism is an equivalence, it preserves epimorphisms and strong monomorphisms.
The universal property of the factorizations therefore gives the unique dotted arrow below, which is the bimodule component of a morphism $\widetilde\kappa_{E_1}\colon\Calc_{\E_2}(kE_1)\to\Calc_{\E_1}(E_1)$ over $\xi_{E_1}^{-1}$.
\[
\begin{tikzcd}[ampersand replacement=\&,row sep=small]
 {\Omega_{e_1E_1,u}} \&\& {\Omega_{e_1E_1}} \\
 \& {\Omega'_{e_1E_1}} \\
 \& {\Omega'_{e_2kE_1}} \\
 {\Omega_{e_2kE_1,u}} \&\& {\Omega_{e_2kE_1}}
 \arrow["h_1", from=1-1, to=1-3]
 \arrow[two heads, from=1-1, to=2-2]
 \arrow["\cong"', from=4-1, to=1-1]
 \arrow["{\gamma_{E_1}}", from=4-3, to=1-3]
 \arrow[tail, from=2-2, to=1-3]
 \arrow["{\widetilde\kappa_{E_1}}"{xshift=0.1cm,yshift=0.1cm}, dotted, from=3-2, to=2-2]
 \arrow[tail, from=3-2, to=4-3]
 \arrow[two heads, from=4-1, to=3-2]
 \arrow["h_2"', from=4-1, to=4-3]
\end{tikzcd}
\]
Naturality and compatibility with identities and composition follow from uniqueness, since a morphism of first order calculi is determined by its monoid component.
\end{proof}

\begin{rmk}\label{rmk:kappa-chi}
By construction, $\widetilde\kappa$ commutes with the $\chi$ of \Cref{remark:functor-Calc_E}, i.e. the following equation holds.
\begin{center}
\begin{tikzcd}[ampersand replacement=\&,row sep=tiny]
 \&\& {\Mon(\V)} \\
 {\E_2} \\
 \&\& {\Calc(\V)} \\
 {\E_1}
 \arrow["{L_u}", from=1-3, to=3-3]
 \arrow[""{name=0, anchor=center, inner sep=0}, "{e_2}", from=2-1, to=1-3]
 \arrow[""{name=1, anchor=center, inner sep=0}, "{\Calc_{\E_2}}"{description}, from=2-1, to=3-3]
 \arrow["k"', from=4-1, to=2-1]
 \arrow[""{name=2, anchor=center, inner sep=0}, "{\Calc_{\E_1}}"', from=4-1, to=3-3]
 \arrow["{\widetilde{\kappa}}"'{xshift=-0.1cm}, between={0.3}{0.7}, Rightarrow, from=1, to=2]
 \arrow["{\chi_2}"{xshift=0.1cm}, shift left=5, between={0.3}{0.7}, Rightarrow, from=0, to=1]
\end{tikzcd}
$=$
\begin{tikzcd}[ampersand replacement=\&,row sep=tiny]
 \&\& {\Mon(\V)} \\
 {\E_2} \\
 \&\& {\Calc(\V)} \\
 {\E_1}
 \arrow["{L_u}", from=1-3, to=3-3]
 \arrow["{e_2}", from=2-1, to=1-3]
 \arrow["k"', from=4-1, to=2-1]
 \arrow[""{name=0, anchor=center, inner sep=0}, "{e_1}", from=4-1, to=1-3]
 \arrow["{\Calc_{\E_1}}"', from=4-1, to=3-3]
\arrow["\cong"{description, pos=0.3}, draw=none, from=2-1, to=0] 
\arrow["{\chi_1}"{xshift=0.1cm}, between={0.3}{0.8}, Rightarrow, from=0, to=3-3]
\end{tikzcd}
\end{center}
Equivalently,
\[
 \chi_1=\widetilde\kappa\circ(\chi_2k)\circ(L_u\xi).
\]
\end{rmk}

\Cref{rmk:comparison-two-canonical} suggests the following definition of morphisms of \suitable~functors, and so of the category $\wgeomfun$.

\begin{defi}\label{def:cat-of-geom-fun}
	The category $\wgeomfun$ has objects \suitable\ functors $e\colon\E\to\Mon(\V)$ and morphisms pairs $(k,\xi)\colon(\E_1,e_1)\to(\E_2,e_2)$ consisting of a functor $k\colon\E_1\to\E_2$ and $\xi$ an invertible natural transformation as below.
	\[\begin{tikzcd}[ampersand replacement=\&,row sep=tiny]
		{\E_1} \\
		\& {\Mon(\mathcal{V})} \\
		{\E_2}
		\arrow[""{name=0, anchor=center, inner sep=0}, "{e_1}", from=1-1, to=2-2]
		\arrow["k"', from=1-1, to=3-1]
		\arrow[""{name=1, anchor=center, inner sep=0}, "{e_2}"', from=3-1, to=2-2]
		\arrow["\xi"'{xshift=-0.1cm}, shift right=2, between={0.3}{0.7}, Rightarrow, from=0, to=1]
	\end{tikzcd}\]
	Composition is given by pasting and the identity is given by $(1_\E,1_{e})\colon(\E,e)\to(\E,e)$. 
	
	The category $\geomfun$ is defined as the full subcategory of $\wgeomfun$ with objects \verysuitable\ functors. 
\end{defi}

\begin{eg}\label{ex:smooth_algebraic_comparison}
	In \Cref{rmk:operads-morph-of-geom-fun} we will see a family of examples coming from the theory of operads. 
	A concrete example, of a morphism of weakly $\Rvect$-geometrical functors, is the corestriction of the functor $(-)_0\colon C^\infty\tx{-Ring}\to \Mon(\mathbb R\tx{-Vect})$ (described in \Cref{eg:Cinfty-rings-pullback}) to $\CMon(\mathbb R\tx{-Vect})$. 
\[\begin{tikzcd}[ampersand replacement=\&,row sep=tiny]
	{	C^\infty\tx{-Ring}} \\
	\& {\Mon(\mathbb R\tx{-Vect})} \\
	{\CMon(\mathbb R\tx{-Vect})}
	\arrow["{(-)_0}", from=1-1, to=2-2]
	\arrow[from=1-1, to=3-1]
	\arrow[hook, from=3-1, to=2-2]
\end{tikzcd}\]
\end{eg}
 
\begin{eg}\label{ex:cmon_comparison}
	It follows from \Cref{ComMon-suitable} and the discussion on $n$-commutative monoids in \Cref{rmk:comm-mon-full-subcat}, that another family of examples of morphisms of $\V$-geometrical functors are given by the factorization of full subcategory inclusions $\CMon_n(\V)\hookrightarrow \CMon_{qn}(\V)\hookrightarrow \Mon(\V)$, for $q\in\mathbb{Z}$. 
	\[\begin{tikzcd}[ampersand replacement=\&,row sep=tiny]
	{	\CMon_n(\V)} \\
	\& {\Mon(\V)} \\
	{\CMon_{qn}(\V)}
	\arrow[hook, from=1-1, to=2-2]
	\arrow[hook, from=1-1, to=3-1]
	\arrow[hook, from=3-1, to=2-2]
\end{tikzcd}\]
\end{eg}

\begin{eg}
	\label{ComMon-canonical-calc}
	Let us consider, for a braided monoidal category $\V$, the category $\E:=\CMon_n(\V)$. 
	In \Cref{ComMon-suitable} we show that the inclusion $\CMon_n(\V)\hookrightarrow\VMon$ is \suitable. 
	Let us moreover assume that $\V$ is a locally presentable additive category such that the tensor product preserves colimits in each variable.
	Then, since $\Mon(\V)$ is locally presentable (see \Cref{l.p.mod.mon}) and $\CMon_n(\V)$ is closed under limits and $\lambda$-filtered colimits (for the same $\lambda$ such that $\otimes$ preserves $\lambda$-filtered colimits), also $\CMon_n(\V)$ is locally presentable (\cite[Corollary~2.48]{AR94:libro}).
	It follows that $\CMon_n(\V)\hookrightarrow\VMon$ is accessible and continuous, therefore by \Cref{inducing-F'1} it is also \verysuitable. 
	
	Thus, we can apply \Cref{thm:canonical-calc} to $\CMon_n(\V)\hookrightarrow \Mon(\V)$. 
	\[\begin{tikzcd}[ampersand replacement=\&]
		{\Mod(\CMon_n(\V))} \&\& {\Mod(\V)} \\
		{\CMon_n(\V)} \& {} \& {\Mon(\V)}
		\arrow[hook, from=1-1, to=1-3]
		\arrow[""{name=0, anchor=center, inner sep=0}, "{U'}"', shift right=2, from=1-1, to=2-1]
		\arrow["\lrcorner"{anchor=center, pos=0.125}, draw=none, from=1-1, to=2-2]
		\arrow[""{name=1, anchor=center, inner sep=0}, "U"'{yshift=-0.1cm}, shift right=2, from=1-3, to=2-3]
		\arrow[""{name=2, anchor=center, inner sep=0}, "{L'}"'{yshift=-0.1cm}, shift right=2, curve={height=6pt}, from=2-1, to=1-1]
		\arrow[hook, from=2-1, to=2-3]
		\arrow[""{name=3, anchor=center, inner sep=0}, "L"', shift right=2, curve={height=6pt}, from=2-3, to=1-3]
		\arrow["\vdash"{description}, draw=none, from=0, to=2]
		\arrow["\vdash"{description}, draw=none, from=1, to=3]
	\end{tikzcd}\]

	Note that, if we set $\V=\kvect$, then $\Mon(\kvect)=\kalg$ is the category of $\bk$-algebras and $\CMon(\kvect)=\kcalg$ the subcategory of the commutative ones.   
	Using \Cref{thm:canonical-calc} in this case, for any commutative $\bk$-algebra $A$ the canonical functor $L'$ gives the classical notion of K\"ahler differentials. 
	We will show this more generally for any braided monoidal category $\V$, using a suitable coequalizer to describe the K\"ahler differentials.

For any $A\in\CMon_n(\V)$ we get $L'A=(A,\Omega)$, where $(\Omega,l,r)$ is an $n$-commutative bimodule and $d\colon A\to\Omega$ is the first order $\E$-differential calculus given by \Cref{thm:canonical-calc}. Let $(\Omega_u^1,d_u)$ be the universal first order differential calculus over the monoid $A$, with actions $l_u,r_u$. By \Cref{universal-calc-left-adj}, there exists a unique bimodule map $d'\colon\Omega_u^1\to\Omega$ such that $d'd_u=d$.

To proceed, we split the remainder of the construction into even and odd cases. The maps displayed below are morphisms in $\V$. When taking their coequalizer in $\AModA$, we first extend them to bimodule maps from the free $A$-bimodule on their common domain, as in \Cref{forgetful}.

First, consider $n=2m+1$. Define $d_K\colon\Omega_u^1\to\Omega_K$ by the simultaneous coequalizer of the two pairs
\[
\begin{tikzcd}[ampersand replacement=\&,column sep=large]
 \Omega_u^1\otimes A \& \Omega_u^1 \& A\otimes\Omega_u^1.
 \arrow["r_u",shift left=2,from=1-1,to=1-2]
 \arrow["{l_u\beta^n_{\Omega_u^1,A}}"',shift right=2,from=1-1,to=1-2]
 \arrow["l_u"',shift right=2,from=1-3,to=1-2]
 \arrow["{r_u\beta^n_{A,\Omega_u^1}}",shift left=2,from=1-3,to=1-2]
\end{tikzcd}
\]
These are the two action relations of \Cref{eg:central-bimod-as-pullback}. In the symmetric case with odd $n$, they are equivalent, and for $\V=\kvect$ this gives the usual K\"ahler quotient; see \cite[Theorem~2.14]{blute2015derivations} and \cite{blute2011kahler,ONeill:master-thesis}.

Next, consider $n=2m$, giving a dyslectic analogue of the K\"ahler differentials. As before, define $\Omega_K$ to be the pushout
\[
\begin{tikzcd}[ampersand replacement=\&]
 \Omega_u^1 \& \Omega_{K_r} \\
 \Omega_{K_l} \& \Omega_K
 \arrow["d_{K_r}",from=1-1,to=1-2]
 \arrow["d_{K_l}"',from=1-1,to=2-1]
 \arrow[from=1-2,to=2-2]
 \arrow[from=2-1,to=2-2]
 \arrow["\lrcorner"{anchor=center,pos=0.125,rotate=180},draw=none,from=2-2,to=1-1]
\end{tikzcd}
\]
of the coequalizers
\[
\begin{tikzcd}[ampersand replacement=\&,column sep=large]
 A\otimes\Omega_u^1 \& \Omega_u^1 \& \Omega_{K_l}, \\
 \Omega_u^1\otimes A \& \Omega_u^1 \& \Omega_{K_r}.
 \arrow["{l_u\beta^n_{A,\Omega_u^1}}",shift left=2,from=1-1,to=1-2]
 \arrow["l_u"',shift right=2,from=1-1,to=1-2]
 \arrow["d_{K_l}",two heads,from=1-2,to=1-3]
 \arrow["{r_u\beta^n_{\Omega_u^1,A}}",shift left=2,from=2-1,to=2-2]
 \arrow["r_u"',shift right=2,from=2-1,to=2-2]
 \arrow["d_{K_r}",two heads,from=2-2,to=2-3]
\end{tikzcd}
\]
and write $d_K\colon\Omega_u^1\to\Omega_K$ for the common composite.

In either case, the forgetful functor preserves these colimits by \Cref{forgetful}, so $d_K$ is a regular epimorphism in $\V$. Since tensoring preserves colimits, $d_K\otimes1_A$ and $1_A\otimes d_K$ are epimorphisms. Naturality of the braiding therefore implies that the induced actions on $\Omega_K$ satisfy both defining $n$-commutativity relations. Hence $(A,\Omega_K)\in\Mod(\CMon_n(\V))$, and $(\Omega_K,d_Kd_u)$ is an ordinary first order differential calculus by \Cref{lem:epi-from-univ-always-calc}.

We can now show that $\Omega_K\cong\Omega$. First, there is a unique $\varphi\colon\Omega_K\to\Omega$ such that $d'=\varphi d_K$: the required coequalities follow by naturality of the braiding, since $d'$ is a bimodule map and $\Omega$ is $n$-commutative. Conversely, the morphism
\[
 (1_A,d_Kd_u)\colon A\longrightarrow A\oplus\Omega_K
\]
lies in $\CMon_n(\V)$, so the adjunction $L'\dashv U'$ gives a unique $\psi\colon\Omega\to\Omega_K$ with $\psi d=d_Kd_u$. Using the universal properties of the coequalizers and pushout, together with the adjunction, we see that $\psi$ and $\varphi$ are inverse. Thus $(\Omega,d)$ is also a first order differential calculus in the ordinary sense.
\end{eg}

\begin{eg}\label{eg:final-Cinfty-rings}
	Consider again $(-)_0\colon C^\infty\tx{-Ring}\to \Mon(\mathbb R\tx{-Vect})$ as in \Cref{eg:Cinfty-rings-pullback}; we know that this is \suitable\ by \Cref{eg:Cinfty-rings-suitable}. Moreover, $C^\infty\tx{-Ring}$ is locally presentable, being the category of models of a Lawvere theory. Finally, $(-)_0$ is accessible and continuous: we have a commutative triangle
	\begin{center}
		\begin{tikzpicture}[baseline=(current  bounding  box.south), scale=2]
			
			\node (a0) at (0,0.7) {$C^\infty\tx{-Ring}$};
			\node (b0) at (1.4,0.7) {$\Mon(\mathbb R\tx{-Vect})$};
			\node (c0) at (0.7,0) {$\bf{Set}$};
			
			\path[font=\scriptsize]
			
			(a0) edge [->] node [above] {$(-)_0$} (b0)
			(a0) edge [->] node [left] {} (c0)
			(b0) edge [->] node [right,xshift=0.1cm] {} (c0);
		\end{tikzpicture}	
	\end{center} 
	where the two vertical arrows are the forgetful functors to $\bf{Set}$; these are conservative, continuous, and preserve filtered colimits by \cite[Proposition~11.8]{ARV2010algebraic}. Thus so is $(-)_0$.
	
	Therefore we can apply \Cref{thm:canonical-calc}. It follows that the map $U_\infty$ below left
	{\begin{center}
		\begin{tikzpicture}[baseline=(current  bounding  box.south), scale=2]
			
			\node (a0) at (0,0.8) {$\Mod(C^\infty\tx{-Ring})$};
			
			\node (b0) at (1.6,0.8) {$\Mod(\mathbb R\tx{-Vect})$};
			
			\node (c0) at (0,0) {$C^\infty\tx{-Ring}$};
			\node (d0) at (1.6,0) {$\Mon(\mathbb R\tx{-Vect})$};
			
			\node (e0) at (0.1,0.4) {$\vdash$};
			\node (f0) at (1.7,0.4) {$\vdash$};
			
			\path[font=\scriptsize]
			
			(a0) edge [->] node [above] {} (b0)
			(a0) edge [->] node [left] {$U_\infty$} (c0)
			(b0) edge [->] node [left] {$U$} (d0)
			(c0) edge [->] node [below] {} (d0)
			
			(c0) edge [bend right=40, ->] node [right] {$L_\infty$} (a0)
			(d0) edge [bend right=40, ->] node [right] {$L$} (b0);
		\end{tikzpicture}	
	\end{center}} 
	has a left adjoint $L_\infty$ which assigns $A\in C^\infty\tx{-Ring}$ to a first order $\E$-differential calculus $(\Omega^\infty,d^\infty\colon A_0\to \Omega_0)$ over $A$. The stronger Leibniz property here says that the pair
	{$$(1_A,d^\infty)\colon A\to U_\infty(L_\infty A)$$}
	defines a morphism of $C^\infty$-rings, not simply of $\mathbb R$-algebras (which would give the usual Leibniz rule); this means that for any $\rho\in C^\infty(\mathbb R^n,\mathbb R)$ and $\bar a=(a_1,\cdots,a_n)\in A^n$ the differential $d^\infty$ satisfies
	$$ d^\infty(\rho_A(\bar a))=  \sum_{i=1}^{n} \left(\frac{\partial \rho}{\partial x_i}\right)_{\!\!\! A}\!(\bar a)\cdot d^\infty(a_i); $$
	these are usually called {\em $C^\infty$-derivations}, see for instance~\cite[Section~2]{dubuc19841} and \cite{C-inf-ring-modality}.\\
	Note now that, given $A\in C^\infty\tx{-Ring}$, the forgetful functor out of $\AModA$ factors as
	\begin{center}
		\begin{tikzpicture}[baseline=(current  bounding  box.south), scale=2]
			
			\node (c0) at (0.9,0) {$_{A_0}\!\tx{CMod}_{A_0}$};
			\node (a0) at (0,-0.5) {$\AModA$};
			\node (b0) at (1.8,-0.5) {$_{A_0}\!\Mod_{A_0}$};
			
			\path[font=\scriptsize]
			
			(a0) edge [->] node [below] {} (b0)
			(a0) edge [->] node [above] {$R\ \ \ $} (c0)
			(c0) edge [->] node [above] {$\ \ \ S$} (b0);
		\end{tikzpicture}	
	\end{center}
	where $_{A_0}\!\tx{CMod}_{A_0}$ is the category of central $A_0$-bimodules. It follows from the explicit description given in \Cref{eg:Cinfty-rings-pullback}, that $R$ is an isomorphism of categories. Moreover $S$ is fully faithful and its image is closed under subobjects and quotients. Thus, by \Cref{lemma:surjectivity-comparison}, $(\Omega^\infty_0,d^\infty)$ is also a first order differential calculus over $A_0$. 
\end{eg}

\begin{eg}\label{eg:Banach-pullback}
	Let $\V=\Banb$ be the category of Banach spaces with bounded linear maps, equipped with the projective tensor product. Note that $\Banb$ is additive, has finite limits and colimits, and the tensor product preserves cokernels in each variable.\\
	Within $\Banb$ we can consider the (non-full) subcategory $\Ban$ spanned by the same objects, but with morphisms the linear contractions. This is complete and cocomplete (even locally presentable) and the projective tensor product defines a symmetric monoidal closed structure. However, $\Ban$ is not additive. 
	
	A monoid in $\Banb$ is a Banach space $\mathbb B$ together with a monoid structure on its underlying $\mathbb C$-vector space which in addition satisfies
	$$ \norm{ a\cdot b}\leq C\norm{a}\norm{b} $$
	for any $a,b\in\mathbb B$, and a given $C\in\mathbb R$. A monoid in $\Ban$ is one as above where we can take $C=1$ and the unit map is contractive; this is commonly called a {\em Banach algebra}. 
	
	We would like to consider a notion of universal first order differential calculus over a Banach algebra, but since $\Ban$ is not additive we cannot get such notion using the results of \Cref{sec:univ first ord calc}. To solve this, we can apply the results of that section to $\Banb$ and then try to restrict to $\Ban$.
	
	First note that $U_b\colon\Mod(\Banb)\to\Mon(\Banb)$ sends a pair $(A,M)$ to the monoid $U(A,M)=A\oplus M$ whose norm can be chosen to be $$\norm{(a,m)}_{U(A,M)}=\norm{a}_A+\norm{m}_M$$ (there are other equivalent norms that make $A\oplus M$ a biproduct in $\Banb$, we are just fixing one). 
	Moreover, observe that the categories $\AModA$, for $A\in\Mon(\Ban)$, and $\Mod(\Ban)$ are well-defined (since these notions do not require additivity), and are subcategories of the corresponding constructions in $\Banb$. 
	The only part that a priori is not well-defined is $U_1\colon \Mod(\Ban)\to\Mon(\Ban)$. But we can take this to be the codomain restriction of the composite
	$$ \Mod(\Ban)\mono \Mod(\Banb)\xrightarrow{\ U\ }\Mon(\Banb) $$ 
	to the subcategory $\Mon(\Ban)\mono\Mon(\Banb)$; this restriction exists since if $(A,M)\in\Mod(\Ban)$ then $U(A,M)=A\oplus M$, with the norm given above, is a Banach algebra, and each morphism  $(A,M)\to (B,N)$ in $\Mod(\Ban)$ induces a contraction $U(A,M)\to U(B,N)$. Therefore we have a commutative square as below.
	\begin{center}
		\begin{tikzpicture}[baseline=(current  bounding  box.south), scale=2]
			
			\node (a0) at (0,0.8) {$\Mod(\Ban)$};
			\node (a0') at (0.3,0.6) {$\lrcorner$};
			\node (b0) at (1.5,0.8) {$\Mod(\Banb)$};
			\node (c0) at (0,0) {$\Mon(\Ban)$};
			\node (d0) at (1.5,0) {$\Mon(\Banb)$};
			
			\path[font=\scriptsize]
			
			(a0) edge [>->] node [above] {} (b0)
			(a0) edge [->] node [left] {$U_1$} (c0)
			(b0) edge [->] node [right] {$U_b$} (d0)
			(c0) edge [>->] node [below] {} (d0);
		\end{tikzpicture}	
	\end{center} 
	It is easy to see that this is a pullback; indeed a pair $(A,M)\in\Mod(\Banb)$ is such that $U(A,M)\in\Mon(\Ban)$ if and only if $A$ is a Banach algebra, and the scalar multiplication on $M$ is a contraction, if and only if $(A,M)\in\Mod(\Ban)$. 
	
	However, $U_1$ does not preserve products, so it cannot have a left adjoint. To see this, set $X=(\mathbb C,\mathbb C)$ with the standard bimodule structure. 
Then in $\Mod(\Ban)$, $X\times X=(\mathbb C^2,\mathbb C^2)$, with coordinatewise algebra and bimodule operations and the supremum norm on each component.	
The canonical comparison
\[
 U_1(X\times X)\longrightarrow U_1(X)\times U_1(X),\qquad
 (a,b;m,n)\longmapsto((a,m),(b,n)),
\]
has an inverse which sends the element $((1,0),(0,1))$ with norm $1$ to the element $(1,0;0,1)$ with norm $2$. 
Thus this inverse is not a contraction. 
Consider then $L_b\dashv U_b$ in the bounded setting. We can apply it to $A\in\Mon(\Ban)$, so that $LA=(A,\Omega_u)$ with unit $(1,d)\colon A\to U(A,\Omega_u)$ in $\Mon(\Banb)$. It is easy to see that $(A,\Omega_u)$ lies in $\Mod(\Ban)$, since by definition $\Omega_u$ is the kernel of $m\colon A\otimes A\to A$ which is a contraction. However, the morphism $(1,d)\colon A\to U(A,\Omega_u)$ is not generally a contraction, and hence does not lie in $\Mon(\Ban)$. 
\end{eg}

\subsection{Algebraic examples}\label{sec:alg-closure}

Throughout this section $\bk$ is a commutative ring and $\Alg_\bk=\Mon(\bk$-$\Mod)$ denotes the category of unital associative $\bk$-algebras. The first class of examples we produce is that generated by the satisfaction of polynomial identities.

\begin{prop}[PI varieties]\label{prop:PI}
	Let $\Sigma$ be any fixed collection of (noncommutative) polynomial identities over $\bk$, and let $\Alg_\bk(\Sigma)$ be the full replete subcategory of algebras satisfying every identity in $\Sigma$. Then
	\[
	\Alg_\bk(\Sigma)\hookrightarrow\Alg_\bk
	\]
	is ($\bk$-$\Mod$)-geometrical.
\end{prop}
\begin{proof}
	The category $\Alg_\bk(\Sigma)$ is a full subcategory of $\Alg_\bk$; moreover, since polynomial identities are inherited by subalgebras, it follows that $\Alg_\bk(\Sigma)$ is closed under subobjects. Thus \cref{prop:implies-suitable} applies. 
	
	Finally, note that $\Alg_\bk(\Sigma)$ is locally presentable and continuously accessibly embedded in $\Alg_\bk$ by \cite[Corollary~3.7]{AR94:libro}, so the inclusion is $\bk$-$\Mod$-geometrical by \Cref{thm:canonical-calc}.
\end{proof}

\begin{rmk}
	This includes commutative algebras, algebras satisfying a standard identity $s_{2n}=0$, Lie-nilpotence identities, and fixed products of commutators.
\end{rmk}

\begin{eg}[Kapranov commutator-nilpotent algebras]
	Let us fix $\bk=\mathbb C$ (the following algebraic filtration is defined over any commutative base ring $\bk$ by the same formula).  Let
	\[
	\gamma_1(A)=A,\qquad \gamma_{r+1}(A)=[A,\gamma_r(A)]
	\]
	be the lower central series of the Lie algebra underlying $A$.  Following \cite[Definitions~1.1.1 and~1.1.5(a)]{kapranov1998noncommutative}, define
	\[
	F^d_{\mathrm{comm}}A
	=\sum_{m\ge1}\ \sum_{\substack{i_1,\ldots,i_m\ge1\\i_1+\cdots+i_m-m=d}}
	A\,\gamma_{i_1}(A)\,A\cdots A\,\gamma_{i_m}(A)\,A .
	\]
	Thus $F^0_{\mathrm{comm}}A=A$, and $F^d_{\mathrm{comm}}A$ is generated by expressions containing $d$ commutator brackets in total.  Then Kapranov defines
	\[
	\mathrm{NCAlg}_{\le d}
	=\{A\in\Alg_\bk\mid F^{d+1}_{\mathrm{comm}}A=0\}.
	\]
	For $T_n=\bk\langle x_1,\ldots,x_n\rangle$ define
	\[
	\Sigma_d:=\bigcup_{n\ge0}F^{d+1}_{\mathrm{comm}}T_n,
	\]
	viewing each element as a polynomial identity. Then it's easy to see that for every $d\ge0$ we have
	\[
	\mathrm{NCAlg}_{\le d}=\Alg_\bk(\Sigma_d).
	\]
	Indeed, if $A\in\mathrm{NCAlg}_{\le d}$, functoriality of the commutator filtration sends
	$F^{d+1}_{\mathrm{comm}}T_n$ to $F^{d+1}_{\mathrm{comm}}A=0$ under every evaluation $T_n\to A$; hence $A$ satisfies $\Sigma_d$.
	Conversely, every generator of $F^{d+1}_{\mathrm{comm}}A$ is an evaluation of an element of $F^{d+1}_{\mathrm{comm}}T_n$ for some finite $n$.  
	If $A$ satisfies $\Sigma_d$, all such generators vanish, so $F^{d+1}_{\mathrm{comm}}A=0$.
	It follows, in particular, that the inclusion $\mathrm{NCAlg}_{\le d}\hookrightarrow\Alg_\bk$ is $\bk$-Mod-geometrical.
\end{eg}

We can also treat the case of a more general formulae, involving implications between polynomial identities; these are called {\em universal Horn sentences} and the categories that they model are called {\em finitary quasivarieties}, see \cite[Section~3.B]{AR94:libro}.

\begin{prop}[Universal Horn theories]\label{prop:Horn}
	Let $\Alg_\bk(\mathcal T)\subseteq\Alg_\bk$ be a full replete subcategory axiomatized, in the usual algebraic language of unital $\bk$-algebras, by universal Horn sentences. Then $\Alg_\bk(\mathcal T)\hookrightarrow\Alg_\bk$ is $\bk$-Mod-geometrical.
\end{prop}
\begin{proof}
	The full subcategory $\Alg_\bk(\mathcal T)$ is closed under subalgebras by \cite[Remark~3.19]{AR94:libro}, so that \cref{prop:implies-suitable} applies. Moreover, $\Alg_\bk(\mathcal T)$ is locally presentable and the inclusion accessible and limit-preserving by \cite[3.22 \& 3.24]{AR94:libro}; therefore it is $\bk$-$\Mod$-geometrical by \Cref{thm:canonical-calc}.
\end{proof}

\begin{eg}[Reduced commutative algebras]\label{ex:reduced}
	The full subcategory of reduced commutative $\bk$-algebras is axiomatised by the family of universal horn sentences
	\[
	x^n=0\;\Longrightarrow\;x=0,\qquad n\ge1.
	\]
	Hence the inclusion into $\Alg_\bk$ is geometrical.
\end{eg}

\begin{eg}[Dubois-Violette--Michor]\label{dvmkc:central-calculus}
	Let us consider $\V=\bk$-$\Mod$ and let $\E$ be the category whose objects are pairs $E=(A,S)$, with $A\in\Alg_\bk$ and $S\subseteq Z(A)$ a unital $\bk$-subalgebra, and whose morphisms $f\colon(A,S)\to(B,T)$ are the algebra maps satisfying $f(S)\subseteq T$.
	The forgetful functor $(-)_0\colon\E\to\Mon(\V)$, given by $E_0=A$, is faithful and is an isofibration, since algebra isomorphisms transport the chosen central subalgebras.
	To check that it is \suitable, consider $f\colon E\to U'X$ with underlying map $f_0=(f_1,f_2)\colon A\to B\oplus M$, where $f_1$ is a monomorphism.
	Writing $U'X=(B\oplus M,T)$ and $V'X=(B,S_X)$, we have the factorization
	\[
 	A\xrightarrow{(1_A,f_2)}A\oplus f_1^*M \xrightarrow{\lambda=f_1\oplus1_M}B\oplus M.
	\]
	Since $\lambda$ is monic, $\lambda^{-1}(T)$ is central in $A\oplus f_1^*M$. Equipping the middle algebra with this
subalgebra, both arrows become morphisms in $\E$. Moreover, $f_1^{-1}(S_X)$ is central in $A$, and the inclusions $S_X\oplus0\subseteq T\subseteq S_X\oplus M$ show that the zero section and projection lift to a retraction with base $(A,f_1^{-1}(S_X))$.
	It follows from Remark~\ref{rmk:other-suitable} that $(-)_0$
is \suitable.

	Further, we can show that $(-)_0$ is \verysuitable.
	For $E=(A,S)$ and $W\in\V$, let
	\[
	 P_W=(A\otimes_\bk W\otimes_\bk A)/[S,A\otimes_\bk W\otimes_\bk A],\qquad
	 L_E(W)=\left((A\oplus P_W,S\oplus0),P_W\right),
	\]
	with the evident retraction over $E$, where $P_W$ is the free $S$-central $A$-bimodule on $W$.
	Given $M\in\EEModE$ with total object $\widehat E=(A\oplus M_0,T)$, we have $S\oplus0\subseteq T$, so $M_0$ is $S$-central. Thus every map $W\to M_0$ in $\V$ extends uniquely to an $A$-bimodule map $P_W\to M_0$.
	The induced algebra map $A\oplus P_W\to A\oplus M_0$ preserves the chosen central subalgebras, since it sends $S\oplus0$ into $T$, and is compatible with the retractions.
	Thus $L_E$ is left adjoint to $V_E$; it follows that $(-)_0$ is \verysuitable.

	We now look at the pullback functor $U'\colon\EMod\to\E$ in this case. 
	Given $E=(A,S)$, consider
	\[
 	\Omega_0=\Omega^1_{A,u}/[S,\Omega^1_{A,u}], \qquad q\colon\Omega^1_{A,u}\twoheadrightarrow\Omega_0, \qquad 	d=q d_u,
	\]
where $[S,\Omega^1_{A,u}]$ denotes the $A$-subbimodule generated by $s\omega-\omega s$, for $s\in S$ and $\omega\in\Omega^1_{A,u}$.
	For $s\in S$ and $a\in A$, the Leibniz rule gives
	\[
 	 0=d_u(sa-as)=[d_us,a]+[s,d_ua].
	\]
	Moreover, $[s,b(d_ua)c]=b[s,d_ua]c$. 
	Then, since $\Omega^1_{A,u}$ is generated as an $A$-bimodule by the $d_ua$, it follows that the defining subbimodule is also generated by the $[d_us,a]$.
	Thus $[s,\omega]=0$ and $[ds,a]=0$ for every $s\in S$, $\omega\in\Omega_0$, and $a\in A$.

	Setting $S^d\colonequals\{(s,ds):s\in S\}$ and using that $d1=0$ and
\[
  (s,ds)(t,dt)=(st,s\,dt+(ds)t)=(st,d(st)),
\]
we see that this is a unital $\bk$-subalgebra of $A\oplus\Omega_0$.
	Further, it is central, since
	\[
 	 [(s,ds),(a,\omega)] =(0,[s,\omega]+[ds,a])=0.
	\]
	Let $T_\Omega$ be the subalgebra generated by $S\oplus0$ and $S^d$. Both generating subalgebras are central, and the projection of $T_\Omega$ to $A$ is $S$. We therefore obtain $\widehat E=(A\oplus\Omega_0,T_\Omega)\in\E$ and a retraction $E\rightleftarrows\widehat E$ lifting the zero section and projection, hence an object $\Omega\in\EEModE$, and a morphism $\hat d\colon E\to\widehat E$ with $\hat d_0=(1_A,d)$.

	Given $f\colon E\to U'X=(B\oplus M,T)$, write its underlying
map as $f_0=(f_1,f_2)$. 
	Since $(f_1(s),f_2(s))\in T$, it commutes with $(0,m)$ for every $m\in M$, so $f_1(s)m=mf_1(s)$.
	The bimodule map $\overline{f_2}\colon\Omega^1_{A,u}\to f_1^*M$ given by Proposition~\ref{universal-calc-left-adj}, with
$\overline{f_2}d_u=f_2$, therefore satisfies
	\[
 	 \overline{f_2}(s\omega-\omega s)  =f_1(s)\overline{f_2}(\omega)  -\overline{f_2}(\omega)f_1(s)=0.
	\]
	Hence it factors uniquely through $q$, which yields $h\colon\Omega_0\to f_1^*M$ with $hd=f_2$.
	Now
	\[
	  (f_1\oplus h)(s,ds)=(f_1(s),f_2(s))\in T,
	\]
and $(f_1(s),0)\in T$, since $p_X\circ f$ and the zero section of $X$ are morphisms in $\E$. Thus $f_1\oplus h$ sends both generating subalgebras of $T_\Omega$ into $T$ and defines a morphism $\widehat E\to U'X$ in $\E$, compatible with the retractions, whose composite with $\hat d$ is $f$.
	Any such factorization in $\EMod$ has algebra component $f_1$, and its bimodule component is determined by $hd=f_2$, since $\Omega_0$ is generated by $dA$.
	This gives the required universal property in $\EMod$.
	Therefore $U'$ has a left adjoint $L'$ given by $L'E=(\widehat E,(A,\Omega_0))$, with base $E$ and unit $\hat d$.
	By Proposition~\ref{F'-gives-calculus2}, $(\Omega,d)$ is the associated first order $\E$-differential calculus.
	Note that $q$ is an epimorphism in $\AModA$, and hence $d$ generates $\Omega_0$ as an $A$-bimodule. 
	Thus $(\Omega_0,d)$ is also a first order differential calculus in the ordinary sense, and is the one obtained by the image construction in Corollary~\ref{remark:functor-Calc_E}.

	Note moreover that, if $M\in\AModA$ is $S$-central, that is, $sm=ms$ for every $s\in S$ and $m\in M$, every derivation
$A\to M$ factors uniquely through $d$. 
	Thus, for $\bk=\mathbb R$ or $\mathbb C$ and $S=Z(A)$, we recover the central-bimodule calculus of \cite[Proposition~2]{dubois1994derivations}.
\end{eg}

\begin{eg}[Kapranov--Corti\~nas]\label{dvmkc:kapranov}\label{dvmkc:cortinas}
	Consider now $\V=\bk$-$\Mod$ and $\E=\mathrm{NCAlg}_{\le\ell}$, for a fixed $\ell\ge0$, with the
commutator filtration of \cite[Definitions~1.1.1 and~1.1.5(a)]{kapranov1998noncommutative}.
	By Proposition~\ref{prop:PI}, the inclusion $(-)_0\colon\E\hookrightarrow\Mon(\V)$ is $\V$-geometrical.
	We now describe the calculus associated to $A\in\E$.
	By functoriality of the commutator filtration, projection onto $A$ sends $F^{\ell+1}_{\mathrm{comm}}(A\oplus\Omega^1_{A,u})$ to $F^{\ell+1}_{\mathrm{comm}}A=0$. 
	It follows that this ideal is $0\oplus N_{\ell,A}$ for an $A$-subbimodule $N_{\ell,A}\subseteq\Omega^1_{A,u}$. 
	We then set
	\[
 	\Omega_0=\Omega^1_{A,u}/N_{\ell,A}, \qquad q\colon\Omega^1_{A,u}\twoheadrightarrow\Omega_0, \qquad d=q d_u.
	\]
	Surjective algebra maps send each commutator-filtration term onto the corresponding term, since the inputs of its generating expressions can be lifted. 
	Hence
	\[
 	A\oplus\Omega_0\cong\frac{A\oplus\Omega^1_{A,u}} {F^{\ell+1}_{\mathrm{comm}}(A\oplus\Omega^1_{A,u})}\in\E,
	\]
and $(1_A,d)$ gives a morphism in $\E$.

	Given $f\colon A\to U'X$ with underlying map $f_0=(f_1,f_2)\colon A\to B\oplus M$, where $B\oplus M\in\E$,
Proposition~\ref{universal-calc-left-adj} gives a unique bimodule map $\overline{f_2}\colon\Omega^1_{A,u}\to f_1^*M$ such that $\overline{f_2}d_u=f_2$.
	The algebra map $f_1\oplus\overline{f_2}\colon A\oplus\Omega^1_{A,u}\to B\oplus M$ vanishes on $F^{\ell+1}_{\mathrm{comm}}(A\oplus\Omega^1_{A,u})$ by functoriality. 
	Therefore $\overline{f_2}$ factors uniquely through $q$, which gives the universal property of $L'A$.
	Thus $KL'A=(A,\Omega_0)$, with unit $(1_A,d)$.
	As in the preceding example, $q$ is an epimorphism, so $(\Omega_0,d)$ is also the canonical first order differential
calculus in the ordinary sense.

	Let us now assume that $\bk$ is a field of characteristic zero and compare this calculus with Corti\~nas's differential forms.
	For a nonnegatively graded algebra $G$, write $F^\bullet_{\mathrm{gr}}G$ for the commutator filtration formed using homogeneous inputs and $[\alpha,\beta]_{\mathrm{gr}}= \alpha\beta-(-1)^{|\alpha||\beta|}\beta\alpha$.
	Equivalently, this is the recursive filtration of \cite[equation~(1)]{cortinas2003rham}, formed using graded commutators.
	If $G$ is a differential graded algebra, the graded Leibniz rule shows that differentiation preserves the number of brackets.
	Thus each $F^r_{\mathrm{gr}}G$ is a differential graded ideal.
	Following \cite[equation~(2)]{cortinas2003rham}, let
	\[
	 \Omega^\bullet_{\mathrm{NC}_\ell,A} =\Omega^\bullet_{A,u}/F^{\ell+1}_{\mathrm{gr}}\Omega^\bullet_{A,u}.
	\]
	To compare its degree-one part with $\Omega_0$, consider the quotient
	\[
	 \pi\colon\Omega^\bullet_{A,u}\longrightarrow \Omega^\bullet_{A,u}/\Omega^{\ge2}_{A,u}=A\oplus\Omega^1_{A,u}.
	\]
	This is an isomorphism in degrees zero and one and sends each filtration term onto the corresponding term. 
	In the target, graded and ordinary commutators agree on homogeneous elements, since products of two degree-one elements vanish; it follows by multilinearity that the two filtrations agree. 
	Hence
	\[
 	\pi(F^{\ell+1}_{\mathrm{gr}}\Omega^\bullet_{A,u})=F^{\ell+1}_{\mathrm{comm}}(A\oplus\Omega^1_{A,u}) =0\oplus 		N_{\ell,A}.
	\]
	Since the filtration ideal is homogeneous and $\pi$ is the identity in degree one, taking degree-one components gives $N_{\ell,A}=(F^{\ell+1}_{\mathrm{gr}}\Omega^\bullet_{A,u})^1$, and thus a natural isomorphism
	\[
	 (\Omega_0,d)\cong \bigl((\Omega^\bullet_{\mathrm{NC}_\ell,A})^1,d\bigr),
	\]
sending the class of $d_ua$ to the class of $d_ua$.
	The canonical first order differential calculus therefore recovers the one-forms of Corti\~nas on the commutator-nilpotent algebras of Kapranov. 
\end{eg}

\subsection{{Algebras over operads and Lawvere theories}}\label{subsec:operads}

Example~\ref{ComMon-canonical-calc} fits in a more general framework, i.e. that of algebras over an operad. 
We start by recalling some definitions and notions regarding operads (for a more complete background we refer to \cite{KvrMay:operads-alg-mod-motives,markl:operads,Leinster_2004}).
Then, we construct $\V$-geometrical functors in \Cref{cor:operad-examples} and \ref{cor:operad-ex-diff-basis}. 

An operad $X$ consists of: for any natural number a set $X(n)$ (which is thought as the set of \emph{operations of arity $n$}); an element $1\in X(1)$ (called the \emph{unit}); for each $n,m_1,\ldots,m_n\in\mathbb{N}$, a function (called \emph{substitution}) 
$$X(n)\times X(m_1)\times\ldots\times X(m_n)\longrightarrow X(\displaystyle\Sigma_{i=1}^nm_i).$$
These data are required to satisfy associativity and unit axioms. 
Moreover, one can also consider \emph{braided/symmetric} operads, where there is an action of the braided/symmetric group on each $X(n)$ (compatible with substitution and unit). 

\begin{eg}\label{eg:operads}
	\begin{enumerate}
		\item[]
		\item Let $\V$ be a monoidal category and $A\in\V$ an object in it. 
		We can define a canonical operad $\Op(A)$ setting $\Op(A)(n):=\V(A^{\otimes n},A)$, unit the identity $1_A\in\V(A,A)$ and substitution defined using functoriality of the tensor product in $\otimes$. 
		
		\item There is a \emph{terminal} operad $\ass$ where $\ass(n)$ consists of a unique object for any $n\in\mathbb{N}$. 
		Unit and composition are defined in the unique possible way.
		
		\item Similarly to $\ass$, there is also a \emph{braided associative} operad $\Brd$, defined as: $\Brd(n):=B_n$ is the braid group of degree $n$, acting on itself by multiplication; the unit is the only element of $B_1$; substitution is given by, for $\xi\in B_m$, $\eta_i\in B_{n_i}$ for $i=1,\ldots,m$, 
		$$n_{\xi1}+\ldots+n_{\xi m}\xrightarrow{<\xi>}n_{1}+\ldots+n_{m}\xrightarrow{\eta_1+\ldots\eta_m}n_{1}+\ldots+n_{m}.$${}
		The terminal braided operad $\Com$ has $\Com(n)=\{*\}$ with the trivial action of $B_n$. There is a unique morphism of braided operads $\Brd\to\Com$.

	\end{enumerate}
\end{eg}
Morphisms of operads $f\colon X\to Y$ consists of, for any $n$, functions $f_n\colon X(n)\to Y(n)$ compatible with the unit and functions of the operads. 

Given an operad $X$ and a monoidal category $\V$, one can define the category $\V^X$ of $X$-algebras in $\V$. 
A concise way of defining $X$-algebras in $\V$ is as objects $A\in\V$ together with an \emph{operad} morphism $a\colon X\to\Op(A)$.
Morphisms of $X$-algebras $\varphi\colon(A,a)\to (B,b)$ in $\V$ are given by maps $\varphi\colon A\to B$ in $\V$ respecting the actions, i.e. making the following square commutative for any $n$. 
\[\begin{tikzcd}[ampersand replacement=\&]
	{X(n)} \& {\V(A^{\otimes n},A)} \\
	{\V(B^{\otimes n},B)} \& {\V(A^{\otimes n},B)}
	\arrow["{a_n}", from=1-1, to=1-2]
	\arrow["{b_n}"', from=1-1, to=2-1]
	\arrow["{\varphi\circ-}", from=1-2, to=2-2]
	\arrow["{-\circ\varphi^{\otimes n}}"', from=2-1, to=2-2]
\end{tikzcd}\]

\begin{eg}\label{eg:algebras-operads}
\begin{enumerate}
\item[]	
\item Let $\V$ be a monoidal category. The category $\V^{\ass}$ of algebras of $\ass$ in $\V$ is isomorphic to $\Mon(\V)$. 
	\item Similarly, for a braided monoidal category $\V$, taking algebras equivariant for the braid-group actions, we have $\V^{\Brd}\cong\Mon(\V)$ and $\V^{\Com}\cong\CMon(\V)$. 
\end{enumerate}
\end{eg}

Clearly, there is a forgetful functor $W^X\colon\V^X\to\V$, sending an algebra $(A,a)$ to its underlying object $A$. 

\begin{prop}[{\cite[Section 3]{LACK199965}}]\label{prop:monadic-alg-operad}
	Let $\V$ a monoidal category which is cocomplete and such that, for some regular cardinal $\alpha$, each $A\otimes -$ and $-\otimes A$ preserve $\alpha$-filtered colimits.
	Then, the forgetful functor $W\colon \V^X\to\V$ is $\alpha$-ary monadic (that is, monadic and preserve $\alpha$-filtered colimits). 
	Therefore, if $\V$ is locally presentable, also $\V^X$ will be such. 
\end{prop}

\begin{rmk}\label{rmk:mon-discr-isofib}
	For any operad $X$, the forgetful functor $W^X\colon \V^X\to\V$ is an isofibration: given an $X$-algebra $(A,a)$ and an isomorphism $\phi\colon A\cong B$ in $\V$, transporting the operations of $A$ along $\phi$ defines an $X$-algebra structure on $B$ for which $\phi$ is an isomorphism of $X$-algebras.
	Actually, in the case of $X=\ass$ we can see that $W^{\ass}\colon\VMon\cong\V^{\ass}\to\V$ is a \emph{discrete isofibration}\footnote{Discrete isofibrations are isofibrations where the lifting of isomorphism is unique.} (for any monoidal category $\V$).
	Indeed, let us consider a monoid $(A,m,i)\in\VMon$ and an isomorphism $f\colon A\cong B\in\V$. Then, there exists a unique morphism $B\otimes B\to B$ making the following diagram commutative. 
	\[\begin{tikzcd}[ampersand replacement=\&]
		{A\otimes A} \& {B\otimes B} \\
		A \& B
		\arrow["{f\otimes f}", from=1-1, to=1-2]
		\arrow["m"', from=1-1, to=2-1]
		\arrow["{\exists!\,n}", dashed, from=1-2, to=2-2]
		\arrow["f"', from=2-1, to=2-2]
	\end{tikzcd}\]
	This morphism $B\otimes B\to B$, together with $j:=fi\colon I\to A\to B$, provides a monoid structure $(B,n,j)\in\VMon$. 
\end{rmk}

From the concise definition of algebra for an operad, it is straightforward to see that any morphism of operads $f\colon X\to Y$ induces a functor between the categories of algebras making the following triangle commutative. 
\begin{equation}\label{eq:operad-map-induces-alg-map}
	\begin{tikzcd}[ampersand replacement=\&]
		{\V^Y} \&\& {\V^X} \\
		\& \V
		\arrow["{\V^f}", from=1-1, to=1-3]
		\arrow["{W^Y}"', from=1-1, to=2-2]
		\arrow["{W^X}", from=1-3, to=2-2]
	\end{tikzcd}
\end{equation} 
	Explicitly, for any $(A,a\colon Y\to\Op(A))\in\V^Y$, $\V^f(A,a)$ is defined as $A$ together with the operad morphism
$$X\xrightarrow{f}Y\xrightarrow{a}\Op(A).$$

We now construct $\V$-geometrical functors from morphisms of operads $f\colon\ass\to X$. We first recall a property of isofibrations. 

\begin{lemma}\label{lemma:isofibrations}
	Let $F\colon\A\to\B$ and $G\colon\B\to\C$ be functors. 
	If $GF$ is an isofibration and $G$ a discrete isofibration, then $F$ is an isofibration as well. 
	
	If, moreover, also $GF$ is a discrete isofibration, then $F$ is discrete as well.
\end{lemma}

\begin{proof}
	Let us consider $X\in\A$, $Y\in\B$ and an isomorphism $\phi\colon FX\cong Y$.
	Then, also $G\phi\colon GFX\cong GY$ is an isomorphism. 
	Using the isofibration $GF$, we know that there exist $X'\in\A$ and $\psi\colon X\cong X'$ such that $GF\psi= G\phi$. 
	Since $G$ is a discrete isofibration, and both $F\psi$ and $\phi$ are lifting of the isomorphism $G\phi$, then they must be equal $F\psi=\phi$ (and in particular, $FX'=Y$). 
	Therefore, $\psi$ is a lifting of $\phi$. 
	
	Clearly, if $GF$ is discrete, then the lifting is unique and so $F$ is discrete as well. 
\end{proof}

We are now ready to move to the main result. 

\begin{lemma}\label{lemma:operad-alg-morph-almost-geom}
	Let $\V$ satisfy the hypotheses of \Cref{prop:monadic-alg-operad}. 
	Given a morphisms of operads $f\colon \ass\to X$, the induced functor $\V^f\colon \V^X\to \V^{\ass}\cong\Mon(\V)$ between the categories of algebras as in \eqref{eq:operad-map-induces-alg-map} is faithful, an isofibration, accessible, and continuous. 
	
\end{lemma}

\begin{proof}
		First, since both $W^X$ and $W^{\ass}$ are faithful (as they are monadic functors by \Cref{prop:monadic-alg-operad}), then $\V^f$ is faithful as well.
		
		The functor $W^X$ is an isofibration (see \Cref{rmk:mon-discr-isofib}).
		Moreover, $W^{\ass}\colon\Mon(\V)\to \V$ is a discrete isofibration (see \Cref{rmk:mon-discr-isofib}). 
		We conclude that $\V^f$ is an isofibration using \Cref{lemma:isofibrations}
		
		Since $W^X$ and $W^{\ass}$ are conservative, accessible, and continuous (see \cite[Section 3]{LACK199965}), then also $\V^f$ is accessible and continuous by \cite[Proposition~11.8]{ARV2010algebraic}. 
\end{proof}

\begin{cor}\label{cor:operad-examples}
	Let $\V$ as in \Cref{thm:canonical-calc}, $X$ be an operad, and $f\colon \ass\to X$ be an operad morphism. 
	Then, the induced functor $\V^f\colon \V^X\to \Mon(\V)$ is {\verysuitable}.
\end{cor}

\begin{proof}
	{By \Cref{prop:monadic-alg-operad}, $\V^X$ is locally presentable, and by \Cref{lemma:operad-alg-morph-almost-geom}, $\V^f$ is a faithful isofibration which is accessible and continuous. In particular, $\V^X$ has pullbacks preserved by $\V^f$, so \Cref{lem:pullbacks-imply-geometrical} applies. The conclusion follows from \Cref{thm:canonical-calc}.}
\end{proof}

\Cref{ComMon-canonical-calc}, for $n=1$, has the same form for the morphism of braided operads $\Brd\to\Com$: the induced functor is the inclusion $\CMon(\V)\hookrightarrow\Mon(\V)$, whose properties were verified directly in that example.

Next we consider a change of the base monoidal category.
First, let us notice that, given any operad $X$, any (lax) monoidal functor $\Phi\colon\V_1\to\V_2$ induces a functor $\Phi^X\colon\V_1^X\to\V_2^X$ between the categories of algebras evaluated in the respective monoidal categories, making the following square commutative.
\[\begin{tikzcd}[ampersand replacement=\&]
	{\V_1^X} \& {\V_2^X} \\
	{\V_1} \& {\V_2}
	\arrow["{\Phi^X}", from=1-1, to=1-2]
	\arrow["{W^X_1}"', from=1-1, to=2-1]
	\arrow["{W^X_2}", from=1-2, to=2-2]
	\arrow["\Phi"', from=2-1, to=2-2]
\end{tikzcd}\]
Precisely, for any algebra $(A,a\colon X\to\Op(A))\in\V_1^X$, we define an algebra structure on $\Phi A\in\V_2$ through the composite
{\begin{equation}\label{eq:phi^X}
\begin{tikzcd}[ampersand replacement=\&]
	X(n) \& {\V_1(A^{\otimes n},A)} \& {\V_2(\,\Phi(A^{\otimes n}),\Phi(A)\,)} \& {\V_2(\,\Phi(A)^{\otimes n},\Phi(A)\,),}
	\arrow["a", from=1-1, to=1-2]
	\arrow["\Phi", from=1-2, to=1-3]
	\arrow["{-\circ\Phi^2}", from=1-3, to=1-4]
\end{tikzcd}
\end{equation}}
where $\Phi^2\colon\Phi(A)^{\otimes n}\to\Phi(A^{\otimes n})$ is the unique morphism generated by the monoidal constraint of $\Phi$. \footnote{Note that this is unique by the coherence theorem for monoidal functors \cite{JoyalStreet:brad-mon-cat}.}

\begin{lemma}\label{lem:change-of-basis-faith-cons-acc-cont}
	Let $\V_1$ and $\V_2$ satisfy the hypotheses of \Cref{prop:monadic-alg-operad}. 
	If $\Phi$ is conservative, accessible, and continuous, then also $\Phi^X$ is such.
\end{lemma}

\begin{proof}
	This follows again by \cite[Section 3]{LACK199965} (which guarantees that $W^X_1$ and $W^X_2$ are conservative, accessible and continuous) and \cite[Proposition~11.8]{ARV2010algebraic}. 
\end{proof}

\begin{cor}\label{cor:operad-ex-diff-basis}
		Let $\V_1$ and $\V_2$ as in \Cref{thm:canonical-calc}, $X$ be an operad, $f\colon \ass\to X$ be an operad morphism, and $\Phi\colon\V_1\to\V_2$ be a lax monoidal functor which is conservative, accessible, continuous, and a discrete isofibration. 
		If there exists a diagonal filler $\Phi^f$ of the square below, then it is a faithful isofibration, conservative, accessible, and continuous. 
		\[\begin{tikzcd}[ampersand replacement=\&]
			{\V_1^X} \& {\Mon(\V_1)} \\
			{\V_2^X} \& {\Mon(\V_2)}
			\arrow["{\V_1^f}", from=1-1, to=1-2]
			\arrow["{\Phi^X}"', from=1-1, to=2-1]
			\arrow["{\Phi^{\ass}}", from=1-2, to=2-2]
			\arrow["{\Phi^f}"{description}, dashed, from=2-1, to=1-2]
			\arrow["{\V_2^f}"', from=2-1, to=2-2]
		\end{tikzcd}\]
		Furthermore, $\Phi^f$ is $\V_1$-geometrical. 
\end{cor}

\begin{proof}
	First, $\Phi^f$ is faithful since $\V_2^f$ is such (by \Cref{lemma:operad-alg-morph-almost-geom}). 
	Moreover, $\Phi^{\ass}$ is a discrete isofibration by \Cref{lemma:isofibrations}, since it makes the following diagram commutative with $W^{\ass}_1$, $W^{\ass}_2$ and $\Phi$ all discrete isofibrations (see \Cref{rmk:mon-discr-isofib} and hypotheses). 
	\[\begin{tikzcd}[ampersand replacement=\&]
		{\Mon(\V_1)} \& {\Mon(\V_2)} \\
		{\V_1} \& {\V_2}
		\arrow["{\Phi^{\ass}}", from=1-1, to=1-2]
		\arrow["{W^{\ass}_1}"', from=1-1, to=2-1]
		\arrow["{W^{\ass}_2}", from=1-2, to=2-2]
		\arrow["\Phi"', from=2-1, to=2-2]
	\end{tikzcd}\]
Therefore, since $\V_2^f$ is an isofibration by \Cref{lemma:operad-alg-morph-almost-geom}, again applying \Cref{lemma:isofibrations} we get that $\Phi^f$ is an isofibration. 

Now, both $\V_2^f$ and $\Phi^{\ass}$ are conservative, accessible and continuous, respectively by \Cref{lemma:operad-alg-morph-almost-geom} and \Cref{lem:change-of-basis-faith-cons-acc-cont}. 
Thus, by \cite[Proposition~11.8]{ARV2010algebraic}, also $\Phi^f$ is such. 

By \Cref{prop:monadic-alg-operad}, $\V_2^X$ is locally presentable. Thus \Cref{thm:canonical-calc} shows that $\Phi^f$ is $\V_1$-geometrical. 
\end{proof}

\begin{rmk}
	It is interesting to notice that the only property of $\ass$ that we used in the proof above is that $W^{\ass}\colon \Mon(\V)\to\V$ is a discrete isofibration (\Cref{rmk:mon-discr-isofib}). 
\end{rmk}

\begin{rmk}\label{rmk:operads-morph-of-geom-fun}
	For the diagonal filler in \Cref{cor:operad-ex-diff-basis}, $(\Phi^X,1)$ is a morphism from $(\V_1^X,\V_1^f)$ to $(\V_2^X,\Phi^f)$ in $\V_1\operatorname{-Geom}$. Indeed, both functors are geometrical by Corollaries~\ref{cor:operad-examples} and~\ref{cor:operad-ex-diff-basis}, and $\Phi^f\circ\Phi^X=\V_1^f$.
\end{rmk}

\begin{cor}\label{cor:lawvere-theory-examples}
Let $\bk$ be a commutative ring and let $\mathbb T$ be a Lawvere theory under the associative Lawvere theory $\mathbb T_{\kalg}$ of unital $\bk$-algebras. Then the restriction functor
\[
(-)_0\colon\operatorname{Mod}_{\mathbf{Set}}(\mathbb T)\longrightarrow\kalg=\Mon(\bk\tx{-}\Mod)
\]
is $\bk\tx{-}\Mod$-geometrical.
\end{cor}
\begin{proof}
The category of models is locally finitely presentable by \cite[Example~6.22(1)]{ARV2010algebraic}. Restriction is faithful and preserves limits and filtered colimits, which are computed on underlying sets; see \cite[Proposition~11.8 and Corollary~11.9]{ARV2010algebraic}. It is an isofibration since algebra isomorphisms transport the $\mathbb T$-operations. The conclusion follows from \Cref{thm:canonical-calc}.
\end{proof}

\begin{eg}\label{eg:Cinfty-lawvere}
The Lawvere theory with objects $\mathbb R^n$ and smooth maps has $C^\infty$-rings as its models \cite[Definition~1]{stel2013cosimplicial}. Polynomial evaluation defines a morphism from the theory of associative unital $\R$-algebras to this theory, factoring through the theory of commutative real algebras. Thus \Cref{cor:lawvere-theory-examples} shows that the usual forgetful functor $C^\infty\tx{-Ring}\to\Mon(\Rvect)$ is $\Rvect$-geometrical, as in \Cref{eg:final-Cinfty-rings}.
\end{eg}

\begin{eg}[Associative Poisson algebras]\label{eg:Poisson-calculus}
Let $\bk$ be a nonzero commutative ring, let $\V=\bk\text{-}\Mod$, and let $\E$ be the category of unital associative Poisson $\bk$-algebras.
Explicitly, an object of $\E$ is a unital associative $\bk$-algebra equipped with a Lie bracket $\{-,-\}$ satisfying
\[
 \{a,bc\}=\{a,b\}c+b\{a,c\},
\]
and a morphism of $\E$ is an algebra morphism preserving the bracket.
The category $\E$ is the category of models of a Lawvere theory under that of unital associative $\bk$-algebras, hence the forgetful functor $(-)_0\colon\E\to\Mon(\V)$ is $\V$-geometrical by \Cref{cor:lawvere-theory-examples}.

We will first describe $\EEModE$.
Fix $E\in\E$ and write $A=E_0$.
By \eqref{def-EModE}, an object $M\in\EEModE$ consists of an $A$-bimodule $M_0$ together with a Poisson bracket on the square-zero extension $A\oplus M_0$ for which the zero section $a\mapsto(a,0)$ and the projection $(a,m)\mapsto a$ are morphisms of $\E$.
Equivalently, the bracket of $A\oplus M_0$ restricts to that of $E$ on $A\oplus0$, and $0\oplus M_0$ is a Lie ideal.
The bracket is therefore of the form
\begin{equation}\label{eq:Poisson-bimodule-bracket}
 \{(a,m),(b,n)\}=\bigl(\{a,b\},\,\{a,n\}+\{m,b\}+\{m,n\}\bigr),
\end{equation}
where $\{a,n\}$, $\{m,b\}$, and $\{m,n\}$ denote the brackets of the corresponding elements of $A\oplus M_0$, all of which lie in $M_0$.
Note that the bracket $\{m,n\}$ of two elements of $M_0$ is part of the structure, and is not required to vanish.

We next construct the canonical calculus on $E$.
Form the coproduct $P=E \amalg E$
in $\E$, with inclusions $i_0,i_1\colon E\to P$ and codiagonal $p\colon P\to E$, so that $pi_0=pi_1=1_E$.
Further, set $I=\ker p$.
Since $p$ is a morphism of $\E$, the kernel $I$ is not only a two-sided ideal but also a Lie ideal: for $x\in P$ and $y\in I$ we have $p\{x,y\}=\{p(x),p(y)\}=0$, that is,
\[
 \{P,I\}\subseteq I.
\]
Let $I^2$ be the square of $I$ formed using the associative multiplication, namely the $\bk$-submodule spanned by the products $yz$ with $y,z\in I$.
It is a two-sided ideal, and the Leibniz rule $\{x,yz\}=\{x,y\}z+y\{x,z\}$ gives
\[
 \{P,I^2\}\subseteq\{P,I\}I+I\{P,I\}\subseteq I^2.
\]
Thus $I^2$ is a Poisson ideal, so that $P/I^2$ is an object of $\E$ and the quotient map $q\colon P\to P/I^2$ is a morphism of $\E$.
Since $I^2\subseteq I$, the codiagonal factors as $p=\bar pq$ for a unique morphism $\bar p\colon P/I^2\to E$ of $\E$, and $\ker\bar p=I/I^2$.
Writing $\bar\imath_j=qi_j$ for $j=0,1$, we have $\bar p\,\bar\imath_0=\bar p\,\bar\imath_1=1_E$.
In particular,
\begin{equation}\label{eq:Poisson-retraction}
 E\xrightarrow{\ \bar\imath_0\ }P/I^2\xrightarrow{\ \bar p\ }E
\end{equation}
is a retraction in $\E$.

Next we identify \eqref{eq:Poisson-retraction} with a square-zero extension.
Let
\[
 \Omega_0:=I/I^2=\ker\bar p,
\]
regarded as an $A$-bimodule via $a\cdot\omega\cdot b=\bar\imath_0(a)\,\omega\,\bar\imath_0(b)$.
(Using $\bar\imath_1$ in place of $\bar\imath_0$ gives the same actions, since $\bar\imath_1(a)-\bar\imath_0(a)\in I/I^2$ and the product of two elements of $I/I^2$ vanishes.)
Since $\bar p$ is split by $\bar\imath_0$, the map
\[
 \phi\colon A\oplus\Omega_0\longrightarrow P/I^2,\qquad
 (a,\omega)\longmapsto\bar\imath_0(a)+\omega,
\]
is an isomorphism of $\bk$-modules, with inverse $x\mapsto\bigl(\bar p(x),\,x-\bar\imath_0\bar p(x)\bigr)$.
Moreover,
\[
 \phi(a,\omega)\,\phi(b,\eta)
 =\bar\imath_0(ab)+a\cdot\eta+\omega\cdot b+\omega\eta
 =\phi(ab,\,a\cdot\eta+\omega\cdot b),
\]
because $\omega\eta=0$ in $P/I^2$.
Hence $\phi$ is an isomorphism of algebras from the square-zero extension of $A$ by $\Omega_0$ onto $P/I^2$, which identifies the zero section with $\bar\imath_0$ and the projection with $\bar p$.
Transporting the Poisson bracket of $P/I^2$ along $\phi$ therefore turns \eqref{eq:Poisson-retraction} into an object $\Omega\in\EEModE$, with underlying bimodule $\Omega_0=I/I^2$ and $U'_E(\Omega)\cong P/I^2$; from now on we identify $U'_E(\Omega)$ with $P/I^2$ along $\phi$.
Its bracket is given by \eqref{eq:Poisson-bimodule-bracket}, where for instance $\{a,\eta\}=\{\bar\imath_0(a),\eta\}$ is computed in $P/I^2$; these brackets lie in $\Omega_0$ because $I/I^2$ is a Lie ideal of $P/I^2$.

The differential is
\[
 d\colon A\longrightarrow\Omega_0,\qquad
 da=\bar\imath_1(a)-\bar\imath_0(a),
\]
which takes values in $\ker\bar p=\Omega_0$ because $\bar p\,\bar\imath_1=\bar p\,\bar\imath_0$.
Then $\bar\imath_1(a)=\bar\imath_0(a)+da=\phi(a,da)$, so the second inclusion induces a section $\hat d:=\bar\imath_1\colon E\to U'_E(\Omega)$ in $\E$ with underlying map $(1_A,d)$.
In other words, $d$ satisfies the strong Leibniz condition, and $(\Omega,d)$ is a generalized first order $\E$-differential calculus over $E$.
Explicitly, the strong Leibniz condition means that the graph $(1_A,d)$ preserves the product of the square-zero extension and the bracket \eqref{eq:Poisson-bimodule-bracket}, that is,
\[
 \begin{split}
 d(ab)&=a\,db+(da)b,\\
 d\{a,b\}&=\{a,db\}+\{da,b\}+\{da,db\}.
 \end{split}
\]

One can then straightforwardly verify that $\Omega$ is the value at $E$ of the left adjoint $L'$ of $U'\colon\EMod\to\E$, with unit $\hat d$.

Forgetting the Poisson structure gives the generalized first order differential calculus $(\Omega_0,d)$ in $\V$ of \Cref{rmk:underlying-generalized-calculus}.
The comparison map $h\colon\Omega^1_{A,u}\to\Omega_0$ in the proof of \Cref{remark:functor-Calc_E} is the $A$-bimodule map with $hd_u=d$, so that $h(a\,d_ub)=a\,db$.
Its image is the $A$-subbimodule of $\Omega_0$ generated by $dA$.
Hence the first order differential calculus of \Cref{remark:functor-Calc_E} is
\[
 \Omega'_0
 =\operatorname{im}\bigl(\Omega^1_{A,u}\to\Omega_0\bigr)
 =AdA,
\]
with differential $d'$ given by corestriction.
Note that $\Omega'_0$ may be a proper subbimodule of $\Omega_0$.
Indeed, the coproduct $P$ is formed in $\E$, and it is generated by $i_0(A)\cup i_1(A)$ as a Poisson algebra, but not necessarily as an algebra.
Accordingly, $\Omega_0$ is generated by $dA$ under the two actions together with the brackets of \eqref{eq:Poisson-bimodule-bracket}, but not necessarily under the two actions alone; compare the discussion preceding \Cref{remark:functor-Calc_E}.

As a toy example, showing the dependence of the construction on the choice of Poisson bracket, consider $A=\bk[x,y]$. One can readily verify that the zero Poisson bracket yields the universal first order differential calculus, whereas the bracket $\{x,y\}=1$ yields the classical K\"ahler first order differential calculus.
\end{eg}

\begin{rmk}\label{rmk:Poisson-not-Beck}
In \Cref{alternative} we pointed out that first order differential calculi could also be approached through the theory of Beck modules.
In the setting of \Cref{eg:Poisson-calculus}, the fibrewise addition
\[
 (A\oplus M_0)\times_A(A\oplus M_0)\longrightarrow A\oplus M_0,\qquad
 (a,m,n)\longmapsto(a,m+n),
\]
of an object $M\in\EEModE$ preserves the bracket \eqref{eq:Poisson-bimodule-bracket} if and only if $\{m,n\}=0$ for all $m,n\in M_0$.
Thus $M$ determines a Beck module precisely when the bracket of any two elements of $M_0$ vanishes.
We leave it to the reader to check that the term $\{da,db\}$ in the strong Leibniz condition can be nonzero, so that the internal calculus $\Omega$ need not be a Beck module in general; this shows that our framework is not captured by the theory of Beck modules.
\end{rmk}

\begin{eg}[Rota--Baxter algebras]\label{eg:rota-baxter}
Let $\bk$ be a nonzero commutative ring, let $\V=\bk\text{-}\Mod$, and fix $\lambda\in \bk$.
Let $\E$ be the category of unital associative Rota--Baxter $\bk$-algebras of weight $\lambda$, as in \cite[Section~2.1.2]{guo2021representations}.
An object $E=(A,R)$ consists of a unital associative $\bk$-algebra $A$ and a $\bk$-linear operator $R\colon A\to A$ satisfying
\[
 R(a)R(b)=R\bigl(R(a)b+aR(b)+\lambda ab\bigr),
\]
and a morphism of $\E$ is an algebra morphism commuting with the operators.
This is again the category of models of a Lawvere theory under that of unital associative $\bk$-algebras.
Hence $(-)_0\colon\E\to\Mon(\V)$ is $\V$-geometrical by \Cref{cor:lawvere-theory-examples}.

We will first describe $\EEModE$.
Fix $E=(A,R)$.
By \eqref{def-EModE}, an object $M\in\EEModE$ consists of an $A$-bimodule $M_0$ together with a Rota--Baxter operator $Q$ on the square-zero extension $A\oplus M_0$ for which the zero section and the projection are morphisms of $\E$.
Compatibility with the zero section gives $Q(a,0)=(R(a),0)$, and compatibility with the projection gives $Q(0,m)\in0\oplus M_0$.
Hence $Q=R\oplus T$ for a unique $\bk$-linear operator $T\colon M_0\to M_0$.
Evaluating the Rota--Baxter identity of $R\oplus T$ on the pairs $\bigl((a,0),(0,m)\bigr)$ and $\bigl((0,m),(a,0)\bigr)$ gives
\[
 \begin{split}
 R(a)T(m)&=T\bigl(R(a)m+aT(m)+\lambda am\bigr),\\
 T(m)R(a)&=T\bigl(T(m)a+mR(a)+\lambda ma\bigr),
 \end{split}
\]
whereas on $\bigl((a,0),(b,0)\bigr)$ it is the identity satisfied by $R$, and on $\bigl((0,m),(0,n)\bigr)$ both sides vanish.
Since both sides of the Rota--Baxter identity are $\bk$-bilinear, $R\oplus T$ is a Rota--Baxter operator if and only if the two displayed equations hold.
Thus an object of $\EEModE$ amounts to a pair $(M_0,T)$ satisfying these equations; these are precisely the strict Rota--Baxter bimodules of \cite[Section~3.2, equations~(9)--(10)]{guo2021representations}.
For a $\bk$-linear map $d\colon A\to M_0$, the strong Leibniz condition means that its graph is a morphism
\[
 (1_A,d)\colon(A,R)\longrightarrow(A\oplus M_0,R\oplus T)
\]
in $\E$.
Equivalently,
\[
 d(ab)=a\,db+(da)b,\qquad
 d(R(a))=T(da).
\]

Next, we give an explicit construct of the canonical calculus on $E$, and leave it to the reader to show that it actually satisfies the required universal property.
Form the coproduct $P=E\amalg E$ in $\E$, with Rota--Baxter operator $R_P$, and define $i_0$, $i_1$, $p$, and $I=\ker p$ as in \Cref{eg:Poisson-calculus}.
Since $pR_P=Rp$, the ideal $I$ is stable under $R_P$; this plays the role of the inclusion $\{P,I\}\subseteq I$ in the Poisson case.
Unlike the Poisson bracket, however, $R_P$ satisfies no Leibniz rule, so there is no reason for $I^2$ to be stable under $R_P$.
We therefore let $J$ be the smallest two-sided ideal of $P$ which contains $I^2$ and is stable under $R_P$, that is, the intersection of all such ideals.
Since $I$ is one of them, $J\subseteq I$.
The operator $R_P$ descends to a Rota--Baxter operator on $P/J$, so that $P/J$ is an object of $\E$ and the quotient map $q\colon P\to P/J$ is a morphism of $\E$.
Moreover, $p$ factors through a morphism $\bar p\colon P/J\to E$ of $\E$ with $\ker\bar p=I/J$, and the product of two elements of $I/J$ vanishes because $I^2\subseteq J$.
Writing $\bar\imath_j=qi_j$, the argument of \Cref{eg:Poisson-calculus} shows that $(a,\omega)\mapsto\bar\imath_0(a)+\omega$ is an isomorphism of algebras from the square-zero extension of $A$ by the $A$-bimodule
\[
 \Omega_0:=I/J
\]
onto $P/J$, which identifies the zero section with $\bar\imath_0$ and the projection with $\bar p$.
Since $\bar\imath_0$ and $\bar p$ commute with the operators, the operator of $P/J$ corresponds to $R\oplus T$, where $T\colon\Omega_0\to\Omega_0$ is the restriction to $I/J$ of the operator induced by $R_P$.
The retraction $E\to P/J\to E$ thus defines an object $\Omega\in\EEModE$, and
\[
 da=\bar\imath_1(a)-\bar\imath_0(a)
\]
defines a map $d\colon A\to\Omega_0$ for which $\hat d:=\bar\imath_1\colon E\to U'_E(\Omega)$ has underlying map $(1_A,d)$.

Forgetting $T$ gives the generalized first order differential calculus $(\Omega_0,d)$ in $\V$.
As in \Cref{eg:Poisson-calculus}, the ordinary calculus of \Cref{remark:functor-Calc_E} is
\[
 \Omega'_0
 =\operatorname{im}\bigl(\Omega^1_{A,u}\to\Omega_0\bigr)
 =AdA,
\]
with differential $d'$ given by corestriction.
The strong Leibniz condition shows that $T$ preserves $dA$, since $T(da)=d(R(a))$; however, the subbimodule $AdA$ need not be preserved by $T$.

Now, as a toy example which demonstrates the sensitivity of the construction to the input data of a given Rota--Baxter operator, consider the case $\bk$ is a field of characteristic $0$, $\lambda=0$, and $A=\bk[x]$.  
One can straightforwardly compute that the zero Rota--Baxter operator yields the universal first order differential calculus, whereas the integration operator $R(x^n)=x^{n+1}/(n+1)$ yields the zero first order differential calculus.
\end{eg}

\section{Differential calculi}\label{sec:extalg}

In this section we generalize the notion of differential calculi (also known as exterior algebras) from associative algebras to our setting (i.e. monoids in a monoidal additive category).

From this point on we fix $\V$ to be a monoidal additive category with finite limits and colimits such that the tensor product preserves finite colimits in each variable. Note that in this setting, by \Cref{prop:otimes_A-preserves-coequ}, we have a well defined tensor product of modules.

Recall from Example~\ref{MonoidInChain} that a monoid in $\Ch(\V)$ is a differentially graded algebra.

\begin{defi}\label{def:ext-alg}
	A \emph{differential calculus} $\Omega^\bullet_d$ over a monoid $(A,m,i)$ in $\V$, is a monoid $(\Omega^{\bullet}_d=\bigoplus_{n\ge 0}\Omega^n_d,d,\wedge,j)$ in $\Ch(\V)$ such that
	\begin{enumerate}
		\item $(\Omega^0_d,\wedge^{0,0},j)=(A,m,i)$;
		\item (\emph{surjectivity condition}) $A$ generates $\Omega^{\bullet}_d$ via $d$ and $\wedge$; that is, for any $n\in\mathbb{N}$, the following composite is an epimorphism in $\V$. 
\[\begin{tikzcd}[ampersand replacement=\&]
	{A^{\otimes n+1}} \&\& {A\otimes\Omega^1_d\otimes(\Omega^1_d)^{\otimes n-1}} \&\&\& {(\Omega^1_d)^{\otimes n}} \\
	\&\&\&\&\& {\Omega_d^n}
	\arrow["{1_{A}\otimes d^{\otimes n}}", from=1-1, to=1-3]
	\arrow[""{name=0, anchor=center, inner sep=0}, "{p_n}"', curve={height=6pt}, two heads, from=1-1, to=2-6]
	\arrow[""{name=1, anchor=center, inner sep=0}, "{\mu_d\otimes1_{(\Omega^1_d)^{\otimes n-1}}}", from=1-3, to=1-6]
	\arrow["\wedge", from=1-6, to=2-6]
	\arrow["{:=}"{marking, allow upside down}, draw=none, from=0, to=1]
\end{tikzcd}\]
	\end{enumerate}
\end{defi}

\begin{rmk}
	Note that if $M^\bullet$ is a monoid in $\Ch(\V)$, then:\begin{itemize}
		\item $M^0$ is a monoid in $\V$ with unit the same as $M^\bullet$ and monoid action
		$$ \wedge_M^{0,0}\colon M^0\otimes M^0\to M^0;$$
		\item each $M^n$ inherits a structure of $M^0$-bimodule with left and right actions given by
		\begin{center}
			$\wedge_M^{0,n}\colon M^0\otimes M^n\to M^n$ \quad and \quad $\wedge_M^{n,0}\colon M^n\otimes M^0\to M^n$.
		\end{center}
	\end{itemize} 
	Indeed, the monoid and bimodule axioms follow from the associativity and unitality axioms of $\wedge_M$. Similarly, if $f^\bullet\colon M^\bullet \to N^\bullet$ is a morphism in $\Ch(\V)$, then $f_0\colon M^0\to N^0$ is a morphism of monoids and 
	$$ f^n\colon M^n\to (f^0)^*N^n $$
	defines a morphism of $M^0$-bimodules.
	In particular, if $\Omega^\bullet_d$ is a differential calculus over a monoid $A$, then $\Omega^n_d$ inherits a structure of $A$-bimodule, for any $n>0$.
\end{rmk}

\subsection{Maximal prolongation}

In this subsection, we show a standard procedure\footnote{In principle, one could consider maximal prolongation procedure in $\Mod(\E)$, yielding a notion of $\E$-differential calculus. We do not pursue this idea here.} to extend a first order differential calculus to a differential calculus, which generalizes what is commonly known as the \emph{maximal prolongation} in noncommutative differential geometry. 

\begin{defi}
	Given a monoid $(A,m,i)$ in $\V$, the maximal prolongation $\Omega^\bullet_u$ of the universal calculus $\Omega^1_u$ is defined by\begin{enumerate}
		\item $\Omega^n_u:= (\Omega^1_u)^{\otimes_An}$;
		\item the operation $\wedge^{n,m}_u\colon \Omega^n_u\otimes \Omega^m_u\to\Omega^{n+m}_u$ is the quotient map
		 $$ \Omega^n_u\otimes \Omega^m_u\longrightarrow\Omega^n_u\otimes_A \Omega^m_u\cong \Omega^{\otimes_An+m}_u $$
		defining the tensor product over $A$.
		\item the differential $d^n_u\colon \Omega^n_u\to \Omega^{n+1}_u$ is 
		$$ d_u^n\colon\Omega_u^n\xrightarrow{\iota_u^n} A^{\otimes n+1}\xrightarrow{d^{\otimes n+1}}\Omega_u^{\otimes n+1}\xrightarrow{\wedge_u}\Omega_u^{n+1}. $$
	\end{enumerate}
\end{defi}

\begin{rmk}\label{compatibility-for d^n}
	The following observation will be useful in the proofs. Given $\Omega_u^\bullet$ (or indeed any differential calculus $\Omega_d^\bullet$) we keep denoting by $d_u=d_u^0\colon A\to\Omega_u^1$ the differential on degree $0$. Then the following diagram commutes 
		\begin{center}
		\begin{tikzpicture}[baseline=(current  bounding  box.south), scale=2]
			
			\node (01) at (-1,0.8) {$A^{\otimes 2}\otimes A^{\otimes n}$};
			\node (a1) at (1.9,0.8) {$\Omega_u^{1}\otimes \Omega_u^{\otimes n-1}$};
			\node (00) at (-1,0) {$A^{\otimes n+1}$};
			\node (a0) at (1.9,0) {$\Omega_u^{n}$};

			\path[font=\scriptsize]

			(01) edge [->] node [above] {$-(d_u\cdot 1_A)\!\otimes\! ((1_A\cdot d_u)\otimes d_u^{\otimes n-2})$} (a1)
			(01) edge [->] node [left] {$1_A\otimes m\otimes 1_{A^{\otimes n-1}}$} (00)
			(a1) edge [->] node [right] {$\wedge_u^{1,n-1}$} (a0)
			(00) edge [->] node [below] {$-(d_u\cdot 1_A)\wedge d_u^{\wedge n-1}$} (a0);
		\end{tikzpicture}	
	\end{center} 
	for any $n\geq 1$. Indeed, this can be broken into the commutativity of the two squares below
		\begin{center}
		\begin{tikzpicture}[baseline=(current  bounding  box.south), scale=2]
			
			\node (01) at (-1.3,0.8) {$A^{\otimes 2}\otimes A^{\otimes n}$};
			\node (11) at (2,0.8) {$\Omega_u^1\otimes A\otimes A\otimes \Omega_u^{\otimes n-1}$};
			\node (a1) at (5,0.8) {$\Omega_u^1\otimes \Omega_u^{\otimes n-1}$};

			\node (00) at (-1.3,0) {$A^{\otimes n+1}$};
			\node (10) at (2,0) {$\Omega_u^1\otimes A\otimes \Omega_u^{\otimes n-1}$};
			\node (a0) at (5,0) {$\Omega_u^{n}$};

			\path[font=\scriptsize]

			(01) edge [->] node [above] {$-(d_u\otimes 1_A)\!\otimes\! (1_A\otimes d_u^{\otimes n-1})$} (11)
			(11) edge [->] node [above] {$\wedge^{1,0}\otimes \wedge^{0,1}\otimes 1_{\Omega_u^{\otimes n-1}}$} (a1)
			
			(01) edge [->] node [right] {$1_A\otimes m\otimes 1_{A^{\otimes n-1}}$} (00)
			(11) edge [->] node [right] {$1_{\Omega_u^1}\otimes m\otimes 1_{\Omega_u^{\otimes n-1}}$} (10)
			(a1) edge [->] node [right] {$\wedge_u^{1,n-1}$} (a0)
			
			(00) edge [->] node [below] {$-d_u\otimes 1_A\otimes d_u^{\otimes n-1}$} (10)
			(10) edge [->] node [below] {$\wedge$} (a0);
		\end{tikzpicture}	
	\end{center} 
	where the left commutes by functoriality of $\otimes$ and the right by associativity of $\wedge$ (for $\Omega_u$ the wedge is associative by definition since it is a coequalizer map and coequalizers commute with other coequalizers).\\
	Since $-(d_u\cdot 1_A)$ is a map of right $A$-modules and $((1_A\cdot d_u)\otimes d_u^{\otimes n-2})$ is a map of left $A$-modules, this says in particular that
	$$-(d_u\cdot 1_A) \otimes_A ((1_A\cdot d_u)\wedge d_u^{\wedge n-2})= -(d_u\cdot 1_A)\wedge d_u^{\wedge n-1} $$
	under the identifications $A^{\otimes 2}\otimes_A A^{\otimes n}\cong A^{\otimes n+1}$ and $\Omega_u^{1}\otimes_A \Omega_u^{n-1}\cong\Omega_u^{n}$.
\end{rmk}

\begin{prop}
	$(\Omega^\bullet_u,\wedge_u,d_u)$ is well-defined and is a differential calculus
\end{prop}
\begin{proof}
	Consider the Amitsur complex $A^\bullet$ (\cite[Remark~4.1.12]{brandenburg2014tensor}) defined by $A^n:=A^{\otimes n+1}$ with differential given by
	$$  d_A^n\colon \sum_{i=0}^{n+1} (-1)^i(1_{A^{\otimes i}} \otimes i \otimes 1_{A^{\otimes n+1-i}}) \colon A^{\otimes n+1}\to A^{\otimes n+2}.$$ This comes equipped with operations
	$$ \wedge_A^{n,m}:=1_{A^{\otimes n} }\otimes m \otimes 1_{A^{\otimes m}}\colon A^{\otimes n+1}\otimes A^{\otimes m+1}\to A^{\otimes n+m+1} $$ 
	that make it a monoid in chain complex over $\V$.\\
	For each $n$ consider the composite
	$$ \iota_u^n\colon \Omega_u^n= (\Omega^1_u)^{\otimes_An}\xrightarrow{(\iota_u^1)^{\otimes_An}} (A\otimes A)^{\otimes_A n}\xrightarrow{\ \cong\ }A^{\otimes n+1}; $$
	where the last isomorphism follows from the definition of $\otimes_A$, see Remark~\ref{Otimes_A cancels A}. Note that by definition we have
	\begin{equation}\label{iota-n+m}
		\iota_u^n\otimes_A\iota_u^m\cong\iota_u^{n+m}
	\end{equation}
	for any $n,m\geq 0$. 
	
	We shall prove that $\iota_u^n$, seen as a morphism in $\V$, has left inverse the map
	$$ p_u^n\colon A^{\otimes n+1}\xrightarrow{(1_A\cdot d)\otimes d^{\otimes n-1}}\Omega_u^1\otimes\Omega_u^{\otimes n-1}\xrightarrow{\ \wedge_u\ }\Omega_u^n  $$
	appearing in the definition of surjectivity condition. This will imply that $\iota_u^n$ is a monomorphism and that the surjectivity condition for $\Omega_u^\bullet$ holds.
	We prove this fact by induction on $n$; for $n=1$ it has been shown in Proposition~\ref{prop:univ-1-d-cal-well-def} since $p^1=(1_A\cdot d)$.

	Let us now assume that the equality $p_u^{n-1}\circ \iota_u^{n-1}=1$ holds for a given $n-1\geq 1$, and prove it for $n$. For that, consider the following chain of equalities:
	\begin{align}
		p_u^{n}\circ \iota_u^{n}&=((1_A\cdot d)\wedge d^{\wedge n-1})\circ (\iota_u^1\otimes_A \iota_u^{n-1}) \tag{def}\\
		&=(-(d\cdot 1_A)\wedge d^{\wedge n-1})\circ (\iota_u^1\otimes_A \iota_u^{n-1}) \tag{$(1_A\cdot d)\circ\iota_u^1=-(d\cdot 1_A)\circ\iota_u^1$}\\
		&= (-(d_u\cdot 1_A) \otimes_A ((1_A\cdot d_u)\wedge d_u^{\wedge n-2}))\circ (\iota_u^1\otimes_A \iota_u^{n-1}) \tag{by \ref{compatibility-for d^n}}\\
		&=( (-d\cdot 1_A)\circ \iota_u^1) \otimes_A (p^{n-1}\circ \iota_u^{n-1})\\
		&= 1\otimes_A 1= 1\tag{inductive hyp.}
		\end{align}

	Since $\wedge_u$ and $\wedge_A$ are both constructed as the quotient maps that define $\otimes_A$ as a coequalizer, and this construction is natural, the isomorphism in~\ref{iota-n+m} says that for any $n,m\geq 0$ the square below commutes.
	\begin{center}
		\begin{tikzpicture}[baseline=(current  bounding  box.south), scale=2]
			
			\node (01) at (-1,0.8) {$\Omega_u^n\otimes \Omega_u^m$};
			\node (a1) at (0.7,0.8) {$\Omega_u^{n+m}$};
			\node (00) at (-1,0) {$A^{\otimes n+1}\otimes A^{\otimes m+1}$};
			\node (a0) at (0.7,0) {$A^{\otimes n+m+1}$};

			\path[font=\scriptsize]

			(01) edge [->] node [above] {$\wedge^{n,m}_u$} (a1)
			(01) edge [>->] node [left] {$\iota_u^n\otimes \iota_u^m$} (00)
			(a1) edge [>->] node [right] {$\iota_u^{n+m}$} (a0)
			(00) edge [->] node [below] {$\wedge^{n,m}_A$} (a0);
		\end{tikzpicture}	
	\end{center} 
	Then, since $A^\bullet$ is a monoid in $\Ch(\V)$ and the vertical maps above are monomorphisms, it will follows that $\Omega_u^\bullet$ is also a monoid in $\Ch(\V)$ once we prove that it has a compatible differential: arguing as in the proof of Proposition~\ref{prop:univ-1-d-cal-well-def}, the monoid axioms for $\Omega_u^\bullet$ hold since those of $A^\bullet$ do.
	
	For any $n$ define the differential of $\Omega_u^\bullet$ as
	$$ d_u^n\colon\Omega_u^n\xrightarrow{\iota_u^n} A^{\otimes n+1}\xrightarrow{d^{\otimes n+1}}\Omega_u^{\otimes n+1}\xrightarrow{\wedge_u}\Omega_u^{n+1}. $$
	To show the compatibility with the differential of $A^\bullet$ we need to show that
	$$ \iota_u^{n+1}\circ d_u^n= d_A^n\circ \iota_u^n. $$
	Consider the following equalities
	\begin{align}
		\iota_u^{n+1}\circ d_u^n&=\iota_u^{n+1}\circ \wedge_u\circ d^{\otimes n+1}\circ\iota_u^n \tag{def of $d_u^n$}\\
		&= \wedge_A\circ \iota_u^{\otimes n+1}\circ d_u^{\otimes n+1}\circ\iota_u^n \tag{def of $\iota_u^{n+1}$}\\
		&=(1_A\otimes m^{\otimes n}\otimes 1_A)\circ( (\iota_u\circ d_u)^{\otimes n+1})\circ\iota_u^n \tag{def of $\wedge_A$}\\
		&=(1_A\otimes m^{\otimes n}\otimes 1_A)\circ \bigotimes_{j=1}^{n+1}(i\otimes 1_A-1_A\otimes i) \circ\iota_u^n \tag{def of $d_u$}\\
		&=(1_A\otimes m^{\otimes n}\otimes 1_A)\circ \sum_J (h_1\otimes \cdots\otimes h_{n+1}) \circ\iota_u^n \notag
	\end{align}
	where in the last equality $h_j\in\{i\otimes 1_A, -1_A\otimes i\}$, and $J$ is the set of all possible choices.
	Now, given a choice of $\bar h:= h_1\otimes \cdots\otimes h_{n+1}$, suppose that there exists $j<n+1$ such that $h_j= i\otimes 1_A$ and $h_{j+1}=-1_A\otimes i$; then $\wedge_A\circ \bar h$ is of the form $f_1\otimes (-m)\otimes f_2$, for some $f_1,f_2$; it follows that $\wedge_A\circ \bar h\circ \iota_u^n=0$ (by definition of $\iota_u^1$ and $\iota_u^n$).
	An easy calculation shows that the only summands that remain are of the form
	$$ (-1_A\otimes i)^{\otimes k}\otimes (i\otimes 1_A)^{\otimes n-k+1} $$
	for $0\leq k\leq n+1$. Continuing the equalities it follows that
	\begin{align}
		\iota_u^{n+1}\circ d_u^n&=(1_A\otimes m^{\otimes n}\otimes 1_A)\circ \sum_{k=0}^{n+1} ((-1_A\otimes i)^{\otimes k}\otimes (i\otimes 1_A)^{\otimes n-k+1} ) \circ\iota_u^n \notag\\
		&= \sum_{k=0}^{n+1} (-1)^k(1_{A^{\otimes k}} \otimes i \otimes 1_{A^{\otimes n-k+1}})\circ \iota_u^n \tag{applying $m$}\\
		&= d_A^n\circ \iota_u^n\tag{def of $d_A^n$}
	\end{align}
	as claimed.

\end{proof}

The following is analogous to Proposition~\ref{universal-prop}

\begin{prop}\label{univ-ext-alg}
	The differential calculus $\Omega_u^\bullet$ is universal; that is, for any other differential calculus $\Omega_d^\bullet$ over $A$, there exists a unique morphism of monoids $f\colon \Omega_u^\bullet\to\Omega_d^\bullet$ in $\Ch(\V)$ with $f^0=1_A$. 
	Such morphism is always an epimorphism (that is, componentwise an epimorphism).
\end{prop}
\begin{proof}
	We define $f^\bullet$ componentwise as the composite
	$$ f^n\colon \Omega^n_u\xrightarrow{\iota_u^n} A^{\otimes n+1}\xrightarrow{(1_A\cdot d)\wedge d^{\wedge n-1}} \Omega^n_d, $$
	where the second map expresses the surjectivity condition on $\Omega^\bullet_d$, and we are denoting $d:=d^0$. We need to check that this is a map of differential calculi. Note that for $n=0$ we have $f^0=1_A$ and for $n=1$ the morphism $f^1$ coincides with the one defined in Proposition~\ref{universal-prop}.
	
	First, we show that $f^\bullet$ respects the differentials. For any $n\geq 0$ we have
	\begin{align}
		f^{n+1}\circ d_u^n&=((1_A\cdot d)\wedge d^{\wedge n})\circ \iota_u^{n+1}\circ d_u^n \tag{def}\\
		&=((1_A\cdot d)\wedge d^{\wedge n})\circ d_A^n\circ\iota_u^n \tag{$\iota^\bullet$ map of monoids}\\
		&=((1_A\cdot d)\wedge d^{\wedge n})\circ (i\otimes 1_{A^{\otimes n+1}})\circ\iota_u^n \tag{def of $d_A$ and $d(i)=0$}\\
		&= d^{\wedge n+1}\circ\iota_u^n\tag{$i$ unit}\\
		&= (d\wedge d^{\wedge n} + 1_A\wedge d^n(d^{\wedge n}) )\circ\iota_u^n\tag{$d\circ d=0$}\\
		&= d^n\circ (1_A\wedge^{0,n} d^{\wedge n})\circ \iota_u^n\tag{n-dimensional Leibniz}\\
		&= d^n\circ ((1_A\cdot d)\wedge d^{\wedge n-1})\circ \iota_u^n \tag{$\cdot=\wedge^{0,n}$}\\
		&= d^n\circ f^n\tag{def}
	\end{align}
	
	Next, we need to show that it respects the $\wedge$; that is, for any $n,m\geq 0$ the following diagram
	\begin{center}
		\begin{tikzpicture}[baseline=(current  bounding  box.south), scale=2]
			
			\node (01) at (-1,0.8) {$\Omega_u^n\otimes \Omega_u^m$};
			\node (a1) at (0.7,0.8) {$\Omega_d^n\otimes \Omega_d^m$};
			\node (00) at (-1,0) {$\Omega_u^{n+m}$};
			\node (a0) at (0.7,0) {$\Omega_d^{n+m}$};

			\path[font=\scriptsize]

			(01) edge [->] node [above] {$f^n\otimes f^m$} (a1)
			(01) edge [->] node [left] {$\wedge^{n,m}_u$} (00)
			(a1) edge [->] node [right] {$\wedge^{n,m}$} (a0)
			(00) edge [->] node [below] {$f^{n+m}$} (a0);
		\end{tikzpicture}	
	\end{center} 
	commutes. For that, it is enough to show that the square below left commutes for any $k\geq 0$, where for simplicity we are writing $\Omega_u=\Omega_u^1$ and $\Omega_d=\Omega_d^1$.
	\begin{center}
		\begin{tikzpicture}[baseline=(current  bounding  box.south), scale=2]
			
			\node (b1) at (-4,1.2) {$\Omega_u^{\otimes k}$};
			\node (c1) at (-2.7,1.2) {$\Omega_d^{\otimes k}$};
			\node (b0) at (-4,0.4) {$\Omega_u^{k}$};
			\node (c0) at (-2.7,0.4) {$\Omega_d^{k}$};
			
			\node (*) at (-2,0.8) {$\stackrel{k=n+m}{=}$};
	
			\node (02) at (-1,1.6) {$\Omega_u^{\otimes n}\otimes \Omega_u^{\otimes m}$};
			\node (a2) at (0.9,1.6) {$\Omega_d^{\otimes n}\otimes \Omega_d^{\otimes m}$};
			\node (01) at (-1,0.8) {$\Omega_u^n\otimes \Omega_u^m$};
			\node (a1) at (0.9,0.8) {$\Omega_d^n\otimes \Omega_d^m$};
			\node (00) at (-1,0) {$\Omega_u^{n+m}$};
			\node (a0) at (0.9,0) {$\Omega_d^{n+m}$};			
			
			\path[font=\scriptsize]
			
			(b1) edge [->] node [above] {$(f^1)^{\otimes k}$} (c1)
			(b1) edge [->] node [left] {$\wedge_u^k$} (b0)
			(c1) edge [->] node [right] {$\wedge^k$} (c0)
			(b0) edge [->] node [below] {$f^{k}$} (c0)
			
			(02) edge [->] node [above] {$(f^1)^{\otimes n}\!\otimes\! (f^1)^{\otimes m}$} (a2)
			(02) edge [->] node [left] {$\wedge_u^n\otimes \wedge_u^m$} (01)
			(a2) edge [->] node [right] {$\wedge^n\otimes \wedge^m$} (a1)
			(01) edge [->] node [above] {$f^n\otimes f^m$} (a1)
			(01) edge [->] node [left] {$\wedge^{n,m}_u$} (00)
			(a1) edge [->] node [right] {$\wedge^{n,m}$} (a0)
			(00) edge [->] node [below] {$f^{n+m}$} (a0);
		\end{tikzpicture}	
	\end{center} 
	Indeed, notice that the equality above holds thanks to associativity of the wedge operation, where the top-right square is obtained by tensoring together two copies of the left square (for $k=n,m$). Then, since $\wedge_u\otimes\wedge_u$ is an epimorphism, the bottom-right square commutes if the left one does. 
	
	We prove commutativity of the left square by induction on $k$. For $k=0,1$ it is trivial; assume then that the square commutes for a given $k\geq 1$ and let us prove it for $k+1$:
	\begin{align}
		f^{k+1}\circ \wedge_u^{k+1}&=f^{k+1}\circ \wedge_u^{1,k}\circ (1_{\Omega_u}\otimes\wedge^k_u) \tag{associativity}\\
		&= ((1_A\cdot d)\wedge d^{\wedge k})\circ \iota_u^{k+1} \circ \wedge_u^{1,k}\circ (1_{\Omega_u}\otimes\wedge^k_u)\tag{def of $f^\bullet$}\\
		&= ((1_A\cdot d)\wedge d^{\wedge k})\circ (\iota_u^1\otimes_A \iota_u^{k})\circ \wedge_u^{1,k}  \circ (1_{\Omega_u}\otimes\wedge^k_u)\tag{def of $\iota_u^\bullet$}\\
		&= ((-d\cdot 1_A)\wedge d^{\wedge k})\circ (\iota_u^1\otimes_A \iota_u^{k})\circ \wedge_u^{1,k}  \circ (1_{\Omega_u}\otimes\wedge^k_u)\tag{ $(1_A\!\cdot\! d)\!\circ\!\iota_u^1\!=\!(-d\!\cdot\! 1_A)\!\circ\!\iota_u^1$}\\
		&= ((-d\cdot 1_A)\wedge d^{\wedge k})\circ \wedge_A^{1,k}\circ (\iota_u^1\otimes \iota_u^{k})  \circ (1_{\Omega_u}\otimes\wedge^k_u) \tag{def of $\iota_u^\bullet$}\\
		&= \wedge^{1,k}\circ \left((-d\cdot 1_A)\otimes ((1_A\cdot d)\wedge d^{\wedge k-1})\right)\circ (\iota_u^1\otimes \iota_u^{k})  \circ (1_{\Omega_u}\otimes\wedge^k_u) \tag{by \ref{compatibility-for d^n}}\\
		&= \wedge^{1,k}\circ \left(((-d\cdot 1_A)\circ \iota_u^1) \otimes f^k\right) \circ (1_{\Omega_u}\otimes\wedge^k_u) \tag{def}\\
		&= \wedge^{1,k}\circ \left((f^1\circ (-d_u\cdot 1_A)\circ \iota_u^1) \otimes f^k\right) \circ (1_{\Omega_u}\otimes\wedge^k_u) \tag{$f^1$ map of diff. calc.}\\
		&= \wedge^{1,k}\circ (f^1 \otimes f^k) \circ (1_{\Omega_u}\otimes\wedge^k_u) \tag{$(-d_u\cdot 1_A)\circ \iota_u^1=1_{\Omega_u}$}\\
		&= \wedge^{1,k}\circ (f^1 \otimes (f^k\circ \wedge^k_u)) \tag{functoriality}\\
		&= \wedge^{1,k}\circ (f^1 \otimes (\wedge^k\circ (f^1)^{\otimes k}))\tag{induction}\\
		&= \wedge^{1,k}\circ (1_{\Omega_d}\otimes \wedge^k)\circ (f^1 \otimes (f^1)^{\otimes k})\tag{functoriality}\\
		&= \wedge^{k+1}\circ (f^1)^{\otimes k+1}.  \tag{associativity}
	\end{align}
	Therefore $f^\bullet\colon\Omega_u^\bullet\to\Omega_d^\bullet$ defines a morphism of differential calculi. As a consequence the triangle 
	\begin{center}
		\begin{tikzpicture}[baseline=(current  bounding  box.south), scale=2]
			
			\node (c0) at (0.6,0) {$A^{\otimes n+1}$};
			\node (a0) at (0,-0.7) {$\Omega_u^n$};
			\node (b0) at (1.2,-0.7) {$\Omega_{d}^n$};
			
			\path[font=\scriptsize]
			
			(a0) edge [->] node [below] {$f^n$} (b0)
			(c0) edge [->>] node [left] {$(1_A\cdot d_u)\wedge_u d_u^{\wedge_u n-1} $} (a0)
			(c0) edge [->>] node [right] {$ (1_A\cdot d)\wedge d^{\wedge n-1}$} (b0);
		\end{tikzpicture}	
	\end{center} 
	commutes for any $n\geq 0$. Since the vertical maps are epimorphisms (by the surjectivity condition) it follows that $f^n$ is an epimorphism for each $n$ and that is the only morphism $\Omega_u^n\to\Omega_d^n$ of $\V$ making such a triangle commute. Therefore $f^\bullet$ is itself an epimorphism and is the only morphism $\Omega_u^\bullet\to\Omega_d^\bullet$.
\end{proof}

\begin{rmk}\label{rmk:univ-ext-alg-relative}
	The same arguments prove the following relative version of \Cref{univ-ext-alg}. Let $h\colon A\to B$ be a morphism of monoids and let $\Theta_\delta^\bullet$ be a differential calculus over $B$. Then there exists a unique morphism of monoids $g^\bullet\colon\Omega^\bullet_{A,u}\to\Theta^\bullet_\delta$ in $\Ch(\V)$ with $g^0=h$; it is given in degrees $n\geq1$ by
	$$ g^n\colon \Omega^n_{A,u}\xrightarrow{\ \iota_u^n\ } A^{\otimes n+1}\xrightarrow{\ h^{\otimes n+1}\ } B^{\otimes n+1}\xrightarrow{(1_B\cdot \delta)\wedge \delta^{\wedge n-1}} \Theta^n_\delta, $$
	where we are denoting $\delta:=\delta^0$.
	The proof of \Cref{univ-ext-alg} goes through after replacing $d$ with $\delta\circ h$ and each $\Omega^n_d$ with the $A$-bimodule $h^*\Theta^n_\delta$.
	Since $h$ is a morphism of monoids, $\delta\circ h$ still satisfies the Leibniz rule and $\delta\circ h\circ i_A=0$. Uniqueness follows as before from the surjectivity condition on $\Omega^\bullet_{A,u}$.
\end{rmk}

In the other direction, analogously to result \Cref{lem:epi-from-univ-always-calc}, we can prove:

\begin{prop}
	Let $f^\bullet\colon \Omega_u^\bullet\to M^\bullet$ be an epimorphism of monoids in $\Ch(\V)$ for which $f^0=1_A$. Then $M^\bullet$ is a differential calculus
\end{prop}
\begin{proof}
	We only need to prove that $M^\bullet$ satisfies the surjectivity condition. But, for any $n\geq 0$, the following triangle 
	\begin{center}
		\begin{tikzpicture}[baseline=(current  bounding  box.south), scale=2]
			
			\node (c0) at (0.6,0) {$A^{\otimes n+1}$};
			\node (a0) at (0,-0.7) {$\Omega_u^n$};
			\node (b0) at (1.2,-0.7) {$M^n$};
			
			\path[font=\scriptsize]
			
			(a0) edge [->>] node [below] {$f^n$} (b0)
			(c0) edge [->>] node [left] {$(1_A\cdot d_u)\wedge_u d_u^{\wedge_u n-1} $} (a0)
			(c0) edge [->] node [right] {$ (1_A\cdot d_M^0)\wedge (d_M^0)^{\wedge n-1}$} (b0);
		\end{tikzpicture}	
	\end{center}
	commutes. Now, since $f^n$ is an epimorphism by hypothesis and so is the left vertical map (by the surjectivity condition on $\Omega_u^\bullet$), it follows that also the right vertical map is an epimorphism. Hence $M^\bullet$ is a differential calculus
\end{proof}

\begin{defi}
	Let $\Omega^1_d$ be a first order differential calculus. We define its maximal prolongation $\Omega_{d,max}^\bullet=\bigoplus_{n\geq 0}\Omega^n_d$ as follows:
	\begin{itemize}
		\item $\Omega_d^0=A$, $\eta\colon I\to A$ and $\wedge_d^{0,0}\colon A\otimes A\to A$ are the monoid unit and multiplication.
		\item $\Omega_d^1:=\Omega^1_d$, $d^0=d$, and $\wedge_d^{0,1}$ and $\wedge_d^{1,0}$ are the left and right actions of $\Omega^1_d$.
		\item For $n\geq 2$, we define $\Omega_d^n$ recursively as the colimit of the solid diagram below
		\begin{center}
			\begin{tikzpicture}[baseline=(current  bounding  box.south), scale=2]
				
				\node (a1) at (-0.3,0.9) {$\bigoplus_{i+j=n}^{i,j\geq 1} \Omega^i_u\otimes \Omega^j_u$};
				\node (b1) at (1.5,0.4) {$\Omega^n_u$};
				\node (c1) at (3,0.9) {$\Omega^{n-1}_u$};
				\node (a0) at (-0.3,0) {$\bigoplus_{i+j=n}^{i,j\geq 1} \Omega^i_d\otimes \Omega^j_d$};
				\node (b0) at (1.5,-0.5) {$\Omega^n_d$};
				\node (c0) at (3,0) {$\Omega^{n-1}_d$};
				
				\path[font=\scriptsize]
				
				(a1) edge [->] node [above] {$\wedge_u$} (b1)
				(b1) edge [<-] node [above] {$d^{n-1}_u$} (c1)
				(a0) edge [dashed, ->] node [below] {$\wedge_d$} (b0)
				(b0) edge [dashed, <-] node [below] {$d^{n-1}$} (c0)
				(a1) edge [->] node [left] {$\bigoplus f^i\otimes f^j$} (a0)
				(b1) edge [dashed, ->] node [right] {$f^n$} (b0)
				(c1) edge [->] node [right] {$f^{n-1}$} (c0);
			\end{tikzpicture}	
		\end{center} 
		where $f^0=1_A$ and $f^1$ is the unique map arising from Proposition~\ref{universal-prop}.
		\item The monoid structure $\wedge_d^{i,j}\colon \Omega^i_d\otimes \Omega^j_d\to \Omega^{i+j}_d$, for $i,j\geq 1$, and the differential $d^n\colon \Omega^n_d\to\Omega^{n+1}_d$ are induced recursively from the construction above; while, $\wedge_d^{0,j}$ and $\wedge_d^{i,0}$ will be constructed from the universal property of the colimit (see proposition below).
	\end{itemize}
\end{defi}

\begin{prop}
	For any first order differential calculus $\Omega_d^1$, its maximal prolongation $\Omega^\bullet_{d,max}$ is a differential calculus
\end{prop}
\begin{proof}
	By definition we have maps $f^n\colon \Omega^n_u\to\Omega^n_d$ for any $n\geq 0$. 
	Note that $f^0$ and $f^1$ are epimorphisms (the first is the identity on $A$ and the second is epi by \ref{universal-prop}), and so is $f^n$: assume inductively that $f^i$ is an epimorphism for any $i<n$, then $f^n$ is obtained from these $f^i$ by taking tensor products, direct sums, and iterative pushouts.
	Then $f^n$ is an epimorphism since all these operations preserve epimorphism.
	
	As mentioned in the definition above, we need to construct $\wedge_d^{0,n}$ and $\wedge_d^{m,0}$ for $n,m\geq0$. Let us consider $\wedge_d^{0,n}$ first and define it recursively on $n$. For $n=0,1$ the map is given by the monoid multiplication of $A$ ($n=0$) and by the left action of $A$ on $\Omega^1_d$ as a left $A$-module ($n=1$). Consider now $n\geq 2$, and assume to have defined $\wedge_d^{0,i}$, for $i<n$, and that these are preserved by $f$. Since $A\otimes -$ preserves finite colimits, the colimit of the solid part of the diagram above is $A\otimes \Omega_d^n$.
	{\begin{center}
		\begin{tikzpicture}[baseline=(current  bounding  box.south), scale=2]
			
			\node (a1) at (-0.3,0.9) {$\bigoplus_{i+j=n}^{i,j\geq 1} A\otimes \Omega^i_u\otimes \Omega^j_u$};
			\node (b1) at (1.5,0.4) {$A\otimes \Omega^n_u$};
			\node (c1) at (3.2,0.9) {$A\otimes \Omega^{n-1}_u$};
			
			\node (a0) at (-0.3,0) {$\bigoplus_{i+j=n}^{i,j\geq 1} A\otimes \Omega^i_d\otimes \Omega^j_d$};
			\node (c0) at (3.2,0) {$A\otimes \Omega^{n-1}_d$};
			
			\node (a01) at (-0.1,-0.9) {$\bigoplus_{i+j=n}^{i,j\geq 1} \Omega^i_{d}\otimes \Omega^j_{d}$};
			\node (b0) at (1.5,-1.4) {$\Omega^n_{d}$};
			
			\path[font=\scriptsize]
			
			(a1) edge [->] node [above] {$1_A\otimes \wedge_u$} (b1)
			(b1) edge [<-] node [above] {$1_A\otimes d^{n-1}_u$} (c1)
			(a1) edge [->] node [left] {$\bigoplus 1_A\otimes f^i\otimes f^j$} (a0)
			(c1) edge [->] node [right] {$1_A\otimes f^{n-1}$} (c0)
			
			(a0) edge [bend right=10, dashed, ->] node [left] {$\bigoplus \wedge_d^{0,i}\otimes 1_{\Omega_d^j}$} (a01)
			(b1) edge [dashed, ->] node [right] {$f^n(\wedge_u^{0,n})$} (b0)
			(c0) edge [bend left=40, dashed, ->] node [right] {$\ d^{n-1}\circ\wedge_d^{0,n-1}- \wedge_d^{1,n-1}\circ(d^0\otimes 1_{\Omega_d^{n-1}})$} (b0)
			(a01) edge [bend right=15, dashed, ->] node [below] {$\wedge_d^{i,j}$} (b0);
		\end{tikzpicture}	
	\end{center}} 
	The dashed part  forms a cocone for the solid  diagram: compatibility with $1_A\otimes\wedge_u$ follows from the associativity of $\wedge_u$ and the inductive hypothesis, and compatibility with $1_A\otimes d^{n-1}_u$ from the graded Leibniz rule $a\cdot d_u\omega=d_u(a\omega)-d_ua\wedge_u\omega$ in $\Omega^\bullet_u$, using that $f^{n-1}$ and $f^n$ preserve the products and differentials already constructed; thus there is a unique induced map $\wedge_d^{0,n}\colon A\otimes \Omega_d^n\to\Omega_d^n$ which distributes over the differential (commutativity of the diagram on the right) and satisfies the associativity condition with respect to the other degrees. To define $\wedge_d^{m,0}$ one argues in a similar way by tensoring with $-\otimes A$.
	
	Now, by definition the following squares
		\begin{center}
		\begin{tikzpicture}[baseline=(current  bounding  box.south), scale=2]
			
			\node (a1) at (0,0.8) {$\Omega^i_u\otimes \Omega^j_u$};
			\node (b1) at (1.5,0.8) {$\Omega^{i+j}_u$};
			\node (a0) at (0,0) {$\ \Omega^i_d\otimes \Omega^j_d$};
			\node (b0) at (1.5,0) {$\Omega^{i+j}_d$};
			
			\node (a11) at (3,0.8) {$\Omega^n_u$};
			\node (b11) at (4.5,0.8) {$\Omega^{n+1}_u$};
			\node (a01) at (3,0) {$\ \Omega^n_d$};
			\node (b01) at (4.5,0) {$\Omega^{n+1}_d$};
			
			\path[font=\scriptsize]
			
			(a11) edge [->] node [above] {$d^n_u$} (b11)
			(a01) edge [->] node [below] {$d^n$} (b01)
			(a11) edge [->>] node [left] {$f^n$} (a01)
			(b11) edge [->>] node [right] {$f^{n+1}$} (b01)
			
			(a1) edge [->] node [above] {$\wedge_u$} (b1)
			(a0) edge [->] node [below] {$\wedge_d$} (b0)
			(a1) edge [->>] node [left] {$f^i\otimes f^j$} (a0)
			(b1) edge [->>] node [right] {$f^n$} (b0);
		\end{tikzpicture}	
	\end{center} 
	commute for any $i,j,n\geq0$. Using this plus the fact that each $f^n$ is an epimorphism it follows easily that $d^n$ is a differential, that $\wedge_d$ satisfies the monoid conditions from Example~\ref{MonoidInChain}, and that the surjectivity condition holds. Indeed, by precomposition with $f^n$ (or any tensors of them) each commutativity condition involving $\Omega^\bullet_{d,max}$ can be seen as a consequence of the corresponding commutativity condition for $\Omega_u^\bullet$.
\end{proof}

\begin{prop}\label{unique-morphism-ealg}
	Let $\Omega^1_d$ be a first order differential calculus and $\Omega^\bullet_{d'}$ be a differential calculus with $\Omega^1_{d'}=\Omega^1_d$ and $d'^0=d$. 
	Then there exists a unique morphism of differential calculi $h\colon \Omega^\bullet_{d,max}\to\Omega^\bullet_{d'}$ with $h^0=1_A$ and $h^1=1_{\Omega^1_d}$.
\end{prop}
\begin{proof}
	By Proposition~\ref{univ-ext-alg} there exists a unique morphism $g\colon \Omega^\bullet_u\to\Omega_{d'}^\bullet$ which is moreover an epimorphism. By uniqueness of $g$, the morphism $h$ we define will satisfy $hf=g$ and will be unique since $f$ is an epimorphism.
	
	Define $h^0=1_A$ and $h^1=1_{\Omega^1_d}$; for $n\geq 2$ we assume to have already defined $h^i$, for $i<n$, and argue recursively. Consider the (dashed) cocone on the diagram which defines $\Omega^n_d$ as a colimit:
	\begin{center}
		\begin{tikzpicture}[baseline=(current  bounding  box.south), scale=2]
			
			\node (a1) at (-0.3,0.9) {$\bigoplus_{i+j=n}^{i,j\geq 1} \Omega^i_u\otimes \Omega^j_u$};
			\node (b1) at (1.5,0.4) {$\Omega^n_u$};
			\node (c1) at (3,0.9) {$\Omega^{n-1}_u$};
			
			\node (a0) at (-0.3,0) {$\bigoplus_{i+j=n}^{i,j\geq 1} \Omega^i_d\otimes \Omega^j_d$};
			\node (c0) at (3,0) {$\Omega^{n-1}_d$};
			
			\node (a01) at (-0.3,-0.9) {$\bigoplus_{i+j=n}^{i,j\geq 1} \Omega^i_{d'}\otimes \Omega^j_{d'}$};
			\node (c01) at (3,-0.9) {$\Omega^{n-1}_{d'}$};
			
			\node (b0) at (1.5,-1.4) {$\Omega^n_{d'}$};
			
			\path[font=\scriptsize]
			
			(a1) edge [->] node [above] {$\wedge_u$} (b1)
			(b1) edge [<-] node [above] {$d^{n-1}_u$} (c1)
			(a1) edge [->] node [left] {$\bigoplus f^i\otimes f^j$} (a0)
			(c1) edge [->] node [right] {$f^{n-1}$} (c0)
			
			(a0) edge [dashed, ->] node [left] {$\bigoplus h^i\otimes h^j$} (a01)
			(b1) edge [dashed, ->] node [right] {$g^n$} (b0)
			(c0) edge [dashed, ->] node [right] {$h^{n-1}$} (c01)
			(a01) edge [dashed, ->] node [below] {$\wedge_d$} (b0)
			(b0) edge [dashed, <-] node [below] {$d^{n-1}$} (c01);
		\end{tikzpicture}	
	\end{center} 
	Then, by the universal property of the colimit defining $\Omega^n_d$, we obtain a map $h^n\colon \Omega^n_d\to\Omega^n_{d'}$ which satisfies the required commutativity properties.
\end{proof}

\subsection{Differential calculi and left adjoints}
\label{sec:ext-alg-left-adj}

In this subsection we show the universal property of the maximal prolongation as a left adjoint. 
Let us begin by denoting by $\EAlg(\V)$ the full subcategory of $\Mon(\Ch(\V))$ whose objects are the differential calculi.

\begin{lemma}\label{unique-extension-ealg}
	Given two differential calculi $\Omega^\bullet_d$ and $\Theta^\bullet_\delta$, and a morphism of monoids $f\colon \Omega_d^0\to \Theta_\delta^0$ there exists at most one morphism of differential calculi
	$$ f^\bullet\colon \Omega^\bullet_d\to\Theta^\bullet_\delta $$
	for which $f^0=f$.
\end{lemma}
\begin{proof}
	Suppose we have a morphism $f^\bullet\colon \Omega^\bullet_d\to\Theta^\bullet_\delta$ in $\EAlg(\V)$, and consider the epimorphisms
	$$ p_d\colon (\Omega^0_d)^{\otimes n+1}\twoheadrightarrow \Omega^n_d \text{\ \ and\ \ } p_\delta\colon (\Theta^0_\delta)^{\otimes n+1}\twoheadrightarrow \Theta^n_\delta $$
	induced by the surjectivity conditions. Since these are constructed using the differentials and the monoid operations of the differential calculi, and $f^\bullet$ preserves these by definition, it follows that the following diagram
	\begin{center}
		\begin{tikzpicture}[baseline=(current  bounding  box.south), scale=2]
			
			\node (a0) at (0,0.8) {$(\Omega^0_d)^{\otimes n+1}$};
			\node (b0) at (0,0) {$(\Theta^0_\delta)^{\otimes n+1}$};
			\node (c0) at (1.3,0.8) {$\Omega^n_d$};
			\node (d0) at (1.3,0) {$\Theta^n_\delta$};
			
			\path[font=\scriptsize]
			
			(a0) edge [->] node [left] {$(f^0)^{\otimes n+1}$} (b0)
			(a0) edge [->>] node [above] {$p_d$} (c0)
			(b0) edge [->>] node [below] {$p_\delta$} (d0)
			(c0) edge [->] node [right] {$f^n$} (d0);
		\end{tikzpicture}	
	\end{center}
	commutes. Since $p_d$ is an epimorphism in $\V$ it follows that, given $f^0$, the map $f^n$ is the only map with this commutativity property. 
\end{proof}

We can then define a functor 
$$ \pi\colon  \EAlg(\V) \longrightarrow \Calc(\V)$$
sending $\Omega^\bullet$ to $\pi(\Omega^\bullet):= (\Omega^0,\Omega^1)$. This is well-defined: if $\Omega^\bullet=(\Omega^\bullet,\eta,\wedge,d)$ is a differential calculus, then:\begin{itemize}
	\item $\Omega^0$ is a a monoid in $\V$ with unit $\eta$ and multiplication
	$$ \wedge^{0,0}\colon\Omega^0\otimes \Omega^0\to\Omega^0;$$
	\item $\Omega^1$ is a $\Omega^0$-bimodule with actions
	$$ \wedge^{0,1}\colon\Omega^0\otimes \Omega^1\to\Omega^1 \text{\ \  and\ \ } \wedge^{1,0}\colon\Omega^1\otimes \Omega^0\to\Omega^1; $$
	\item $(\Omega^1,d^0)$ is a first order differential calculus over $\Omega^0$, where the surjectivity condition for this follows from the surjectivity condition for $\Omega^\bullet$ at degree $1$.
\end{itemize}

\begin{prop}\label{prop:right-and-left-adj-of-EAlg-Calc}
	The projection $\pi\colon \EAlg(\V)\to\Calc(\V)$ has:\begin{enumerate}
		\item a right adjoint, which is given by sending $(A,\Omega_d)$ to the trivial extension $\Omega_{d,0}^\bullet$ defined by $\Omega_{d,0}^0=A$, $\Omega_{d,0}^1=\Omega_d$, $d^0=d$, and is trivial everywhere else; 
		\item a left adjoint, which is given by sending $(A,\Omega_d)$ to the maximal prolongation $\Omega_{d,max}^\bullet$.
	\end{enumerate} 
\end{prop}
\begin{proof}
	(i). We need to show that
	$$ \EAlg(\V)(\Theta_\delta^\bullet, \Omega^\bullet_{d,0})\cong\Calc(\V)(\pi(\Theta_\delta^\bullet),(A,\Omega_d)) $$
	naturally in $\Theta_\delta^\bullet$ in $\EAlg(\V)$ and $(A,\Omega_d)$ in $\Calc(\V)$.  But, since $\Omega^n_{d,0}=0$ for $n>1$, a map $f^\bullet\colon \Theta_\delta^\bullet\to\Omega_{d,0}^\bullet$ is determined by $f^0\colon\Theta_\delta^0\to A$ and $f^1\colon\Theta_\delta^1\to\Omega_d$ (because $f^n=0$ for any $n>1$). It follows from the definition of $\pi(\Theta_\delta^\bullet):= (\Theta_\delta^0,\Theta_\delta^1)$, that $f^\bullet\colon \Theta_\delta^\bullet\to\Omega_{d,0}^\bullet$ is a map of differential calculi if and only if $(f^0,f^1)\colon\pi(\Theta_\delta^\bullet)\to (A,\Omega_d)$ is a map of calculi.
	
	(ii). We need to prove that 
	$$\EAlg(\V)(\Omega_{d,max}^\bullet, \Theta_\delta^\bullet) \cong \Calc(\V)((A,\Omega_d),\pi(\Theta_\delta^\bullet)) $$
	naturally in the two variables. Given a morphism of differential calculi $h\colon\Omega_{d,max}^\bullet\to\Theta_\delta^\bullet$, by restricting on degrees $0$ and $1$ we obtain a morphism of calculi
	$$ (h^0,h^1)\colon (A,\Omega_d)\to(\Theta_\delta^0,\Theta_\delta^1)=\pi(\Theta_\delta^\bullet). $$
	By Lemma~\ref{unique-extension-ealg} the assignment $h^\bullet\mapsto(h^0,h^1)$ is monic. Thus, we only need to prove that any map of calculi $ (h^0,h^1)\colon (A,\Omega_d)\to\pi(\Theta_\delta^\bullet)$ extends to a map of differential calculi $h^\bullet\colon \Omega_{d,max}^\bullet\to\Theta_\delta^\bullet$. This is done as in Proposition~\ref{unique-morphism-ealg}, using in place of the comparison of Proposition~\ref{univ-ext-alg} the unique morphism $g^\bullet\colon\Omega^\bullet_{A,u}\to\Theta^\bullet_\delta$ with $g^0=h^0$ provided by \Cref{rmk:univ-ext-alg-relative}. Note that $g^1=h^1\circ f^1$, since both are morphisms of $A$-bimodules $\Omega^1_{A,u}\to (h^0)^*\Theta^1_\delta$ whose composite with $d_u$ is $\delta^0\circ h^0$ (see the proof of \Cref{prop:univ-calc-adjunction}). For $n\geq 2$, the map $h^n$ is then induced recursively by the cocone given by $g^n$, $\wedge_\delta\circ(h^i\otimes h^j)$, and $\delta^{n-1}\circ h^{n-1}$.
\end{proof}

{\par}\begin{eg}[Higher central forms]\label{dvmkc:central-prolongation}
Let us take $\bk=\mathbb C$ and $S=Z(A)$ as in Example~\ref{dvmkc:central-calculus}.
Following \cite[Section~7]{dubois2001lectures}, let $I_Z$ be the differential graded ideal of $\Omega^\bullet_{A,u}$ generated by $[Z(A),\Omega^1_{A,u}]$, and put $\Omega_{Z(A)}^\bullet=\Omega^\bullet_{A,u}/I_Z$.
By Proposition~\ref{prop:right-and-left-adj-of-EAlg-Calc}(ii), this is precisely the maximal prolongation of $(\Omega_0,d)$, so that
	\[
 	\Omega^\bullet_{d,\mathrm{max}}\cong\Omega_{Z(A)}^\bullet.
	\]
Thus we also recover the higher central forms of Dubois-Violette.
Their universal property among differential calculi over $A$ with
central homogeneous components is given in \cite[Proposition~7 and Corollary~2]{dubois2001lectures}.
\end{eg}

\begin{eg}[Higher Kapranov-Corti\~nas forms]\label{dvmkc:prolongation}
For the calculus in \Cref{dvmkc:kapranov}, let $J$ be the differential graded ideal generated by $N_{\ell,A}$.
The maximal prolongation of \Cref{prop:right-and-left-adj-of-EAlg-Calc}(ii) is $\Omega^\bullet_{d,\mathrm{max}}=\Omega^\bullet_{A,u}/J$.
Since $N_{\ell,A}=(F^{\ell+1}_{\mathrm{gr}}\Omega^\bullet_{A,u})^1$, we have $J\subseteq F^{\ell+1}_{\mathrm{gr}}\Omega^\bullet_{A,u}$ and obtain a natural epimorphism
	\[
 	\Omega^\bullet_{d,\mathrm{max}}\twoheadrightarrow \Omega^\bullet_{\mathrm{NC}_\ell,A}.
	\]
This epimorphism is an isomorphism in degrees zero and one, but need not be in higher degrees. 
For instance, setting $\ell=1$ and $A=\bk[x,y]$, one can show that $[y,(dx)^2]$ is nonzero in
$\Omega^2_{d,\mathrm{max}}$, whereas its image in
$\Omega^2_{\mathrm{NC}_1,A}$ vanishes, since
\[
  2[y,(dx)^2]=[y,[dx,dx]_{\mathrm{gr}}]_{\mathrm{gr}}
  \in F^2_{\mathrm{gr}}\Omega^\bullet_{A,u}.
\]\end{eg}

\subsection{de Rham Functors}\label{subsec:deRham}

In this subsection, we will construct a generalized \emph{de Rham} functor in our setting. 
We recall that in classical differential geometry the de Rham functor associates to any smooth manifold $M$ its de Rham complex $\Omega^\bullet(M)$ (the classical exterior algebra), which is the cochain complex with degree $n$ consisting of the differential forms of degree $n$. 
In our setting, this will be represented by a functor $\E\to\EAlg(\V)$ for a suitable category $\E$.

\begin{defi}\label{def:de-Rham-funct}
	Let $(-)_0\colon\E\to\Mon(\V)$ be a weakly $\V$-geometrical functor satisfying the hypotheses of \Cref{thm:canonical-calc}. 
	We define the \emph{de Rham functor} $\dr_\E$ on $\E$ as the composite below, 
	\[\begin{tikzcd}[ampersand replacement=\&]
		\E \&\& {\Calc(\V)} \&\& {\EAlg(\V)}
		\arrow["{\Calc_\E}", from=1-1, to=1-3]
		\arrow["{(-)_{\tx{max}}}", from=1-3, to=1-5]
	\end{tikzcd}\]
	where $\Calc_\E$ is defined in \Cref{remark:functor-Calc_E} and $(-)_{\tx{max}}$ is the left adjoint defined in \Cref{prop:right-and-left-adj-of-EAlg-Calc}.
\end{defi}

\begin{rmk}\label{rmk:diffeos}
Let $(-)_0\colon\E\to\Mon(\V)$ be a weakly $\V$-geometrical functor satisfying the hypotheses of \Cref{thm:canonical-calc}. 
	Then by faithfulness of $(-)_0$ and functoriality of $\dr_\E$ we can regard morphisms and isomorphisms in $\E$ as generalizations of smooth maps and diffeomorphisms, respectively, in $\E$.
\end{rmk}

\begin{prop}\label{rmk:comparison-of-de-Rham-functors}
 Let $e_i\colon\E_i\to\Mon(\V)$ (for $i=1,2$) be two \suitable\ functors satisfying the hypotheses of \Cref{thm:canonical-calc}, and let $(k,\xi)\colon(\E_1,e_1)\to(\E_2,e_2)$ be as in \Cref{rmk:comparison-two-canonical}. Then $(k,\xi)$ induces a natural transformation $\mathring\kappa\colon\dr_{\E_2}\circ k\Rightarrow\dr_{\E_1}$ between the associated de Rham functors, as on the right below.
 \[\begin{tikzcd}[ampersand replacement=\&,row sep=tiny]
 	{\E_1} \&\&\& {\E_1} \\
 	\& {\Mon(\mathcal{V})} \&\&\&\& {\EAlg(\V)} \\
 	{\E_2} \&\&\& {\E_2}
 	\arrow[""{name=0, anchor=center, inner sep=0}, "{e_1}", from=1-1, to=2-2]
 	\arrow["k"', from=1-1, to=3-1]
 	\arrow[""{name=1, anchor=center, inner sep=0}, "{\dr_{\E_1}}", from=1-4, to=2-6]
 	\arrow["k"', from=1-4, to=3-4]
 	\arrow[""{name=2, anchor=center, inner sep=0}, "{e_2}"', from=3-1, to=2-2]
 	\arrow[""{name=3, anchor=center, inner sep=0}, "{\dr_{\E_2}}"', from=3-4, to=2-6]
 	\arrow["\xi"'{xshift=-0.1cm}, shift right=2, between={0.3}{0.7}, Rightarrow, from=0, to=2]
 	\arrow["{\mathring{\kappa}}"'{xshift=0.1cm,scale=1.1}, between={0.3}{0.7}, Rightarrow, from=3, to=1]
 \end{tikzcd}\]
 Moreover, this mapping is functorial.\footnote{Precisely, we mean that there is a functor from the full subcategory of $\wgeomfun$ on the weakly $\V$-geometrical functors satisfying the hypotheses of \Cref{thm:canonical-calc} to the lax slice $\operatorname{Cat}\sslash\EAlg(\V)$, whose objects are pairs $(\D,F\colon\D\to\EAlg(\V))$ and whose morphisms $(\D_1,F_1)\to(\D_2,F_2)$ are pairs $(k,\alpha)$ with $k\colon\D_1\to\D_2$ and $\alpha\colon F_2k\Rightarrow F_1$. It sends $(\E,e)$ to $(\E,\dr_\E)$ and $(k,\xi)$ to $(k,\mathring\kappa_{(k,\xi)})$.}
\end{prop}

\begin{proof}
	\Cref{rmk:comparison-two-canonical} gives a natural transformation $\widetilde{\kappa}\colon\Calc_{\E_2}k\Rightarrow\Calc_{\E_1}$, and we set $\mathring\kappa:=(-)_{\tx{max}}\widetilde\kappa$.
	\[\begin{tikzcd}[ampersand replacement=\&,row sep=tiny]
		{\E_1} \\
		\&\& {\Calc(\V)} \&\& {\EAlg(\V)} \\
		{\E_2}
		\arrow[""{name=0, anchor=center, inner sep=0}, "{\Calc_{\E_1}}"{description}, from=1-1, to=2-3]
		\arrow[""{name=1, anchor=center, inner sep=0}, "{\dr_{\E_1}}", curve={height=-12pt}, from=1-1, to=2-5]
		\arrow["k"', from=1-1, to=3-1]
		\arrow["{(-)_{\tx{max}}}"{description}, from=2-3, to=2-5]
		\arrow[""{name=2, anchor=center, inner sep=0}, "{\Calc_{\E_2}}"{description}, from=3-1, to=2-3]
		\arrow[""{name=3, anchor=center, inner sep=0}, "{\dr_{\E_2}}"', curve={height=12pt}, from=3-1, to=2-5]
		\arrow["{\widetilde{\kappa}}"'{xshift=0.1cm,scale=1.1}, between={0.3}{0.7}, Rightarrow, from=2, to=0]
		\arrow["{:=}"{marking, allow upside down,scale=0.8}, draw=none, from=1, to=2-3]
		\arrow["{:=}"{marking, allow upside down,scale=0.8}, draw=none, from=3, to=2-3]
	\end{tikzcd}\]
	Functoriality follows from that of $\widetilde\kappa$, since $(-)_{\tx{max}}$ is a functor. Alternatively, by \Cref{unique-extension-ealg} the functor $\EAlg(\V)\to\Mon(\V)$, $\Omega^\bullet\mapsto\Omega^0$, is faithful, so the required identities can be checked on monoid components, as in the proof of \Cref{rmk:comparison-two-canonical}.
\end{proof}

The proposition above further motivates the definition of morphisms in the categories $\wgeomfun$ and $\geomfun$ (see \Cref{def:cat-of-geom-fun}). 

In the setting above, we can define the cohomology functor
$$ H^\bullet\colon \Ch(\V)\longrightarrow \tx{Gr}(\V)$$
into the category of graded objects in $\V$. Indeed, given a cochain complex $A^\bullet$ in $\V$ and $n\geq0$ we can first take the kernel od the differential at level $n$
$$ \tx{Ker}(d_{n})\rightarrowtail A^{n}\xrightarrow{\ d^n} A^{n+1} $$
and then define $H^n(A^\bullet)$ as the cokernel
$$ A^{n-1}\xrightarrow{\tilde d^{n-1}} \tx{Ker}(d_{n}) \twoheadrightarrow \tx{Coker}(\tilde d^{n-1})=:H^n(A^\bullet) $$
where $\tilde d^{n-1}$ is induced by the universal property of the kernel since $d^{n}\circ d^{n-1}=0$. For $n=0$ we set $A^{-1}=0$, so that $H^0( A^\bullet)=\tx{Ker}(d^0)$.
It follows that we can define 
$$H^\bullet(A^\bullet):= (H^n(A^\bullet))_{n\geq 0};$$
the fact that this action is functorial is an easy consequence of the universal property of kernels and cokernels.

\begin{defi}\label{rmk:deRhamCohomology}
Given $(-)_0\colon\E\to\Mon(\V)$ as in \Cref{def:de-Rham-funct}, we can consider the composite, which we denote $H^\bullet_{\dr_\E}$ and term the \emph{de Rham cohomology functor on $(-)_0\colon\E\to\Mon(\V)$ }
$$ \E\xrightarrow{ \Calc_\E(-) } \Calc(\V)\xrightarrow{(-)_{\tx{max}}} \EAlg(\V)\xrightarrow{\ J\ } \Ch(\V)\xrightarrow{ H^\bullet } \tx{Gr}(\V), $$
where $J$ is the forgetful functor. 
\end{defi}

\begin{rmk}\label{rmk:comparison-of-de-Rham-cohomology-functors}
Given two $e_i\colon\E_i\to\VMon$ (for $i=1,2$) satisfying the hypotheses of \Cref{rmk:comparison-of-de-Rham-functors}, using \Cref{rmk:comparison-two-canonical} and \Cref{rmk:comparison-of-de-Rham-functors}, we obtain a natural transformation  $H^\bullet(J\mathring\kappa)\colon H^\bullet_{\dr_{\E_2}}\circ k\Rightarrow H^\bullet_{\dr_{\E_1}}$ between their corresponding de Rham cohomology functors. %
 
\end{rmk}

\begin{eg}
\label{ex:smooth_algebraic_de_rham}
It follows from \Cref{ex:smooth_algebraic_comparison}  and \Cref{rmk:comparison-of-de-Rham-cohomology-functors} that the corestriction $k$ to $\CMon(\mathbb R\tx{-Vect})$ of the functor $(-)_0\colon C^\infty\tx{-Ring}\to \Mon(\mathbb R\tx{-Vect})$ (described in \Cref{eg:Cinfty-rings-pullback}) induces a natural  transformation
$$ H^\bullet_{\dr_{\CMon(\mathbb R\tx{-Vect})}}\circ k\Longrightarrow H^\bullet_{\dr_{C^\infty\tx{-Ring}}}, $$
from the algebraic de Rham cohomology  of the underlying $\mathbb R$-algebra of a $C^\infty$-ring to its smooth de Rham cohomology.
\end{eg}

\begin{eg}
\label{ex:cmon_de_rham}
\Cref{ex:cmon_comparison} and \Cref{rmk:comm-mon-full-subcat} imply that for $q\in\mathbb{Z}$ the  factorization of the inclusion $\CMon_n(\V)\hookrightarrow \Mon(\V)$ through $\CMon_{qn}(\V)$ induces a natural  transformation as in \Cref{rmk:comparison-of-de-Rham-cohomology-functors}  , from the de Rham cohomology  functor arising from $\CMon_{qn}(\V)\hookrightarrow\Mon(\V)$  , restricted to $\CMon_n(\V)$, to the one arising from $\CMon_n(\V)\hookrightarrow \Mon(\V)$.
\end{eg}

\bibliography{biblio}
\bibliographystyle{plain}

\end{document}